\documentclass[12pt,a4paper,dvipdfm]{amsbook}

\usepackage{amssymb}
\usepackage[abbrev]{amsrefs}

\theoremstyle{plain}
  \newtheorem{thm}{Theorem}[chapter]
  \newtheorem{lem}[thm]{Lemma}
  
  \newtheorem{cor}[thm]{Corollary}
  \newtheorem{prop}[thm]{Proposition}
  \newtheorem{clm}[thm]{Claim}

\theoremstyle{definition}
  \newtheorem{defn}[thm]{Definition}
  
  \newtheorem{ex}[thm]{Example}

  \newtheorem{asmp}[thm]{Assumption}

\theoremstyle{remark}
  \newtheorem{rem}[thm]{Remark}

  \newtheorem*{ack}{Acknowledgment}

\numberwithin{equation}{chapter}

\DeclareMathOperator{\vol}{vol}

\DeclareMathOperator{\diam}{diam}

\DeclareMathOperator{\supp}{supp}

\DeclareMathOperator{\grad}{grad}

\DeclareMathOperator{\Ric}{Ric}

\DeclareMathOperator{\lm}{lm}

\DeclareMathOperator{\proj}{proj}
\DeclareMathOperator{\Ent}{Ent}
\DeclareMathOperator{\CD}{CD}
\DeclareMathOperator{\per}{per}

\newcommand{\ud}{\underline{d}}

\newcommand{\field}[1]{\mathbb{#1}}

\newcommand{\C}{\field{C}}
\newcommand{\R}{\field{R}}

\newcommand{\N}{\field{N}}

\DeclareMathOperator{\defi}{def}
\DeclareMathOperator{\dKF}{{\it d_{KF}}}
\DeclareMathOperator{\pr}{pr}
\DeclareMathOperator{\id}{id}
\DeclareMathOperator{\dconc}{{\it d}_{{\rm conc}}}
\newcommand{\Lip}{\mathcal{L}{\it ip}}
\DeclareMathOperator{\ObsDiam}{ObsDiam}
\DeclareMathOperator{\LeRad}{LeRad}
\DeclareMathOperator{\Sep}{Sep}
\DeclareMathOperator{\Avr}{Avr}
\DeclareMathOperator{\Dir}{Dir}
\renewcommand{\Cap}{\mathop{\mbox{\rm Cap}}\nolimits}
\DeclareMathOperator{\Cov}{Cov}
\DeclareMathOperator{\Exp}{Exp}

\newcommand{\cA}{\mathcal{A}}
\newcommand{\cB}{\mathcal{B}}
\newcommand{\cC}{\mathcal{C}}
\newcommand{\cE}{\mathcal{E}}
\newcommand{\cF}{\mathcal{F}}
\newcommand{\cL}{\mathcal{L}}
\newcommand{\cM}{\mathcal{M}}
\newcommand{\cN}{\mathcal{N}}
\newcommand{\cP}{\mathcal{P}}
\newcommand{\cT}{\mathcal{T}}
\newcommand{\cX}{\mathcal{X}}
\newcommand{\cY}{\mathcal{Y}}

\newcommand{\Lo}{\mathcal{L}_1}

\makeindex

\title[Metric Measure Geometry]
{Metric Measure Geometry\\
  {\rm\large --Gromov's Theory of Convergence and Concentration\\ of Metrics and Measures--\\
\today}}
\author{Takashi Shioya}
\address{Mathematical Institute, Tohoku University, Sendai 980-8578, Japan}
\email{shioya@math.tohoku.ac.jp}
\date{\today}

\begin{document}

\maketitle
\tableofcontents

\chapter*{Introduction}

In this book, we study Gromov's metric geometric theory
\cite{Gromov}*{\S 3$\frac{1}{2}$}
on the space of metric measure spaces,
based on the idea of concentration of measure phenomenon
due to L\'evy and Milman.
Although most of the details are omitted in the original
article \cite{Gromov}*{\S 3$\frac{1}{2}$},
we present complete and detailed proofs for some main parts,
in which we prove several claims
that are not mentioned in any literature.
We also discuss concentration with a lower bound of curvature,
which is originally studied in \cite{FS}.

The concentration of measure phenomenon was first discovered by P.~L\'evy
\cite{Levy}
and further put forward by V.~Milman \cite{Mil:Dvoretzky,Mil:heritage}.
It has many applications in various areas, such as, geometry, analysis,
probability theory, and discrete mathematics
(see \cite{Ldx:book} and the references therein).
The phenomenon is stated as that
any $1$-Lipschitz continuous function is close to a constant
on a domain with almost full measure,
which is often observed for high-dimensional spaces.
As a most fundamental example, we observe it
in the high-dimensional unit spheres $S^n(1) \subset \R^{n+1}$,
i.e., any $1$-Lipschitz continuous function on $S^n(1)$ is close to a constant
on a domain with almost full measure if $n$ is large enough.
In general, it is described for a sequence of
metric measure spaces.
In this book, we assume a metric measure space, an \emph{mm-space} for short,
to be a triple $(X,d_X,\mu_X)$, where
$(X,d_X)$ is a complete separable metric space
and $\mu_X$ a Borel
probability\footnote{In \cite{Gromov}*{\S 3$\frac{1}{2}$},
the measures of mm-spaces are not necessarily probability.
However, all our proofs easily extend to the case of non-probability
mm-spaces.} measure on $X$.
A sequence of mm-spaces $X_n$, $n=1,2,\dots$, is called a \emph{L\'evy family}
if
\[
\lim_{n\to\infty} \inf_{c\in\R} \mu_{X_n}(|f_n-c| > \varepsilon) = 0
\]
for any sequence of $1$-Lipschitz continuous functions $f_n : X_n \to \R$,
$n=1,2,\dots$, and for any $\varepsilon > 0$.
The sequence of the unit spheres $S^n(1)$, $n=1,2,\dots$, is a L\'evy family,
where the measure on $S^n(1)$ is taken to be the Riemannian volume measure
normalized as the total measure to be one.

One of central themes in this book is the study of the observable distance.
The \emph{observable distance $\dconc(X,Y)$ between
two mm-spaces $X$ and $Y$}
is, roughly speaking, the difference between $1$-Lipschitz functions on $X$
and those on $Y$ (see Definition \ref{defn:obs-dist} for the precise
definition).
A sequence of mm-spaces is a L\'evy family if and only if
it $\dconc$-converges to a one-point mm-space, where we note that
any $1$-Lipschitz function on a one-point mm-space is constant.
Thus, $\dconc$-convergence of mm-spaces
can be considered as a generalization of the L\'evy property.
We call $\dconc$-convergence of mm-spaces \emph{concentration of mm-spaces}.
A typical example of a concentration $X_n\to Y$ is obtained
by a fibration
\[
F_n \to X_n \to Y
\]
such that $\{F_n\}_{n=1}^\infty$ is a L\'evy family,
which example makes us to notice that
concentration of mm-spaces is an analogue of collapsing
of Riemannian manifolds.  Concentration is strictly weaker than
measured Gromov-Hausdorff convergence
and is more suitable for the study of a sequence of manifolds whose dimensions
are unbounded.

Although $\dconc$ is not easy to investigate,
we have a more elementary distance, called the \emph{box distance},
between mm-spaces.
The box distance function is fit for well-known measured Gromov-Hausdorff
convergence of mm-spaces (see Remark \ref{rem:box-mGH}).
Concentration of mm-spaces is rephrased as convergence of associated pyramids
using the box distance function,
where a \emph{pyramid} is a family of mm-spaces
that forms a directed set with respect to some natural order relation
between mm-spaces, called the \emph{Lipschitz order}
(see Definitions \ref{defn:dom} and \ref{defn:pyramid}).
We have a metric $\rho$ on the set of pyramids, say $\Pi$,
induced from the box distance function (see Definition \ref{defn:metric-Pi}
and \cite{Shioya:mmlim}).
Each mm-space $X$ is associated with the pyramid, say $\cP_X$,
consisting of all descendants of
the mm-space (i.e., smaller mm-spaces with respect to the Lipschitz order).
Denote the set of mm-spaces by $\cX$.
We prove that the map
\[
\iota : \cX \ni X \longmapsto \cP_X \in \Pi
\]
is a $1$-Lipschitz continuous topological embedding map with respect to
$\dconc$ and $\rho$.
This means that concentration of mm-spaces is expressed
only by the box distance function,
since $\rho$ is induced from the box distance function.
We also prove that $\Pi$ is a compactification of $\cX$ with $\dconc$.
Such a concrete compactification is far more valuable
than just an abstract one.

It is also interesting to study a sequence of mm-spaces
that $\dconc$-diverges but have proper asymptotic behavior.
A sequence of mm-spaces $X_n$, $n=1,2,\dots$, is said to be \emph{asymptotic}
if the associated pyramid $\cP_{X_n}$ converges in $\Pi$.
We say that a sequence of mm-spaces \emph{asymptotically concentrates}
if it is a $\dconc$-Cauchy sequence.
Any asymptotically concentrating sequence of mm-spaces
is asymptotic.
For example, the sequence of the Riemannian product spaces
\[
S^1(1) \times S^2(1) \times \dots \times S^n(1),
\quad n=1,2,\dots,
\]
$\dconc$-diverges and asymptotically concentrates
(see Example \ref{ex:prod-sph}).
The sequence of the spheres $S^n(\sqrt{n})$ of radius $\sqrt{n}$,
$n=1,2,\dots$, does not even asymptotically concentrate
but is asymptotic (see Theorem \ref{thm:sphere-Gaussian},
Corollary \ref{cor:sphere-Gaussian}, and \cite{Shioya:mmlim}).
One of main theorems in this book states that the map $\iota : \cX \to \Pi$
extends to the $\dconc$-completion of $\cX$, so that
the space $\Pi$ of pyramids is also a compactification of
the $\dconc$-completion of the space $\cX$ of mm-spaces
(see Theorem \ref{thm:emb-pyramid}).
Let $\gamma^n$ denote the standard Gaussian measure on $\R^n$.
Then, the associated pyramids $\cP_{S^n(\sqrt{n})}$ and $\cP_{(\R^n,\gamma^n)}$
both converge to a common pyramid as $n\to\infty$
(see Theorem \ref{thm:sphere-Gaussian} and \cite{Shioya:mmlim}),
which can be thought as a generalization of
the Maxwell-Boltzmann distribution law
(or the Poincar\'e limit theorem).

The spectral property is deeply related with
the asymptotic behavior of a sequence of mm-spaces.
The \emph{spectral compactness} of a family of mm-spaces
is defined by the Gromov-Hausdorff compactness
of the energy sublevel sets of $L_2$ functions
(see Definition \ref{defn:spec-cpt})
and is closely related with the notion of
asymptotic compactness of Dirichlet energy forms
(see \cite{KS}).
For a family of compact Riemannian manifolds,
it is equivalent to the discreteness of the limit set of the spectrums
of the Laplacians of the manifolds (see Proposition \ref{prop:spec-cpt}).
(In this book, manifolds may have nonempty boundary.)
We prove that any spectrally compact and asymptotic sequence of mm-spaces
is asymptotically concentrates if the observable diameter
is bounded from above (see Theorem \ref{thm:spec-conc}).
We say that a sequence of mm-spaces \emph{spectrally concentrates}
if it spectrally compact and asymptotically concentrates.
For example, let
\[
X_n := F_1 \times F_2 \times \dots \times F_n
\]
be the Riemannian product of compact Riemannian manifolds $F_n$, $n=1,2,\dots$.
If $\lambda_1(F_n)$ diverges to infinity as $n\to\infty$,
then $\{X_n\}$ spectrally concentrates (see Corollary \ref{cor:prod-spec-conc}).

There is a notion of dissipation for a sequence of mm-spaces,
which is opposite to concentration and means that
the mm-spaces disperse into many small pieces far apart each other.
A sequence of mm-spaces \emph{$\delta$-dissipates}, $\delta > 0$,
if and only if
any limit of the associated pyramids contains
all mm-spaces with diameter $\le \delta$.
The sequence \emph{infinitely dissipates} if and only if
the associated pyramid converges to the space of mm-spaces
(see Proposition \ref{prop:dissipate}).
On one hand, for a disconnected mm-space $F$,
the sequence of the $n^{th}$ power product spaces $F^n$, $n=1,2,\dots$,
with $l_\infty$ metric $\delta$-dissipates for some $\delta > 0$
(see Proposition \ref{prop:disconn-dissipation}).
On the other hand, the non-dissipation theorem
(Theorem \ref{thm:non-dissipation}) states that
the sequence $\{F^n\}$ does not $\delta$-dissipate for any $\delta > 0$
if $F$ is connected and locally connected.
The proof of the non-dissipation theorem relies on the study
of the obstruction condition for dissipation.
For example, a sequence of compact Riemannian manifolds $X_n$
does not dissipate if $\lambda_1(X_n)$ is bounded away from zero
(see Corollary \ref{cor:prod-diss}),
which is one of essential statements in the proof of
the non-dissipation theorem.

It is interesting to study the relation between curvature
and concentration.
The concept of Ricci curvature bounded below is generalized
to the \emph{curvature-dimension condition} for an mm-space
by Lott-Villani-Sturm \cites{LV,Sturm:geoI,Sturm:geoII}
via the optimal mass-transport theory.
We prove that if a sequence of mm-spaces satisfying
the curvature-dimension condition concentrates to an mm-space,
then the limit also satisfies the curvature-dimension condition
(see \cite{FS}).
This stability result of the curvature-dimension condition
has an important application to the eigenvalues of Laplacian
on Riemannian manifolds.
In fact, under the nonnegativity of Ricci curvature,
the $k^{th}$ eigenvalue of the Laplacian
of a closed Riemannian manifold
is dominated by a constant multiple of
the first eigenvalue, where the constant depends only on
$k$ and is independent of the dimension of the manifold.
This dimension-free estimate cannot be obtained
by the ordinary technique.
Combining this estimate with Gromov-V.~Milman's and
E.~Milman's results \cites{GroMil,Emil:isop,Emil:role},
we have the equivalence:
\begin{align*}
  &\text{$\{X_n\}$ is a L\'evy family}\\
  \Longleftrightarrow\  &\lambda_1(X_n) \to +\infty\\
  \Longleftrightarrow\  &\lambda_k(X_n) \to +\infty
  \ \text{for some $k$}
\end{align*}
for a sequence of closed Riemannian manifolds $X_n$, $n=1,2,\dots$,
with nonnegative Ricci curvature.

The organization of this book is as follows.

In Chapter \ref{chap:conv-meas}, we define
weak and vague convergence of measures,
the Prohorov distance, transportation, the Ky Fan metric,
convergence in measure of maps,
and present those basic facts.

Chapter \ref{chap:LM-conc} is devoted to
a minimal introduction to L\'evy-Milman concentration phenomenon.
We define the observable diameter, the separation distance,
and the Lipschitz order.
We prove the normal law \'a la L\'evy for $S^n(\sqrt{n})$ stating that
any limit of the push-forward of the normalized volume measure
on $S^n(\sqrt{n})$ by a $1$-Lipschitz continuous function on $S^n(\sqrt{n})$
as $n\to\infty$
is the push-forward of the $1$-dimensional standard Gaussian measure
by some $1$-Lipschitz continuous function on $\R$.
From this we derive the asymptotic estimate of the observable diameter
of $S^n(1)$ and $\C P^n$.  We also prove the relation between
the $k^{th}$ eigenvalue of the Laplacian and the separation distance
for a compact Riemannian manifold,
which yields some examples of L\'evy families.
Most of the contents in this chapter are already known for specialists.

Chapter \ref{chap:GH-dist-matrix} presents some basic facts
on metric geometry, such as,
the Hausdorff distance and the Gromov-Hausdorff distance.
We also prove the equivalence between the Gromov-Hausdorff convergence 
and the convergence of the distance matrices of compact metric spaces.

Chapter \ref{chap:box-dist} deals with the box distance between
mm-spaces, which is one of fundamental tools in this book.
We prove that the Lipschitz order is stable under box convergence,
and that any mm-space can be approximated by a monotone nondecreasing
sequence of finite-dimensional mm-spaces.
We investigate the convergence of finite product spaces to the infinite product.

Chapter \ref{chap:obs-dist-measurement}
discusses the observable distance and the measurements,
where the \emph{$N$-measurement of an mm-space} is defined to be
the set of push-forwards of the measure of the mm-space
by $1$-Lipschitz maps to $\R^N$ with $l_\infty$ norm.
The measurements have the complete information of the mm-space
and can be treated easier than the mm-space itself.
We prove that the concentration of mm-spaces
is equivalent to
the convergence of the corresponding measurements,
which is one of the essential points for
the investigation of convergence of pyramids.

Chapter \ref{chap:pyramid} is devoted to the space of pyramids.
We define a metric on the space of pyramids
and prove its compactness.
The metric is originally due to \cite{Shioya:mmlim}.

In Chapter \ref{chap:asymp-conc},
we finally complete the proof of the theorem that the $\dconc$-completion of
the space of mm-spaces is embedded into the space of pyramids,
which is one of main theorems in this book.
We study the asymptotic concentration of finite product spaces
and the asymptotic property of
the pyramids $\cP_{S^n(\sqrt{n})}$ and $\cP_{(\R^n,\gamma^n)}$ (see \cite{Shioya:mmlim}).
We also study spectral compactness
and prove that any spectrally compact and asymptotic sequence of mm-spaces
asymptotically concentrates if the observable diameter is bounded from above.

Chapter \ref{chap:dissipation} discusses dissipation.
After the basics of dissipation, we present some examples of dissipation.
One of the interesting examples is the sequence of the spheres $S^n(r_n)$
of radius $r_n$.
It infinitely dissipates if only if
$r_n/\sqrt{n} \to +\infty$ as $n\to\infty$ (see \cite{Shioya:mmlim}).
We also study some obstruction for dissipation
and prove the non-dissipation theorem.

The final Chapter \ref{chap:curv-conc} is an exposition of \cite{FS}.
We prove the stability theorem of the curvature-dimension condition
for concentration, and apply it to the study of the eigenvalues of
Laplacian on closed Riemannian manifolds.
We also prove the stability of a lower bound of Alexandrov curvature.

\begin{ack}
  The author would like to thank Prof.~Mikhail Gromov and Prof.~Vitali Milman
  for their comments and encouragement.
  He also thanks to Prof.~Asuka Takatsu,
  Mr.~Takuya Higashi, Mr.~Daisuke Kazukawa, Ms.~Yumi Kume,
  and Mr.~Hirotaka Nakajima
  for checking a draft version of the manuscript.
  Thanks to Prof.~Atsushi Katsuda, Prof.~Takefumi Kondo, and
  Mr.~Ryunosuke Ozawa for
  valuable discussions and comments.
\end{ack}

\chapter{Preliminaries from measure theory}
\label{chap:conv-meas}

\section{Some basics}

In this section, we enumerate some basic definitions and
facts on measure theory.
We refer to \cites{Bog,Bil,Kechris} for more details.

\index{measure}
A \emph{measure} $\mu$ on a set $X$ is a nonnegative countably additive
set function on a $\sigma$-algebra over $X$.
\index{measure space}
We call a pair $(X,\mu)$ of a set $X$ and a measure $\mu$ on $X$
a \emph{measure space}.
\index{Borel measure}
A measure on a topological space $X$ is called
a \emph{Borel measure} if it is defined on the Borel $\sigma$-algebra
over $X$.
A measure $\mu$ on a set $X$ is said to be \emph{finite} if $\mu(X) < +\infty$.
A \emph{probability measure $\mu$ on $X$} is defined to be
a measure on $X$ with $\mu(X) = 1$.
\index{finite measure} \index{probability measure}

A map $f : X \to Y$ from a measure space $(X,\mu)$ to a topological space $Y$
is said to be ($\mu$-)\emph{measurable}
if for any Borel subset $A \subset Y$ the preimage $f^{-1}(A)$ belongs to
the associated $\sigma$-algebra of $(X,\mu)$.
A map $f : X \to Y$ between two topological spaces $X$ and $Y$
is \emph{Borel measurable}
if for any Borel subset $A \subset Y$ the preimage $f^{-1}(A)$
is a Borel subset of $X$.

\begin{defn}[Inner and outer regular measure]
  \index{inner regular} \index{tight}
  Let $\mu$ be a Borel measure on a topological space $X$.
  $\mu$ is said to be \emph{inner regular} (or \emph{tight})
  if for any Borel subset $A \subset X$ and for any real number
  $\varepsilon > 0$,
  there exists a compact set $K$ contained in $A$ such that
  $\mu(A) \le \mu(K)+\varepsilon$.
  \index{outer regular}
  $\mu$ is said to be \emph{outer regular} if for any Borel subset
  $A \subset X$ and for any $\varepsilon > 0$, there exists
  an open set $U$ containing $A$
  such that $\mu(A) \ge \mu(U) - \varepsilon$.
\end{defn}

\begin{thm}
  Any finite Borel measure on a complete separable metric space
  is inner and outer regular.
\end{thm}

\begin{defn}[Absolute continuity]
  \index{absolutely continuous}
  Let $\mu$ and $\nu$ be two Borel measures on a topological space $X$.
  $\mu$ is said to be \emph{absolutely continuous with respect to $\nu$}
  if $\mu(A) = 0$ for any Borel subset $A \subset X$ with $\nu(A) = 0$.
\end{defn}

\begin{thm}[Radon-Nikodym theorem] \label{thm:RN}
  \index{Radon-Nikodym theorem}
  Let $X$ be a topological space.
  If a Borel measure $\mu$ on $X$ is absolutely continuous with respect to
  a Borel measure $\nu$ on $X$, then there exists a Borel measurable function
  $f : X \to [\,0,+\infty\,)$ such that
  \[
  \mu(A) = \int_A f \,d\nu
  \]
  for any Borel subset $A \subset X$.
  Moreover, $f$ is unique $\nu$-a.e.
\end{thm}

\begin{defn}[Radon-Nikodym derivative]
  \index{Radon-Nikodym derivative} \index{density}
  The function $f$ in Theorem \ref{thm:RN} is called
  the \emph{Radon-Nikodym derivative}
  (or \emph{density}) \emph{of $\mu$ with respect to $\nu$}
  and is denoted by
  \[
  \frac{d\mu}{d\nu}.
  \]
  \index{dmudnu@$\frac{d\mu}{d\nu}$}
\end{defn}

\begin{defn}[Push-forward]
  Let $p : X \to Y$ be a measurable map from
  a measure space $(X,\mu)$ to a topological space $Y$.
  We define a Borel measure $p_*\mu$ on $Y$ by
  \[
  p_*\mu(A) := \mu(p^{-1}(A))
  \]
  for any Borel subset $A \subset Y$.
  We call $p_*\mu$ the \emph{push-forward of $\mu$ by the map $p$}.
\end{defn}

\begin{thm}[Disintegration theorem] \label{thm:disintegration}
  \index{disintegration theorem}
  Let $p : X \to Y$ be a Borel measurable map between
  two complete separable metric spaces $X$ and $Y$.
  Then, for any finite Borel measure $\mu$ on $X$,
  there exists a 
  family of probability measures $\mu_y$, $y\in Y$, on $X$ such that
  \begin{enumerate}
  \item the map $Y \ni y \mapsto \mu_y$ is Borel measurable, i.e.,
    $Y \ni y \mapsto \mu_y(A)$ is a Borel measurable function for
    any Borel subset $A \subset X$,
  \item $\mu_y(X \setminus p^{-1}(y)) = 0$ for $p_*\mu$-a.e.~$y \in Y$,
  \item for any Borel measurable function $f : X \to [\,0,+\infty\,)$,
    \[
    \int_X f(x) \, d\mu(x)
    = \int_Y \int_{p^{-1}(y)} f(x) \, d\mu_y(x) d(p_*\mu)(y).
    \]
  \end{enumerate}
  Moreover, $\{\mu_y\}_{y\in Y}$ is unique $p_*\mu$-a.e.
\end{thm}

\begin{defn}[Disintegration]
  \index{disintegration}
  The family $\{\mu_y\}_{y\in Y}$ as in Theorem \ref{thm:disintegration}
  is called the \emph{disintegration of $\mu$ for $p : X \to Y$}.
\end{defn}

\begin{defn}[Median and L\'evy mean]
  \index{median} \index{Levy mean@L\'evy mean}
  Let $X$ be a measure space with probability measure $\mu$
  and $f : X \to \R$ a measurable function.
  A real number $m$ is called a \emph{median of $f$}
  if it satisfies
  \begin{align*}
    \mu(f\geq m) \geq \frac{1}{2} \quad \text{and}\quad
    \mu(f\leq m) \geq \frac{1}{2},
  \end{align*}
  where $\mu(P)$ for a conditional formula $P$ denotes the $\mu$-measure
  of the set of points where $P$ holds.
  It is easy to see that the set of medians of $f$
  is a closed and bounded interval.
  The \emph{L\'evy mean of $f$ with respect to the measure $\mu$}
  is defined to be \index{mf@$m_f$} \index{lm(f;muX)@$\lm(f;\mu)$}
  \[
  m_f := \lm(f;\mu) := \frac{a_f + b_f}{2},
  \]
  where $a_f$ is the minimum of medians of $f$
  and $b_f$ the maximum of medians of $f$.
\end{defn}

\section{Convergence of measures}
\label{sec:obs-sphere}

For a while, let $X$ be a metric space with metric $d_X$.

\begin{defn}[Weak and vague convergence of measures]
  \index{weak convergence} \index{converge weakly}
  Let $\mu$ and $\mu_n$, $n=1,2,\dots$, be finite Borel measures on $X$.
  We say that $\mu_n$ \emph{converges weakly} to $\mu$
  and write \emph{$\mu_n \to \mu$ weakly} as $n\to\infty$ if
  \begin{equation}
    \label{eq:weak}
      \lim_{n\to\infty} \int_X f \; d\mu_n = \int_X f \; d\mu
  \end{equation}
  for any bounded and continuous function $f : X \to \R$.
  \index{vague convergence} \index{converge vaguely}
  We say that $\mu_n$ \emph{converges vaguely} to $\mu$
  and write \emph{$\mu_n \to \mu$ vaguely} $n\to\infty$ if
  \eqref{eq:weak} holds for any continuous function $f : X \to \R$
  with compact support.
\end{defn}

Any weakly convergent sequence is vaguely convergent,
but the converse is not necessarily true.
For example, the sequence of Dirac's delta measures $\delta_n$,
$n=1,2,\dots$, converges vaguely to the zero measure on $\R$,
but it does not converge weakly, where \emph{Dirac's delta measure $\delta_x$
at a point $x$ in a space $X$} is defined by
\index{Dirac's delta measure} \index{deltax@$\delta_x$}
\[
\delta_x(A) :=
\begin{cases}
  1 & \text{if $x \in A$},\\
  0 & \text{if $x \notin A$}
\end{cases}
\]
for $A \subset X$.
For any weakly convergent sequence $\mu_n \to \mu$,
the total measure $\mu_n(X)$ converges to $\mu(X)$.

\begin{lem}
  Assume that a sequence of finite Borel measures $\mu_n$, $n=1,2,\dots$,
  converges weakly {\rm(}resp.~vaguely{\rm)}
  to a finite Borel measure $\mu$ on a metric space
  {\rm(}resp.~locally compact metric space{\rm)} $X$.
  Then, for any Borel {\rm(}resp.~relatively compact Borel{\rm)}
  subset $A \subset X$, we have
  \[
  \mu(A^\circ) \le
  \liminf_{n\to\infty} \mu_n(A) \le \limsup_{n\to\infty} \mu_n(A) \le \mu(\bar{A}),
  \]
  where $A^\circ$ and $\bar{A}$ denote the interior and the closure
  of $A$, respectively.
\end{lem}

\begin{defn}[Prohorov distance]
  \index{Prohorov distance}
  The \emph{Prohorov distance $d_P(\mu,\nu)$
    between two Borel probability measures $\mu$ and $\nu$ on $X$}
  is defined to be the infimum of $\varepsilon > 0$ satisfying
  \begin{equation} \label{eq:Proh}
    \mu(U_\varepsilon(A)) \ge \nu(A) - \varepsilon
  \end{equation}
  for any Borel subset $A \subset X$, where
  \[
  U_\varepsilon(A) := \{\; x \in X \mid d_X(x,A) < \varepsilon\;\}.
  \]
  \index{UepsilonA@$U_\varepsilon(A)$}
  The distance function $d_P$ is called the \emph{Prohorov metric}.
\end{defn}

We have $d_P(\mu,\nu) \le 1$ for any two Borel probability measures
$\mu$ and $\nu$ on $X$.

Note that \eqref{eq:Proh} holds
for any Borel subset $A \subset X$ if and only if
\begin{equation}
  \label{eq:Proh2}
  \nu(U_\varepsilon(A)) \ge \mu(A) - \varepsilon
\end{equation}
for any Borel subset $A \subset X$.
In fact, since $U_\varepsilon(X \setminus U_\varepsilon(A)) \subset X \setminus A$, 
\eqref{eq:Proh} for $X \setminus U_\varepsilon(A)$ yields
\[
\mu(X \setminus A) \ge \mu(U_\varepsilon(X \setminus U_\varepsilon(A)))
\ge \nu(X \setminus U_\varepsilon(A)) - \varepsilon,
\]
which implies \eqref{eq:Proh2}.
In particular, we have $d_P(\nu,\mu) = d_P(\mu,\nu)$.

\begin{prop}[cf.~\cite{Bil}*{\S 6}]
  Let $X$ be a metric space.
  The Prohorov metric $d_P$ is a metric on the set of Borel probability
  measures on $X$.
\end{prop}

The following is sometimes useful.

\begin{lem}
  For any two Borel probability measures $\mu$ and $\nu$ on $X$,
  we have
  \begin{align*}
    d_P(\mu,\nu) = \inf\{\;\varepsilon \ge 0 &\mid
    \mu(B_\varepsilon(A)) \ge \nu(A) - \varepsilon\\
    &\quad\text{for any Borel subset $A \subset X$}\;\},
  \end{align*}
  where
  \[
  B_\varepsilon(A) := \{\;x \in X \mid d_X(x,A) \le \varepsilon\;\}.
  \]
  \index{BepsilonA@$B_\varepsilon(A)$}
\end{lem}

\begin{proof}
  Since $U_\varepsilon(A) \subset B_\varepsilon(A)$,
  the right-hand side is not greater than $d_P(\mu,\nu)$.

  If $\mu(B_\varepsilon(A)) \ge \nu(A) - \varepsilon$
  for a real number $\varepsilon > 0$, then
  $\mu(U_{\varepsilon'}(A)) \ge \nu(A) - \varepsilon'$
  for any $\varepsilon'$ with $\varepsilon' > \varepsilon$.
  This proves that the right-hand side is not less than $d_P(\mu,\nu)$.
\end{proof}

\begin{lem}[cf.~\cite{Bil}*{\S 5--6}] \label{lem:conv-meas}
  \begin{enumerate}
  \item If $X$ is separable, then we have
    \[
    \mu_n \to \mu \ \text{weakly}\ \Longleftrightarrow\ d_P(\mu_n,\mu) \to 0
    \]
    for any Borel probability measures $\mu$ and $\mu_n$, $n=1,2,\dots$,
    on $X$.
  \item
    If $X$ is separable {\rm(}resp.~separable and complete{\rm)},
    then the set of Borel probability measures on $X$
    is separable {\rm(}resp.~separable and complete{\rm)} with respect to $d_P$.
  \item\label{it:conv-meas-w}
    If $X$ is compact, then any sequence of Borel measures $\mu_n$,
    $n=1,2,\dots$, on $X$ with $\sup_n\mu_n(X) < +\infty$
    has a weakly convergent subsequence, and
    in particular, the set of Borel probability measures
    on $X$ is $d_P$-compact.
  \item\label{it:conv-meas-v}
    If $X$ is proper, then any sequence of Borel measures $\mu_n$,
    $n=1,2,\dots$, on $X$ with $\sup_n\mu_n(X) < +\infty$
    has a vaguely convergent subsequence,
    where $X$ is said to be \emph{proper} if any bounded subset of $X$
    is relatively compact. \index{proper metric space}
  \end{enumerate}
\end{lem}

\begin{proof}
  We refer to \cite{Bil}*{\S 5--6} for the proof of (1)--\eqref{it:conv-meas-w}.

  We give the proof of \eqref{it:conv-meas-v}.
  Let $\mu_n$, $n=1,2,\dots$, be Borel measures on a proper metric space $X$
  with $\sup_n\mu_n(X) < +\infty$.
  By \cite{Kechris}*{(5.3)}, the one-point compactification of $X$,
  say $\hat{X}$, is metrizable.
  Applying \eqref{it:conv-meas-w} yields that there exists
  a weakly convergent subsequence of $\{\mu_n\}$ on $\hat{X}$,
  which is a vaguely convergent sequence on $X$ if each $\mu_n$
  is restricted on $X$.
  This completes the proof.
\end{proof}

\begin{defn}[Tight]
  \index{tight}
  Let $\cM$ be a family of Borel measures on a topological space $X$.
  We say that $\cM$ is \emph{tight} if
  for any real number $\varepsilon > 0$ there exists a compact subset
  $K_\varepsilon \subset X$ such that $\mu(X \setminus K_\varepsilon) < \varepsilon$
  for every $\mu \in \cM$.
\end{defn}

\begin{thm}[Prohorov's theorem]
  \index{Prohorov's theorem} \label{thm:Proh}
  Let $\cM$ be a family of Borel probability measures
  on a complete separable metric space.
  Then the following {\rm(1)} and {\rm(2)} are equivalent to each other.
  \begin{enumerate}
  \item $\cM$ is tight.
  \item $\cM$ is relatively compact with respect to $d_P$.
  \end{enumerate}
\end{thm}

\begin{defn}[Transport plan]
  \index{transport plan} \index{coupling}
  Let $\mu$ and $\nu$ be two finite Borel measures on $X$.
  A Borel measure $m$ on $X \times X$ is called
  a \emph{transport plan} (or \emph{coupling}) \emph{between $\mu$ and $\nu$}
  if
  \[
  m(A \times X) = \mu(A) \quad\text{and}\quad m(X \times A) = \nu(A)
  \]
  for any Borel subset $A \subset X$.
\end{defn}

\begin{defn}[$\varepsilon$-Transportation]
  \index{epsilon transportation@$\varepsilon$-transportation}
  \index{transportation}
  Let $\mu$ and $\nu$ be two Borel probability measures on $X$.
  A Borel measure $m$ on $X \times X$ is called
  an \emph{$\varepsilon$-transportation between $\mu$ and $\nu$}
  if the following (1) and (2) are satisfied.
  \begin{enumerate}
  \item There exist two Borel measures $\mu'$ and $\nu'$ on $X$
    with $\mu' \le \mu$ and $\nu' \le \nu$
    such that $m$ is a transport plan between $\mu'$ and $\nu'$.
  \item We have
    \[
    \supp m \subset
    \Delta_\varepsilon
    := \{\;(x,y) \in X \times X \mid d_X(x,y) \le \varepsilon\;\},
    \]
    where $\supp m$ is the \emph{support of $m$}, i.e.,
    the set of points $x$ such that any open neighborhood of $x$
    has positive $m$-measure.
    \index{supp@$\supp$} \index{support of a measure}
  \end{enumerate}
  For an $\varepsilon$-transportation $m$ between $\mu$ and $\nu$,
  the \emph{deficiency of $m$} is defined to be
  \index{deficiency} 
  \[
  \defi m := 1-m(X\times X).
  \]
\end{defn}

\begin{thm}[Strassen's theorem; cf.~\cite{Villani:topics}*{Corollary 1.28}]
  \label{thm:di-tra} \index{Strassen's theorem}
  For any two Borel probability measures $\mu$ and $\nu$ on $X$, we have
  \begin{align*}
    d_P(\mu,\nu) = \inf\{\;\varepsilon > 0 &\mid
    \text{There exists an $\varepsilon$-transportation $m$}\\
    &\quad\text{between $\mu$ and $\nu$ with $\defi m \le \varepsilon$}\;\}.
  \end{align*}
\end{thm}

\section{Convergence in measure of maps}

\begin{defn}[Convergence in measure, Ky Fan metric]
  \index{convergence in measure} \index{Ky Fan metric}
  \index{dKF mu@$d_{KF}^\mu$} \index{dKF@$\dKF$}
  Let $(X,\mu)$ be a measure space and $Y$ a metric space.
  For two $\mu$-measurable maps $f,g : X \to Y$, we define
  $\dKF(f,g) = d_{KF}^\mu(f,g)$
  to be the infimum of $\varepsilon \ge 0$ satisfying
  \begin{align} \label{eq:me}
    \mu(\{\;x \in X \mid d_Y(f(x),g(x)) > \varepsilon\;\}) \le \varepsilon.
  \end{align}
  We call $d_{KF}^\mu$ the \emph{Ky Fan metric} on the set of
  $\mu$-measurable maps from $X$ to $Y$.
  We say that a sequence of $\mu$-measurable maps $f_n : X \to Y$,
  $n=1,2,\dots$, \emph{converges in measure} to a $\mu$-measurable map
  $f : X \to Y$ if
  \[
  \lim_{n\to\infty} d_{KF}^\mu(f_n,f) = 0.
  \]
\end{defn}

Note that $\dKF(f,g) \le 1$ for any $\mu$-measurable maps $f,g : X \to Y$
provided that $\mu$ is a probability measure.

\begin{rem} \label{rem:me}
  Since $\varepsilon \mapsto
  \mu(\{\;x \in X \mid d_Y(f(x),g(x)) > \varepsilon\;\})$
  is right-continuous, there is the minimum of $\varepsilon \ge 0$
  satisfying \eqref{eq:me}.
  For a real number $\varepsilon$,
  $\dKF(f,g) \le \varepsilon$ holds if and only if \eqref{eq:me} holds.
\end{rem}

\begin{lem}
  Let $(X,\mu)$ be a measure space and $Y$ a metric space.
  Then, $d_{KF}^\mu$ is a metric on the set of $\mu$-measurable maps from $X$ to $Y$
  by identifying two measurable maps from $X$ to $Y$ if
  they are equal to each other $\mu$-almost everywhere.
\end{lem}

\begin{proof}
  Let $f, g, h : X \to Y$ be three $\mu$-measurable maps.

  It is obvious that $f = g$ $\mu$-a.e. if and only if $\dKF(f,g) = 0$.

  It is also clear that $\dKF(g,f) = \dKF(f,g)$.

  We prove the triangle inequality $\dKF(f,h) \le \dKF(f,g) + \dKF(g,h)$.
  Setting $\varepsilon := \dKF(f,g)$ and $\delta := \dKF(g,h)$, we have
  (see Remark \ref{rem:me})
  \begin{align*}
    \mu(\{\;x \in X \mid d_Y(f(x),g(x)) > \varepsilon\;\}) &\le \varepsilon,\\
    \mu(\{\;x \in X \mid d_Y(g(x),h(x)) > \delta\;\}) &\le \delta.
  \end{align*}
  If $d_Y(f(x),g(x))+d_Y(g(x),h(x)) > \varepsilon+\delta$ for a point $x \in X$,
  then we have at least one of $d_Y(f(x),g(x)) > \varepsilon$ and
  $d_Y(g(x),h(x)) > \delta$.  Therefore,
  \begin{align*}
    &\mu(\{\;x \in X \mid d_Y(f(x),h(x)) > \varepsilon+\delta\;\})\\
    &\le
    \mu(\{\;x \in X \mid d_Y(f(x),g(x))+d_Y(g(x),h(x)) >
    \varepsilon+\delta\;\})\\
    &\le
    \mu(\{\;x \in X \mid d_Y(f(x),g(x)) > \varepsilon\;\})\\
    &\ +\mu(\{\;x \in X \mid d_Y(g(x),h(x)) > \delta\;\})\\
    &\le \varepsilon+\delta,
  \end{align*}
  which implies $\dKF(f,h) \le \varepsilon+\delta$.
  This completes the proof.
\end{proof}

\begin{lem} \label{lem:di-me}
  Let $X$ be a topological space with a Borel probability measure $\mu$
  and $Y$ a metric space.
  For any two $\mu$-measurable maps $f,g : X \to Y$, we have
  \[
  d_P(f_*\mu,g_*\mu) \le d_{KF}^\mu(f,g).
  \]
\end{lem}

\begin{proof}
  Let $\varepsilon := d_{KF}^\mu(f,g)$.
  We take any Borel subset $A \subset Y$.
  It suffices to prove that $f_*\mu(B_\varepsilon(A)) \ge g_*\mu(A)-\varepsilon$.
  Setting $X_0 := \{\;x \in X \mid d_Y(f(x),g(x)) \le \varepsilon\;\}$,
  we have $\mu(X \setminus X_0) \le \varepsilon$.

  We prove that $g^{-1}(A) \cap X_0 \subset f^{-1}(B_\varepsilon(A))$.
  In fact, if we take any point $x \in g^{-1}(A) \cap X_0$, then
  $g(x) \in A$ and $x \in X_0$, which imply
  $f(x) \in B_\varepsilon(A)$ and so $x \in f^{-1}(B_\varepsilon(A))$.
  Thus, $g^{-1}(A) \cap X_0 \subset f^{-1}(B_\varepsilon(A))$.

  Since $\mu(g^{-1}(A) \setminus X_0) \le \mu(X \setminus X_0) \le \varepsilon$,
  \begin{align*}
    g_*\mu(A) &= \mu(g^{-1}(A))
    = \mu(g^{-1}(A) \cap X_0) + \mu(g^{-1}(A) \setminus X_0)\\
    &\le \mu(f^{-1}(B_\varepsilon(A))) + \varepsilon
    = f_*\mu(B_\varepsilon(A)) + \varepsilon
  \end{align*}
  This completes the proof.
\end{proof}

The proof of the following lemma is left to the reader.

\begin{lem}
  Let $X$ be a topological space with a Borel probability measure $\mu$,
  and $Y$ a metric space.
  For any Borel measurable map $f : X \to Y$ and for any point $c \in Y$,
  we have
  \[
  d_P(f_*\mu,\delta_c) = d_{KF}^\mu(f,c).
  \]
\end{lem}

\chapter{L\'evy-Milman concentration phenomenon}
\label{chap:LM-conc}

\section{Observation of spheres}

Let $S^n(r)$ \index{Snr@$S^n(r)$} be the sphere of radius $r > 0$
centered at the origin in the $(n+1)$-dimensional Euclidean space $\R^{n+1}$
and $\sigma^n$ \index{sigman@$\sigma^n$} the Riemannian volume measure
on $S^n(r)$ normalized as $\sigma^n(S^n(r)) = 1$.
We assume the distance between points in $S^n(r)$
to be the geodesic distance.
Let $k \le n$.
Identifying $\R^k$ with the subspace
$\R^k \times \{(0,0,\dots,0)\} \subset \R^{n+1}$,
we consider the orthogonal projection from $\R^{n+1}$ to $\R^k$,
and denote the restriction of it on $S^n(\sqrt{n})$ by
$\pi_{n,k} : S^n(\sqrt{n}) \to \R^k$.
Note that $\pi_{n,k} : S^n(\sqrt{n}) \to \R^k$ is $1$-Lipschitz
continuous, i.e., Lipschitz continuous functions with
Lipschitz constant $1$.

\index{standard Gaussian measure} \index{gammak@$\gamma^k$}
\index{Gaussian measure}
Denote by $\gamma^k$
the \emph{$k$-dimensional standard Gaussian measure on $\R^k$},
i.e.,
\[
d\gamma^k(x) := \frac{1}{(2\pi)^{k/2}} e^{-\frac{1}{2}\|x\|_2^2} \; dx,
\qquad x \in \R^k,
\]
where $\|x\|_2$ is the Euclidean norm of $x$
and $dx$ the $k$-dimensional Lebesgue measure on $\R^k$.
\index{bar two@$\Vert\cdot\Vert_2$}

\begin{prop}[Maxwell-Boltzmann distribution law\footnotemark]
  \label{prop:MB-law}
  \index{Maxwell-Boltzmann distribution law} \index{Poincar\'e's limit}
  \index{the Poincar\'e limit theorem}
  \footnotetext{This is also called the Poincar\'e limit theorem
    in many literature.
    However, there is no evidence that Poincar\'e proved this
    (see \cite{DF}).}
  For any natural number $k$ we have
  \[
  \frac{d(\pi_{n,k})_*\sigma^n}{dx} \to \frac{d\gamma^k}{dx}
  \qquad \text{as $n \to \infty$},
  \]
  where $(\pi_{n,k})_*\sigma^n$ is the push-forward of $\sigma^n$
  by $\pi_{n,k}$.  In particular,
  \[
  (\pi_{n,k})_*\sigma^n \to \gamma^k \ \text{weakly}
  \qquad \text{as $n \to \infty$}.
  \]
\end{prop}

\begin{proof}
  Denote by $\vol_l$ the $l$-dimensional volume measure.
  Since $\pi_{n,k}^{-1}(x)$, $x \in \R^k$, is isometric to
  $S^{n-k}((n-\|x\|_2^2)^{1/2})$, we have
  \begin{align*}
    \frac{d(\pi_{n,k})_*\sigma^n}{dx}
    &= \frac{\vol_{n-k}\pi_{n,k}^{-1}(x)}{\vol_n S^n(\sqrt{n})}
    = \frac{(n-\|x\|_2^2)^{(n-k)/2}}
    {\int_{\|x\|_2 \le \sqrt{n}} (n-\|x\|_2^2)^{(n-k)/2} \; dx}\\
    &\overset{n\to\infty}\longrightarrow
    \frac{e^{-\frac{1}{2}\|x\|_2^2}}{\int_{\R^k} e^{-\frac{1}{2}\|x\|_2^2}\,dx}
    = \frac{1}{(2\pi)^{k/2}} e^{-\frac{1}{2}\|x\|_2^2}
    = \frac{d\gamma^k}{dx}.
  \end{align*}
\end{proof}

The purpose of this section is to prove the following

\begin{thm}[Normal law \`a la L\'evy] \label{thm:normal}
  \index{normal law a la Levy@normal law \`a la L\'evy}
  Let $f_n : S^n(\sqrt{n}) \to \R$, $n=1,2,\dots$, be
  $1$-Lipschitz functions.
  Assume that, for a subsequence $\{f_{n_i}\}$ of $\{f_n\}$,
  the push-forward $(f_{n_i})_*\sigma^{n_i}$ converges vaguely to
  a Borel measure $\sigma_\infty$ on $\R$.
  Then, there exists a $1$-Lipschitz function $\alpha : \R \to \R$
  such that
  \[
  \alpha_*\gamma^1 = \sigma_\infty
  \]
  unless $\sigma_\infty$ is identically equal to zero.
\end{thm}

Lemma \ref{lem:conv-meas}\eqref{it:conv-meas-v} implies
the existence of a subsequence $\{f_{n_i}\}$
such that $(f_{n_i})_*\sigma^{n_i}$ converges vaguely to
some finite Borel measure on $\R$.

With the notation of Definition \ref{defn:dom}, we have
\[
(\R,|\cdot|,\sigma_\infty) \prec (\R,|\cdot|,\gamma^1).
\]

We need some claims for the proof of Theorem \ref{thm:normal}.
The following theorem is well-known.

\begin{thm}[L\'evy's isoperimetric inequality \cites{Levy,FLM}]
  \label{thm:Levy-isop}
  \index{Levys isoperimetric inequality@L\'evy's isoperimetric inequality}
  For any closed subset $\Omega \subset S^n(1)$,
  we take a metric ball $B_\Omega$ of $S^n(1)$ with
  $\sigma^n(B_\Omega) = \sigma^n(\Omega)$.
  Then we have
  \[
  \sigma^n(U_r(\Omega)) \ge \sigma^n(U_r(B_\Omega))
  \]
  for any $r > 0$.
\end{thm}

Assume the condition of Theorem \ref{thm:normal}.
We consider a natural compactification $\bar\R := \R \cup \{-\infty,+\infty\}$
of $\R$.  Then, by replacing $\{f_{n_i}\}$ with a subsequence,
$\{(f_{n_i})_*\sigma^{n_i}\}$ converges weakly to a Borel probability measure
$\bar{\sigma}_\infty$ on $\bar{\R}$
(see Lemma \ref{lem:conv-meas}\eqref{it:conv-meas-w})
such that $\bar{\sigma}_\infty|_{\R} = \sigma_\infty$.
We prove the following

\begin{lem}\label{lem:normal1}
  Let $x$ and $x'$ be two given real numbers.
  If $\gamma^1(\,-\infty,x\,] = \bar\sigma_\infty[\,-\infty,x'\,]$
  and if $\sigma_\infty\{x'\} = 0$, then
  \[
  \sigma_\infty[\,x'-\varepsilon_1,x'+\varepsilon_2\,]
  \ge \gamma^1[\,x-\varepsilon_1,x+\varepsilon_2\,]
  \]
  for all real numbers $\varepsilon_1,\varepsilon_2 \ge 0$.
  In particular, if $\bar\sigma_\infty \neq \delta_{\pm\infty}$,
  then $\bar\sigma_\infty\{-\infty,+\infty\} = 0$
  and $\sigma_\infty$ is a probability measure on $\R$.
\end{lem}

\begin{proof}
  We set $\Omega_+ := \{\,f_{n_i} \ge x'\,\}$ and
  $\Omega_- := \{\,f_{n_i} \le x'\,\}$.
  We have $\Omega_+ \cup \Omega_- = S^{n_i}(\sqrt{n_i})$.
  Let us prove
  \begin{align}\label{eq:normal-omega1}
    U_{\varepsilon_1}(\Omega_+) \cap U_{\varepsilon_2}(\Omega_-)
    \subset \{\,x'-\varepsilon_1 \le f_{n_i} \le x'+\varepsilon_2\,\}.
  \end{align}
  In fact, for any point $\xi \in U_{\varepsilon_1}(\Omega_+)$,
  there is a point $\xi' \in \Omega_+$ such that the geodesic distance
  between $\xi$ and $\xi'$ is not greater than $\varepsilon_1$.
  The $1$-Lipschitz continuity of $f_{n_i}$ proves that
  $f_{n_i}(\xi) \ge f_{n_i}(\xi') - \varepsilon_1 \ge x'-\varepsilon_1$.
  Thus we have
  $U_{\varepsilon_1}(\Omega_+) \subset \{\,x'-\varepsilon_1 \le f_{n_i}\;\}$
  and, in the same way,
  $U_{\varepsilon_2}(\Omega_-) \subset \{\,f_{n_i} \le x'+\varepsilon_2\,\}$.
  Combining these two inclusions implies \eqref{eq:normal-omega1}.

  It follows from \eqref{eq:normal-omega1} and
  $U_{\varepsilon_1}(\Omega_+) \cup U_{\varepsilon_2}(\Omega_-)= S^{n_i}(\sqrt{n_i})$
  that
  \begin{align*}
    &(f_{n_i})_*\sigma^{n_i}[\,x'-\varepsilon_1,x'+\varepsilon_2\,]\\
    &= \sigma^{n_i}(x'-\varepsilon_1 \le f_{n_i} \le x'+\varepsilon_2)
    \ge \sigma^{n_i}(U_{\varepsilon_1}(\Omega_+) \cap U_{\varepsilon_2}(\Omega_-))\\
    &= \sigma^{n_i}(U_{\varepsilon_1}(\Omega_+)) + \sigma^{n_i}(U_{\varepsilon_2}(\Omega_-))
    -1.
  \end{align*}
  The L\'evy's isoperimetric inequality (Theorem \ref{thm:Levy-isop})
  implies
  $\sigma^{n_i}(U_{\varepsilon_1}(\Omega_+)) \ge \sigma^{n_i}(U_{\varepsilon_1}(B_{\Omega_+}))$
  and
  $\sigma^{n_i}(U_{\varepsilon_2}(\Omega_-)) \ge \sigma^{n_i}(U_{\varepsilon_2}(B_{\Omega_-}))$,
  so that
  \[
  (f_{n_i})_*\sigma^{n_i}[\,x'-\varepsilon_1,x'+\varepsilon_2\,]
  \ge \sigma^{n_i}(U_{\varepsilon_1}(B_{\Omega_+})) + \sigma^{n_i}(U_{\varepsilon_2}(B_{\Omega_-}))
  -1.
  \]
  It follows from $\sigma_\infty\{x'\} = 0$ that $\sigma^{n_i}(\Omega_+)$ 
  converges to $\bar\sigma_\infty(\,x',+\infty\,]$ as $n\to\infty$.
  We besides have
  $\bar\sigma_\infty(\,x',+\infty\,] = \gamma^1[\,x,+\infty\,) \neq 0,1$,
  so that $\sigma^{n_i}(\Omega_+) \neq 0,1$ for all sufficiently large $i$.
  Let $a_i$ and $b_i$ be two real numbers such that
  $\sigma^{n_i}(\Omega_+) = (\pi_{n,1})_*\sigma^{n_i}[\,a_i,+\infty\,)$
  and $\sigma^{n_i}(U_{\epsilon_1}(B_{\Omega_+}))
  = (\pi_{n_i,1})_*\sigma^{n_i}[\,b_i,+\infty\,)$.
  By the Maxwell-Boltzmann distribution law (Proposition \ref{prop:MB-law})
  and by remarking that the radius of the sphere is divergent to infinity,
  we see that $a_i$ and $b_i$ converges to $x$ and $x-\varepsilon_1$,
  respectively, as $i\to\infty$.
  In particular we obtain
  \[
  \lim_{i\to\infty} \sigma^{n_i}(U_{\varepsilon_1}(B_{\Omega_+}))
  = \gamma^1[\,x-\varepsilon_1,+\infty\,)
  \]
  as well as
  \[
  \lim_{i\to\infty} \sigma^{n_i}(U_{\varepsilon_2}(B_{\Omega_-}))
  = \gamma^1(\,-\infty,x+\varepsilon_2\,].
  \]
  Therefore,
  \begin{align*}
    \sigma_\infty[\,x'-\varepsilon_1,x'+\varepsilon_2\,]
    &\ge \liminf_{i\to\infty} (f_{n_i})_*\sigma^{n_i}[\,x'-\varepsilon_1,x'+\varepsilon_2\,]\\
    &\ge \gamma^1[\,x-\varepsilon_1,+\infty\,) + \gamma^1(\,-\infty,x+\varepsilon_2\,]
    -1\\
    &= \gamma^1[\,x-\varepsilon_1,x+\varepsilon_2\,].
  \end{align*}
  The first part of the lemma is obtained.

  Assume that $\bar\sigma_\infty \neq \delta_{\pm\infty}$.
  The rest of the proof is to show that $\sigma_\infty(\R) = 1$.
  Suppose $\sigma_\infty(\R) < 1$.
  Then, there is a non-atomic point $x' \in \R$ of $\sigma_\infty$
  (i.e., a point $x'$ with $\sigma_\infty\{x'\} = 0$)
  such that
  $0 < \bar\sigma_\infty[\,-\infty,x'\,) < 1$.
  We find a real number $x$ in such a way that
  $\gamma^1(\,-\infty,x\,] = \bar\sigma_\infty[\,-\infty,x'\,]$.
  The first part of the lemma implies
  $\sigma_\infty[\,x'-\varepsilon_1,x'+\varepsilon_2\,]
  \ge \gamma^1[\,x-\varepsilon_1,x+\varepsilon_2\,]$
  for all $\varepsilon_1,\varepsilon_2 \ge 0$.
  Taking the limit as $\varepsilon_1,\varepsilon_2 \to +\infty$,
  we obtain $\sigma_\infty(\R) = 1$.
  This completes the proof.
\end{proof}

\begin{lem} \label{lem:normal2}
  $\supp\sigma_\infty$ is a closed interval.
\end{lem}

\begin{proof}
  $\supp\sigma_\infty$ is a closed set by the definition of the support
  of a measure.
  It then suffices to prove the connectivity of $\supp\sigma_\infty$.
  Suppose not.  Then, there are numbers $x'$ and $\varepsilon > 0$
  such that
  $\sigma_\infty(\,-\infty,x'-\varepsilon\,) > 0$,
  $\sigma_\infty[\,x'-\varepsilon,x'+\varepsilon\,] = 0$,
  and $\sigma_\infty(\,x'+\varepsilon,+\infty\,) > 0$.
  There is a number $x$ such that
  $\gamma^1(\,-\infty,x\,] = \sigma_\infty(\,-\infty,x'\,]$.
  Lemma \ref{lem:normal1} shows that
  $\sigma_\infty[\,x'-\varepsilon,x'+\varepsilon\,]
  \ge \gamma^1[\,x-\varepsilon,x+\varepsilon\,] > 0$,
  which is a contradiction.
  The lemma has been proved.
\end{proof}

\begin{proof}[Proof of Theorem \ref{thm:normal}]
  For any given real number $x$, there exists a smallest number $x'$
  satisfying $\gamma^1(\,-\infty,x\,] \le \sigma_\infty(\,-\infty,x'\,]$.
  The existence of $x'$ follows from the right-continuity and the monotonicity
  of $y \mapsto \sigma_\infty(\,-\infty,y\,]$.
  Setting
  $\alpha(x) := x'$ we have a function
  $\alpha : \R \to \R$, which is monotone nondecreasing.
  It is easy to see that
  $(\supp\sigma_\infty)^\circ \subset \alpha(\R) \subset \supp\sigma_\infty$.

  We first prove the continuity of $\alpha$ in the following.
  Take any two numbers $x_1$ and $x_2$ with $x_1 < x_2$.
  We have
  $\gamma^1(\,-\infty,x_1\,] \le \sigma_\infty(\,-\infty,\alpha(x_1)\,]$
  and $\gamma^1(\,-\infty,x_2\,] \ge \sigma_\infty(\,-\infty,\alpha(x_2)\,)$,
  which imply
  \begin{align} \label{eq:normal2}
    \gamma^1[\,x_1,x_2\,] \ge \sigma_\infty(\,\alpha(x_1),\alpha(x_2)\,).
  \end{align}
  This shows that, as $x_1 \to a-0$ and $x_2 \to a+0$ for a number $a$,
  we have
  $\sigma_\infty(\,\alpha(x_1),\alpha(x_2)\,) \to 0$, which together with
  Lemma \ref{lem:normal2} implies $\alpha(x_2) - \alpha(x_1) \to 0$.
  Thus, $\alpha$ is continuous on $\R$.

  Let us next prove the $1$-Lipschitz continuity of $\alpha$.
  We take two numbers $x$ and $\varepsilon > 0$ and fix them.
  It suffices to prove that
  \[
  \Delta\alpha := \alpha(x+\varepsilon)-\alpha(x) \le \varepsilon.
  \]

  \begin{clm}
    If $\sigma_\infty\{\alpha(x)\} = 0$, then $\Delta\alpha \le \varepsilon$.
  \end{clm}

  \begin{proof}
    The claim is trivial if $\Delta\alpha = 0$.
    We thus assume $\Delta\alpha > 0$.
    Since $\sigma_\infty\{\alpha(x)\} = 0$, we have
    $\gamma^1(\,-\infty,x\,] = \sigma_\infty(\,-\infty,\alpha(x)\,]$,
    so that Lemma \ref{lem:normal1} implies that
    \begin{align} \label{eq:normal1}
      \sigma_\infty[\,\alpha(x),\alpha(x)+\delta\,]
      \ge \gamma^1[\,x,x+\delta\,]
    \end{align}
    for all $\delta \ge 0$.
    By \eqref{eq:normal2} and \eqref{eq:normal1},
    \begin{align*}
      \gamma^1[\,x,x+\varepsilon\,]
      &\ge \sigma_\infty(\,\alpha(x),\alpha(x+\varepsilon)\,)\\
      &= \sigma_\infty[\,\alpha(x),\alpha(x)+\Delta\alpha\,)\\
      &= \lim_{\delta \to \Delta\alpha-0} \sigma_\infty[\,\alpha(x),\alpha(x)+\delta\,]\\
      &\ge \lim_{\delta \to \Delta\alpha-0} \gamma^1[\,x,x+\delta\,]\\
      &= \gamma^1[\,x,x+\Delta\alpha\,],
    \end{align*}
    which implies $\Delta\alpha \le \varepsilon$.
  \end{proof}

  We next prove that $\Delta\alpha \le \varepsilon$ in the case where
  $\sigma_\infty\{\alpha(x)\} > 0$.
  We may assume that $\Delta\alpha > 0$.
  Let $x_+ := \sup \alpha^{-1}(\alpha(x))$.
  It follows from $\alpha(x) < \alpha(x+\varepsilon)$
  that $x_+ < x+\varepsilon$.
  The continuity of $\alpha$ implies that $\alpha(x_+) = \alpha(x)$.
  There is a sequence of positive numbers
  $\varepsilon_i \to 0$ such that
  $\sigma_\infty\{\alpha(x_++\varepsilon_i)\} = 0$.
  By applying the claim above,
  \[
  \alpha(x_++\varepsilon_i+\varepsilon)-\alpha(x_++\varepsilon_i) \le \varepsilon.
  \]
  Moreover we have
  $\alpha(x+\varepsilon) \le \alpha(x_++\varepsilon_i+\varepsilon)$
  and $\alpha(x_++\varepsilon_i) \to \alpha(x_+) = \alpha(x)$ as $i\to\infty$.
  Thus,
  \[
  \alpha(x+\varepsilon) - \alpha(x) \le \varepsilon
  \]
  and so $\alpha$ is $1$-Lipschitz continuous.

  The rest is to prove that $\alpha_*\gamma^1 = \sigma_\infty$.
  Take any number $x' \in \alpha(\R)$ and fix it.
  Set $x := \sup \alpha^{-1}(x') \;(\le +\infty)$.
  We then have $\alpha(x) = x'$ provided $x < +\infty$.
  Since $x$ is the largest number to satisfy
  $\gamma^1(\,-\infty,x\,] \le \sigma_\infty(\,-\infty,x'\,]$,
  we have $\gamma^1(\,-\infty,x\,] = \sigma_\infty(\,-\infty,x'\,]$,
  where we agree $\gamma^1(\,-\infty,+\infty\,] = 1$.
  By the monotonicity of $\alpha$, we obtain
  \[
  \alpha_*\gamma^1(\,-\infty,x'\,] = \gamma^1(\alpha^{-1}(\,-\infty,x'\,])
  = \gamma^1(\,-\infty,x\,] = \sigma_\infty(\,-\infty,x'\,],
  \]
  which implies that $\alpha_*\gamma^1 = \sigma_\infty$
  because $\sigma_\infty$ is a Borel probability measure.
  This completes the proof.
\end{proof}

\begin{cor}[L\'evy's lemma] \label{cor:Levy}
  \index{Levy's lemma@L\'evy's lemma}
  Let $f_n : S^n(1) \to \R$, $n=1,2,\dots$, be $1$-Lipschitz functions
  such that $\int_{S^n(1)} f_n \; d\sigma^n = 0$.
  Then we have
  \[
  (f_n)_*\sigma^n \to \delta_0 \ \text{weakly as $n\to\infty$},
  \]
  or equivalently,
  $f_n$ converges in measure to zero as $n\to\infty$.
\end{cor}

\begin{proof}
  Suppose that the lemma is false, so that
  we find $1$-Lipschitz functions $f_n : S^n(1) \to \R$, $n=1,2,\dots$,
  with $\int_{S^n(1)} f_n \; d\sigma^n = 0$,
  and a subsequence $\{f_{n_i}\}$ of $\{f_n\}$ such that
  \begin{equation}
    \label{eq:Levy}
    \liminf_{i\to\infty} d_P((f_{n_i})_*\sigma^{n_i},\delta_0) > 0.
  \end{equation}
  We denote by $\iota_n : S^n(\sqrt{n}) \to S^n(1)$ a natural map.
  $\tilde{f}_n := \sqrt{n}\, f_n \circ \iota_n : S^n(\sqrt{n}) \to \R$
  is $1$-Lipschitz continuous.
  Let $m_n$ be a median of $f_n$.
  Then, $\sqrt{n}\,m_n$ is a median of $\tilde{f}_n$.
  We consider the measure
  $\tilde\sigma^{n_i} := (\tilde{f}_{n_i}-\sqrt{n_i}\,m_{n_i})_*\sigma^{n_i}$.
  Lemma \ref{lem:conv-meas}\eqref{it:conv-meas-w} implies that
  there is a subsequence of $\{\tilde\sigma^{n_i}\}$ that is weakly convergent on $\bar\R$.
  Replace $\{\tilde\sigma^{n_i}\}$ with such a subsequence.
  Since
  $\tilde\sigma^{n_i}(\,-\infty,0\,], \tilde\sigma^{n_i}[\,0,+\infty\,) \ge 1/2$
  and by Lemma \ref{lem:normal1},
  the limit of $\tilde\sigma^{n_i}$
  is a Borel probability measure on $\R$.
  Therefore, $(f_{n_i} - m_{n_i})_*\sigma^{n_i}
  = ((1/\sqrt{n_i})\tilde{f}_{n_i}-m_{n_i})_*\sigma^{n_i}$
  converges weakly to $\delta_0$ as $i\to\infty$.
  Since the geodesic distance between any two points in $S^n(1)$
  is at most $\pi$ 
  and $\int_{S^n(1)} f_n d\sigma^n = 0$,
  we have $-\pi \le f_n \le \pi$, so that
  $m_{n_i}$ converges to zero as $i\to\infty$.
  We then obtain that $(f_{n_i})_*\sigma^{n_i}$ converges weakly to $\delta_0$
  as $i\to\infty$, which contradicts \eqref{eq:Levy}.
  This completes the proof.
\end{proof}

\section{mm-Isomorphism and Lipschitz order}

\begin{defn}[mm-Space]
  \index{mm-space}
  Let $(X,d_X)$ be a complete separable metric space
  and $\mu_X$ a Borel probability measure on $X$.
  We call the triple $(X,d_X,\mu_X)$ an \emph{mm-space}.
  We sometimes say that $X$ is an mm-space, in which case
  the metric and measure of $X$ are respectively indicated by $d_X$ and $\mu_X$.
\end{defn}

In this book, manifolds may have nonempty boundary
unless otherwise stated.
For a complete Riemannian manifold $X$ with finite volume,
we always equip $X$ with the Riemannian distance function $d_X$
and with the volume measure $\mu_X$ normalized as $\mu_X(X) = 1$,
i.e., $\mu_X := \vol_X/\vol_X(X)$,
where $\vol_X$ is the Riemannian volume measure on $X$.
Then, $(X,d_X,\mu_X)$ is an mm-space.

A complete Riemannian manifold with finite diameter is always compact.
However, an mm-space with finite diameter is not necessarily compact.
Such an example is obtained as the discrete countable space
$X = \{x_i\}_{i=1}^\infty$ with $d_X(x_i,x_j) = 1-\delta_{ij}$
and $\mu_X = \sum_{i=1}^\infty 2^{-i}\delta_{x_i}$,
where $\delta_{ii} = 1$ and $\delta_{ij} = 0$ if $i \neq j$.

\begin{defn}[mm-Isomorphism]
  \index{mm-isomorphism} \index{mm-isomorphic}
  Two mm-spaces $X$ and $Y$ are said to be \emph{mm-isomorphic}
  to each other if there exists an isometry $f : \supp\mu_X \to \supp\mu_Y$
  such that $f_*\mu_X = \mu_Y$.
  Such an isometry $f$ is called an \emph{mm-isomorphism}.
  The mm-isomorphism relation is an equivalence relation
  on the set of mm-spaces.
  Denote by $\cX$ \index{X@$\cX$}
  the set of mm-isomorphism classes of mm-spaces.
\end{defn}

Any mm-isomorphism between mm-spaces is automatically surjective,
even if we do not assume it.
Note that $X$ is mm-isomorphic to $(\supp\mu_X,d_X,\mu_X)$.

\emph{We assume that any mm-space $X$ satisfies
\[
X = \supp\mu_X
\]
unless otherwise stated.}

\begin{defn}[Lipschitz order] \label{defn:dom}
  \index{Lipschitz order}
  Let $X$ and $Y$ be two mm-spaces.
  We say that $X$ (\emph{Lipschitz}) \emph{dominates} $Y$
  \index{dominate} \index{Lipschitz dominate} \index{less than@$\prec$}
  and write $Y \prec X$ if
  there exists a $1$-Lipschitz map $f : X \to Y$ satisfying
  \[
  f_*\mu_X = \mu_Y.
  \]
  We call the relation $\prec$ on $\cX$ the \emph{Lipschitz order}.
\end{defn}

\begin{prop} \label{prop:Liporder}
  The Lipschitz order $\prec$ is a partial order relation on $\cX$, i.e.,
  we have the following {\rm(1)}, {\rm(2)}, and {\rm(3)}
  for any mm-spaces $X$, $Y$, and $Z$.
  \begin{enumerate}
  \item $X \prec X$.
  \item If $X \prec Y$ and $Y \prec X$, then $X$ and $Y$ are
    mm-isomorphic to each other.
  \item If $X \prec Y$ and $Y \prec Z$, then $X \prec Z$.
  \end{enumerate}
\end{prop}

(1) and (3) are obvious.
For the proof of (2), we need a lemma.

Let $\varphi : [\,0,+\infty\,) \to [\,0,+\infty\,)$
be a bounded, continuous, and strictly monotone increasing function.
For an mm-space $X$ we define
\[
\Avr_\varphi(X) := \int_{X \times X} \varphi(d_X(x,x'))\;d(\mu_X\otimes\mu_X)(x,x').
\]

\begin{lem} \label{lem:avr}
  Let $X$ and $Y$ be two mm-spaces.
  \begin{enumerate}
  \item If $X \prec Y$, then $\Avr_\varphi(X) \le \Avr_\varphi(Y)$.
  \item If $X \prec Y$ and if $\Avr_\varphi(X) = \Avr_\varphi(Y)$, then
    $X$ and $Y$ are mm-isomorphic to each other.
  \end{enumerate}
\end{lem}

\begin{proof}
  We prove (1).  By $X \prec Y$, we have a $1$-Lipschitz map $f : Y \to X$
  with $f_*\mu_Y = \mu_X$.
  We have
  \begin{align*}
    \Avr_\varphi(X) &= \int_{X \times X} \varphi(d_X(x,x'))
    \;d(f_*\mu_Y\otimes f_*\mu_Y)(x,x')\\
    &= \int_{Y \times Y} \varphi(d_X(f(y),f(y')))\;d(\mu_Y\otimes \mu_Y)(y,y')\\
    &\le \int_{Y \times Y} \varphi(d_Y(y,y'))\;d(\mu_Y\otimes \mu_Y)(y,y')\\
    &= \Avr_\varphi(Y).
  \end{align*}

  We prove (2).  Since the equality holds in the above,
  we have
  \[
  \varphi(d_X(f(y),f(y'))) = \varphi(d_Y(y,y'))
  \]
  $\mu_Y\otimes \mu_Y$-a.e.~$(y,y') \in Y \times Y$,
  which implies that $f : Y \to X$ is isometric.
  It follows from $f_*\mu_Y = \mu_X$ that
  the image $f(Y)$ is dense in $X$,
  which together with the completeness of $Y$
  proves $f(Y) = X$.
  Thus, $f$ is an mm-isomorphism between $X$ and $Y$.
  This completes the proof.
\end{proof}

\begin{proof}[Proof of Proposition \ref{prop:Liporder}]
  It suffices to prove (2).
  Assume that $X \prec Y$ and $Y \prec X$.
  By Lemma \ref{lem:avr}(1) we have $\Avr_\varphi(X) = \Avr_\varphi(Y)$,
  so that Lemma \ref{lem:avr}(2) implies that
  $X$ and $Y$ are mm-isomorphic to each other.
  This completes the proof.
\end{proof}

\section{Observable diameter}

The observable diameter is one of the most fundamental invariants
of an mm-space.

\begin{defn}[Partial and observable diameter]
  \index{observable diameter}
  Let $X$ be an mm-space and $Y$ a metric space.
  For a real number $\alpha \le 1$, we define
  the \emph{partial diameter
    $\diam(X;\alpha) = \diam(\mu_X;\alpha)$ of $X$}
  \index{partial diameter}
  \index{diam@$\diam(X;\alpha)$} \index{diam@$\diam(\mu_X;\alpha)$}
  to be the infimum of $\diam A$,
  where $A \subset X$ runs over all Borel subsets
  with $\mu_X(A) \ge \alpha$, and the \emph{deameter $\diam A$ of $A$}
  is defined by $\diam A := \sup_{x,y\in A} d_X(x,y)$ for $A \neq \emptyset$
  and $\diam\emptyset := 0$.
  For a real number $\kappa > 0$ we define
  \begin{align*}
    \ObsDiam_Y(X;-\kappa) &:= \sup\{\;\diam(f_*\mu_X;1-\kappa) \mid\\
    &\qquad\qquad\text{$f : X \to Y$ is $1$-Lipschitz}\;\},\\
    \ObsDiam_Y(X) &:= \inf_{\kappa > 0} \max\{\ObsDiam_Y(X;-\kappa),\kappa\}.
  \end{align*}
  \index{obsdiam@$\ObsDiam_Y(\cdots)$, $\ObsDiam(\cdots)$}
  We call $\ObsDiam_Y(X)$ (resp.~$\ObsDiam_Y(X;-\kappa)$)
  the \emph{observable diameter of $X$ with screen $Y$}
  (resp.~\emph{$\kappa$-observable diameter of $X$ with screen $Y$}).
  \index{screen}
  The case $Y = \R$ is most important and we set
  \begin{align*}
    \ObsDiam(X;-\kappa) &:= \ObsDiam_\R(X;-\kappa),\\
    \ObsDiam(X) &:= \ObsDiam_\R(X).
  \end{align*}
\end{defn}

The observable diameter is invariant under mm-isomorphism.
Note that $\ObsDiam_Y(X;-\kappa) = \diam(X;1-\kappa) = 0$ for $\kappa \ge 1$
and we always have $\ObsDiam_Y(X) \le 1$.
We see that $\diam(\mu_X;1-\kappa)$ and $\ObsDiam_Y(X;-\kappa)$ are
both monotone nonincreasing in $\kappa$.

\begin{defn}[L\'evy family]
  \index{Levy family@L\'evy family}
  A sequence of mm-spaces $X_n$, $n=1,2,\dots$,
  is called a \emph{L\'evy family} if
  \[
  \lim_{n\to\infty} \ObsDiam(X_n) = 0,
  \]
  or equivalently
  \[
  \lim_{n\to\infty} \ObsDiam(X_n;-\kappa) = 0
  \]
  for any $\kappa > 0$.
\end{defn}

It follows from the definition that
$\{X_n\}_{n=1}^\infty$ is a L\'evy family if and only if
\begin{itemize}
\item for any $1$-Lipschitz functions $f_n : X_n \to \R$, $n=1,2,\dots$,
  there exist real numbers $c_n$ such that
  \[
  \lim_{n\to\infty} \dKF(f_n,c_n) = 0.
  \]
\end{itemize}
In particular, L\'evy's lemma (Corollary \ref{cor:Levy}) implies

\begin{thm}
  $\{S^n(1)\}_{n=1}^\infty$ is a L\'evy family.
\end{thm}

\begin{rem}
  $\diam(S^n(1);1-\kappa)$ does not converge to zero
  as $n\to\infty$ for $0 < \kappa < 1$.
\end{rem}

\begin{rem}
  For a L\'evy family $\{X_n\}$,
  the above constant $c_n$ can always be taken to be a median of $f_n$,
  where the width of the interval of medians of $f_n$ shrinks to zero
  as $n\to\infty$.
  In the case where $\diam X_n$ is bounded from above, such as $S^n(1)$,
  the difference between the average and a median of $f_n$ tends to
  zero as $n\to\infty$ for the L\'evy family $\{X_n\}$.
  However, this is not true in general.
  For instance, considering the sequence of the measures
  \[
  \mu_n := (1-1/n)\delta_0+(1/n)\delta_n, \quad n=1,2,\dots,
  \]
  we see that $\{(\R,\mu_n)\}$ is a L\'evy family.
  The map $f_n(x) = x$, $x \in \R$, has $\mu_n$-average $1$ for any $n$,
  but $0$ is the unique median of $f_n$ for $n \ge 3$.
\end{rem}



\begin{prop} \label{prop:diam-ObsDiam-dom}
  Let $X$ and $Y$ be two mm-spaces and $\kappa > 0$ a real number.
  \begin{enumerate}
  \item If $X$ is dominated by $Y$, then
    \[
    \diam(X;1-\kappa) \le \diam(Y;1-\kappa).
    \]
  \item We have
    \[
    \ObsDiam(X;-\kappa) \le \diam(X;1-\kappa).
    \]
  \item If $X$ is dominated by $Y$, then
    \[
    \ObsDiam(X;-\kappa) \le \ObsDiam(Y;-\kappa).
    \]
  \end{enumerate}
\end{prop}

\begin{proof}
  We prove (1).
  Since $X \prec Y$, there is a $1$-Lipschitz map $F : Y \to X$
  such that $F_*\mu_Y = \mu_X$.
  Let $A$ be any Borel subset of $Y$ with $\mu_Y(A) \ge 1-\kappa$
  and $\overline{F(A)}$ the closure of $F(A)$.
  \index{overline@$\overline{\ \cdot\ }$}
  We have $\mu_X(\overline{F(A)}) = \mu_Y(F^{-1}(\overline{F(A)}))
  \ge \mu_Y(A) \ge 1-\kappa$ and,
  by the $1$-Lipschitz continuity of $F$,
  $\diam(\overline{F(A)}) \le \diam A$.
  Therefore, $\diam(X;1-\kappa) \le \diam A$.
  Taking the infimum of $\diam A$ over all $A$'s
  yields (1).

  We prove (2).  Let $f : X \to \R$ be any $1$-Lipschitz
  function.  Since $(\R,|\cdot|,f_*\mu_X)$ is dominated by $X$,
  (1) implies that $\diam(f_*\mu_X;1-\kappa) \le \diam(X;1-\kappa)$.
  This proves (2).

  We prove (3).
  By $X \prec Y$, there is a $1$-Lipschitz map $F : Y \to X$ with
  $F_*\mu_Y = \mu_X$.
  For any $1$-Lipschitz function $f : X \to \R$,
  we have $f_*\mu_X = f_*F_*\mu_Y = (f\circ F)_*\mu_Y$.
  Since $f\circ F : Y \to \R$ is also a $1$-Lipschitz function,
  \[
  \diam(f_*\mu_X;1-\kappa) = \diam((f\circ F)_*\mu_Y;1-\kappa)
  \le \ObsDiam(Y;-\kappa).
  \]
  This completes the proof.
\end{proof}

\begin{prop} \label{prop:scale-ObsDiam}
  Let $X$ be an mm-space.
  Then, for any real number $t > 0$ we have
  \[
  \ObsDiam(tX;-\kappa) = t\ObsDiam(X;-\kappa),
  \]
  where $tX := (X,td_X,\mu_X)$. \index{tX@$tX$}
\end{prop}

\begin{proof}
  We have
  \begin{align*}
    &\ObsDiam(tX;-\kappa)\\
    &= \sup\{\;\diam(f_*\mu_X;1-\kappa) \mid
    \text{$f : tX \to \R$ $1$-Lipschitz}\;\}\\
    &= \sup\{\;\diam(f_*\mu_X;1-\kappa) \mid
    \text{$t^{-1}f : X \to \R$ $1$-Lipschitz}\;\}\\
    &= \sup\{\;\diam((tg)_*\mu_X;1-\kappa) \mid
    \text{$g : X \to \R$ $1$-Lipschitz}\;\}\\
    &= t\ObsDiam(X;-\kappa).
  \end{align*}
  This completes the proof.
\end{proof}

Denote the \emph{$l_\infty$ norm} on $\R^N$ by $\|\cdot\|_\infty$, i.e.,
\[
\|x\|_\infty := \max_{i=1}^N |x_i|
\]
for $x = (x_1,x_2,\dots,x_N) \in \R^N$.
\index{bar infinity@$\Vert\cdot\Vert_\infty$}
\index{l infinity norm@$l_\infty$ norm}

\begin{lem} \label{lem:ObsDiamRN-ObsDiam}
  Let $X$ be an mm-space.
  For any real number $\kappa > 0$ and any natural number $N$,
  we have
  \[
  \ObsDiam_{(\R^N,\|\cdot\|_\infty)}(X;-N\kappa)
  \le \ObsDiam(X;-\kappa).
  \]
\end{lem}

\begin{proof}
  Assume $\ObsDiam(X;-\kappa) < \varepsilon$ for a number
  $\varepsilon$.
  Let $F : X \to (\R^N,\|\cdot\|_\infty)$ be any $1$-Lipschitz map.
  By setting $(f_1,f_2,\dots,f_N) := F$,
  each $f_i$ is $1$-Lipschitz continuous and so
  $\diam((f_i)_*\mu_X;1-\kappa) < \varepsilon$.
  There is a Borel subset $A_i \subset \R$ for each $i$ such that
  $(f_i)_*\mu_X(A_i) \ge 1-\kappa$ and $\diam A_i < \varepsilon$.
  Letting $A := A_1 \times A_2 \times \dots \times A_N$,
  we have
  \[
  F_*\mu_X(A) = \mu_X(F^{-1}(A))
  = \mu_X(f_1^{-1}(A_1) \cap \dots \cap f_N^{-1}(A_N))
  \ge 1-N\kappa.
  \]
  Since $\diam A < \varepsilon$,
  we have $\diam(F_*\mu_X;1-N\kappa) < \varepsilon$.
  This completes the proof.
\end{proof}

\begin{thm} \label{thm:ObsDiam-Sn}
  For any real number $\kappa$ with $0 < \kappa < 1$,
  we have
  \begin{align}
    \tag{1}
    \lim_{n\to\infty} \ObsDiam(S^n(\sqrt{n});-\kappa)
    &= \diam(\gamma^1;1-\kappa) = 2I^{-1}((1-\kappa)/2),\\
    \tag{2}
    \ObsDiam(S^n(1);-\kappa) &= O(n^{-1/2}),
  \end{align}
  where $I(r) := \gamma^1[\,0,r\,]$.
\end{thm}

The theorem implies the L\'evy property of $\{S^n(1)\}_{n=1}^\infty$.

\begin{proof}
  (1) follows from the normal law \`a la L\'evy (Theorem \ref{thm:normal})
  and the Maxwell-Boltzmann distribution law (Proposition \ref{prop:MB-law}).

  We prove (2).
  By Proposition \ref{prop:scale-ObsDiam} we have
  \[
  \ObsDiam(S^n(1);-\kappa) = \frac{1}{\sqrt{n}}\ObsDiam(S^n(\sqrt{n});\kappa),
  \]
  which together with (1) implies (2).
  This completes the proof.
\end{proof}

Combining Theorem \ref{thm:ObsDiam-Sn} and
Proposition \ref{prop:scale-ObsDiam} proves the following

\begin{cor}[\cite{GroMil}*{\S 1.1}] \label{cor:ObsDiam-Sn}
  Let $r_n$, $n=1,2,\dots$, be positive real numbers.
  Then, $\{S^n(r_n)\}$ is a L\'evy family if and only if
  $r_n/\sqrt{n} \to 0$ as $n\to\infty$.
\end{cor}

\begin{ex} \label{ex:CPn}
  We see that the Hopf fibration $f_n : S^{2n+1}(1) \to \C P^n$ is $1$-Lipschitz
  continuous with respect to the Fubini-Study metric on $\C P^n$, and
  that the push-forward $(f_n)_*\sigma^{2n+1}$ coincides with
  the normalized volume measure on $\C P^n$ induced from the Fubini-Study
  metric.
  This together with Proposition \ref{prop:diam-ObsDiam-dom}(3) implies
  \[
  \ObsDiam(\C P^n;-\kappa) \le \ObsDiam(S^{2n+1};-\kappa) = O(n^{-1/2})
  \]
  for any $\kappa$ with $0 < \kappa < 1$.
  In particular, $\{\C P^n\}_{n=1}^\infty$ is a L\'evy family.
\end{ex}

\section{Separation distance}

\begin{defn}[Separation distance]
  \index{separation distance} \index{sep@$\Sep(\cdots)$}
  Let $X$ be an mm-space.
  For any real numbers $\kappa_0,\kappa_1,\cdots,\kappa_N > 0$
  with $N\geq 1$,
  we define the \emph{separation distance}
  \[
  \Sep(X;\kappa_0,\kappa_1, \cdots, \kappa_N)
  \]
  of $X$ as the supremum of $\min_{i\neq j} d_X(A_i,A_j)$
  over all sequences of $N+1$ Borel subsets $A_0,A_2, \cdots, A_N \subset X$
  satisfying that $\mu_X(A_i) \geq \kappa_i$ for all $i=0,1,\cdots,N$,
  where $d_X(A_i,A_j) := \inf_{x\in A_i,y\in A_j} d_X(x,y)$.
  If there exists no sequence $A_0,\dots,A_N \subset X$
  with $\mu_X(A_i) \ge \kappa_i$, $i=0,1,\cdots,N$, then
  we define
  \[
  \Sep(X;\kappa_0,\kappa_1, \cdots, \kappa_N) := 0.
  \]
\end{defn}

We see that $\Sep(X;\kappa_0,\kappa_1, \cdots, \kappa_N)$ is
monotone nonincreasing in $\kappa_i$ for each $i=0,1,\dots,N$.
The separation distance is an invariant under mm-isomorphism.

\begin{lem} \label{lem:Sep-prec}
  Let $X$ and $Y$ be two mm-spaces.
  If $X$ is dominated by $Y$, then we have,
  for any real numbers $\kappa_0,\dots,\kappa_N > 0$,
  \[
  \Sep(X;\kappa_0,\dots,\kappa_N) \le \Sep(Y;\kappa_0,\dots,\kappa_N).
  \]
\end{lem}

\begin{proof}
  If we assume $X \prec Y$, then there is a $1$-Lipschitz map
  $f : Y \to X$ such that $f_*\mu_Y = \mu_X$.
  We take any Borel subsets $A_0,A_1,\dots,A_N \subset X$
  such that $\mu_X(A_i) \ge \kappa_i$ for any $i$.
  If we have no such sequence of Borel subsets, the the lemma is trivial.
  Since $f$ is $1$-Lipschitz continuous and since
  $\mu_Y(f^{-1}(A_i)) = \mu_X(A_i) \ge \kappa_i$, we have
  \[
  \min_{i\neq j} d_X(A_i,A_j)
  \le \min_{i\neq j} d_Y(f^{-1}(A_i),f^{-1}(A_j))
  \le \Sep(Y;\kappa_0,\dots,\kappa_N).
  \]
  This completes the proof.
\end{proof}

\begin{prop} \label{prop:ObsDiam-Sep}
  For any mm-space $X$ and any real numbers $\kappa$ and $\kappa'$
  with $\kappa > \kappa' > 0$, we have
  \begin{align}
    \tag{1} &\ObsDiam(X;-2\kappa) \le \Sep(X;\kappa,\kappa),\\
    \tag{2} &\Sep(X;\kappa,\kappa) \le \ObsDiam(X;-\kappa').
  \end{align}
\end{prop}

\begin{proof}
  We prove (1).
  If $\kappa \ge 1/2$, then the left-hand side of (1)
  becomes zero and (1) is trivial.

  We assume $\kappa < 1/2$.
  Let $f : X \to \R$ be any $1$-Lipschitz function and set
  \begin{align*}
    \rho_- &:= \sup\{\;t \in \R \mid f_*\mu_X(\,-\infty,t\,) \le \kappa\;\},\\
    \rho_+ &:= \inf\{\;t \in \R \mid f_*\mu_X(\,t,+\infty\,) \le \kappa\;\}.
  \end{align*}
  We then see that
  \begin{align*}
    & f_*\mu_X(\,-\infty,\rho_-\,) \le \kappa, \qquad
    f_*\mu_X(\,-\infty,\rho_-\,] \ge \kappa,\\
    & f_*\mu_X(\,\rho_+,+\infty\,) \le \kappa, \qquad
    f_*\mu_X[\,\rho_+,+\infty\,) \ge \kappa,\\
    & \qquad\qquad\qquad\text{and}\quad \rho_- \le \rho_+.
  \end{align*}
  Since $f_*\mu_X[\,\rho_-,\rho_+\,] \ge 1-2\kappa$, we have
  \begin{align*}
    \diam(f_*\mu_X;1-2\kappa) &\le \rho_+ - \rho_-
    = d_{\R}((\,-\infty,\rho_-\,],[\,\rho_+,+\infty\,))\\
    &\le \Sep((\R,f_*\mu_X);\kappa,\kappa)
    \le \Sep(X;\kappa,\kappa),
  \end{align*}
  where the last inequality follows from Lemma \ref{lem:Sep-prec}.
  This proves (1).

  We prove (2).
  Let $A_i \subset X$, $i=1,2$, be any Borel subsets
  with $\mu_X(A_i) \ge \kappa$.
  We define $f(x) := d_X(x,A_1)$, $x \in X$.
  Then, $f : X \to \R$ is a $1$-Lipschitz function.
  Let us estimate $\diam(f_*\mu_X;1-\kappa')$.
  Any interval $I \subset \R$ with $f_*\mu_X(I) \ge 1-\kappa'$
  intersects both $f(A_1)$ and $f(A_2)$,
  because
  \[
  f_*\mu_X(f(A_i)) + f_*\mu_X(I) \ge \mu_X(A_i) + f_*\mu_X(I)
  \ge \kappa + 1-\kappa' > 1
  \]
  for $i=1,2$.
  Therefore,
  \[
  \diam I \ge d_{\R}(f(A_1),f(A_2)) = \inf_{x\in A_2} f(x) = d_X(A_1,A_2)
  \]
  and so $\ObsDiam(X;-\kappa') \ge \diam(f_*\mu_X;1-\kappa') \ge d_X(A_1,A_2)$.
  This completes the proof of the proposition.
\end{proof}

\begin{rem}
  Proposition \ref{prop:ObsDiam-Sep}(2) does not hold
  for $\kappa = \kappa'$.
  In fact, let us consider an mm-space $X := \{x_1,x_2\}$ with
  $\mu_X := (1/2)\delta_{x_1} + (1/2)\delta_{x_2}$.
  Then we have
  \begin{align*}
    \Sep(X;1/2,1/2) &= d_X(x_1,x_2),\\
    \ObsDiam(X;-1/2) &= \diam(X;1/2) = 0.
  \end{align*}
\end{rem}

It follows from Proposition \ref{prop:ObsDiam-Sep} that
$\{X_n\}$ is a L\'evy family if and only if
\begin{itemize}
\item for any Borel subsets $A_{in} \subset X_n$, $i=1,2$, $n=1,2,\dots$,
  such that $\liminf_{n\to\infty} \mu_{X_n}(A_{in}) > 0$, we have
  \[
  \lim_{n\to\infty} d_{X_n}(A_{1n},A_{2n}) = 0.
  \]
\end{itemize}
It is easy to see that this is also equivalent to
\begin{itemize}
\item for any Borel subsets $A_n \subset X_n$, $n=1,2,\dots$, such that\\
  $\liminf_{n\to\infty} \mu_{X_n}(A_n) > 0$ and
  $\limsup_{n\to\infty} \mu_{X_n}(A_n) < 1$, we have
  \[
  \lim_{n\to\infty} \mu_{X_n}(B_\varepsilon(A_n) \cap B_\varepsilon(X_n \setminus A_n)) = 1
  \]
  for any $\varepsilon > 0$.
\end{itemize}
In the case where $X_n$ are Riemannian manifolds, the L\'evy property
of $\{X_n\}$ means that
the measure $\mu_{X_n}$ on $X_n$ concentrates around the boundary of any such
$A_n$ for large $n$.

\begin{rem}
  Let $X$ be an mm-space.
  One more well-known invariant is
  the \emph{concentration function} $\alpha_X(r)$, $r > 0$,
  \index{concentration function} \index{alphaX@$\alpha_X$}
  defined by
  \[
  \alpha_X(r) := \sup\{\;1-\mu_X(U_r(A)) \mid A \subset X : \text{Borel},
  \ \mu_X(A) \ge 1/2\;\}.
  \]
  We have
  \begin{align}
    \tag{1} &\ObsDiam(X;-\kappa)
    \le 2\,\inf\{\; r > 0 \mid \alpha_X(r) \le \kappa/2\;\},\\
    \tag{2} &\alpha_X(r)
    \le \sup\{\;\kappa > 0 \mid \ObsDiam(X;-\kappa) \ge r\;\}.
  \end{align}
  We refer to \cite{Ldx:book}*{Proposition 1.12} for the proof of (1).
  We now prove (2).
  Let $\varepsilon > 0$ be any small real number.
  There is a Borel subset $A \subset X$ such that
  $\mu_X(A) \ge 1/2$ and $\mu_X(X \setminus U_r(A)) > \alpha_X(r)-\varepsilon$.
  Since $d_X(A,X \setminus U_r(A)) \ge r$,
  we have $\Sep(X;1/2,\alpha_X(r)-\varepsilon) \ge r$.
  By Proposition \ref{prop:ObsDiam-Sep},
  $\ObsDiam(X;-(\alpha_X(r)-2\varepsilon)) \ge r$, so that
  $\alpha_X(r)-2\varepsilon$ is not greater than the right-hand side
  of (2) for any small $\varepsilon > 0$.  This proves (2).
\end{rem}

\section{Comparison theorem for observable diameter}

In this section, we prove the following theorem by using
the L\'evy-Gromov isoperimetric inequality.
A manifold is said to be \emph{closed} if it is compact and has no boundary.
\index{closed manifold}
$\Ric_X$ denotes the Ricci curvature of a Riemannian manifold $X$.
\index{Ric@$\Ric_X$}

\begin{thm} \label{thm:comp}
  Let $X$ be a closed and connected $n$-dimensional Riemannian manifold
  of $\Ric_X \ge n-1$, $n \ge 2$.
  Then, for any $\kappa$ with $0 < \kappa \le 1$, we have
  \begin{align*}
  \ObsDiam(X;-\kappa) &\le \ObsDiam(S^n(1);-\kappa) = \pi-2v^{-1}(\kappa/2)\\
  &\le \frac{2\sqrt{2}}{\sqrt{n-1}} \sqrt{-\log\left(\sqrt{\frac{2}{\pi}}\kappa\right)},
  \end{align*}
  where
  \[
  v(r) := \frac{\int_0^r \sin^{n-1} t \; dt}{\int_0^\pi \sin^{n-1} t \; dt}
  \]
  is the $\sigma^n$-measure of a metric ball of radius $r$ in $S^n(1)$.
\end{thm}

\begin{rem}
  Theorem \ref{thm:comp} also leads to Theorem \ref{thm:ObsDiam-Sn}.
\end{rem}

Throughout this section, let $X$ be a closed and connected $n$-dimensional
Riemannian manifold
of $\Ric_X \ge n-1$, $n \ge 2$, with the normalized volume measure $\mu_X$,
and let $0 < \kappa \le 1$.

The following is a generalization of L\'evy's isoperimetric inequality
(Theorem \ref{thm:Levy-isop}).

\begin{thm}[L\'evy-Gromov isoperimetric inequality
\cite{Gromov}*{Appendix $C_+$}] 
  For any closed subset $\Omega \subset X$, we take a metric ball $B_\Omega$
  of $S^n(1)$ with $\sigma^n(B_\Omega) = \mu_X(\Omega)$.  \label{thm:LGisop}
  Then we have
  \[
  \mu_X(U_r(\Omega)) \ge \sigma^n(U_r(B_\Omega))
  \]
  for any $r > 0$.
\end{thm}

Using this theorem we prove

\begin{lem} \label{lem:comp1}
  We have
  \[
  \Sep(X;\kappa/2,\kappa/2) \le \Sep(S^n(1);\kappa/2,\kappa/2)
  = \pi - 2v^{-1}(\kappa/2).
  \]
\end{lem}

\begin{proof}
  Since the distance between the metric ball centered at the north pole of $S^n(1)$
  with $\sigma^n$-measure $\kappa/2$ and the metric ball centered at the
  south pole of $S^n(1)$ with $\sigma^n$-measure $\kappa/2$
  is equal to $\pi - 2v^{-1}(\kappa/2)$,
  we see that
  \[
  \Sep(S^n(1);\kappa/2,\kappa/2) \ge \pi - 2v^{-1}(\kappa/2).
  \]
  The rest of the proof is to show
  $\Sep(X;\kappa/2,\kappa/2)  \le \pi - 2v^{-1}(\kappa/2)$.
  
  Let $\Omega, \Omega' \subset X$ be two mutually disjoint closed subsets
  such that $\mu_X(\Omega) = \mu_X(\Omega') = \kappa/2$,
  and let $r := d_X(\Omega,\Omega')$.
  It suffices to prove that
  $r \le \pi - 2v^{-1}(\kappa/2)$.
  Since $\Omega'$ and $U_r(\Omega)$ are disjoint to each other,
  the L\'evy-Gromov isoperimetric inequality (Theorem \ref{thm:LGisop}) proves
  \begin{align*}
    &\frac{\kappa}{2} = \mu_X(\Omega') \le 1 - \mu_X(U_r(\Omega))\\
    &\le 1 - \sigma^n(U_r(B_\Omega))
    = 1 - v(r+v^{-1}(\kappa/2))
  \end{align*}
  and hence
  \[
  r + v^{-1}(\kappa/2) \le v^{-1}(1-\kappa/2) = \pi - v^{-1}(\kappa/2).
  \]
  This completes the proof.
\end{proof}

\begin{lem} \label{lem:comp2}
  We have
  \[
  \ObsDiam(S^n(1);-\kappa) = \Sep(S^n(1);\kappa/2,\kappa/2).
  \]
\end{lem}

\begin{proof}
  Proposition \ref{prop:ObsDiam-Sep}(1) implies
  \[
  \ObsDiam(S^n(1);-\kappa) \le \Sep(S^n(1);\kappa/2,\kappa/2).
  \]
  We prove the reverse inequality.
  Let $f(x) := d_{S^n(1)}(x_0,x)$, $x \in S^n(1)$, be the distance function
  from a fixed point $x_0 \in S^n(1)$, where $d_{S^n(1)}$
  is the geodesic distance function on $S^n(1)$.
  We see
  \[
  \frac{df_*\sigma^n}{dx}(r)
  = \frac{\sin^{n-1}r}{\int_0^\pi \sin^{n-1}t \; dt},
  \]
  which implies $\diam(f_*\sigma^n;1-\kappa) = \pi - 2v^{-1}(\kappa/2)$.
  Therefore,
  \[
  \ObsDiam(S^n(1);-\kappa) \ge \pi - 2v^{-1}(\kappa/2)
  = \Sep(S^n(1);\kappa/2,\kappa/2)
  \]
  by Lemma \ref{lem:comp1}.
\end{proof}

Combining Proposition \ref{prop:ObsDiam-Sep},
Lemmas \ref{lem:comp1} and \ref{lem:comp2} yields
\[
\ObsDiam(X;-\kappa) \le \ObsDiam(S^n(1);-\kappa) = \pi-2v^{-1}(\kappa/2).
\]
The rest is to estimate $\pi-2v^{-1}(\kappa/2)$ from above.
For that, we need the following

\begin{lem} \label{lem:comp-tail}
  We have
  \[
  \int_r^\infty \frac{1}{\sqrt{2\pi}} e^{-\frac{t^2}{2}} \; dt \le \frac{1}{2} e^{-\frac{r^2}{2}}.
  \]
  for any $r \ge 0$.
\end{lem}

\begin{proof}
  Let
  \[
  f(r) :=  \frac{1}{2} e^{-\frac{r^2}{2}}
  - \int_r^\infty \frac{1}{\sqrt{2\pi}} e^{-\frac{t^2}{2}}\;dt.
  \]
  Then, $f(0) = f(\infty) = 0$, $f'(r) \ge 0$ for $0 \le r \le 2/\sqrt{2\pi}$,
  and $f'(r) \le 0$ for $r \ge 2/\sqrt{2\pi}$.
  This completes the proof.
\end{proof}

The following lemma completes the proof of Theorem \ref{thm:comp}.

\begin{lem}
  \[
  \pi-2v^{-1}(\kappa/2)
  \le \frac{2\sqrt{2}}{\sqrt{n-1}} \sqrt{-\log\left(\sqrt{\frac{2}{\pi}}\kappa\right)}.
  \]
\end{lem}

\begin{proof}
  Setting $r := \pi/2 - v^{-1}(\kappa/2)$ and
  $w_n := \int_0^\pi \sin^{n-1} t \; dt$, we see
  \begin{align*}
  \frac{\kappa}{2} &= \frac{1}{w_n} \int_{\pi/2+r}^\pi \sin^{n-1} t\;dt
  = \frac{1}{w_n} \int_r^{\pi/2} \cos^{n-1} t\;dt\\
  &= \frac{1}{w_n\sqrt{n-1}} \int_{r\sqrt{n-1}}^{(\pi/2)\sqrt{n-1}}
  \cos^{n-1}\frac{s}{\sqrt{n-1}}\;ds
  \intertext{and, by $\cos x \le e^{-\frac{x^2}{2}}$ for $0 \le x \le \pi/2$
  and by Lemma \ref{lem:comp-tail},}
  &\le \frac{1}{w_n\sqrt{n-1}} \int_{r\sqrt{n-1}}^\infty e^{-\frac{s^2}{2}}\;ds
  \le \frac{\sqrt{\pi}}{w_n\sqrt{2(n-1)}} e^{-\frac{n-1}{2}r^2}.
  \end{align*}
  Let $n \ge 2$.
  By the integration by parts, we see
  $w_{n+2} = n(n+1)^{-1} w_n \ge \sqrt{(n-1)(n+1)^{-1}}\, w_n$,
  so that $\sqrt{n+1}\, w_{n+2} \ge \sqrt{n-1}\, w_n$.
  This together with
  $w_2 = 2$ and $\sqrt{2} \, w_3 = \pi/\sqrt{2} \ge 2$
  implies $\sqrt{n-1} \, w_n \ge 2$.
  Therefore,
  \[
  \kappa \le \frac{\sqrt{\pi}}{\sqrt{2}} e^{-\frac{n-1}{2}r^2},
  \]
  which proves the lemma.
\end{proof}

We have the following corollary to Theorem \ref{thm:comp}.

\begin{cor}
  Let $n \ge 2$.
  If a closed and connected $n$-dimensional Riemannian manifold $X$
  satisfies $\Ric_X \ge K$ for a constant $K > 0$, then
  \[
  \ObsDiam(X;-\kappa) \le \sqrt{\frac{n-1}{K}} \ObsDiam(S^n(1);-\kappa)
  = O(K^{-1/2})
  \]
  for any $\kappa > 0$.
\end{cor}

\begin{proof}
  Assume that $\Ric_X \ge K > 0$.
  Setting $t := \sqrt{K/(n-1)}$, we have $\Ric_{tX} \ge K/t^2 = n-1$.
  Apply Theorem \ref{thm:comp} for $tX$
  and use Proposition \ref{prop:scale-ObsDiam}.
\end{proof}

\begin{ex} \label{ex:SOSUSp}
  Let $SO(n)$, $SU(n)$, and $Sp(n)$ be the special orthogonal group,
  the special unitary group, and the compact symplectic group, respectively.
  For $X_n = SO(n),SU(n),Sp(n)$ we have
  \[
  \Ric_{X_n} = \left(\frac{\beta(n+2)}{4}-1\right)g_{X_n},
  \]
  where $\beta = 1$ for $X_n = SO(n)$,
  $\beta = 2$ for $X_n = SU(n)$, and $\beta = 4$ for $X_n = Sp(n)$
  (see \cite{AGZ}*{(F.6)}).
  Therefore,
  $\{SO(n)\}_{n=1}^\infty$, $\{SU(n)\}_{n=1}^\infty$, and $\{Sp(n)\}_{n=1}^\infty$
  are all L\'evy families with
  \[
  \ObsDiam(X_n;-\kappa) \le O(n^{-1/2}).
  \]
\end{ex}

\section{Spectrum of Laplacian and separation distance}
\label{sec:spec-sep}

The first nonzero eigenvalue of the Laplacian is useful to
detect L\'evy families.
Let $X$ be a compact Riemannian manifold and
$\Delta$ the (nonnegative) Laplacian on $X$.
We equip $X$ with the volume measure $\mu_X$ normalized as $\mu_X(X) =1$
as always.
It is known that the spectrum of $\Delta$ consists of
eigenvalues \index{lambdakX@$\lambda_k(X)$} \index{lambda1X@$\lambda_1(X)$}
\[
0 = \lambda_0(X) \le \lambda_1(X) \le \lambda_2(X) \le \dots
\le \lambda_k(X) \le \dots,
\]
(with multiplicity),
and $\lambda_k(X)$ is divergent to infinity as $k\to\infty$.
Denote by $L_2(X)$ \index{L2X@$L_2(X)$}
the space of square integrable functions on $X$ with respect to $\mu_X$,
and by $\Lip(X)$ the set of Lipschitz functions on $X$.
A min-max principle \index{min-max principle} says that
\begin{align*}
  \lambda_k(X) = \inf_L \sup_{u \in L \setminus \{0\}} R(u),
\end{align*}
where $L$ runs over all $(k+1)$-dimensional linear subspaces of
$L_2(X) \cap \Lip(X)$ and
$R(u) := \|\grad u\|_{L_2}^2/\|u\|_{L_2}^2$ is the \emph{Rayleigh quotient}.
\index{Rayleigh quotient}
Note that any Lipschitz function on $X$ is differentiable a.e.~by
Rademacher's theorem, so that its gradient vector filed is defined a.e.~on $X$.

The compactness of $X$ implies that $X$ has at most finitely many
connected components.
The number of connected components of $X$ coincides with
the smallest number $k$ with $\lambda_k(X) > 0$.
In particular, $X$ is connected if and only if $\lambda_1(X) > 0$.
Note that the Riemannian distance between two
different connected components is defined to be infinity and that
the separation distance takes values in $[\,0,+\infty\,]$.

\begin{prop} \label{prop:lamk-Sep}
  Let $X$ be a compact Riemannian manifold.
  Then we have
  \[
   \Sep(X;\kappa_0,\kappa_1,\dots,\kappa_k)
  \le \frac{2}{\sqrt{\lambda_k(X)\min_{i=0,1,\dots,k}\kappa_i}}
  \]
  for any $\kappa_0,\kappa_1,\dots,\kappa_k > 0$.
\end{prop}

We see a more refined estimate in \cite{CGY}.

\begin{proof}
  We set $S := \Sep(X;\kappa_0,\dots,\kappa_k)$ for simplicity.
  If $S = 0$, then the proposition is trivial.
  Assume $S > 0$.
  Let $r$ be any number with $0 < r < S$.
  There are Borel subsets $A_0,\dots,A_k \subset X$
  such that $\mu_X(A_i) \ge \kappa_i$ and $d_X(A_i,A_j) > r$
  for any different $i$ and $j$.
  Let
  \[
  f_i(x) := \max\{1-\frac{2}{r}d_X(x,A_i),0\}, \quad x \in X.
  \]
  $f_i$ are Lipschitz functions on $X$ and are differentiable
  a.e.~on $X$.
  We see that $f_0,f_1,\dots,f_k$ are $L_2$ orthogonal to each other,
  $|\grad f_i| \le 2/r$ a.e., and $\|f_i\|_{L_2}^2 \ge \kappa_i$.
  Denote by $L_0$ the linear subspace of $L_2(X)$ generated by
  $f_0,\dots,f_k$.  We obtain
  \begin{align*}
    \lambda_k(X) = \inf_{L\; :\;\dim L = k+1} \sup_{u \in L \setminus\{0\}} R(u)
    \le \sup_{u \in L_0\setminus\{0\}} R(u),
  \end{align*}
  where
  $R(u) := \|\grad u\|_{L_2}^2/\|u\|_{L_2}^2$.
  It is easy to prove that, for any function
  $u = \sum_{i=0}^k a_i f_i \in L_0 \setminus \{0\}$,
  \[
  \|u\|_{L_2}^2 \ge \min_{j=0,\dots,k} \kappa_j \sum_{i=0}^k a_i^2
  \quad\text{and}\quad
  \|\grad u\|_{L_2}^2 \le \frac{4}{r^2} \sum_{i=0}^k a_i^2,
  \]
  so that $R(u) \le 4/(r^2\min_i\kappa_i)$.
  Therefore, $\lambda_k(X) \le 4/(r^2\min_i\kappa_i)$.
  By the arbitrariness of $r$, we obtain the proposition.
\end{proof}

Propositions \ref{prop:ObsDiam-Sep} and \ref{prop:lamk-Sep}
together imply

\begin{cor} \label{cor:ObsDiam-spec}
  Let $X$ be a compact Riemannian manifold.
  For any $\kappa > 0$ we have
  \[
  \ObsDiam(X;-2\kappa) \le \Sep(X;\kappa,\kappa)
  \le \frac{2}{\sqrt{\lambda_1(X)\,\kappa}}.
  \]
  In particular, if $\{X_n\}$ is a sequence of compact Riemannian manifolds
  such that $\lambda_1(X_n) \to +\infty$ as $n\to\infty$,
  then $\{X_n\}$ is a L\'evy family.
\end{cor}

\begin{rem}
  It is a famous theorem of Lichnerowicz that
  if the Ricci curvature of a closed and connected $n$-dimensional
  Riemannian manifold $X$
  satisfies $\Ric_X \ge n-1$, $n \ge 2$, then
  \[
  \lambda_1(X) \ge n.
  \]
  This together with Corollary \ref{cor:ObsDiam-spec} implies
  \[
  \ObsDiam(X;-2\kappa) \le \frac{2}{\sqrt{n\kappa}},
  \]
  which is weaker than Theorem \ref{thm:comp} for small $\kappa > 0$.
\end{rem}

\chapter{Gromov-Hausdorff distance
  and distance matrix}
\label{chap:GH-dist-matrix}

\section{Net, covering number, and capacity}

In this and next sections, we define some of the basic terminologies
from metric geometry.

Assume that $X$ is a metric space.

\begin{defn}[Net, $\varepsilon$-covering number, and $\varepsilon$-capacity]
  \index{net} \index{epsilon-covering number@$\varepsilon$-covering number}
  \index{epsilon-capacity@$\varepsilon$-capacity}
  We call a discrete subset of $X$ a \emph{net of $X$}.
  Let $\varepsilon > 0$.
  A net $\cN$ of $X$ is called an \emph{$\varepsilon$-net}
  \index{epsilon-net@$\varepsilon$-net}
  if $B_\varepsilon(\cN) = X$.
  A net $\cN$ of $X$ is said to be \emph{$\varepsilon$-discrete}
  \index{epsilon-discrete net@$\varepsilon$-discrete net}
  if $d_X(x,y) > \varepsilon$ for any different
  $x,y \in \cN$.
  Define
  \begin{align*}
    \Cov_\varepsilon(X) &:= \inf\{\; \#\cN \mid
    \text{$\cN$ is an $\varepsilon$-net of $X$}\;\},\\
    \Cap_\varepsilon(X) &:= \sup\{\; \#\cN \mid
    \text{$\cN$ is an $\varepsilon$-discrete net of $X$}\;\},
  \end{align*}
  where $\#\cN$ is the number of points in $\cN$.
  $\Cov_\varepsilon(X)$ and $\Cap_\varepsilon(X)$ are respectively
  called the \emph{$\varepsilon$-covering number}
  and the \emph{$\varepsilon$-capacity} of $X$.
\end{defn}

\begin{lem}
  For any $\varepsilon > 0$ we have
  \[
  \Cap_{2\varepsilon}(X) \le \Cov_\varepsilon(X) \le \Cap_\varepsilon(X).
  \]
\end{lem}

\begin{proof}
  Since a maximal $\varepsilon$-discrete net of $X$ is an $\varepsilon$-net,
  we have $\Cov_\varepsilon(X) \le \Cap_\varepsilon(X)$.

  To prove the other inequality, we take any $2\varepsilon$-discrete net
  $\cN$ of $X$ and any $\varepsilon$-net $\cN'$ of $X$.
  Since any ball of radius $\varepsilon$ covers at most one point in
  $\cN$, we have $\#\cN \le \#\cN'$,
  which implies $\Cap_{2\varepsilon}(X) \le \Cov_\varepsilon(X)$.
\end{proof}

\begin{defn}[$\varepsilon$-Projection and nearest point projection]
  \index{epsilon-projection@$\varepsilon$-projection}
  \index{nearest point projection}
  Let $A \subset X$ be a subset and let $\varepsilon \ge 0$.
  A map $\pi : X \to A$ is called an \emph{$\varepsilon$-projection to $A$}
  if
  \[
  d_X(x,\pi(x)) \le d_X(x,A) + \varepsilon
  \]
  for any $x \in X$.
  A $0$-projection to $A$ is called a \emph{nearest point projection to $A$}.
\end{defn}

\begin{lem} \label{lem:Borel-proj}
  For any finite net $\cN \subset X$,
  there exists a Borel measurable nearest point projection to $\cN$.
\end{lem}

\begin{proof}
  We set $\{a_i\}_{i=1}^N := \cN$.
  For a given $x \in X$, let $\{a_{i_j}\}_{j=1}^k$, $k \le N$, be
  the set of nearest points in $\cN$ to $x$, and let
  \[
  i(x) := \min_{j=1}^k i_j,\qquad \pi(x) := a_{i(x)}.
  \]
  This defines a nearest point projection $\pi : X \to \cN$.

  We prove that $\pi$ is Borel measurable.
  In fact, for any $a_i \in \cN$,
  \begin{align*}
    \pi^{-1}(a_i) = \{\;x \in X \mid
    \ &d_X(x,a_i) < d_X(x,a_j)\ \text{for any $j < i$},\\
    &d_X(x,a_i) \le d_X(x,a_j)\ \text{for any $j > i$}\;\}
  \end{align*}
  is a Borel subset of $X$, which implies the Borel measurability of
  $\pi : X \to \cN$.
\end{proof}

\begin{lem} \label{lem:eps-proj}
  Let $X$ be a separable metric space.
  For any subset $A \subset X$ and any $\varepsilon > 0$,
  there exists a Borel measurable $\varepsilon$-projection $\pi$ to $A$
  such that, if $A$ is a Borel subset of $X$, then $\pi|A = \id_A$.
\end{lem}

\begin{proof}
  Since $X$ is separable, there is a
  a dense countable subset $\{x_i\}_{i=1}^\infty \subset X$.
  For each $x_i$ there is a point $a_i \in A$ such that
  $d_X(x_i,a_i) \le d_X(x_i,A) + \varepsilon/3$.
  By setting
  \[
  B_1 := B_{\varepsilon/3}(x_1) \quad\text{and}\quad
  B_{i+1} := B_{\varepsilon/3}(x_{i+1}) \setminus \bigcup_{j=1}^i B_{\varepsilon/3}(x_j)
  \]
  for $i \ge 1$, the sequence $\{B_i\}_{i=1}^\infty$ is a disjoint covering of $X$.
  For $i$ with $B_i \neq \emptyset$ and for $x \in B_i$,
  we define $\pi(x) := a_i$ to obtain
  a Borel measurable $\varepsilon$-projection $\pi : X \to A$.
  In fact, the preimage of any set by $\pi$ is a Borel subset of $X$.

  If $A$ is a Borel subset of $X$, then we obtain a desired
  $\varepsilon$-projection $\pi'$ by
  \[
  \pi'(x) :=
  \begin{cases}
    \pi(x) &\text{if $x \in X \setminus A$},\\
    x &\text{if $x \in A$,}
  \end{cases}
  \]
  which is Borel measurable.
  This completes the proof.
\end{proof}

\section{Hausdorff and Gromov-Hausdorff distance}

\begin{defn}[Hausdorff distance]
  \index{Hausdorff distance} \index{dH@$d_H$}
  Let $Z$ be a metric space and $X, Y \subset Z$ two subsets.
  The \emph{Hausdorff distance $d_H(X,Y)$ between $X$ and $Y$}
  is defined to be the infimum of $\varepsilon \ge 0$ such that
  \[
  X \subset B_\varepsilon(Y) \quad\text{and}\quad Y \subset B_\varepsilon(X).
  \]
\end{defn}

\begin{lem}[cf. \cite{BBI}*{\S 7.3}] \label{lem:dH}
  For a metric space $Z$ we denote by $\cF(Z)$ the set of closed
  subsets of $Z$.
  \begin{enumerate}
  \item If $Z$ is complete, then so is $(\cF(Z),d_H)$.
  \item If $Z$ is compact, then so is $(\cF(Z),d_H)$.
  \end{enumerate}
\end{lem}

\begin{defn}[Gromov-Hausdorff distance]
  \index{Gromov-Hausdorff distance}
  \index{dGH@$d_{GH}$}
  Let $X$ and $Y$ be two compact metric spaces.
  We embed $X$ and $Y$ into some metric space $Z$ isometrically
  and define the \emph{Gromov-Hausdorff distance $d_{GH}(X,Y)$
    between $X$ and $Y$}
  to be the infimum of $d_H(X,Y)$ over all such $Z$ and
  all such isometric embeddings $X, Y \hookrightarrow Z$.
\end{defn}

\begin{lem}[cf. \cite{BBI}*{\S 7.3}] \label{lem:dGH-complete}
  The Gromov-Hausdorff metric $d_{GH}$ is a complete metric on
  the set of isometry classes of compact metric spaces.
\end{lem}

\begin{defn}[$\varepsilon$-Isometric map and $\varepsilon$-isometry]
  \index{epsilon-isometric map@$\varepsilon$-isometric map}
  \index{epsilon-isometry@$\varepsilon$-isometry}
  Let $X$ and $Y$ be two metric spaces and let $\varepsilon \ge 0$.
  A map $f : X \to Y$ is said to be \emph{$\varepsilon$-isometric}
  if
  \[
  |\;d_Y(f(x),f(y)) - d_X(x,y)\;| \le \varepsilon
  \]
  for any $x,y \in X$.
  A map $f : X \to Y$ is called an \emph{$\varepsilon$-isometry}
  if $f$ is $\varepsilon$-isometric and if
  \[
  B_\varepsilon(f(X)) = Y.
  \]
\end{defn}

\begin{lem}[cf. \cite{BBI}*{\S 7.3}] \label{lem:dGH-isom}
  For a real number $\varepsilon > 0$ we have the following
  {\rm(1)} and {\rm(2)}.
  \begin{enumerate}
  \item If $d_{GH}(X,Y) < \varepsilon$, then
    there exists a $2\varepsilon$-isometry $f : X \to Y$.
  \item If there exists an $\varepsilon$-isometry $f : X \to Y$,
    then $d_{GH}(X,Y) < 2\varepsilon$.
  \end{enumerate}
\end{lem}

\begin{lem}[cf. \cite{BBI}*{\S 7.4}] \label{lem:dGH-precpt}
  Let $\cC$ be a set of isometry classes of compact metric spaces
  such that $\sup_{X \in \cC} \diam X < +\infty$,
  where $\diam X := \sup_{x,y\in X} d_X(x,y)$ is the diameter of $X$.
  Then, the following {\rm(1)}, {\rm(2)}, and {\rm(3)}
  are all equivalent to each other.
  \begin{enumerate}
  \item $\cC$ is $d_{GH}$-precompact.
  \item For any $\varepsilon > 0$ we have
    \[
    \sup_{X \in \cC} \Cap_\varepsilon(X) < +\infty.
    \]
  \item For any $\varepsilon > 0$ we have
    \[
    \sup_{X \in \cC} \Cov_\varepsilon(X) < +\infty.
    \]
  \end{enumerate}
\end{lem}



\section{Distance matrix}

\begin{defn}[Distance matrix]
  \index{distance matrix} \index{KNX@$K_N(X)$}
  For a metric space $X$ and a natural number $N$,
  the \emph{distance matrix $K_N(X)$ of $X$ of order $N$}
  is defined to be the set of symmetric matrices $(d_X(x_i,x_j))_{ij}$
  of order $N$, where $x_i$, $i=1,2,\dots,N$, run over all
  points in $X$.
\end{defn}

If $X$ is compact, then $K_N(X)$ is compact for any $N$.

\begin{lem} \label{lem:KN-isom}
  Let $X$ and $Y$ be two compact metric spaces.
  If $K_N(X) = K_N(Y)$ for every natural number $N$,
  then $X$ and $Y$ are isometric to each other.
\end{lem}

\begin{proof}
  Assume that $K_N(X) = K_N(Y)$ for all natural number $N$.
  We take any $\varepsilon > 0$ and fix it.
  For any net $\{x_i\}_{i=1}^N \subset X$,
  since $(d_X(x_i,x_j))_{ij} \in K_N(X) = K_N(Y)$,
  there is a net $\{y_i\}_{i=1}^N \subset Y$ such that
  $d_X(x_i,x_j) = d_Y(y_i,y_j)$ for all $i,j=1,2,\dots,N$.
  It holds that $\{x_i\}_{i=1}^N$ is $\varepsilon$-discrete if and only if
  so is $\{y_i\}_{i=1}^N$.
  Therefore we have $\Cap_\varepsilon(X) = \Cap_\varepsilon(Y)$.
  By setting $N := \Cap_\varepsilon(X) = \Cap_\varepsilon(Y)$,
  there are two $\varepsilon$-discrete nets $\{x_i\}_{i=1}^N \subset X$ and
  $\{y_i\}_{i=1}^N \subset Y$ such that 
  $d_X(x_i,x_j) = d_Y(y_i,y_j)$ for all $i,j=1,2,\dots,N$.
  Note that $\{x_i\}_{i=1}^N$ and $\{y_i\}_{i=1}^N$ are $\varepsilon$-nets
  of $X$ and $Y$, respectively.
  For any given point $x \in X$ we find an $i$ in such a way that
  $d_X(x_i,x) \le \varepsilon$ and then set $f(x) := y_i$.
  This defines a map $f : X \to Y$.
  We see that $f$ is a $2\varepsilon$-isometry.
  Lemma \ref{lem:dGH-isom} implies that $d_{GH}(X,Y) < 4\varepsilon$.
  By the arbitrariness of $\varepsilon > 0$, we have $d_{GH}(X,Y) = 0$,
  so that $X$ and $Y$ are isometric to each other.
\end{proof}

The \emph{$l_\infty$ norm} \index{l infinity norm@$l_\infty$ norm}
of a square matrix $A = (a_{ij})$ of order $N$
is defined to be
\[
\|A\|_\infty := \max_{i,j=1}^N |a_{ij}|.
\]
Of course, the $l_\infty$ norm induces a metric,
called the \emph{$l_\infty$ metric}, \index{l infinity metric@$l_\infty$ metric}
on the set of square matrices of order $N$.

\begin{lem} \label{lem:KN-GH}
  For any two compact metric spaces $X$ and $Y$
  and for any natural number $N$, we have
  \[
  d_H(K_N(X),K_N(Y)) \le 2\,d_{GH}(X,Y),
  \]
  where $d_H$ in the left-hand side is the Hausdorff distance
  defined by the $l_\infty$ metric on the set of square matrices of order $N$.
\end{lem}

\begin{proof}
  Assume that $d_{GH}(X,Y) < \varepsilon$ for a number $\varepsilon$.
  $X$ and $Y$ are embedded into some compact metric space $Z$
  such that $d_H(X,Y) < \varepsilon$.
  We take any matrix $A \in K_N(X)$ and set
  $A = (d_X(x_i,x_j))_{ij}$, $\{x_i\}_{i=1}^N \subset X$.
  By $d_H(X,Y) < \varepsilon$, there is a point $y_i \in Y$ for each $x_i$
  such that $d_Z(x_i,y_i) < \varepsilon$.
  Since
  \[
  |\;d_X(x_i,x_j)-d_Y(y_i,y_j)\;|
  \le d_Z(x_i,y_i)+d_Z(x_j,y_j) < 2\varepsilon
  \]
  for any $i,j = 1,2,\dots,N$,
  the matrix $B := (d_Y(y_i,y_j))_{ij}$ satisfies $B \in K_N(Y)$ and
  $\|A-B\|_\infty < 2\varepsilon$.
  This proves that $K_N(X) \subset B_{2\varepsilon}(K_N(Y))$.
  Since this also holds by exchanging $X$ and $Y$, we obtain
  \[
  d_H(K_N(X),K_N(Y)) \le 2\varepsilon.
  \]
  This completes the proof.
\end{proof}

\begin{lem} \label{lem:KN-GH-2}
  Let $X$ and $X_n$, $n=1,2,\dots$, be compact metric spaces.
  If $K_N(X_n)$ Hausdorff converges to $K_N(X)$ as $n\to\infty$
  with respect to the $l_\infty$ norm for any natural number $N$,
  then $X_n$ Gromov-Hausdorff converges to $X$ as $n\to\infty$.
\end{lem}

\begin{proof}
  We first prove that $\{X_n\}$ is $d_{GH}$-precompact.
  Setting $\delta_n := d_H(K_N(X_n),K_N(X))$
  we have
  $K_N(X_n) \subset B_{\delta_n}(K_N(X))$.
  This proves that, for any net $\{x_i\}_{i=1}^N \subset X_n$,
  there is a net $\{y_i\}_{i=1}^N \subset X$ such that
  \[
  |\;d_{X_n}(x_i,x_j) - d_X(y_i,y_j)\;| \le \delta_n
  \]
  for all $i,j=1,2,\dots,N$.
  We take any $\varepsilon > 0$ and any $n$ with $\delta_n < \varepsilon/2$.
  If $\{x_i\}_{i=1}^N$ is $\varepsilon$-discrete, then
  $\{y_i\}_{i=1}^N$ is $\varepsilon/2$-discrete.
  Thus we have $\Cap_\varepsilon(X_n) \le \Cap_{\varepsilon/2}(X)$,
  which holds for all sufficiently large $n$, so that
  $\sup_{n=1,2,\dots} \Cap_\varepsilon(X_n) < +\infty$.
  It is also easy to see that $\sup_{n=1,2,\dots} \diam X_n < +\infty$.
  By Lemma \ref{lem:dGH-precpt}, $\{X_n\}$ is $d_{GH}$-precompact.
  
  For any $d_{GH}$-convergent subsequence $\{X_{n_i}\}$ of $\{X_n\}$,
  Lemma \ref{lem:KN-GH} implies that $d_H(K_N(X_{n_i}),K_N(X')) \to 0$
  as $i\to\infty$ for any $N$,
  where $X'$ is the limit of $\{X_{n_i}\}$.
  Therefore we have $K_N(X) = K_N(X')$ for any $N$.
  By Lemma \ref{lem:KN-isom},
  $X$ and $X'$ are isometric to each other.
  This proves that $X_n$ $d_{GH}$-converges to $X$ as $n\to\infty$.
\end{proof}

\chapter{Box distance}
\label{chap:box-dist}

\section{Basics for the box distance}

In this section, we define the box distance between mm-spaces
and prove its basic properties.
The proofs in this section are mostly taken from \cite{Funano:thesis}.

\begin{defn}[Parameter]
  \index{parameter} \index{I@$I$}
  Let $I := [\,0,1\,)$ and let $X$ be a topological space
  with a Borel probability measure $\mu_X$.
  A map $\varphi : I \to X$ is called a \emph{parameter of $X$}
  if $\varphi$ is a Borel measurable map such that
  \[
  \varphi_*\cL^1 = \mu_X,
  \]
  where $\cL^1$ denotes the one-dimensional Lebesgue measure on $I$.
  \index{L1@$\cL^1$}
\end{defn}

Note that $\varphi(I)$ has full measure in $X$.

\begin{lem}
  Any mm-space has a parameter.
\end{lem}

\begin{proof}
  Let $X$ be an mm-space and let
  $A = \{x_i\}_{i=1}^N$, $N \le \infty$, be the set of atoms of $\mu_X$,
  where a point $x \in X$ is called an \emph{atom of $\mu_X$} \index{atom}
  if $\mu_X\{x\} > 0$.
  We assume that $x_i \neq x_j$ for $i\neq j$.
  Putting $a_i := \mu_X\{x_i\}$, $b_0 := 0$, $b_i := \sum_{j=1}^i a_j$
  for $1 \le i \le N$,
  we have $b_N = \mu_X(A)$.
  By \cite{Kechris}*{(17.41)},
  there is a Borel isomorphism $\varphi : [\,b_N,1\,) \to X \setminus A$
  such that $\varphi_*(\cL^1|_{[\,b_N,1\,)}) = \mu_X - \sum_i a_i\delta_{x_i}$.
  Setting $\varphi|_{[\,b_{i-1},b_i\,)} := x_i$ for $i \ge 1$ defines
  a Borel measurable map $\varphi : I \to X$ such that
  $\varphi_*\cL^1 = \mu_X$.
\end{proof}

\begin{defn}[Pseudo-metric]
  \index{pseudo-metric}
  A \emph{pseudo-metric} $\rho$ on a set $S$ is defined to be
  a function $\rho : S \times S \to [\,0,+\infty\,)$ satisfying that,
  for any $x,y,z \in S$,
  \begin{enumerate}
  \item $\rho(x,x) = 0$,
  \item $\rho(y,x) = \rho(x,y)$,
  \item $\rho(x,z) \le \rho(x,y) + \rho(y,z)$.
  \end{enumerate}
\end{defn}

Note that $\rho(x,y) = 0$ does not necessarily imply $x = y$.

A Lipschitz map between two spaces with pseudo-metrics
is defined in the same way as usual.

\begin{defn}[Box distance]
  \index{box distance} \index{boxrho1rho2@$\square(\rho_1,\rho_2)$}
  For two pseudo-metrics $\rho_1$ and $\rho_2$ on $I$, we define
  $\square(\rho_1,\rho_2)$ to be the infimum of $\varepsilon \ge 0$
  satisfying that there exists a Borel subset $I_0 \subset I$
  such that
  \begin{align}
    & |\,\rho_1(s,t)-\rho_2(s,t)\,| \le \varepsilon
    \quad\text{for any $s,t \in I_0$},\tag{1}\\
    & \cL^1(I_0) \ge 1-\varepsilon.\tag{2}
  \end{align}
  We define the \emph{box distance $\square(X,Y)$ between
    two mm-spaces $X$ and $Y$} \index{boxXY@$\square(X,Y)$} to be
  the infimum of $\square(\varphi^*d_X,\psi^*d_Y)$, where
  $\varphi : I \to X$ and $\psi : I \to Y$ run over all parameters
  of $X$ and $Y$, respectively, and where
  $\varphi^*d_X(s,t) := d_X(\varphi(s),\varphi(t))$ for $s,t \in I$.
  \index{phistardX@$\varphi^*d_X$}
\end{defn}

Note that $\square(\rho_1,\rho_2) \le 1$ for any two pseudo-metrics
$\rho_1$ and $\rho_2$,
so that $\square(X,Y) \le 1$ for any two mm-spaces $X$ and $Y$.

We are going to prove that $\square$ is a metric on the set of
mm-isomorphism classes of mm-spaces.

\begin{defn}[Distance matrix distribution]
  \index{distance matrix distribution}
  Denote by $M_N$ the set of real symmetric matrices of order $N$.
  For an mm-space $X$, we define a map $\kappa_N : X^N \to M_N$ by
  \[
  \kappa_N(x_1,x_2,\dots,x_N) := (d_X(x_i,x_j))_{ij}
  \]
  for $(x_1,x_2,\dots,x_N) \in X^N$.
  The \emph{distance matrix distribution of $X$ of order $N$}
  is defined to be the push-forward
  \[
  \underline{\mu}_N^X := (\kappa_N)_*\mu_X^{\otimes N}
  \]
  \index{muNX@$\underline{\mu}_N^X$}
  of the $N$-times product of $\mu_X$ by the map $\kappa_N$.
\end{defn}

\begin{lem} \label{lem:box}
  If $\square(X,Y) = 0$ for two mm-spaces, then
  \[
  \underline{\mu}_N^X = \underline{\mu}_N^Y
  \]
  for any natural number $N$.
\end{lem}

\begin{proof}
  Assume that $\square(X,Y) = 0$ for two mm-spaces $X$ and $Y$.
  Then, there are parameters $\varphi_n$ and $\psi_n$ of $X$ and $Y$
  respectively, $n=1,2,\dots$, such that
  $\square(\varphi_n^*d_X,\psi_n^*d_Y) \to 0$ as $n\to\infty$.
  We have $\varepsilon_n \to 0+$ and $I_n \subset I$ in such a way that
  $\cL^1(I_n) \ge 1-\varepsilon_n$ and
  \[
  |\;\varphi_n^*d_X(s,t) - \psi_n^*d_Y(s,t)\;| \le \varepsilon_n
  \]
  for any $s,t \in I_n$.
  Let $f : M_N \to \R$ be a uniformly continuous and bounded function
  and set
  \begin{align*}
    f_{\varphi_n}(s_1,s_2,\dots,s_N) &:= f((\varphi_n^*d_X(s_i,s_j))_{ij}),\\
    f_{\psi_n}(s_1,s_2,\dots,s_N) &:= f((\psi_n^*d_X(s_i,s_j))_{ij})
  \end{align*}
  for $(s_1,s_2,\dots,s_N) \in I^N$.
  Note that $f_{\varphi_n}$ and $f_{\psi_n}$ are both Borel measurable functions.
  The uniform continuity of $f$ shows that
  \begin{equation}
    \label{eq:dist-distr1}
    \lim_{n\to\infty} \sup_{I_n^N} |\;f_{\varphi_n}-f_{\psi_n}\;| = 0.
  \end{equation}
  We have
  \begin{equation}
    \label{eq:dist-distr2}
    \int_{M_N} f \; d\underline{\mu}_N^X
    = \int_{X^N} f\circ\kappa_N \; d\mu_X^{\otimes N}
    = \int_{I^N} f_{\varphi_n} \; d\cL^N,
  \end{equation}
  where $\cL^N$ denotes the $N$-dimensional Lebesgue measure.
  \index{LN@$\cL^N$}
  \eqref{eq:dist-distr2} also holds
  if we replace $X$ and $\varphi_n$ with $Y$ and $\psi_n$, respectively.
  Since
  \[
  \limsup_{n\to\infty} \left| \int_{I^N \setminus I_n^N} f_{\varphi_n}
    \; d\cL^N \right|
  \le \sup|f| \lim_{n\to\infty} \cL^N(I^N \setminus I_n^N) = 0,
  \]
  \eqref{eq:dist-distr1} and \eqref{eq:dist-distr2} together imply
  \begin{align*}
    \int_{M_N} f \; d\underline{\mu}_N^X
    &= \lim_{n\to\infty} \int_{I_n^N} f_{\varphi_n} \; d\cL^N
    = \lim_{n\to\infty} \int_{I_n^N} f_{\psi_n} \; d\cL^N
    = \int_{M_N} f \; d\underline{\mu}_N^Y.
  \end{align*}
  This completes the proof.
\end{proof}

\begin{thm}[mm-Reconstruction theorem] \label{thm:mm-reconst}
  \index{mm-reconstruction theorem}
  Let $X$ and $Y$ be two mm-spaces.
  If $\underline{\mu}_N^X = \underline{\mu}_N^Y$ for any natural number $N$, then
  $X$ and $Y$ are mm-isomorphic to each other.
\end{thm}

For the complete proof of this theorem, we refer to \cite{Gromov}*{3$\frac12$.5},
\cite{Vershik}*{\S 2}, and \cite{Kondo}.
The following is taken from \cite{Kondo}.

\begin{proof}[Sketch of Proof]
 We define the distance matrix distribution
 \[
 \underline{\mu}_\infty^X := (\kappa_\infty)_*\mu_X^{\otimes\infty}
 \]
 of $X$ of infinite order
 as an Borel probability measure
 on the set, say $M_\infty$, of real symmetric matrices of infinite order
 in the same manner as before,
 where $\kappa_\infty : X^\infty \to M_\infty$ is defined by
 \[
 \kappa_\infty((x_i)_{i=1}^\infty) := (d_X(x_i,x_j))_{i,j=1}^\infty
 \]
 and $M_\infty$ is equipped with the product topology.
 Taking the completion of $\underline{\mu}_\infty^X$,
 we assume that $\underline{\mu}_\infty^X$ is a complete measure.
 
 Assume that $\underline\mu_N^X = \underline\mu_N^Y$ for any natural number $N$.
 Let us first prove $\underline{\mu}_\infty^X = \underline{\mu}_\infty^Y$.
 In fact, it is easy to see that, for $1 \le N < N' \le \infty$, the measure
 $\underline\mu_N^X$ coincides with the push-forward of
 $\underline{\mu}_{N'}^X$ by the natural projection from $M_{N'}$ to $M_N$.
 The Kolmogorov extension theorem tells us that
 $\underline{\mu}_\infty^X$ is determined only by $\underline\mu_N^X$, $N = 1,2,\dots$.
 Therefore, by the assumption, we have $\underline{\mu}_\infty^X = \underline{\mu}_\infty^Y$.
 
 Let $E_X \subset X^\infty$ be the set of uniformly distributed sequences on $X$, i.e.,
 $(x_i)_{i=1}^\infty \in E_X$ if and only if
 \[
 \lim_{N\to\infty} \frac{1}{N} \sum_{i=1}^N f(x_i)
 = \int_X f \; d\mu_X
 \]
 for any bounded continuous function $f : X \to \R$.
 Note that each $(x_i)_{i=1}^\infty \in E_X$ is dense in $X$.
 It is well-known that $\mu_X^{\otimes\infty}(E_X) = 1$.
 Since $\kappa_\infty^{-1}(\kappa_\infty(E_X)) \supset E_X$,
 the set $\kappa_\infty^{-1}(\kappa_\infty(E_X))$ is $\mu_X^{\otimes\infty}$-measurable,
 which implies the $\underline{\mu}_\infty^X$-measurablity of $\kappa_\infty(E_X)$.
 We also obtain the $\underline{\mu}_\infty^Y$-measurablity of $\kappa_\infty(E_Y)$
 by the same reason.
 We have $\underline{\mu}_\infty^X(\kappa_\infty(E_X))
 = \underline{\mu}_\infty^Y(\kappa_\infty(E_Y)) = 1$, which together with
 $\underline{\mu}_\infty^X = \underline{\mu}_\infty^Y$ imlies
 $\kappa_\infty(E_X) \cap \kappa_\infty(E_Y) \neq \emptyset$.
 Therefore, there are two sequences $(x_i)_{i=1}^\infty \in E_X$ and $(y_i)_{i=1}^\infty \in E_Y$
 such that $d_X(x_i,x_j) = d_Y(y_i,y_j)$ for all $i,j  = 1,2,\dots$.
 The map $x_i \mapsto y_i$ extends to an isometry $F : X \to Y$.
 For any bounded continuous function $f : Y \to \R$, we have
 \[
 \int_Y f \; d\mu_Y = \lim_{N\to\infty} \frac{1}{N} \sum_{i=1}^N f(y_i)
 = \lim_{N\to\infty} \frac{1}{N} \sum_{i=1}^N f(F(x_i))
 = \int_X f\circ F \; d\mu_X,
 \]
 which implies $F_*\mu_X = \mu_Y$.
 Thus, $X$ and $Y$ are mm-isomorphic to each other.
\end{proof}

\begin{lem} \label{lem:box-pseudo-tri}
  $\square$ satisfies a triangle inequality between
  pseudo-metrics.
\end{lem}

\begin{proof}
  Assume that $\square(\rho_1,\rho_2) < \varepsilon$
  and $\square(\rho_2,\rho_3) < \delta$ for three pseudo-metrics $\rho_1$,
  $\rho_2$, and $\rho_3$ and for two numbers $\varepsilon, \delta > 0$.
  It suffices to prove that $\square(\rho_1,\rho_3) \le \varepsilon+\delta$.
  There are two Borel subsets $I_0,I_0' \subset I$ such that
  \begin{align*}
    &\cL^1(I_0) \ge 1-\varepsilon,\qquad
    \cL^1(I_0') \ge 1-\delta,\\
    &|\;\rho_1(s,t)-\rho_2(s,t)\;| \le \varepsilon
    \quad\text{for any $s,t \in I_0$},\\
    &|\;\rho_2(s,t)-\rho_3(s,t)\;| \le \delta
    \quad\text{for any $s,t \in I_0'$}.
  \end{align*}
  Setting $I_0'' := I_0 \cap I_0'$, we have
  \begin{align*}
    \cL^1(I_0'') &= \cL^1(I_0) + \cL^1(I_0')
    - \cL^1(I_0 \cup I_0')\\
    &\ge 2-\varepsilon-\delta- \cL^1(I_0 \cup I_0')
    \ge 1-(\varepsilon+\delta)
  \end{align*}
  and, for any $s,t \in I_0''$,
  \begin{align*}
    |\;\rho_1(s,t)-\rho_3(s,t)\;| &\le
    |\;\rho_1(s,t)-\rho_2(s,t)\;| + |\;\rho_2(s,t)-\rho_3(s,t)\;|\\
    &\le \varepsilon + \delta.
  \end{align*}
  Therefore we obtain $\square(\rho_1,\rho_3) \le \varepsilon+\delta$.
  This completes the proof.
\end{proof}

The following lemma is needed to prove a triangle inequality
for the box distance between mm-spaces.

\begin{lem} \label{lem:box-tri}
  Let $\varphi : I \to X$ and $\psi : I \to X$ be two parameters
  of an mm-space $X$.  Then, for any $\varepsilon > 0$
  there exist two Borel isomorphisms $f : I \to I$ and $g : I \to I$
  such that
  \begin{enumerate}
  \item $f_*\cL^1 = g_*\cL^1 = \cL^1$,
  \item $\square((\varphi\circ f)^*d_X,(\psi\circ g)^*d_X) \le \varepsilon$.
  \end{enumerate}
\end{lem}

\begin{proof}
  Let $\varepsilon > 0$ be a given number.
  There is a sequence $\{B_i\}_{i=1}^N$ of disjoint Borel subsets of $X$,
  $N \le \infty$,
  such that $\bigcup_{i=1}^N B_i = X$ and $\diam B_i \le \varepsilon/2$
  for any $i$.
  We set $a_i := \mu_X(B_i) = \cL^1(\varphi^{-1}(B_i))$, $b_0 := 0$,
  and $b_i := \sum_{j \le i} a_j$.
  Since each $\varphi^{-1}(B_i)$ is a Borel subset of $I$,
  it is a standard Borel space.
  Thus, there is a Borel isomorphism
  $f_i : [\,b_{i-1},b_i\,) \to \varphi^{-1}(B_i)$ for each $i \ge 1$
  such that
  $(f_i)_*(\cL^1|_{[\,b_{i-1},b_i\,)}) = \cL^1|_{\varphi^{-1}(B_i)}$.
  Combining all $f_i$'s defines a Borel isomorphism
  $f : I \to I$ with the property that $f([\,b_{i-1},b_i\,)) = \varphi^{-1}(B_i)$
  for any $i \ge 1$ and $f_*\cL^1 = \cL^1$.
  In the same way, we have a Borel isomorphism $g : I \to I$ such that
  $g([\,b_{i-1},b_i\,)) = \psi^{-1}(B_i)$
  for any $i \ge 1$ and $g_*\cL^1 = \cL^1$.
  For any $s \in I$, both
  $(\varphi\circ f)(s)$ and $(\psi\circ g)(s)$ belong to
  a common $B_i$ and hence
  $d_X((\varphi\circ f)(s),(\psi\circ g)(s)) \le \varepsilon/2$.
  Therefore, for any $s,t \in I$,
  \[
  |\;d_X((\varphi\circ f)(s),(\varphi\circ f)(t)) - 
  d_X((\psi\circ g)(s),(\psi\circ g)(t))\;| \le \varepsilon,
  \]
  which implies that
  $\square((\varphi\circ f)^*d_X,(\psi\circ g)^*d_X) \le \varepsilon$.
  This completes the proof.
\end{proof}

We are now ready to prove the following

\begin{thm} \label{thm:box}
  The box metric $\square$ is a metric on the set $\cX$ of
  mm-isomorphism classes of mm-spaces.
\end{thm}

\begin{proof}
  Combining Lemma \ref{lem:box} and Theorem \ref{thm:mm-reconst} yields
  that $\square(X,Y) = 0$ if and only if $X$ and $Y$ are mm-isomorphic to
  each other.

  The symmetricity for $\square$ is clear.

  Let us prove the triangle inequality
  $\square(X,Z) \le \square(X,Y) + \square(Y,Z)$
  for three mm-spaces $X$, $Y$, and $Z$.
  We take any parameters $\varphi : I \to X$, $\psi,\psi' : I \to Y$,
  and $\xi : I \to Z$, and fix them.
  Lemma \ref{lem:box-tri} implies that, for any $\varepsilon > 0$,
  there are two Borel isomorphisms
  $f : I \to I$ and $g : I \to I$ such that
  $\square((\psi\circ f)^*d_Y,(\psi'\circ g)^*d_Y) \le \varepsilon$.
  It follows from Lemma \ref{lem:box-pseudo-tri} that
  \begin{align*}
    &\square(\varphi^*d_X,\psi^*d_Y) + \square({\psi'}^*d_Y,\xi^*d_Z)\\
    &= \square((\varphi\circ f)^*d_X,(\psi\circ f)^*d_Y)
    + \square((\psi'\circ g)^*d_Y,(\xi\circ g)^*d_Z)\\
    &\ge \square((\varphi\circ f)^*d_X,(\xi\circ g)^*d_Z)
    - \square((\psi\circ f)^*d_Y,(\psi'\circ g)^*d_Y)\\
    &\ge \square(X,Z) - \varepsilon.
  \end{align*}
  Taking the infimum of the left-hand side for any parameters
  $\varphi,\psi,\psi',\xi$ yields
  $\square(X,Y) + \square(Y,Z) \ge \square(X,Z) - \varepsilon$.
  Since $\varepsilon > 0$ is arbitrary, this completes the proof.
\end{proof}

Recall that $tX := (X,td_X,\mu_X)$
for an mm-space $X$ and for $t > 0$.

\begin{lem} \label{lem:box-scale}
  For any two mm-spaces $X$ and $Y$, we have the following {\rm(1)}
  and {\rm(2)}.
  \begin{enumerate}
  \item $\square(tX,tY)$ is monotone nondecreasing in $t > 0$.
  \item $t^{-1}\square(tX,tY)$ is monotone nonincreasing in $t > 0$.
  \end{enumerate}
\end{lem}

The proof of the lemma is left for the reader.

\begin{prop} \label{prop:box-di}
  Let $X$ be a complete separable metric space.
  For any two Borel probability measures $\mu$ and $\nu$ on $X$,
  we have
  \[
  \frac{1}{2}\,\square((X,\mu),(X,\nu)) \le
  \square((2^{-1}X,\mu),(2^{-1}X,\nu)) \le d_P(\mu,\nu).
  \]
\end{prop}

\begin{proof}
  The first inequality follows from Lemma \ref{lem:box-scale}(2).
  We prove the second.
  Assume that $d_P(\mu,\nu) < \varepsilon$ for a number $\varepsilon$.
  It suffices to prove that
  $\square((2^{-1}X,\mu),(2^{-1}X,\nu)) \le \varepsilon$.
  By Strassen's theorem (Theorem \ref{thm:di-tra}),
  there is an $\varepsilon$-transportation $m$ between $\mu$ and $\nu$
  with $\defi m \le \varepsilon$.
  Namely, $m$ is a transport plan between two measures $\mu'$ and $\nu'$
  such that
  $\mu' \le \mu$, $\nu' \le \nu$, $1-m(X\times X) \le \varepsilon$,
  and $\supp m \subset \Delta_\varepsilon$.
  If $m(X \times X) = 1$, then we set $m_+ := m$.
  If otherwise, we set
  \[
  m_+ := m + \frac{(\mu-\mu') \times (\nu-\nu')}{1-m(X\times X)}.
  \]
  Note that $m(X\times X) = \mu'(X) = \nu'(X)$.
  It is easy to see that $m_+$ is a transport plan between $\mu$ and $\nu$.
  Let $\Phi : I \to X \times X$ be a parameter of $(X \times X,m_+)$,
  and let $\varphi := \pr_1\circ\Phi : I \to X$ and
  $\psi := \pr_2\circ\Phi : I \to X$, where
  $\pr_i : X \times X \to X$, $i=1,2$, are the projections.
  Let us prove that $\varphi$ is a parameter of $(X,\mu)$.
  In fact, for any Borel subset $A \subset X$,
  \begin{align*}
    \varphi_*\cL^1(A) &= \cL^1(\varphi^{-1}(A))
    = \cL^1(\Phi^{-1}(A \times X))\\
    &= \Phi_*\cL^1(A \times X) = m_+(A \times X) = \mu(A).
  \end{align*}
  In the same way, we see that $\psi$ is a parameter of $(X,\nu)$.
  Set $I_0 := \Phi^{-1}(\Delta_\varepsilon)$.
  To obtain $\square((2^{-1}X,\mu),(2^{-1}X,\nu)) \le \varepsilon$,
  it suffices to prove that
  \begin{enumerate}
  \item $\cL^1(I_0) \ge 1-\varepsilon$,
  \item $|d_X(\varphi(s),\varphi(t)) - d_X(\psi(s),\psi(t))| \le 2\varepsilon$
    for any $s,t \in I_0$.
  \end{enumerate}

  For (1), we have
  \[
  \cL^1(I_0) = \Phi_*\cL^1(\Delta_\varepsilon)
  = m_+(\Delta_\varepsilon) \ge m(\Delta_\varepsilon)
  = m(X \times X) \ge 1-\varepsilon.
  \]

  We prove (2).
  Take any $s,t \in I_0$.
  Since $\Phi(s),\Phi(t) \in \Delta_\varepsilon$,
  we have $d_X(\varphi(s),\psi(s)) \le \varepsilon$ and
  $d_X(\varphi(t),\psi(t)) \le \varepsilon$, so that,
  by a triangle inequality,
  the left-hand side of (2) is
  \[
  \le d_X(\varphi(s),\psi(s)) + d_X(\varphi(t),\psi(t))
  \le 2\varepsilon.
  \]
  This completes the proof.
\end{proof}

To obtain the completeness of $\square$ on $\cX$,
we need the following

\begin{lem}[Union lemma] \label{lem:union}
  \index{union lemma}
  Let $X_n$, $n=1,2,\dots$, be mm-spaces such that
  \[
  \square(2^{-1}X_n,2^{-1}X_{n+1}) < \varepsilon_n
  \]
  for any $n$ and
  for a sequence of real numbers $\varepsilon_n$, $n=1,2,\dots$.
  Then, there exists a metric on the disjoint union of all $X_n$'s
  that is an extension of each $X_n$ such that
  \[
  d_P(\mu_{X_n},\mu_{X_{n+1}}) \le \varepsilon_n
  \]
  for any $n=1,2,\dots$.
\end{lem}

\begin{proof}
  Assume that $\square(2^{-1}X_n,2^{-1}X_{n+1}) < \varepsilon_n$
  for any $n=1,2,\dots$.
  Then, for each $n$ there are two parameters $\varphi_n : I \to X_n$,
  $\psi_n : I \to X_{n+1}$, and a Borel subset $I_n \subset I$ such that
  $\cL^1(I_n) \ge 1-\varepsilon_n$ and
  \[
  |\;d_{X_n}(\varphi_n(s),\varphi_n(t))-d_{X_{n+1}}(\psi_n(s),\psi_n(t))\;|
  \le 2\varepsilon_n
  \]
  for any $s,t \in I_n$.
  Let $Z$ be the disjoint union of all $X_n$'s.
  We define a function $d_Z : Z \times Z \to [\,0,+\infty\,)$ as follows:
  For $x,y \in X_n$, we set
  \[
  d_Z(x,y) := d_{X_n}(x,y).
  \]
  For $x \in X_n$ and $y \in X_{n+1}$, 
  \[
  d_Z(x,y) := \inf_{s \in I_n}(d_{X_n}(x,\varphi_n(s)) + d_{X_{n+1}}(\psi_n(s),y))
  + \varepsilon_n
  \]
  For $x \in X_n$ and $y \in X_m$ with $n + 2 \le m$,
  \[
  d_Z(x,y) := \inf \sum_{i=n}^{m-1} d_Z(x_i,x_{i+1}),
  \]
  where $x_n := x$, $x_m := y$, and where $x_i$ with $n < i < m$ runs over all
  points in $X_i$.
  For $x \in X_n$ and $y \in X_m$ with $n > m$,
  \[
  d_Z(x,y) := d_Z(y,x).
  \]
  It is easy to check that $d_Z$ is a metric on $Z$.
  Let $f_n : I \to Z \times Z$ be a map defined by
  $f_n(s) := (\varphi_n(s),\psi_n(s))$, and let $m := (f_n)_*\cL^1|_{I_n}$.
  Setting $\Delta_{\varepsilon_n}
  := \{\;(x,y) \in Z \times Z \mid d_Z(x,y) \le \varepsilon_n\;\}$,
  we have  
  \begin{align*}
    &m(Z \times Z \setminus \Delta_{\varepsilon_n})
    = \cL^1(I_n \cap f_n^{-1}(Z \times Z \setminus \Delta_{\varepsilon_n}))\\
    &= \cL^1(\{\;s \in I_n
    \mid d_Z(\varphi_n(s),\psi_n(s)) > \varepsilon_n\;\}) = 0
  \end{align*}
  and also, for any Borel subset $A \subset Z$,
  \[
  m(A \times Z) = \cL^1(\varphi_n^{-1}(A) \cap I_n)
  \le \cL^1(\varphi_n^{-1}(A)) = \mu_{X_n}(A)
  \]
  as well as $m(Z \times A) \le \mu_{X_{n+1}}(A)$.
  Therefore $m$ is an $\varepsilon_n$-transportation between $\mu_{X_n}$
  and $\mu_{X_{n+1}}$.
  Note that $X_n \cup X_{n+1}$ is a complete separable subspace of $Z$.
  Since $\defi m = 1-m(Z\times Z) = 1-\cL^1(I_n) \le \varepsilon_n$,
  Strassen's theorem (Theorem \ref{thm:di-tra}) tells us that
  $d_P(\mu_{X_n},\mu_{X_{n+1}}) \le \varepsilon_n$.
  This completes the proof.
\end{proof}

\begin{thm} \label{thm:box-complete}
  The box metric $\square$ is complete on $\cX$.
\end{thm}

\begin{proof}
  Let $\{X_n\}_{n=1}^\infty$ be a $\square$-Cauchy sequence of mm-spaces.
  It suffices to prove that $\{X_n\}_{n=1}^\infty$ has a $\square$-convergent
  subsequence.
  Replacing $\{X_n\}_{n=1}^\infty$ with a subsequence, we assume that
  $\square(X_n,X_{n+1}) < 2^{-n}$ for any $n$.
  By the union lemma (Lemma \ref{lem:union}), we have a metric $d_Z$
  on the disjoint union
  $Z$ of all $X_n$'s such that $d_P(\mu_{X_n},\mu_{X_{n+1}}) \le 2^{-n}$ for any $n$.
  We see that $\{\mu_{X_n}\}_{n=1}^\infty$ is a $d_P$-Cauchy sequence.
  Let $(\bar Z,d_{\bar Z})$ be the completion of $(Z,d_Z)$.
  Note that $\bar Z$ is a complete separable metric space.
  Since the set of Borel probability measures on $\bar Z$ is $d_P$-complete
  (see Lemma \ref{lem:conv-meas}),
  the sequence $\{\mu_{X_n}\}_{n=1}^\infty$ $d_P$-converges
  to some Borel probability measure, say $\mu_\infty$, on $\bar Z$.
  Proposition \ref{prop:box-di} proves that
  $X_n$ $\square$-converges to $(\bar Z,\mu_\infty)$.
  This completes the proof.
\end{proof}

\begin{ex}
  We have $\lim_{n\to\infty} \square(S^n(1),S^{n+1}(1)) = 0$.
  This is because, embedding $S^n(1)$ into $S^{n+1}(1)$ naturally, we have
  \[
  \lim_{n\to\infty} \sigma^{n+1}(S^{n+1}(1) \setminus U_\varepsilon(S^n(1))) = 0
  \]
  for any $\varepsilon > 0$, which implies
  \[
  \square(S^n(1),S^{n+1}(1)) \le 2 d_P(\sigma^n,\sigma^{n+1})
  \to 0 \quad\text{as $n\to\infty$.}
  \]
  However, $\{S^n(1)\}$ is not a $\square$-convergent sequence
  as is seen in Corollary \ref{cor:homog}.
\end{ex}

\begin{rem}
  \index{Gromov-Prohorov distance} \index{dGP@$d_{GP}$}
  We define the \emph{Gromov-Prohorov distance $d_{GP}(X,Y)$
  between two mm-spaces $X$ and $Y$}
  to be the infimum of $d_P(\mu_X,\mu_Y)$ for all metrics on
  the disjoint union of $X$ and $Y$ that are extensions
  of $d_X$ and $d_Y$.
  As a direct consequence of
  Proposition \ref{prop:box-di} and the union lemma,
  we obtain
  \[
  \square(2^{-1}X,2^{-1}Y) = d_{GP}(X,Y)
  \]
  and, in particular,
  \[
  \frac{1}{2}\,\square(X,Y) \le d_{GP}(X,Y) \le \square(X,Y)
  \]
  for any two mm-spaces $X$ and $Y$.
  This is proved by L\"ohr \cite{Lohr}.
\end{rem}

\section{Finite approximation}

In this section, we prove that
any mm-space can be approximated by a finite mm-space,
i.e., an mm-space consisting of at most finitely many points.
By using this, we study an $\varepsilon$-mm-isomorphism
and the $\square$-compactness of a family of mm-spaces.

\begin{defn}[$\varepsilon$-Supporting net]
  \index{epsilon-supporting net@$\varepsilon$-supporting net}
  Let $X$ be an mm-space and $\cN$ a net of $X$.
  For a real number $\varepsilon > 0$, we say that
  $\cN$ \emph{$\varepsilon$-supports $X$}
  if
  \[
  \mu_X(B_\varepsilon(\cN)) \ge 1-\varepsilon.
  \]
\end{defn}

\begin{lem} \label{lem:eps-supp}
  Let $X$ be an mm-space.
  For any $\varepsilon > 0$ there exists a finite net of $X$
  that $\varepsilon$-supports $X$.
\end{lem}

\begin{proof}
  Since $X$ is separable,
  we have a dense countable subset $\{a_i\}_{i=1}^\infty \subset X$.
  Since $\bigcup_{i=1}^\infty B_\varepsilon(a_i) = X$, we have
  $\lim_{n\to\infty} \mu_X(\bigcup_{i=1}^n B_\varepsilon(a_i)) = \mu(X) = 1$.
  Therefore, there is a number $n$ such that
  $\mu_X(\bigcup_{i=1}^n B_\varepsilon(a_i)) \ge 1-\varepsilon$.
  $\cN := \{a_i\}_{i=1}^n$ is a desired net.
\end{proof}

Let $X$ be an mm-space
and $\cN$ a finite net $\varepsilon$-supporting $X$
for a number $\varepsilon > 0$.
We have a Borel measurable nearest point projection
$\pi_{\cN} : X \to \cN$
(see Lemma \ref{lem:Borel-proj}).

\begin{lem} \label{lem:pi-di}
  We have
  \[
  d_P((\pi_{\cN})_*\mu_X,\mu_X) \le \varepsilon.
  \]
\end{lem}

\begin{proof}
  Since
  \[
  \mu_X(\{\;x \in X \mid d_X(\pi_{\cN}(x),x) > \varepsilon\;\})
  = \mu_X(X \setminus B_\varepsilon(\cN)) \le \varepsilon,
  \]
  we have $\dKF(\pi_{\cN},\id_X) \le \varepsilon$,
  where $\id_X : X \to X$ is the identity map.
  By Lemma \ref{lem:di-me},
  $d_P((\pi_{\cN})_*\mu_X,\mu_X) \le \dKF(\pi_{\cN},\id_X)
  \le \varepsilon$.
\end{proof}

\begin{prop} \label{prop:fin-approx}
  Let $X$ be an mm-space.
  For any $\varepsilon > 0$ there exists a finite mm-space $\dot X$
  such that
  \[
  \square(X,\dot X) \le \varepsilon.
  \]
\end{prop}

\begin{proof}
  By Lemma \ref{lem:eps-supp}, there is a finite net
  $\cN \subset X$ $(\varepsilon/2)$-supporting $X$.
  By letting $\dot X := (\cN,d_X,(\pi_{\cN})_*\mu_X)$,
  Proposition \ref{prop:box-di} and Lemma \ref{lem:pi-di} together imply
  that $\square(X,\dot X) \le 2 d_P(\mu_X,(\pi_{\cN})_*\mu_X) \le \varepsilon$.
\end{proof}

Let $X$ and $Y$ be two mm-spaces and $f : X \to Y$ a Borel measurable map.

\begin{defn}[$\varepsilon$-mm-Isomorphism] \label{defn:mm-iso}
  \index{epsilon-mm-isomorphism@$\varepsilon$-mm-isomorphism}
  Let $\varepsilon \ge 0$ be a real number.
  We say that $f$ is \emph{$\varepsilon$-mm-isomorphism} if
  there exists a Borel subset $X_0 \subset X$ such that
  \begin{align*}
    & \mu_X(X_0) \ge 1-\varepsilon,\tag{1}\\
    & |\;d_X(x,y) - d_Y(f(x),f(y))\;| \le \varepsilon
    \qquad\text{for any $x,y \in X_0$,}\tag{2}\\
    & d_P(f_*\mu_X,\mu_Y) \le \varepsilon.\tag{3}
  \end{align*}
  We call $X_0$ a \emph{non-exceptional domain of $f$}.
  \index{non-exceptional domain}
\end{defn}

Clearly, a $0$-mm-isomorphism means an mm-isomorphism.

\begin{lem} \label{lem:box-eps-mm-iso}
  \begin{enumerate}
  \item If there exists an $\varepsilon$-mm-isomorphism $: X \to Y$,
    then $\square(X,Y) \le 3\varepsilon$.
  \item If $\square(X,Y) < \varepsilon$, then
    there exists a $3\varepsilon$-mm-isomorphism $: X \to Y$.
  \end{enumerate}
\end{lem}

\begin{proof}
  We prove (1).
  Assume that there is an $\varepsilon$-mm-isomorphism $f : X \to Y$.
  Let $X_0 \subset X$ be a non-exceptional domain of $f$.
  We take a parameter $\varphi : I \to X$.
  We see that $\psi := f \circ \varphi : I \to Y$ is a parameter of
  $(Y,f_*\mu_X)$.
  Setting $I_0 := \varphi^{-1}(X_0)$, we have,
  by Definition \ref{defn:mm-iso}(2),
  \[
  |d_X(\varphi(s),\varphi(t))-d_Y(\psi(s),\psi(t))| \le \varepsilon
  \]
  for any $s,t \in I_0$.
  Since $\cL^1(I_0) = \varphi_*\cL^1(X_0) = \mu_X(X_0)
  \ge 1-\varepsilon$, we have
  $\square(X,(Y,f_*\mu_X)) \le \varepsilon$.
  By Proposition \ref{prop:box-di} and Definition \ref{defn:mm-iso}(3),
  $\square((Y,f_*\mu_X),Y) \le 2 d_P(f_*\mu_X,\mu_Y) \le 2\varepsilon$.
  A triangle inequality proves that
  $\square(X,Y) \le 3\varepsilon$.

  We prove (2).  Assume that $\square(X,Y) < \varepsilon$.
  Let $\varepsilon' := (\varepsilon-\square(X,Y))/4$.
  There is a finite net $\cN \subset X$ $\varepsilon'$-supporting
  $X$ such that, by setting
  $\dot X := (\cN,d_X,(\pi_{\cN})_*\mu_X)$,
  we have $\square(X,\dot X) \le 2\varepsilon'$ and
  $\mu_X(B_\varepsilon(\cN)) \ge 1-\varepsilon$.
  A triangle inequality proves that
  \[
  \square(\dot X,Y) \le 2\varepsilon' + \square(X,Y) < \varepsilon.
  \]
  Therefore, there are two parameters $\varphi$ and $\psi$ of $\dot{X}$ and $Y$,
  respectively, and a Borel subset $I_0 \subset I$ such that
  \begin{align}
    &\cL^1(I_0) \ge 1-\varepsilon, \label{eq:box-eps-mm-iso-1}\\
    &| d_{\dot{X}}(\varphi(s),\varphi(t)) - d_Y(\psi(s),\psi(t))| \le \varepsilon
    \quad\text{for any $s,t \in I_0$} \label{eq:box-eps-mm-iso-2}.
  \end{align}
  For any point $x \in \dot X$ we take a point $s \in \varphi^{-1}(x)$.
  We assume that $s \in \varphi^{-1}(x) \cap I_0$ if $x \in \varphi(I_0)$.
  Letting $\varphi'(x) := s$ defines a map $\varphi' : \dot X \to I$.
  It is obvious that $\varphi'$ is Borel measurable,
  $\varphi'(\varphi(I_0)) \subset I_0$, and $\varphi\circ\varphi' = \id_{\dot X}$.
  Set $\dot f := \psi\circ\varphi' : \dot{X} \to Y$
  and $\dot X_0 := \varphi(I_0)$.
  We have $\mu_{\dot X}(\dot X_0) = \varphi_*\cL^1(\varphi(I_0))
  = \cL^1(\varphi^{-1}(\varphi(I_0))) \ge \cL^1(I_0)
  \ge 1-\varepsilon$.
  Moreover, by \eqref{eq:box-eps-mm-iso-2},
  \begin{align} \label{eq:box-eps-mm-iso-3}
    |d_{\dot X}(x,y) - d_Y(\dot f(x),\dot f(y))| \le \varepsilon
    \quad\text{for any $x,y \in \dot X_0$}.
  \end{align}
  Let us prove that $d_P(\dot f_*\mu_{\dot X},\mu_Y) \le \varepsilon$.
  For that, it suffices to prove that
  $\dot f_*\mu_{\dot X}(B_\varepsilon(A)) \ge \mu_Y(A) -\varepsilon$
  for any Borel subset $A \subset Y$.
  \begin{clm}
    For any Borel subset $A \subset Y$ we have
    \[
    \psi^{-1}(A) \cap I_0 \subset
    \varphi^{-1}({\varphi'}^{-1}(\psi^{-1}(B_\varepsilon(A)))).
    \]
  \end{clm}
  \begin{proof}
    Take any $s \in \psi^{-1}(A) \cap I_0$ and set
    $t := \varphi'(\varphi(s))$.
    Since $s \in I_0$, we have $t \in \varphi'(\varphi(I_0)) \subset I_0$
    and so
    \[
    |d_{\dot X}(\varphi(s),\varphi(t))-d_Y(\psi(s),\psi(t))| \le \varepsilon.
    \]
    It follows from $\varphi(t) = \varphi(\varphi'(\varphi(s))) = \varphi(s)$
    that $d_Y(\psi(s),\psi(t)) \le \varepsilon$.
    By $\psi(s) \in A$ we see that
    $\psi(\varphi'(\varphi(s))) = \psi(t) \in B_\varepsilon(A)$,
    so that $s \in \varphi^{-1}({\varphi'}^{-1}(\psi^{-1}(B_\varepsilon(A))))$.
    This proves the claim.
  \end{proof}
  By the claim and \eqref{eq:box-eps-mm-iso-1},
  \begin{align*}
    \dot f_*\mu_{\dot X}(B_\varepsilon(A))
    &= \mu_{\dot X}({\varphi'}^{-1}(\psi^{-1}(B_\varepsilon(A))))
    = \cL^1(\varphi^{-1}({\varphi'}^{-1}(\psi^{-1}(B_\varepsilon(A)))))\\
    &\ge \cL^1(\psi^{-1}(A) \cap I_0)
    = \cL^1(\psi^{-1}(A)) - \cL^1(\psi^{-1}(A) \setminus I_0)\\
    &\ge \cL^1(\psi^{-1}(A)) - \varepsilon
    = \mu_Y(A) - \varepsilon
  \end{align*}
  and therefore $d_P(\dot f_*\mu_{\dot X},\mu_Y) \le \varepsilon$.

  By setting $f := \dot f \circ \pi_{\cN} : X \to Y$,
  it is a Borel measurable map.
  It follows from $\mu_{\dot X} = (\pi_{\cN})_*\mu_X$ that
  $f_*\mu_X = \dot f_* (\pi_{\cN})_* \mu_X = \dot f_*\mu_{\dot X}$
  and hence $d_P(f_*\mu_X,\mu_Y) \le \varepsilon$.
  Let $X_0 := B_\varepsilon(\cN) \cap \pi_{\cN}^{-1}(\dot{X}_0)$.
  Then, any point $x \in X_0$ satisfies that
  $d_X(\pi_{\cN}(x),x) \le \varepsilon$ and $\pi_{\cN}(x) \in \dot{X}_0$,
  which together with \eqref{eq:box-eps-mm-iso-3}
  proves that $f$ is $3\varepsilon$-isometric on $X_0$.
  We also have
  \begin{align*}
    \mu_X(X \setminus X_0)
    &= \mu_X((X \setminus B_\varepsilon(\cN))
    \cup (X \setminus \pi_{\cN}^{-1}(\dot{X}_0)))\\
    &\le \mu_X(X \setminus B_\varepsilon(\cN))
    + \mu_X(X \setminus \pi_{\cN}^{-1}(\dot{X}_0))
    \le 2\varepsilon.
  \end{align*}
  Thus, $f : X \to Y$ is a $3\varepsilon$-mm-isomorphism.
  This completes the proof.
\end{proof}

\begin{prop} \label{prop:separable}
  The set $\cX$ of mm-isomorphism classes of mm-spaces
  is separable with respect to the box metric $\square$.
\end{prop}

\begin{proof}
  Let $n$ be a natural number.
  We set
  \begin{align*}
    R_n &:= \{\;(r_{ij})_{i,j=1,\dots,n;\ i<j} \in \R^{n(n-1)/2} \mid\\
    &\qquad r_{ij} > 0,\ r_{ik} \le r_{ij} + r_{jk}
    \ \text{for any $i < j < k$}\;\},\\
    W_n &:= \{\;(w_k)_{k=1,\dots,n-1} \in \R^{n-1} \mid
    w_k > 0,\ \sum_{k=1}^{n-1} w_k < 1\;\}.
  \end{align*}
  For $r \in R_n$ and $w \in W_n$, let $X(r,w) = \{x_1,x_2,\dots,x_n\}$
  be an $n$-point space equipped with the following mm-structure:
  \[
  d_{X(r,w)}(x_i,x_j) := r_{ij}
  \quad\text{and}\quad
  \mu_{X(r,w)} := \sum_{k=1}^n w_k \delta_{x_k},
  \]
  where $r_{ii} := 0$, $r_{ij} := r_{ji}$ for $i > j$,
  and $w_n := 1- \sum_{k=1}^{n-1} w_k$.
  Let $\varepsilon \ge 0$.
  For any $r,r' \in R_n$ and $w,w' \in W_n$
  such that
  \[
  | r_{ij} - r_{ij}' | \le \varepsilon
  \quad\text{and}\quad
  | w_k - w_k' | \le \varepsilon
  \quad\text{for any $i < j$ and $k$},
  \]
  we see $d_P(\mu_{X(r,w)},\mu_{X(r',w')}) \le n\varepsilon$
  and therefore, the identity map from $X(r,w)$ to $X(r',w')$
  is an $n\varepsilon$-mm-isomorphism.
  By Lemma \ref{lem:box-eps-mm-iso}, we have
  $\square(X(r,w),X(r',w')) \le 3n\varepsilon$.
  Denote by $\cX_n$ the mm-isomorphism classes of $n$-point mm-spaces.
  The map $R_n \times W_n \ni (r,w) \mapsto X(r,w) \in \cX_n$
  is Lipschitz continuous with respect to the $l_\infty$-norm on $R_n \times W_n$
  and $\square$.  Since $R_n \times W_n$ is separable,
  so is $(\cX_n,\square)$ for any natural number $n$.
  Therefore, the set of mm-isomorphism classes of finite mm-spaces,
  say $\cX_{<\infty}$, is separable with respect to $\square$.
  Since $\cX_{<\infty}$ is dense in $\cX$,
  this completes the proof.
\end{proof}

\begin{lem} \label{lem:precpt}
  For a family $\cY \subset \cX$, the following are equivalent
  to each other.
  \begin{enumerate}
  \item $\cY$ is $\square$-precompact.
  \item For any $\varepsilon > 0$ there exists a positive number
    $\Delta(\varepsilon)$ such that
    for any $X \in \cY$ we have a finite mm-space $X' \in \cX$
    such that $\square(X,X') \le \varepsilon$, $\# X' \le \Delta(\varepsilon)$,
    and $\diam X' \le \Delta(\varepsilon)$.
  \item For any $\varepsilon > 0$ there exists a positive number
    $\Delta(\varepsilon)$ such that
    for any $X \in \cY$ we have a finite net $\cN \subset X$
    $\varepsilon$-supporting $X$
    such that
    $\#\cN \le \Delta(\varepsilon)$ and
    $\diam\cN \le \Delta(\varepsilon)$.
  \item For any $\varepsilon > 0$ there exists a positive number
    $\Delta(\varepsilon)$ such that
    for any $X \in \cY$ we have Borel subsets
    $K_1,K_2,\dots,K_N \subset X$, $N \le \Delta(\varepsilon)$,
    such that $\diam K_i \le \varepsilon$ for any $i$,
    $\diam\bigcup_{i=1}^N K_i \le \Delta(\varepsilon)$, and
    $\mu_X(X \setminus \bigcup_{i=1}^N K_i) \le \varepsilon$.
  \end{enumerate}
\end{lem}

\begin{proof}
  We prove `(1) $\implies$ (2)'.
  By (1), for any $\varepsilon > 0$
  there are finitely many mm-spaces $X_1,\dots,X_{N(\varepsilon)} \in \cY$
  such that
  $\cY \subset \bigcup_{i=1}^{N(\varepsilon)} B_{\varepsilon/2}(X_i)$.
  By Proposition \ref{prop:fin-approx}, we have a finite mm-space
  $X_i'$ for each $i$ such that
  $\square(X_i,X_i') \le \varepsilon/2$.
  We define $\Delta(\varepsilon)$ to be the maximum of
  $\# X_i'$ and $\diam X_i'$ for all $i=1,2,\dots,N(\varepsilon)$.
  We then obtain (2).

  We prove `(2) $\implies$ (1)'.
  For a number $D > 0$ we set
  \[
  \cX_D := \{\;X \in \cX \mid \# X \le D,\ \diam X \le D\;\}.
  \]
  It is obvious that $\cX_D$ is $\square$-compact.
  Take any number $\varepsilon > 0$.
  By (2), there is a number $\Delta(\varepsilon/2)$ such that
  $\cY \subset B_{\varepsilon/2}(\cX_{\Delta(\varepsilon/2)})$.
  The $\square$-compactness of $\cX_{\Delta(\varepsilon/2)}$ implies that
  there is an $(\varepsilon/2)$-net
  $\{X_i\}_{i=1}^{N(\varepsilon)} \subset \cX_{\Delta(\varepsilon/2)}$, i.e.,
  $\cX_{\Delta(\varepsilon/2)} \subset B_{\varepsilon/2}(\{X_i\}_{i=1}^{N(\varepsilon)})$.
  We see that
  \[
  \cY \subset B_{\varepsilon/2}(\cX_{\Delta(\varepsilon/2)})
  \subset \bigcup_{i=1}^{N(\varepsilon)} B_\varepsilon(X_i),
  \]
  which implies (1).

  We prove `(2) $\implies$ (4)'.
  For any $\varepsilon > 0$ we have $\Delta(\varepsilon/4)$ as in (2).
  Namely, for any $X \in \cY$ there is a finite mm-space
  $X' \in \cX$ such that $\square(X,X') \le \varepsilon/4$,
  $\# X' \le \Delta(\varepsilon/4)$, and $\diam X' \le \Delta(\varepsilon/4)$.
  We find an $\varepsilon$-mm-isomorphism $f : X \to X'$.
  Let $\{x_1,\dots,x_N\} := X'$ and $K_i := f^{-1}(x_i) \cap X_0$,
  where $X_0$ is a non-exceptional domain of $f$.
  We have $\diam K_i \le \varepsilon$.
  Since $\diam X' \le \Delta(\varepsilon/4)$,
  it holds that
  $\diam \bigcup_{i=1}^N K_i \le \Delta(\varepsilon/4) + \varepsilon$.
  It follows from $\bigcup_{i=1}^N K_i = X_0$ that
  $\mu_X(X \setminus \bigcup_{i=1}^N K_i) \le \varepsilon$.
  (4) is obtained.

  We prove `(4) $\implies$ (3)'.
  For any $\varepsilon > 0$ there is a number $\Delta(\varepsilon)$
  such that for any $X \in \cX$ we have
  subsets $K_1,\dots,K_N \subset X$ as in (4).
  We take a point $x_i$ in each $K_i$ and set $\cN := \{x_i\}_{i=1}^N$.
  This satisfies (3).

  We prove `(3) $\implies$ (2)'.
  For any $\varepsilon > 0$ there is $\Delta(\varepsilon/2)$
  such that for any $X \in \cX$ we have a net $\cN \subset X$ as in (3).
  Let $\pi_{\cN} : X \to \cN$ be a Borel measurable
  nearest point projection and let
  $\dot X := (\cN,d_X,(\pi_{\cN})_*\mu_X)$.
  Proposition \ref{prop:box-di} and Lemma \ref{lem:pi-di} imply
  $\square(X,\dot X) \le \varepsilon$.
  Since $\#\dot X \le \Delta(\varepsilon/2)$
  and $\diam\dot X \le \Delta(\varepsilon/2)$,
  we obtain (2).

  This completes the proof of the lemma.
\end{proof}

\begin{defn}[Uniform family of mm-spaces]
  \index{uniform family}
  A family $\cY \subset \cX$ of mm-isomorphism classes of mm-spaces
  is said to be \emph{uniform} if
  \[
  \sup_{X\in\cY} \diam X < + \infty \quad\text{and}\quad
  \inf_{X\in\cY}\inf_{x \in X} \mu_X(B_\varepsilon(x)) > 0
  \]
  for any real number $\varepsilon > 0$.
\end{defn}

\begin{cor} \label{cor:precpt}
  Any uniform family in $\cX$ is $\square$-precompact.
\end{cor}

\begin{proof}
  Let $\cY \subset \cX$ be a uniform family.
  We will check (3) of Lemma \ref{lem:precpt} for $\cY$.
  Set
  \[
  D := \sup_{X\in\cY} \diam X < + \infty \quad\text{and}\quad
  a_\varepsilon := \inf_{X\in\cY}\inf_{x \in X} \mu_X(B_\varepsilon(x)) > 0.
  \]
  For any $\varepsilon > 0$ and $X \in \cY$,
  we find an $\varepsilon$-maximal net $\cN$ of $X$.
  The $\varepsilon$-maximality of $\cN$ implies that
  $B_\varepsilon(\cN) = \cN$ and in particular $\cN$ $\varepsilon$-supports
  $X$.
  It is clear that $\diam\cN \le D$.
  Since $\mu_X(B_{\varepsilon/2}(x)) \ge a_{\varepsilon/2} > 0$ for any $x \in X$
  and since $B_{\varepsilon/2}(x)$, $x \in \cN$, are disjoint to each other,
  the net $\cN$ has at most $[1/a_{\varepsilon/2}]$ points,
  where $[a]$ for a real number $a$ indicates
  the largest integer not greater than $a$.
  Lemma \ref{lem:precpt} proves the corollary.
\end{proof}

\begin{defn}[Doubling condition]
  \index{doubling condition}
  We say that an mm-space $X$ satisfies the \emph{doubling condition}
  if there exists a constant $C > 0$ such that
  \[
  \mu_X(B_{2r}(x)) \le C \mu_X(B_r(x))
  \]
  for any $x \in X$ and $r > 0$.
  The constant $C$ is called a \emph{doubling constant}.
  \index{doubling constant}
\end{defn}

\begin{rem} \label{rem:precpt}
  If an mm-space $X$ has diameter $\le D$ and doubling constant $C$,
  then, for any $x \in X$ and $\varepsilon > 0$,
  \[
  \mu_X(B_\varepsilon(x)) \ge \frac{\mu_X(B_{2\varepsilon}(x))}{C}
  \ge \dots \ge \frac{\mu_X(B_{2^n\varepsilon}(x))}{C^n} = C^{-n}
  \]
  for a natural number $n$ with $2^n\varepsilon \ge D$ and therefore
  \[
  \mu_X(B_\varepsilon(x)) \ge C^{-\log_2([D/\varepsilon]+1)}
  = ([D/\varepsilon]+1)^{-\log_2C}.
  \]
  Thus, if a family of mm-isomorphism classes of mm-spaces has
  an upper diameter bound and an upper bound of doubling constant,
  then the family is uniform and is $\square$-precompact.
  It follows from the Bishop-Gromov volume comparison theorem that
  the set of isometry classes of closed Riemannian manifold
  with an upper dimension bound, an upper diameter bound, and
  a lower bound of Ricci curvature has an upper bound of doubling constant
  and then it is uniform.
\end{rem}

\begin{defn}[Measured Gromov-Hausdorff convergence] \label{defn:mGH}
  \index{measured Gromov-Hausdorff convergence}
  Let $X$ and $X_n$, $n=1,2,\dots$, be compact mm-spaces.
  We say that \emph{$X_n$ measured Gromov-Hausdorff converges to $X$}
  as $n\to\infty$ if there exist Borel measurable
  $\varepsilon_n$-isometries $f_n : X_n \to X$, $n=1,2,\dots$,
  with $\varepsilon_n \to 0$ such that
  $(f_n)_*\mu_{X_n}$ converges weakly to $\mu_X$ as $n\to\infty$.
\end{defn}

If $X_n$ measured Gromov-Hausdorff converges to $X$, then
$X_n$ Gromov-Hausdorff converges to $X$.

\begin{rem} \label{rem:box-mGH}
  Let $X$ and $X_n$, $n=1,2,\dots$, be compact mm-spaces.
  If $X_n$ measured Gromov-Hausdorff converges to $X$ as $n\to\infty$,
  then the map $f_n$ as above is an $\varepsilon_n'$-mm-isomorphism
  with $\varepsilon_n' \to 0$, and, by Lemma \ref{lem:box-eps-mm-iso},
  $X_n$ $\square$-converges to $X$ as $n\to\infty$.

  If $X_n$ $\square$-converges to $X$ as $n\to\infty$,
  then $X_n$ does not necessarily converge to $X$
  in the sense of measured Gromov-Hausdorff convergence.
  In fact, we consider the sequence of mm-spaces $X_n$, $n=1,2,\dots$,
  defined by
  \begin{align*}
    & X_n := \{x_i\}_{i=0}^n, \qquad d_{X_n}(x_i,x_j) := 1-\delta_{ij},\\
    & \mu_{X_n} := \left(1-\frac{1}{n}\right)\delta_{x_0}
    + \sum_{i=1}^n \frac{1}{n}\delta_{x_i}.
  \end{align*}
  As $n\to\infty$, $X_n$ $\square$-converges to the one-point mm-space
  $(\{x_0\},\delta_{x_0})$.
  However, $\{X_n\}$ is not $d_{GH}$-precompact
  and has no $d_{GH}$-convergent subsequence.

  It is not difficult to prove that,
  if $X_n$ $\square$-converges to $X$ as $n\to\infty$ and if
  \[
  \inf_n \inf_{x \in X_n} \mu_{X_n}(B_\varepsilon(x)) > 0
  \]
  for any $\varepsilon > 0$,
  then $X_n$ measured Gromov-Hausdorff converges to $X$.
  The proof is left to the reader.
\end{rem}

\section{Lipschitz order and box convergence}

The purpose of this section is to prove the following theorem.

\begin{thm} \label{thm:dom}
  Let $X$, $Y$, $X_n$, and $Y_n$ be mm-spaces, $n=1,2,\dots$.
  If $X_n$ and $Y_n$ $\square$-converge
  to $X$ and $Y$ respectively as $n\to\infty$
  and if $X_n \prec Y_n$ for any $n$, then $X \prec Y$.
\end{thm}

For the proof, we need some lemmas.

\begin{lem} \label{lem:push-di-Lip}
  Let $X$ and $Y$ be two metric spaces, $f : X \to Y$ a Borel measurable map,
  $\mu$ and $\nu$ two Borel probability measures on $X$,
  and $\varepsilon > 0$ a real number.
  If there exists a Borel subset $X_0 \subset X$ such that
  $\mu(X_0) \ge 1-\varepsilon$, $\nu(X_0) \ge 1-\varepsilon$,
  and
  \[
  d_Y(f(x),f(y)) \le d_X(x,y) + \varepsilon
  \]
  for any $x,y \in X_0$, then we have
  \[
  d_P(f_*\mu,f_*\nu) \le d_P(\mu,\nu) + 2\varepsilon.
  \]
\end{lem}

\begin{proof}
  Take any Borel subset $A \subset Y$ and a real number $\delta > 0$.
  Let us first prove
  \begin{align} \label{eq:push-di-Lip}
    f^{-1}(U_{\varepsilon+\delta}(A)) \supset
    U_\delta(f^{-1}(A) \cap X_0) \cap X_0.
  \end{align}
  In fact, taking any point
  $x \in U_\delta(f^{-1}(A) \cap X_0) \cap X_0$
  we find a point $x' \in f^{-1}(A) \cap X_0$ such that
  $d_X(x,x') < \delta$.
  By the assumption,
  \[
  d_Y(f(x),f(x')) \le d_X(x,x') + \varepsilon < \varepsilon + \delta.
  \]
  Since $f(x') \in A$, we have $f(x) \in U_{\varepsilon+\delta}(A)$
  and so $x \in f^{-1}(U_{\varepsilon+\delta}(A))$.
  \eqref{eq:push-di-Lip} has been proved.
  
  We assume that $d_P(\mu,\nu) < \delta$ for a number $\delta$.
  It follows from \eqref{eq:push-di-Lip} and
  $\mu(X \setminus X_0),\nu(X \setminus X_0) \le \varepsilon$ that
  \begin{align*}
    &f_*\mu(U_{\varepsilon+\delta}(A)) = \mu(f^{-1}(U_{\varepsilon+\delta}(A)))
    \ge \mu(U_\delta(f^{-1}(A) \cap X_0) \cap X_0)\\
    &\ge \mu(U_\delta(f^{-1}(A) \cap X_0)) - \varepsilon
    \ge \nu(f^{-1}(A) \cap X_0)) - \varepsilon - \delta\\
    &\ge \nu(f^{-1}(A)) - 2\varepsilon - \delta
    = f_*\nu(A) - 2\varepsilon - \delta,
  \end{align*}
  which proves that $d_P(f_*\mu,f_*\nu) \le 2\varepsilon + \delta$.
  By the arbitrariness of $\delta$ we have the lemma.
\end{proof}

The following is a direct consequence of Lemma \ref{lem:push-di-Lip}.

\begin{cor} \label{cor:push-di-Lip}
  Let $X$ and $Y$ be a metric space and
  $f : X \to Y$ a $1$-Lipschitz map.
  For any Borel probability measures $\mu$ and $\nu$ on $X$,
  we have
  \[
  d_P(f_*\mu,f_*\nu) \le d_P(\mu,\nu).
  \]
\end{cor}

\begin{defn}[$1$-Lipschitz up to an additive error] \label{defn:1-Lip-err}
  \index{1-Lipschitz up to an additive error@$1$-Lipschitz up to an additive error}
  \index{1-Lipschitz up to epsilon@$1$-Lipschitz up to $\varepsilon$}
  \index{additive error}
  Let $X$ and $Y$ be two mm-spaces.
  A map $f : X \to Y$ is said to be \emph{$1$-Lipschitz up to}
    (\emph{an additive error}) $\varepsilon$ if there exists a Borel subset
  $X_0 \subset X$ such that
  \begin{align*}
    &\mu_X(X_0) \ge 1-\varepsilon,\tag{1}\\
    &d_Y(f(x),f(y)) \le d_X(x,y) + \varepsilon
    \qquad\text{for any $x,y \in X_0$.}\tag{2}
  \end{align*}
  We call such a set $X_0$ a \emph{non-exceptional domain of $f$} and
  the complement $X \setminus X_0$ an \emph{exceptional domain of $f$}.
  \index{non-exceptional domain} \index{exceptional domain}

  Even in the case where the metrics of $X$ and $Y$ are pseudo-metrics,
  a map that is $1$-Lipschitz up to an additive error
  is defined in the same way as above.
\end{defn}

\begin{lem} \label{lem:dom1}
  Let $X$ and $Y$ be two mm-spaces.
  Assume that there exist Borel measurable maps $p_n : X \to Y$,
  $n=1,2,\dots$, that are $1$-Lipschitz up to additive errors
  $\varepsilon_n$ with $\varepsilon_n \to 0$ such that
  \[
  d_P((p_n)_*\mu_X,\mu_Y) < \varepsilon_n
  \]
  for any $n$.
  Then, $X$ dominates $Y$.
\end{lem}

\begin{proof}
  We find an increasing sequence of compact subsets $C_N \subset Y$,
  $N=1,2,\dots$,
  with the property that $\mu_Y(C_N) \ge 1-1/N$ for any $N$.
  We have
  \[
  \mu_X(p_n^{-1}(B_{\varepsilon_n}(C_N))) = (p_n)_*\mu_X(B_{\varepsilon_n}(C_N))
  \ge \mu_Y(C_N) - \varepsilon_n \ge 1-\frac{1}{N}-\varepsilon_n.
  \]
  For any natural number $n$, there is a Borel subset
  $\tilde{X}_n \subset X$ such that
  $\mu_X(\tilde{X}_n) \ge 1-\varepsilon_n$ and
  \[
  d_Y(p_n(x),p_n(y)) \le d_X(x,y) + \varepsilon_n
  \]
  for any points $x,y \in \tilde{X}_n$.
  We see that
  $\mu_X(p_n^{-1}(B_{\varepsilon_n}(C_N)) \cap \tilde{X}_n) \ge 1-1/N-2\varepsilon_n$.
  Setting
  \[
  E_N := \bigcap_{m=1}^\infty \bigcup_{n=m}^\infty
  p_n^{-1}(B_{\varepsilon_n}(C_N)) \cap \tilde{X}_n,
  \]
  we have
  \[
  \mu_X(E_N)
  \ge \liminf_{n\to\infty} \mu_X(p_n^{-1}(B_{\varepsilon_n}(C_N)) \cap \tilde{X}_n)
  \ge 1-\frac{1}{N}.
  \]
  Since $\bigcup_{N=1}^\infty E_N$ is fully measured, it is dense in $X$.
  We take a dense countable subset
  $\{x_k\}_{k=1}^\infty \subset \bigcup_{N=1}^\infty E_N$.
  For any natural number $k$, we find a number $N(k)$ in such a way that
  $x_1,x_2,\dots,x_k \in E_{N(k)}$.
  Then, for any natural number $m$ there is a number $n \ge m$
  such that $p_n(x_k) \in B_{\varepsilon_n}(C_{N(k)})$ and $x_k \in \tilde{X}_n$.
  Therefore, there is a subsequence $\{p_{n^1_i}(x_1)\}_{i=1}^\infty$
  of $\{p_n(x_1)\}$
  that converges to a point in $C_{N(1)}$ and
  $x_1 \in \tilde{X}_{n^1_i}$ for any $i$.
  In the same way, there is a subsequence $\{p_{n^2_i}(x_2)\}_{i=1}^\infty$ of
  $\{p_{n^1_i}(x_2)\}_{i=1}^\infty$ that converges to a point in $C_{N(2)}$
  and $x_2 \in \tilde{X}_{n^2_i}$ for any $i$.
  Repeating this procedure we have an infinite sequence of subsequences
  $\{p_{n^1_i}\}, \{p_{n^2_i}\}, \{p_{n^3_i}\}, \dots$.
  By a diagonal argument, we choose a common subsequence 
  $\{p_{n_i}\}$ of $\{p_n\}$ such that
  $\{p_{n_i}(x_k)\}_{i=1}^\infty$ converges for each $k$ and
  $x_k \in \tilde{X}_{n_i}$ for any $k$ and $i$.
  Letting $f(x_k) := \lim_{i\to\infty} p_{n_i}(x_k)$ defines
  a $1$-Lipschitz map $f : \{x_k\}_{k=1}^\infty \to Y$.
  Since $\{x_k\}_{k=1}^\infty$ is dense in $X$, this map extends to
  a $1$-Lipschitz map $f : X \to Y$.
  For any $\varepsilon > 0$,
  there is a number $k_0$ such that
  $\mu_X( \bigcup_{k=1}^{k_0} B_\varepsilon(x_k)) \ge 1-\varepsilon$
  and $1/N(k_0) \le \varepsilon$.
  We see that
  $\mu_X( \bigcup_{k=1}^{k_0} B_\varepsilon(x_k) \cap E_{N(k_0)})
  \ge 1-1/N(k_0)-\varepsilon \ge 1-2\varepsilon$.
  For the number $\varepsilon$, there is a natural number $i_0$
  such that
  for any number $i \ge i_0$ we have
  $\varepsilon_{n_i} \le \varepsilon$ and
  $d_Y(p_{n_i}(x_k),f(x_k)) \le \varepsilon$ for all $k=1,2,\dots,k_0$.
  We set
  \[
  D_{\varepsilon,i} := \bigcup_{k=1}^{k_0}
  B_\varepsilon(x_k) \cap E_{N(k_0)} \cap \tilde{X}_{n_i}.
  \]
  Let $i$ be any number with $i \ge i_0$.
  It holds that $\mu_X(D_{\varepsilon,i}) \ge 1-3\varepsilon$.
  For any point $x \in D_{\varepsilon,i}$ there is a number
  $k(x) \in \{1,2,\dots,k_0\}$ such that $d_X(x,x_{k(x)}) \le \varepsilon$.
  Since $x,x_{k(x)} \in \tilde{X}_{n_i}$,
  we have $d_Y(p_{n_i}(x),p_{n_i}(x_{k(x)})) \le d_X(x,x_{k(x)})+\varepsilon_{n_i}
  \le 2\varepsilon$.
  We also have $d_Y(p_{n_i}(x_{k(x)}),f(x_{k(x)})) \le \varepsilon$.
  It follows from the $1$-Lipschitz continuity of $f$ that
  $d_Y(f(x_{k(x)}),f(x)) \le d_X(x,x_{k(x)}) \le \varepsilon$.
  Combining these inequalities yields $d_Y(p_{n_i}(x),f(x)) \le 4\varepsilon$
  for any $x \in D_{\varepsilon,i}$, so that, by Lemma \ref{lem:di-me},
  $d_P((p_{n_i})_*\mu_X,f_*\mu_X) \le \dKF(p_{n_i},f) \le 4\varepsilon$.
  As $i\to\infty$, $(p_{n_i})_*\mu_X$ converges weakly to $f_*\mu_X$.
  Since the assumption implies the weak convergence
  $(p_{n_i})_*\mu_X \to \mu_Y$, we obtain $f_*\mu_X = \mu_Y$.
  This completes the proof.
\end{proof}

\begin{proof}[Proof of Theorem \ref{thm:dom}]
  By the assumption, there are a sequence $\varepsilon_n \to 0$,
  $\varepsilon_n$-mm-isomorphisms $\varphi_n : X_n \to X$ and
  $\psi_n : Y \to Y_n$.
  In addition, there is a $1$-Lipschitz map $f_n : Y_n \to X_n$ with
  $(f_n)_*\mu_{Y_n} = \mu_{X_n}$.
  Since $d_P((\psi_n)_*\mu_Y,\mu_{Y_n}) \le \varepsilon_n$,
  Corollary \ref{cor:push-di-Lip} implies
  \begin{align}
    \label{eq:dom-conv}
    d_P((f_n\circ\psi_n)_*\mu_Y,\mu_{X_n})
    &= d_P((f_n)_*(\psi_n)_*\mu_Y,(f_n)_*\mu_{Y_n})\\
    \notag &\le d_P((\psi_n)_*\mu_Y,\mu_{Y_n}) \le \varepsilon_n.
  \end{align}
  For the $\varepsilon_n$-mm-isomorphism $\varphi_n : X_n \to X$,
  there is a Borel subset $\tilde{X}_n \subset X_n$ such that
  $\mu_{X_n}(\tilde{X}_n) \ge 1-\varepsilon_n$ and
  $|d_X(\varphi_n(x),\varphi_n(x'))-d_{X_n}(x,x')| \le \varepsilon_n$
  for any $x,x' \in \tilde{X}_n$.
  By Lemma \ref{lem:eps-proj}, we have a Borel measurable
  $\varepsilon_n$-projection $\pi_n : X_n \to \tilde{X}_n$
  with $\pi_n|_{\tilde{X}_n} = \id_{\tilde{X}_n}$.
  Let $\varphi_n' := \varphi_n\circ\pi_n : X_n \to X$.
  It follows from $\mu_{X_n}(\tilde{X}_n) \ge 1-\varepsilon_n$ that
  $d_{KF}^{\mu_{X_n}}(\varphi_n,\varphi_n') \le \varepsilon_n$, which together
  with Lemma \ref{lem:di-me} implies
  $d_P((\varphi_n)_*\mu_{X_n},(\varphi_n')_*\mu_{X_n}) \le \varepsilon_n$.
  For any $x,x' \in B_{\varepsilon_n}(\tilde{X}_n)$ we have
  \begin{align} \label{eq:varphi-prime}
    |d_X(\varphi_n'(x),\varphi_n'(x'))-d_{X_n}(x,x')| \le 3\varepsilon_n.
  \end{align}
  Moreover, \eqref{eq:dom-conv} implies
  \[
  (f_n\circ\psi_n)_*\mu_Y(B_{\varepsilon_n}(\tilde{X}_n))
  \ge \mu_{X_n}(\tilde{X}_n)-\varepsilon_n \ge 1-2\varepsilon_n.
  \]
  Let
  \[
  F_n := \varphi_n'\circ f_n\circ\psi_n : Y \to X.
  \]
  This is a Borel measurable map.
  It follows from Lemma \ref{lem:push-di-Lip} that
  \begin{align*}
    d_P((F_n)_*\mu_Y,(\varphi_n')_*\mu_{X_n})
    &= d_P((\varphi_n')_*(f_n\circ\psi_n)_*\mu_Y,(\varphi_n')_*\mu_{X_n})\\
    &\le d_P((f_n\circ\psi_n)_*\mu_Y,\mu_{X_n}) + 6\varepsilon_n
    \le 7\varepsilon_n.
  \end{align*}
  Since $d_P((\varphi_n')_*\mu_{X_n},\mu_X)
  \le d_P((\varphi_n')_*\mu_{X_n},(\varphi_n)_*\mu_{X_n})
  + d_P((\varphi_n)_*\mu_{X_n},\mu_X) \le 2\varepsilon_n$,
  we have
  \begin{align} \label{eq:Fn-push}
    d_P((F_n)_*\mu_Y,\mu_X) \le 9\varepsilon_n.
  \end{align}
  For the $\varepsilon_n$-mm-isomorphism $\psi_n : Y \to Y_n$,
  there is a Borel subset $\tilde{Y}_n \subset Y$ such that
  $\mu_Y(\tilde{Y}_n) \ge 1-\varepsilon_n$ and
  $|d_{Y_n}(\psi_n(y),\psi_n(y'))-d_Y(y,y')| \le \varepsilon_n$
  for any $y,y' \in \tilde{Y}_n$.
  We set
  \[
  \tilde{Y}'_n := \tilde{Y}_n \cap
  (f_n\circ\psi_n)^{-1}(B_{\varepsilon_n}(\tilde{X}_n)).
  \]
  Since $\mu_Y((f_n\circ\psi_n)^{-1}(B_{\varepsilon_n}(\tilde{X}_n))
  = (f_n\circ\psi_n)_*\mu_Y(B_{\varepsilon_n}(\tilde{X}_n)) \ge 1-2\varepsilon_n$,
  we have $\mu_Y(\tilde{Y}'_n) \ge 1-3\varepsilon_n$.
  For any $y,y' \in \tilde{Y}'_n$, we have
  \[
  d_{X_n}(f_n\circ\psi_n(y),f_n\circ\psi_n(y'))
  \le d_{Y_n}(\psi_n(y),\psi_n(y'))
  \le d_Y(y,y')+\varepsilon_n,
  \]
  so that,
  by \eqref{eq:varphi-prime},
  \[
  d_X(F_n(y),F_n(y')) \le d_Y(y,y') + 4\varepsilon_n.
  \]
  This proves that $F_n$ is $1$-Lipschitz up to $4\varepsilon_n$.
  Applying Lemma \ref{lem:dom1} completes the proof.
\end{proof}

\section{Finite-dimensional approximation}

The purpose in this section is to prove that
any mm-space can be approximated by $\R^N$ with a Borel probability measure.

\begin{defn}[$\Lip_1(X)$, $\underline\mu_N$, and $\underline{X}_N$]
  \index{Lip1X@$\Lip_1(X)$}
  \index{muN@$\underline\mu_N$}
  \index{XN@$\underline{X}_N$}
  Let $X$ be an mm-space.
  Denote by $\Lip_1(X)$ the set of $1$-Lipschitz functions on $X$.
  For $\varphi_i \in \Lip_1(X)$, $i=1,2,\dots,N$, we define
  \begin{align*}
    \Phi_N &:= (\varphi_1,\dots,\varphi_N) : X \to \R^N,\\
    \underline{\mu}_N &:= (\Phi_N)_*\mu_X,\\
    \underline{X}_N &:= (\R^N,\|\cdot\|_\infty,\underline{\mu}_N).
  \end{align*}
\end{defn}

\begin{prop}
  We have
  \[
  \underline{X}_1 \prec \underline{X}_2 \prec \dots \prec \underline{X}_N
  \prec X.
  \]
\end{prop}

\begin{proof}
  We prove $\underline{X}_N \prec X$.
  In fact, for any $x,y \in X$, we have
  \[
  \|\Phi_N(x)-\Phi_N(y)\|_\infty
  = \max_{i=1,2,\dots,N} |\varphi_i(x)-\varphi_i(y)|
  \le d_X(x,y),
  \]
  i.e., $\Phi_N : X \to \R^N$ is $1$-Lipschitz continuous.
  We therefore obtain $\underline{X}_N \prec X$.

  We prove $\underline{X}_n \prec \underline{X}_{n+1}$.
  Since the projection $\pr : \underline{X}_{n+1} \to \underline{X}_n$
  is $1$-Lipschitz continuous and since $\pr \circ \Phi_{n+1} = \Phi_n$,
  we have
  \[
  \pr_*\underline{\mu}_{n+1} = \pr_*(\Phi_{n+1})_*\mu_X
  = (\pr\circ\Phi_{n+1})_*\mu_X = (\Phi_n)_*\mu_X = \underline{\mu}_n.
  \]
  Therefore, $\underline{X}_n \prec \underline{X}_{n+1}$.
  This completes the proof.
\end{proof}

\begin{defn}[$\Lo(X)$]
  \index{l1@$\Lo(X)$}
  We define an action of $\R$ on $\Lip_1(X)$ by
  \[
  \R \times \Lip_1(X) \ni (c,f) \mapsto f+c \in \Lip_1(X).
  \]
  Let
  \[
  \Lo(X) := \Lip_1(X)/\R
  \]
  be the quotient space of $\Lip_1(X)$ by the $\R$-action.
  We indicate by $[f]$ the $\R$-orbit of $f \in \Lip_1(X)$.
  For two orbits $[f], [g] \in \Lo(X)$,
  we define
  \[
  \dKF([f],[g]) := \inf_{f'\in [f],\; g'\in [g]}
  \dKF(f',g').
  \]
  \index{dKF@$\dKF$}
\end{defn}

Since any $\R$-orbit is closed in $\Lip_1(X)$,
we see that $\dKF$ is a metric on $\Lo(X)$.

It is easy to prove the continuity of the $\R$-action on $\Lip_1(X)$
with respect to $\dKF$, which implies the following lemma.

\begin{lem} \label{lem:me-const}
  For any two functions $f,g \in \Lip_1(X)$,
  there exists a real number $c$ such that
  \[
  \dKF([f],[g]) = \dKF(f,g+c).
  \]
\end{lem}

To prove the $\dKF$-compactness of $\Lo(X)$ we need the following

\begin{lem} \label{lem:cpt-1Lip}
  Let $Y$ be a proper metric space
  and $F_i : X \to Y$, $i = 1,2,\dots$, $1$-Lipschitz maps.
  If $\{F_i(x_0)\}_{i=1}^\infty$ has a convergent subsequence
  for a point $x_0 \in X$, then
  some subsequence of $\{F_i\}$ converges in measure
  to a $1$-Lipschitz map $F : X \to Y$.
\end{lem}

\begin{proof}
  We take a monotone decreasing sequence $\varepsilon_j \to 0$
  as $j \to \infty$.
  There is an increasing sequence of compact subsets
  $K_1 \subset K_2 \subset \dots$ of $X$ such that
  $x_0 \in K_1$, $\mu_X(X \setminus K_j) < \varepsilon_j$ for any $j$,
  and $\bigcup_{j=1}^\infty K_j = X$.
  By the Arzel\`a-Ascoli theorem,
  $\{F_i\}$ has a subsequence $\{F_{i_1(k)}\}$ that
  converges to a $1$-Lipschitz map $G_1$ uniformly on $K_1$.
  $\{F_{i_1(k)}\}$ has a subsequence $\{F_{i_2(k)}\}$ that
  converges to a $1$-Lipschitz map $G_2$ uniformly on $K_2$.
  Repeat this procedure to get subsequences
  $\{F_{i_j(k)}\}$, $j=1,2,\dots$, such that, for each $j$,
  $\{F_{i_j(k)}\}$ converges to a $1$-Lipschitz map $G_j$ uniformly on $K_j$.
  By the diagonal argument, there is a common subsequence $\{F_{i(k)}\}$
  that converges to $G_j$ uniformly on each $K_j$.
  We see that $G_{j+1}|_{K_j} = G_j$ for any $j$.
  We define a map $F : X \to Y$ by $F|_{K_j} := G_j$ for any $j$.
  Then, $F : X \to Y$ is $1$-Lipschitz continuous.
  Since $F_{i(k)}$ converges to $F$ uniformly on each $K_j$,
  it converges in measure to $F$ on $X$.
  This completes the proof.
\end{proof}

\begin{prop}
  $\Lo(X)$ is $\dKF$-compact.
\end{prop}

\begin{proof}
  Take a sequence $[\varphi_i] \in \Lo(X)$, $i = 1,2,\dots$.
  We may assume that $\varphi_i(x_0) = 0$ for each $i$.
  We apply Lemma \ref{lem:cpt-1Lip} to obtain
  a subsequence $\{\varphi_{i(j)}\}$ of $\{\varphi_i\}$
  that converges in measure to a function $\varphi \in \Lip_1(X)$.
  $[\varphi_{i(j)}]$ $\dKF$-converges to $[\varphi]$ as $j\to\infty$.
  This completes the proof.
\end{proof}

\begin{thm} \label{thm:XN}
  Let $X$ be an mm-space.
  For any $\varepsilon > 0$, there exists a real number
  $\delta = \delta(X,\varepsilon) > 0$ such that
  if $\{[\varphi_i]\}_{i=1}^N \subset \Lo(X)$ is a $\delta$-net
  with respect to $\dKF$, then
  \[
  \square(\underline{X}_N,X) \le \varepsilon.
  \]
\end{thm}

\begin{proof}
  Take any $\varepsilon > 0$ and fix it.
  For a number $\rho > 0$ we set
  \[
  Z_\rho := \{\; x \in X \mid \mu_X(B_\varepsilon(x)) \le \rho\;\}.
  \]
  $Z_\rho$ is monotone nondecreasing in $\rho$.
  Since $\mu_X(B_\varepsilon(x)) > 0$ for any $x \in X$, we have
  $\bigcap_{\rho > 0} Z_\rho = \emptyset$ and so
  $\lim_{\rho\to 0+} \mu_X(Z_\rho) = 0$.
  Therefore, there is a number
  $\delta = \delta(X,\varepsilon) > 0$ such that
  $\delta \le \varepsilon$ and $\mu_X(Z_\delta) < \varepsilon$.
  We take any two points $x,y \in X \setminus Z_\delta$ and fix them.
  The definition of $Z_\delta$ implies that
  \begin{align}
    \mu_X(B_\varepsilon(x)) > \delta \quad\text{and}\quad
    \mu_X(B_\varepsilon(y)) > \delta. \label{eq:XN}
  \end{align}
  Let $\{[\varphi_i]\}_{i=1}^N \subset \Lo(X)$ be a $\delta$-net.
  There is a number $i_0 \in \{1,2,\dots,N\}$ such that
  $\dKF([d_x],[\varphi_{i_0}]) \le \delta$,
  where $d_x(z) := d_X(x,z)$ for $x,z \in X$.
  We find a real number $c$ with $\dKF(d_x,\varphi_{i_0}^{(c)}) \le \delta$
  (see Lemma \ref{lem:me-const}),
  where $\varphi_{i_0}^{(c)} := \varphi_{i_0} + c$.
  Since $\mu_X(|d_x-\varphi_{i_0}^{(c)}| > \delta) \le \delta$
  (see Remark \ref{rem:me}) and by \eqref{eq:XN},
  the two balls $B_\varepsilon(x)$ and $B_\varepsilon(y)$ both intersect
  the set $\{\,|d_x-\varphi_{i_0}^{(c)}| \le \delta\,\}$,
  so that there are two points $x' \in B_\varepsilon(x)$ and
  $y' \in B_\varepsilon(y)$ such that
  \begin{align}
    |d_X(x,x')-\varphi_{i_0}^{(c)}(x')| \le \delta
    \quad\text{and}\quad
    |d_X(x,y')-\varphi_{i_0}^{(c)}(y')| \le \delta. \label{eq:XN2}
  \end{align}
  Therefore,
  \begin{align}
    &|d_X(x,y)-\varphi_{i_0}^{(c)}(y)| \label{eq:XN3}\\
    &\le |d_X(x,y)-d_X(x,y')| + |d_X(x,y')-\varphi_{i_0}^{(c)}(y')|
    +|\varphi_{i_0}^{(c)}(y')-\varphi_{i_0}^{(c)}(y)| \notag\\
    &\le d_X(y,y') + \delta + d_X(y,y')
    \le 2\varepsilon + \delta \le 3\varepsilon. \notag
  \end{align}
  The $1$-Lipschitz continuity of $|\varphi_{i_0}^{(c)}|$ and \eqref{eq:XN2}
  show
  \begin{align}
    |\varphi_{i_0}^{(c)}(x)| \le |\varphi_{i_0}^{(c)}(x')| + d_X(x,x')
    \le 2\, d_X(x,x')+\delta \le 3\varepsilon. \label{eq:XN4}
  \end{align}
  Combining \eqref{eq:XN3} and \eqref{eq:XN4} yields
  \begin{align*}
    d_X(x,y) &\le \varphi_{i_0}^{(c)}(y) + 3\varepsilon
    \le \varphi_{i_0}^{(c)}(y) - \varphi_{i_0}^{(c)}(x) + 6\varepsilon\\
    &= \varphi_{i_0}(y) - \varphi_{i_0}(x) + 6\varepsilon
    \le \|\Phi_N(x) - \Phi_N(y)\|_\infty + 6\varepsilon,
  \end{align*}
  which together with the $1$-Lipschitz continuity of $\Phi_N$
  implies that
  \[
  |\; \|\Phi_N(x)-\Phi_N(y)\|_\infty - d_X(x,y) \;| \le 6\varepsilon.
  \]
  for any $x,y \in X \setminus Z_\delta$.
  By recalling $\mu_X(Z_\delta) < \varepsilon$,
  the map $\Phi_N : X \to \underline{X}_N$ turns out to be
  a $6\varepsilon$-mm-isomorphism.
  By Lemma \ref{lem:box-eps-mm-iso}(1) we obtain
  $\square(\underline{X}_N,X) \le 18\varepsilon$.
  This completes the proof.
\end{proof}

\begin{cor} \label{cor:XN}
  If a sequence $\{[\varphi_i]\}_{i=1}^\infty \subset \Lo(X)$ is $\dKF$-dense,
  then $\underline{X}_N$ $\square$-converges to $X$ as $N \to \infty$.
\end{cor}

Note that the compactness of $\Lo(X)$ implies
the existence of a dense countable subset of $\Lo(X)$,
so that any mm-space $X$ can be approximated by
some finite-dimensional space $\underline{X}_N$.

\begin{rem}
  The infinite-product $[\,0,1\,]^\infty$ of the interval $[\,0,1\,]$
  is not separable with respect to $\|\cdot\|_\infty$,
  the proof of which is seen in Remark \ref{rem:nonsep}.
  Therefore, $([\,0,1\,]^\infty,\|\cdot\|_\infty,\cL^\infty)$
  is not an mm-space, where $\cL^\infty$ is the infinite-product
  of the one-dimensional Lebesgue measure on $[\,0,1\,]$.
  Besides, $\Lo([\,0,1\,]^\infty,\|\cdot\|_\infty,\cL^\infty)$
  is not $\dKF$-compact.
  In fact, letting $\varphi_i : [\,0,1\,]^\infty \to \R$ be
  the $i^{th}$ projection, we have
  $\dKF([\varphi_i],[\varphi_j]) = c\,(1-\delta_{ij})$ for some constant $c > 0$,
  where $\delta_{ii} = 1$ and $\delta_{ij} = 0$ if $i \neq j$.
\end{rem}

\begin{prop}
  Let $X$ be an mm-space
  and $\{[\varphi_i]\}_{i=1}^\infty \subset \Lo(X)$ an $\dKF$-dense countable subset.
  We define
  \begin{align*}
    \Phi_\infty &:= (\varphi_1,\varphi_2,\dots) : X \to \R^\infty,\\
    \underline{\mu}_\infty &:= (\Phi_\infty)_*\mu_X,\\
    \underline{X}_\infty &:= (\Phi_\infty(X),\|\cdot\|_\infty,\underline{\mu}_\infty).
  \end{align*}
  Then, $\Phi_\infty : X \to \underline{X}_\infty$ is an isometry.
  In particular, the image $\Phi_\infty(X)$ is separable
  with respect to $\|\cdot\|_\infty$,
  the space $\underline{X}_\infty$ is an mm-space,
  and $\Phi_\infty : X \to \underline{X}_\infty$
  is an mm-isomorphism.
\end{prop}

\begin{proof}
  It suffices to prove that $\Phi_\infty : X \to \underline{X}_\infty$
  is an isometry.
  The proof is similar to that of Theorem \ref{thm:XN}.
  Since each $\varphi_i$ is $1$-Lipschitz,
  $\Phi_\infty : X \to \underline{X}_\infty$ is $1$-Lipschitz continuous.
  Take any two points $x,y \in X$.
  Since $\{[\varphi_i]\}_{i=1}^\infty \subset \Lo(X)$ is $\dKF$-dense,
  there is a subsequence $\{[\varphi_{i_j}]\}_{j=1}^\infty$ of
  $\{[\varphi_i]\}_{i=1}^\infty$ such that
  $\lim_{j\to\infty} \dKF([\varphi_{i_j}],[d_x]) = 0$
  and hence $\varphi_{i_j}^{(c_j)}$ converges in measure to $d_x$ as $j\to\infty$
  for some $c_j$, where $\varphi_{i_j}^{(c_j)} := \varphi_{i_j} + c_j$.
  This proves that, for any $\varepsilon > 0$,
  the two balls $B_\varepsilon(x)$ and $B_\varepsilon(y)$ both intersect
  $\{\,|\varphi_{i_j}^{(c_j)}-d_x| \le \varepsilon\,\}$ for all $j$ large enough.
  Therefore,
  we find two sequences $\{x_j\}_{j=1}^\infty$ and $\{y_j\}_{j=1}^\infty$
  of points in $X$ respectively converging to $x$ and $y$ such that
  \[
  \lim_{j\to\infty} \varphi_{i_j}^{(c_j)}(x_j) = d_x(x) = 0 \quad\text{and}\quad
  \lim_{j\to\infty} \varphi_{i_j}^{(c_j)}(y_j) = d_x(y) = d_X(x,y).
  \]
  Thus,
  \begin{align*}
     &\|\Phi_\infty(x) - \Phi_\infty(y)\|_\infty
     \ge |\varphi_{i_j}(x)-\varphi_{i_j}(y)|
     = |\varphi_{i_j}^{(c_j)}(x)-\varphi_{i_j}^{(c_j)}(y)|\\
     &\ge |\varphi_{i_j}^{(c_j)}(x_j)-\varphi_{i_j}^{(c_j)}(y_j)|
     - d_X(x,x_j) - d_X(y,y_j)
     \overset{j\to\infty}{\longrightarrow} d_X(x,y),
  \end{align*}
  which together with the $1$-Lipschitz continuity of $\Phi_\infty$
  proves that $\Phi_\infty$ is an isometry.
  This completes the proof.
\end{proof}

\section{Infinite product, I}

In this section, we prove that a finite product space
$\square$-converges to the infinite product.

Let $p$ be an extend real number with $1 \le p \le +\infty$,
and $F_n$, $n=1,2,\dots$, be mm-spaces.
We set
\begin{align*}
  X_n := F_1 \times F_2 \times \dots \times F_n
  \quad\text{and}\quad
  X_\infty := \prod_{n=1}^\infty F_n.
\end{align*}
Define, for two points $x = (x_i)_{i=1}^n, y = (y_i)_{i=1}^n \in X_n$,
$1 \le n \le \infty$,
\[
d_{l_p}(x,y) :=
\begin{cases}
  \left( \sum\limits_{i=1}^n d_{F_i}(x_i,y_i)^p \right)^{\frac{1}{p}}
  &\text{if $p < +\infty$,}\\
  \sup\limits_{i=1}^n d_{F_i}(x_i,y_i)
  &\text{if $p = +\infty$,}
\end{cases}
\]
\index{dlp@$d_{l_p}$}
Note that $d_{l_p}(x,y) \le +\infty$.

\begin{asmp} \label{asmp:finite}
  If $p < +\infty$, then
  \[
  \sum_{n=1}^\infty (\diam F_n)^p < +\infty.
  \]
  If $p = +\infty$, then
  \[
  \lim_{n\to\infty} \diam F_n = 0.
  \]
\end{asmp}

\begin{lem} \label{lem:prod-top}
  \begin{enumerate}
  \item The topology on $X_\infty$ induced from $d_{l_p}$
    is not weaker than the product topology.
  \item Under Assumption \ref{asmp:finite},
    the topology on $X_\infty$ induced from $d_{l_p}$
    coincides with the product topology.
  \end{enumerate}
\end{lem}

\begin{proof}
  We prove (1).
  Recall that the product topology is generated by
  \[
  \cB := \left\{\prod_{n=1}^\infty O_n \mid
  \text{$O_n \subset F_n$ is open,
    $\exists k$ : $O_n = F_n$ for any $n \ge k$}\;\right\}
  \]
  as an open basis.
  It suffices to prove that any set $\prod_{n=1}^\infty O_n \in \cB$ is an open set
  with respect to $d_{l_p}$.
  We have $k$ such that $O_n = F_n$ for any $n \ge k$.
  Take any point $x = (x_n)_{n=1}^\infty \in \prod_{n=1}^\infty O_n$.
  Since each $O_n$ is open, for any number $n < k$ there is a number
  $\delta_n > 0$ such that $U_{\delta_n}(x_n) \subset O_n$.
  Let $\delta := \min\{\delta_1,\dots,\delta_k\}$.
  For any point $y = (y_n)_{n=1}^\infty \in U_\delta(x)$,
  since $d_{F_n}(x_n,y_n) \le d_{l_p}(x,y) < \delta$, we have $y_n \in O_n$
  for any $n$.
  $U_\delta(x)$ is contained in $\prod_{n=1}^\infty O_n$.
  By the arbitrariness of $x \in \prod_{n=1}^\infty O_n$,
  the set $\prod_{n=1}^\infty O_n$ is open with respect to $d_{l_p}$.
  (1) is obtained.

  We prove (2) in the case of $p < +\infty$.
  ((2) in the case of $p = +\infty$ is proved
  in the same way.)
  Recall that the family of open $d_{l_p}$-metric balls is an open basis
  for the topology induced from $d_{l_p}$.
  It suffices to prove that
  any open metric ball $U_\varepsilon(x)$, $\varepsilon > 0$, $x \in X_\infty$,
  is open with respect to the product topology.
  Take any point $y \in U_\varepsilon(x)$.
  Since
  \[
  \left( \sum_{n=1}^\infty d_{F_n}(x_n,y_n)^p \right)^{\frac{1}{p}} = d_{l_p}(x,y)
  < \varepsilon,
  \]
  there is a natural number $k$ such that
  \[
  \left( \sum_{n=k+1}^\infty (\diam F_n)^p \right)^{\frac{1}{p}}
  < \varepsilon - d_{l_p}(x,y).
  \]
  We find $k$ real numbers $\delta_1,\dots,\delta_k > 0$
  in such a way that
  \[
  \left( \sum_{n=1}^k \delta_n^p \right)^{\frac{1}{p}}
  < \varepsilon - d_{l_p}(x,y)
  - \left( \sum_{n=k+1}^\infty (\diam F_n)^p \right)^{\frac{1}{p}}.
  \]
  The set
  \[
  O := U_{\delta_1}(y_1) \times \dots \times U_{\delta_k}(y_k) \times
  F_{k+1} \times F_{k+2} \times \dots
  \]
  is open with respect to the product topology
  and contains $y$.
  It suffices to prove that $O \subset U_\varepsilon(x)$.
  If $z \in O$, then
  \begin{align*}
    d_{l_p}(y,z)
    &\le \left( \sum_{n=1}^k \delta_n^p
      + \sum_{n=k+1}^\infty (\diam F_n)^p \right)^{\frac{1}{p}}\\
    &\le \left( \sum_{n=1}^k \delta_n^p \right)^{\frac{1}{p}}
    + \left( \sum_{n=k+1}^\infty (\diam F_n)^p \right)^{\frac{1}{p}}\\
    &< \varepsilon - d_{l_p}(x,y)
  \end{align*}
  and so $d_{l_p}(x,z) \le d_{l_p}(x,y) + d_{l_p}(y,z) < \varepsilon$.
  We obtain $O \subset U_\varepsilon(x)$.

  This completes the proof of the lemma.
\end{proof}

\begin{rem}
  Supposing that Assumption \ref{asmp:finite} does not hold,
  the topology on $X_\infty$ induced from $d_{l_p}$ is strictly stronger than
  the product topology.
\end{rem}

\begin{proof}
  If $p < +\infty$ and if $\sum_{n=1}^\infty (\diam F_n)^p = +\infty$,
  then any set in the open basis $\cB$ as in the proof of
  Lemma \ref{lem:prod-top}
  has infinite $d_{l_p}$-diameter and is not contained in
  any $d_{l_p}$-metric ball,
  so that any $d_{l_p}$-metric open ball is not open with respect to
  the product topology.

  If $p = +\infty$ and if $\delta := \liminf_{n\to\infty} \diam F_n > 0$,
  then any set in $\cB$ has $d_{l_\infty}$-diameter $\ge \delta$
  and is not contained in any $d_{l_\infty}$-metric ball of radius $< \delta/2$.
  Any $d_{l_\infty}$-metric open ball of radius $< \delta/2$
  is not open with respect to the product topology in this case.
  This completes the proof.
\end{proof}

\begin{lem} \label{lem:comp-sep}
  Under Assumption \ref{asmp:finite} we have the following.
  \begin{enumerate}
  \item If each $F_n$ is complete, then so is $(X_\infty,d_{l_p})$.
  \item If each $F_n$ is separable, then so is $(X_\infty,d_{l_p})$.
  \end{enumerate}
\end{lem}

\begin{proof}
  We prove (1).
  Let $\{x_i\}_{i=1}^\infty$ be a Cauchy sequence in $(X_\infty,d_{l_p})$.
  We set $(x_{in})_{n=1}^\infty := x_i$.
  Since $d_{F_n}(x_{in},x_{jn}) \le d_{l_p}(x_i,x_j)$,
  the sequence $\{x_{in}\}_i$ is a Cauchy sequence in $F_n$ for each $n$,
  so that it is a convergent sequence.
  Let
  \[
  x_{\infty,n} := \lim_{i\to\infty} x_{in}
  \quad\text{and}\quad
  x_\infty := (x_{\infty,n})_{n=1}^\infty.
  \]
  If $p < +\infty$, then
  \begin{align*}
    d_{l_p}(x_i,x_\infty)^p
    &= \sum_{n=1}^\infty \lim_{j\to\infty} d_{F_i}(x_{in},x_{jn})^p\\
    &\le \liminf_{j\to\infty} \sum_{n=1}^\infty d_{F_n}(x_{in},x_{jn})^p\\
    &= \liminf_{j\to\infty} d_{l_p}(x_i,x_j)^p\\
    &\to 0 \ \text{as}\ i\to\infty,
  \end{align*}
  and therefore $x_i$ converges to $x_\infty$ as $i\to\infty$.
  In the case of $p = +\infty$, we see that $x_i$ converges to $x_\infty$
  in the same way.
  (1) has been proved.
  

  We prove (2).
  By the second countability of each $F_n$,
  there is a countable open basis $\{O_{ni}\}_{i=1}^\infty$ of $F_n$,
  where we assume that $F_n$ belongs to $\{O_{ni}\}_{i=1}^\infty$.
  By Lemma \ref{lem:prod-top}, the topology of $(X_\infty,d_{l_p})$
  coincides with the product topology, so that
  \begin{align*}
    \cB := \{\;& O_{1i_1} \times O_{2i_2} \times \dots \times O_{ki_k}
    \times F_{k+1} \times F_{k+2} \times \dots\\
    &|\ k=1,2,\dots, \ i_1,i_2,\dots,i_k=1,2,\dots\}
  \end{align*}
  is an open basis over $X_\infty$.
  For each $k$, the family
  \begin{align*}
    \cB_k := \{\;& O_{1i_1} \times O_{2i_2} \times \dots \times O_{ki_k}
    \times F_{k+1} \times F_{k+2} \times \dots\\
    &| \ i_1,i_2,\dots,i_k=1,2,\dots,k\;\}
  \end{align*}
  is finite and satisfies $\bigcup_{k=1}^\infty\cB_k = \cB$.
  Therefore, $\cB$ is a countable basis.
  This completes the proof.
\end{proof}

\begin{rem} \label{rem:nonsep}
  Let $F_n$, $n=1,2,\dots$, be metric spaces.
  If $\inf_n \diam F_n > 0$, then
  $(X_\infty,d_{l_\infty})$ is not separable.
  In particular, the infinite product space $F^\infty$ of an mm-space $F$
  with the $d_{l_p}$ metric and the infinite product measure
  is not an mm-space whenever $F$ has
  two different points.
\end{rem}

\begin{proof}
  Let $\delta := \inf_n \diam F_n > 0$.
  We take any countable set $\{x_k\}_{k=1}^\infty \subset X_\infty$
  and put
  \[
  x_k = (x_{k1},x_{k2},\dots),\quad x_{kn} \in F_n.
  \]
  For any two natural numbers $k$ and $n$,
  there is a point $y_{kn} \in F_n$ such that
  $d_{F_n}(x_{kn},y_{kn}) \ge \delta/2$.
  Setting $y := (y_{11},y_{22},\dots) \in X_\infty$ we have
  \[
  d_{l_\infty}(x_k,y) = \sup_n d_{F_n}(x_{kn},y_{nn})
  \ge d_{F_k}(x_{kk},y_{kk}) \ge \delta/2.
  \]
  Therefore, $\{x_k\}_{k=1}^\infty$ is not dense in $X_\infty$.
  This completes the proof.
\end{proof}

\begin{prop}
  Let $F_n$, $n=1,2,\dots$, be mm-spaces and let $1 \le p \le +\infty$.
  We equip $X_n$ and $X_\infty$ with the product measures
  $\bigotimes_{k=1}^n \mu_{F_k}$ and $\bigotimes_{k=1}^\infty \mu_{F_k}$
  and with the $d_{l_p}$ metrics.
  Under Assumption \ref{asmp:finite}, we have the following
  {\rm(1)}, {\rm(2)}, and {\rm(3)}.
  \begin{enumerate}
  \item $X_\infty$ is an mm-space.
  \item We have
    \[
    X_1 \prec X_2 \prec \dots \prec X_n \prec X_\infty,
    \quad n=1,2,\dots.
    \]
  \item $X_n$ $\square$-converges to $X_\infty$ as $n\to\infty$.
  \end{enumerate}
\end{prop}

\begin{proof}
  We prove (1).
  Lemma \ref{lem:comp-sep} says that $X_\infty$ is complete separable
  and hence an mm-space.
  
  (2) is obvious.

  We prove (3).
  Fixing a point $x_0 = (x_{0n})_{n=1}^\infty \in X_\infty$,
  we define an isometric embedding map
  \[
  \iota_n : X_n \ni (x_1,\dots,x_n) \mapsto
  (x_1,\dots,x_n,x_{0,n+1},x_{0,n+2},\dots) \in X_\infty.
  \]
  According to \cite{Bog}*{8.2.16},
  we obtain that $(\iota_n)_*\mu_{X_n}$ converges weakly to $\mu_{X_\infty}$
  as $n\to\infty$.
  Moreover, $X_n$ is mm-isomorphic to $(X_\infty,(\iota_n)_*\mu_{X_n})$.
  Proposition \ref{prop:box-di} proves (3).
  This completes the proof.
\end{proof}

\chapter{Observable distance and measurement}
\label{chap:obs-dist-measurement}

\section{Basics for the observable distance}

We define the observable distance $\dconc(X,Y)$
between two mm-spaces $X$ and $Y$,
and study its basic properties.

\begin{defn}[$\dconc$ between pseudo-metrics]
  \index{dconcrho1rho2@$\dconc(\rho_1,\rho_2)$}
  For a pseudo-metric $\rho$ on $I$, let $\Lip_1(\rho)$ denotes
  the set of $1$-Lipschitz functions on $I$ with respect to $\rho$.
  Denote by $\mathbf{D}$ the set of pseudo-metrics $\rho$ on $I$
  such that every element of $\Lip_1(\rho)$ is a Lebesgue measurable function.
  For two pseudo-metrics $\rho_1,\rho_2 \in \mathbf{D}$, we set
  \[
  \dconc(\rho_1,\rho_2) := d_H(\Lip_1(\rho_1),\Lip_1(\rho_2)),
  \]
  where the Hausdorff distance $d_H$
  is defined with respect to the Ky Fan metric $\dKF$.
\end{defn}

Since $\dKF \le 1$, we have $\dconc(\rho_1,\rho_2) \le 1$
for any $\rho_1,\rho_2 \in \mathbf{D}$.

\begin{lem}
  Let $X$ be an mm-space.
  For any parameter $\varphi$ of $X$ we have
  \[
  \Lip_1(\varphi^*d_X) = \varphi^*\Lip_1(X)
  := \{\;f\circ\varphi \mid f \in \Lip_1(X)\;\}
  \]
  and, in particular, $\varphi^*d_X$ belongs to $\mathbf{D}$. 
\end{lem}

\begin{proof}
  It is easy to see that $\Lip_1(\varphi^*d_X) \supset \varphi^*\Lip_1(X)$.

  We prove $\Lip_1(\varphi^*d_X) \subset \varphi^*\Lip_1(X)$.
  For any function $f \in \Lip_1(\varphi^*d_X)$,
  we have
  \begin{align} \label{eq:Lip1}
    |f(s)-f(t)| \le d_X(\varphi(s),\varphi(t))
  \end{align}
  for any $s,t \in I$.
  In particular, if $\varphi(s) = \varphi(t)$, then $f(s) = f(t)$.
  For a given point $x \in \varphi(I)$, we take a point $s \in I$
  with $\varphi(s) = x$.
  Then, $f(s)$ depends only on $x$ and independent of $s$.
  We define $\tilde{f}(x) := f(s)$.
  It follows from \eqref{eq:Lip1} that $\tilde{f} : \varphi(I) \to \R$
  is $1$-Lipschitz continuous and extends to a $1$-Lipschitz function
  $\tilde{f} : X \to \R$.
  The definition of $\tilde{f}$ implies that $f = \tilde{f}\circ\varphi$.
  This completes the proof.
\end{proof}

\begin{defn}[Observable distance $\dconc(X,Y)$] \label{defn:obs-dist}
  \index{observable distance} \index{dconcXY@$\dconc(X,Y)$}
  We define the \emph{observable distance $\dconc(X,Y)$ between
    two mm-spaces $X$ and $Y$} by
  \[
  \dconc(X,Y) := \inf_{\varphi,\psi} \dconc(\varphi^*d_X,\psi^*d_Y),
  \]
  where $\varphi : I \to X$ and $\psi : I \to Y$ run over all parameters
  of $X$ and $Y$, respectively.
  We say that a sequence of mm-spaces $X_n$, $n=1,2,\dots$,
  \emph{concentrates to} \index{concentrate} an mm-space $X$
  if
  \[
  \lim_{n\to\infty} \dconc(X_n,X) = 0.
  \]
\end{defn}

Note that $\dconc(X,Y) \le 1$ for any two mm-spaces $X$ and $Y$.

We prove that $\dconc$ is a metric on $\cX$ later in Theorem \ref{thm:dconc}.

\begin{lem} \label{lem:Lip-approx}
  Let $S$ be a topological space with a Borel probability measure $\mu_S$,
  and $\rho$ a pseudo-metric on $S$ such that any $1$-Lipschitz
  function on $S$ with respect to $\rho$ is Borel measurable.
  For any Borel measurable map $f : S \to (\R^N,\|\cdot\|_\infty)$
  that is $1$-Lipschitz up to an additive error $\varepsilon \ge 0$
  with respect to $\rho$,
  there exists a $1$-Lipschitz map $\tilde f : S \to (\R^N,\|\cdot\|_\infty)$
  such that
   \begin{align*}
     \dKF(\tilde{f},f) \le \varepsilon.
   \end{align*}
\end{lem}

\begin{proof}
  By the assumption, there is a (non-exceptional) Borel subset $S_0 \subset S$
  such that $\mu_S(S_0) \ge 1-\varepsilon$ and
  \[
  \|f(x)-f(y)\|_\infty \le \rho(x,y) + \varepsilon
  \]
  for any $x,y \in S_0$.
  We set $(f_1,f_2,\dots,f_N) := f$.
  For any $i$ and $x,y \in S_0$,
  \[
  |f_i(x)-f_i(y)| \le \|f(x)-f(y)\|_\infty \le \rho(x,y) + \varepsilon.
  \]
  We define, for $x \in S$,
  \[
  \tilde{f}_i(x) := \inf_{y \in S_0} (f_i(y)+\rho(x,y)).
  \]
  Then, for any $x,y \in S$,
  \begin{align*}
    \tilde{f}_i(x)-\tilde{f}_i(y)
    &= \inf_{x' \in S_0} (f_i(x')+\rho(x,x'))
    - \inf_{y' \in S_0} (f_i(y')+\rho(y,y'))\\
    &\le \sup_{x' \in S_0} (f_i(x')+\rho(x,x')-f_i(x')-\rho(y,x'))\\
    &\le \rho(x,y).
  \end{align*}
  Since this also holds if we exchange $x$ and $y$,
  the function $\tilde{f}_i$ is $1$-Lipschitz continuous and so is
  $\tilde{f} := (\tilde{f}_1,\tilde{f}_2,\dots,\tilde{f}_N)
  : S \to (\R^N,\|\cdot\|_\infty)$.
  For any $x \in S_0$, we have $\tilde{f}_i(x) \le f_i(x)$
  and
  \[
  f_i(x) - \tilde{f}_i(x) = \sup_{y \in S_0} (f_i(x)-f_i(y)-\rho(x,y))
  \le \varepsilon.
  \]
  Since this holds for any $i$, we have
  \[
  \|\tilde{f}(x)-f(x)\|_\infty \le \varepsilon
  \]
  for any $x \in S_0$.
  We obtain $\dKF(\tilde{f},f) \le \varepsilon$.
  This completes the proof.
\end{proof}

\begin{prop} \label{prop:dconc-box}
  \begin{enumerate}
  \item For any two pseudo-metrics $\rho_1, \rho_2 \in \mathbf{D}$,
    we have
    \[
    \dconc(\rho_1,\rho_2) \le \square(\rho_1,\rho_2).
    \]
  \item For any two mm-spaces $X$ and $Y$ we have
    \[
    \dconc(X,Y) \le \square(X,Y).
    \]
  \end{enumerate}
\end{prop}

\begin{proof}
  (2) follows from (1).

  We prove (1).
  Assume that $\square(\rho_1,\rho_2) < \varepsilon$ for two pseudo-metrics
  $\rho_1,\rho_2 \in \mathbf{D}$ and for a number $\varepsilon$.
  There is a Borel subset $I_0 \subset I$ such that
  $\cL^1(I_0) > 1-\varepsilon$ and
  \[
  |\;\rho_1(s,t) - \rho_2(s,t)\;| < \varepsilon
  \]
  for any $s,t \in I_0$.
  Take any function $f \in \Lip_1(\rho_1)$.
  We then see that $f$ is $1$-Lipschitz up to the additive error $\varepsilon$
  with respect to $\rho_2$.
  Apply Lemma \ref{lem:Lip-approx} to obtain a function
  $\tilde{f} \in \Lip_1(\rho_2)$ with $\dKF(\tilde{f},f) \le \varepsilon$.
  Therefore, $\Lip_1(\rho_1) \subset B_\varepsilon(\Lip_1(\rho_2))$.
  Since this also holds by exchanging $\rho_1$ and $\rho_2$, we have
  $d_H(\Lip_1(\rho_1),\Lip_1(\rho_2)) \le \varepsilon$.
  This completes the proof.
\end{proof}

We denote by $*$ a one-point mm-space, i.e.,
$*$ consists of a single point with trivial metric and Dirac's delta measure.

\begin{lem} \label{lem:dconc-pt}
  Let $X$ be an mm-space.  Then we have
  \[
  \dconc(X,*) = \sup_{f \in \Lip_1(X)}\inf_{c \in \R} \dKF(f,c).
  \]
\end{lem}

\begin{proof}
  Let $\varphi : I \to X$ and $\psi : I \to *$ be two parameters.
  Since $\psi^*\Lip_1(*)$ is the set of constant functions on $I$,
  we have
  \begin{align*}
    d_H(\varphi^*\Lip_1(X),\psi^*\Lip_1(*))
    &= \sup_{f \in \Lip_1(X)}\inf_{c\in\R} \dKF(f\circ\varphi,c)\\
    &= \sup_{f \in \Lip_1(X)}\inf_{c\in\R} \dKF(f,c).
  \end{align*}
\end{proof}

\begin{prop} \label{prop:ObsDiam-dconc}
  For any mm-space $X$ we have
  \[
  \dconc(X,*) \le \ObsDiam(X) \le 2\dconc(X,*).
  \]
\end{prop}

\begin{proof}
  We prove the first inequality.
  Assume that $\ObsDiam(X) < \varepsilon$ for a number $\varepsilon$.
  Then we have $\ObsDiam(X;-\varepsilon) < \varepsilon$
  and so $\diam(f_*\mu_X;1-\varepsilon) < \varepsilon$
  for any function $f \in \Lip_1(X)$.
  There are two numbers $a < b$ such that
  $f_*\mu_X([\,a,b\,]) \ge 1-\varepsilon$ and $b-a < \varepsilon$.
  We set $c := (a+b)/2$.
  Since $[\,a,b\,] \subset (\,c-\varepsilon/2,c+\varepsilon/2\,)$,
  we have
  $\mu_X(|f-c| < \varepsilon/2) \ge 1-\varepsilon$,
  which implies $\dKF(f,c) \le \varepsilon$.
  By Lemma \ref{lem:dconc-pt}, $\dconc(X,*) \le \varepsilon$.
  The first inequality has been proved.

  We prove the second inequality.
  Let $\varepsilon := \dconc(X,*)$.
  By Lemma \ref{lem:dconc-pt},
  for any function $f \in \Lip_1(X)$ there is a real number $c$
  such that $\dKF(f,c) \le \varepsilon$.
  Since $f_*\mu_X([\,c-\varepsilon,c+\varepsilon\,])
  = \mu_X(|f-c| \le \varepsilon) \ge 1-\varepsilon$,
  we have
  $\diam(f_*\mu_X;1-\varepsilon) \le 2\varepsilon$
  and so $\ObsDiam(X;-\varepsilon) \le 2\varepsilon$.
  This completes the proof.
\end{proof}

Proposition \ref{prop:ObsDiam-dconc} implies the following

\begin{cor} \label{cor:ObsDiam-dconc}
  Let $X_n$, $n=1,2,\dots$, be mm-spaces.
  The sequence $\{X_n\}_{n=1}^\infty$ is a L\'evy family if and only if
  it concentrates to a one-point mm-space.
\end{cor}

\begin{defn}[$h$-Homogeneous measure]
  \index{h-homogeneous measure@$h$-homogeneous measure}
  A Borel measure $\mu$ on a metric space $X$ is said to be
  $h$-\emph{homogeneous}, $h \ge 1$,
  if
  \[
  \mu(B_r(x)) \le h\,\mu(B_r(y))
  \]
  for any two points $x,y \in X$ and for any $r > 0$.
\end{defn}

Note that any $h$-homogeneous Borel measure on a metric space
has full support.

The following is a slight extension of a result in \cite{Funano:est-box}.

\begin{prop} \label{prop:homog}
  Let $X_n$, $n=1,2,\dots$, be mm-spaces with $h$-homogeneous measure
  for a real number $h \ge 1$.
  If $\{X_n\}$ is a L\'evy family and if
  the diameter of $X_n$ is bounded away from zero,
  then $\{X_n\}$ has no $\square$-convergent subsequence.
\end{prop}

\begin{proof}
  Let $X_n$, $n=1,2,\dots$, be mm-spaces with $h$-homogeneous measure,
  and assume that $\{X_n\}$ is a L\'evy family and
  the diameter of $X_n$ is bounded away from zero.
  Suppose that $\{X_n\}$ has a $\square$-convergent subsequence.
  Replacing $\{X_n\}$ with the subsequence,
  we assume that $X_n$ is $\square$-convergent as $n\to\infty$.
  It follows from Proposition \ref{prop:dconc-box} and Corollary \ref{cor:ObsDiam-dconc}
  that $X_n$ $\square$-converges to the one-point mm-space $*$.
  We set
  \[
  r := \min\left\{\frac{1}{2(1+h)},\,\frac{1}{5}\inf_n \diam X_n\right\}.
  \]
  There is a number $n_0$ such that $\square(X_n,*) < r$ for all $n \ge n_0$.
  For each $n \ge n_0$,
  there are a parameter $\varphi_n : I \to X_n$ and a measurable set
  $I_n \subset I$ 
  such that $\cL^1(I_n) \ge 1-r$ and
  $d_{X_n}(\varphi_n(s),\varphi_n(t)) \le r$
  for any $s,t \in I_n$.
  We take a point $s_n \in I_n$ for each $n$ and fix it.
  Since $\diam X_n \ge 5r$,
  there is a point $x_n \in X_n$ for each $n \ge n_0$ such that
  $d_{X_n}(\varphi_n(s_n),x_n) > 2r$.
  Since $B_r(\varphi_n(s_n))$ and $B_r(x_n)$ are disjoint to each other,
  we have
  \begin{align*}
    1 &\ge \mu_{X_n}(B_r(\varphi_n(s_n))) + \mu_{X_n}(B_r(x_n))\\
    &\ge (1+h^{-1})\mu_{X_n}(B_r(\varphi_n(s_n)))
    = (1+h^{-1})\cL^1(\varphi_n^{-1}(B_r(\varphi_n(s_n))))
    \intertext{and, since $\varphi_n^{-1}(B_r(\varphi_n(s_n))) \supset I_n$,
    the last term is}
    &\ge (1+h^{-1})\cL^1(I_n) \ge (1+h^{-1})(1-r)\\
    &\ge \left(1+\frac{1}{h}\right)\left(1-\frac{1}{2(1+h)}\right)
    = 1+ \frac{1}{2h} > 1,
  \end{align*}
  which is a contradiction.
  This completes the proof.
\end{proof}

Since $\{S^n(1)\}$, $\{\C P^n\}$, $\{SO(n)\}$, $\{SU(n)\}$,
and $\{Sp(n)\}$ are all L\'evy families of mm-spaces with $1$-homogeneous measure
(see Examples \ref{ex:CPn} and \ref{ex:SOSUSp}), we have the following.

\begin{cor} \label{cor:homog}
  Any subsequence of $\{S^n(1)\}$, $\{\C P^n\}$, $\{SO(n)\}$, $\{SU(n)\}$,
  and $\{Sp(n)\}$ is $\square$-divergent.
\end{cor}

\section{$N$-Measurement
  and nondegeneracy of the observable distance}

In this section, we prove that the observable distance function
is a metric on $\cX$ by using measurements.

\begin{defn}[$N$-Measurement]
  \index{measurement} \index{N-measurement@$N$-measurement}
  Let $X$ be an mm-space and $N$ a natural number.
  Denote by $\cM(N)$ the set of Borel probability measures on $\R^N$
  equipped with the Prohorov metric $d_P$.
  We call the subset of $\cM(N)$
  \[
  \cM(X;N) := \{\; F_*\mu_X \mid F : X \to (\R^N,\|\cdot\|_\infty)
  \ \text{is $1$-Lipschitz}\;\}
  \]
  \index{MXN@$\cM(X;N)$}
  the \emph{$N$-measurement of $X$}.
\end{defn}

\begin{lem}
  The $N$-measurement $\cM(X;N)$ of $X$ is a closed subset of $\cM(N)$.
\end{lem}

\begin{proof}
  Assume that a sequence $(F_i)_*\mu_X \in \cM(X;N)$ $d_P$-converges
  to a measure $\mu \in \cM(N)$, where $F_i : X \to (\R^N,\|\cdot\|_\infty)$, $i=1,2,\dots$,
  are $1$-Lipschitz maps.
  It suffices to prove that $\mu$ belongs to $\cM(X;N)$.

  Let us first prove that $\{F_i(x_0)\}_{i=1}^\infty$ is a bounded sequence
  in $\R^N$, where $x_0$ is a point in $X$.
  In fact, the $1$-Lipschitz continuity of $F_i$
  implies $F_i(B_1(x_0)) \subset B_1(F_i(x_0))$ and so
  \[
  (F_i)_*\mu_X(B_1(F_i(x_0))) = \mu_X(F_i^{-1}(B_1(F_i(x_0))))
  \ge \mu_X(B_1(x_0)) > 0.
  \]
  If $\{F_i\}$ has a subsequence $\{F_{i_j}\}$ for which
  $\|F_{i_j}(x_0)\|_\infty$ diverges to infinity as $j \to \infty$,
  then, since $B_1(F_{i_j}(x_0))$ does not intersect 
  $U^N_R(o) := \{\,x \in \R^N \mid \|x\|_\infty < R\,\}$ for any fixed $R > 0$
  and for every sufficiently large $j$,
  we have
  \begin{align*}
    \mu(\R^N \setminus U^N_R(o))
    &\ge \liminf_{j\to\infty} (F_{i_j})_*\mu_X(\R^N \setminus U^N_R(o))\\
    &\ge \liminf_{j\to\infty} (F_{i_j})_*\mu_X(B_1(F_{i_j}(x_0)))
    \ge \mu_X(B_1(x_0)) > 0,
  \end{align*}
  which is a contradiction to
  $\lim_{R\to +\infty} \mu(\R^N \setminus U^N_R(o)) = 0$.
  Therefore, $\{F_i(x_0)\}$ is bounded.

  Applying Lemma \ref{lem:cpt-1Lip} yields that $\{F_i\}$ has
  a subsequence $\{F_{i_j}\}$ that converges in measure to
  a $1$-Lipschitz map $F : X \to (\R^N,\|\cdot\|_\infty)$.
  By Lemma \ref{lem:di-me},
  $(F_{i_j})_*\mu_X$ $d_P$-converges to $F_*\mu_X$.
  We thus obtain $\mu = F_*\mu_X \in \cM(X;N)$.
\end{proof}

\begin{lem} \label{lem:dom-M}
  For two mm-spaces $X$ and $Y$, the following {\rm(1)} and {\rm(2)}
  are equivalent to each other.
  \begin{enumerate}
  \item $X$ is dominated by $Y$.
  \item $\cM(X;N) \subset \cM(Y;N)$ for any natural number $N$.
  \end{enumerate}
\end{lem}

\begin{proof}
  We prove `(1) $\implies$ (2)'.
  $X \prec Y$ implies that
  there is a $1$-Lipschitz map $f : Y \to X$ such that $f_*\mu_Y = \mu_X$.
  Take any $F_*\mu_X \in \cM(X;N)$, where
  $F : X \to (\R^N,\|\cdot\|_\infty)$ is any $1$-Lipschitz map.
  The composition $F \circ f : Y \to (\R^N,\|\cdot\|_\infty)$ is $1$-Lipschitz
  and
  $F_*\mu_X = F_*f_*\mu_Y = (F \circ f)_*\mu_Y \in \cM(Y;N)$,
  so that we have $\cM(X;N) \subset \cM(Y;N)$.

  We prove `(2) $\implies$ (1)'.
  Assume that $\cM(X;N) \subset \cM(Y;N)$ for any natural number $N$.
  According to Corollary \ref{cor:XN},
  there is a sequence of measures $\underline{\mu}_N \in \cM(X;N)$,
  $N=1,2,\dots$, such that
  $\underline{X}_N = (\R^N,\|\cdot\|_\infty,\underline{\mu}_N)$
  $\square$-converges to $X$ as $N\to\infty$.
  The assumption implies $\underline{\mu}_N \in \cM(Y;N)$,
  so that we have a $1$-Lipschitz map $F : Y \to \R^N$
  with $\underline{\mu}_N = F_*\mu_Y$,
  which means that $\underline{X}_N$ is dominated by $Y$.
  By Theorem \ref{thm:dom}, $X$ is dominated by $Y$.
  This completes the proof.
\end{proof}

\begin{lem} \label{lem:M-dconc}
  Let $X$ and $Y$ be two mm-spaces.
  For any natural number $N$ we have
  \[
  d_H(\cM(X;N),\cM(Y;N)) \le N\cdot\dconc(X,Y),
  \]
  where the Hausdorff distance $d_H$ is defined with respect to
  the Prohorov metric $d_P$.
\end{lem}

\begin{proof}
  Assume that $\dconc(X,Y) < \varepsilon$ for a number $\varepsilon$.
  There are two parameters $\varphi : I \to X$ and $\psi : I \to Y$
  such that
  \begin{equation}
    \label{eq:M-dconc}
    d_H(\varphi^*\Lip_1(X),\psi^*\Lip_1(Y)) < \varepsilon.    
  \end{equation}
  Let us prove that $\cM(X;N) \subset B_{N\varepsilon}(\cM(Y;N))$.
  Take any $F_*\mu_X \in \cM(X;N)$, where $F : X \to (\R^N,\|\cdot\|_\infty)$
  is a $1$-Lipschitz map.
  Setting $(f_1,\dots,f_N) := F$ we have $f_i \in \Lip_1(X)$ and so
  $f_i \circ\varphi \in \varphi^*\Lip_1(X)$.
  By \eqref{eq:M-dconc},
  there is a function $g_i \in \Lip_1(Y)$ such that
  $\dKF(f_i\circ\varphi,g_i\circ\psi) < \varepsilon$.
  Since $G := (g_1,\dots,g_N) : Y \to (\R^N,\|\cdot\|_\infty)$ is $1$-Lipschitz,
  we have $G_*\mu_Y \in \cM(Y;N)$.
  We prove $d_P(F_*\mu_X,G_*\mu_Y) \le N\varepsilon$ in the following.
  For this, it suffices to prove
  $F_*\mu_X(B_\varepsilon(A)) \ge G_*\mu_Y(A) - N\varepsilon$
  for any Borel subset $A \subset \R^N$.
  Since $F_*\mu_X = (F\circ\varphi)_*\cL^1$
  and $G_*\mu_Y = (G\circ\psi)_*\cL^1$, we have
  \[
  F_*\mu_X(B_\varepsilon(A)) = \cL^1((F\circ\varphi)^{-1}(B_\varepsilon(A))),
  \quad
  G_*\mu_Y(A) = \cL^1((G\circ\psi)^{-1}(A)).
  \]
  It is sufficient to prove
  \[
  \cL^1((G\circ\psi)^{-1}(A) \setminus (F\circ\varphi)^{-1}(B_\varepsilon(A)))  \le N\varepsilon.
  \]
  If we take
  $s \in (G\circ\psi)^{-1}(A) \setminus (F\circ\varphi)^{-1}(B_\varepsilon(A))$,
  then
  $G\circ\psi(s) \in A$ and $F \circ\varphi(s) \notin B_\varepsilon(A)$
  together imply
  \[
  \|F\circ\varphi(s) - G\circ\psi(s)\|_\infty > \varepsilon
  \]
  and therefore
  \begin{align*}
    &\cL^1((G\circ\psi)^{-1}(A) \setminus
    (F\circ\varphi)^{-1}(B_\varepsilon(A)))\\
    &\le \cL^1(\{\;s \in I \mid
    \|F\circ\varphi(s) - G\circ\psi(s)\|_\infty > \varepsilon\;\})\\
    &= \cL^1\left( \bigcup_{i=1}^N \{\;s\in I \mid
      |f_i\circ\varphi(s)-g_i\circ\psi(s)| > \varepsilon\;\}\right)\\
    &\le \sum_{i=1}^N \cL^1(\{\;s\in I \mid
    |f_i\circ\varphi(s)-g_i\circ\psi(s)| > \varepsilon\;\})\\
    &\le N\varepsilon,
  \end{align*}
  where the last inequality follows from
  $\dKF(f_i\circ\varphi,g_i\circ\psi) < \varepsilon$.
  We thus obtain $d_P(F_*\mu_X,G_*\mu_Y) \le N\varepsilon$,
  so that $\cM(X;N) \subset B_{N\varepsilon}(\cM(Y;N))$.
  Since this also holds if we exchange $X$ and $Y$,
  we have
  \[
  d_H(\cM(X;N),\cM(Y;N)) \le N\varepsilon.
  \]
  This completes the proof.
\end{proof}

\begin{thm} \label{thm:dconc}
  The function $\dconc$ is a metric on $\cX$.
\end{thm}

\begin{proof}
  The symmetricity is clear.

  A triangle inequality is obtained in the same way as in the proof of
  Theorem \ref{thm:box}, by using Lemma \ref{lem:box-tri}
  and Proposition \ref{prop:dconc-box}.

  We prove the nondegeneracy.
  Assume that $\dconc(X,Y) = 0$ for two mm-spaces $X$ and $Y$.
  Then, Lemma \ref{lem:M-dconc} implies that $\cM(X;N) = \cM(Y;N)$
  for any $N$, which together with Lemma \ref{lem:dom-M} yields
  that $X \prec Y$ and $Y \prec X$.  By Proposition \ref{prop:Liporder},
  $X$ and $Y$ are mm-isomorphic to each other.
  This completes the proof.
\end{proof}

\begin{rem}
  Without the mm-reconstruction theorem, we obtain the
  nondegeneracy of the box metric $\square$ in the same way as in
  the proof of Theorem \ref{thm:dconc}.
\end{rem}

\begin{defn}[Concentration topology]
  \index{concentration topology}
  We call the topology on $\cX$ induced from $\dconc$
  the \emph{concentration topology}.
\end{defn}

\begin{prop} \label{prop:box-dconc-precpt}
  Let $\cY \subset \cX$ be a $\square$-precompact family of
  mm-isomorphism classes of mm-spaces.
  Then, the concentration topology coincides with
  the topology induced from the box metric
  on the $\square$-closure of $\cY$.
\end{prop}

\begin{proof}
  Let $\overline{\cY}^\square$ be the $\square$-closure of $\cY$.
  It follows from the completeness of $\cX$ (see Theorem \ref{thm:box-complete})
  that $\overline{\cY}^\square$ is $\square$-compact.
  Applying the homeomorphism theorem for the identity map
  $\id_\cX : (\overline{\cY}^\square,\square) \to (\overline{\cY}^\square,\dconc)$
  yields that it is a homeomorphism.
  This completes the proof.
\end{proof}

Combining Proposition \ref{prop:box-dconc-precpt}
and Corollary \ref{cor:precpt} implies

\begin{cor} \label{cor:box-dconc-precpt}
  Let $\cY \subset \cX$ be a uniform family of mm-isomorphism classes
  of mm-spaces.
  Then, the concentration topology coincides with
  the topology induced from the box metric on the $\square$-closure of $\cY$.
  In particular, if $\{X_n\}_{n=1}^\infty$ is a uniform sequence of mm-spaces
  that concentrates to an mm-space $X$,
  then $X_n$ $\square$-converges to $X$.
\end{cor}

\begin{rem}
  Recall that $\{S^n(1)\}$ concentrates to a one-point space,
  but has no $\square$-convergent subsequence (see Corollary \ref{cor:homog}).
  It is a non-uniform sequence.
  A non-uniform sequence of mm-spaces is more interesting
  than uniform one for the study of concentration.
\end{rem}

\section{Convergence of $N$-measurements}

In this section we prove the following.

\begin{thm}[Observable criterion for concentration] \label{thm:A}
  \index{observable criterion for concentration}
  Let $X$ and $X_n$, $n=1,2,\dots$, be mm-spaces.
  Then, the following {\rm(1)} and {\rm(2)} are equivalent to each other.
  \begin{enumerate}
  \item The sequence $\{X_n\}$ concentrates to $X$.
  \item For any natural number $N$, the $N$-measurement $\cM(X_n;N)$ of $X_n$
    converges to $\cM(X;N)$ with respect to
    the Hausdorff distance defined from the Prohorov metric.
  \end{enumerate}
\end{thm}

`(1) $\implies$ (2)' follows from Lemma \ref{lem:M-dconc}.

To prove `(2) $\implies$ (1)' we need several lemmas.

From now on let $X$ and $Y$ be two mm-spaces.

\begin{lem} \label{lem:pullback-me}
  Let $p : X \to Y$ be a Borel measurable map such that $p_*\mu_X = \mu_Y$.
  For any two Borel measurable functions $f,g : Y \to \R$, we have
  \[
  \dKF(p^*f,p^*g) = \dKF(f,g),
  \]
  where $p^*f := f\circ p$.
\end{lem}

\begin{proof}
  The lemma follows from
  \begin{align*}
    \mu_Y(|f-g| > \varepsilon) &= p_*\mu_X(|f-g| > \varepsilon)
    = \mu_X(p^{-1}(\{|f-g| > \varepsilon\}))\\
    &= \mu_X(|p^*f-p^*g| > \varepsilon).
  \end{align*}
\end{proof}

\begin{defn}[Enforce $\varepsilon$-concentration]
  \index{enforce epsilon-concentration@enforce $\varepsilon$-concentration}
  A Borel measurable map $p : X \to Y$ is said to
  \emph{enforce $\varepsilon$-concentration of $X$ to $Y$}
  if
  \[
  d_H(\Lip_1(X),p^*\Lip_1(Y)) \le \varepsilon.
  \]
\end{defn}

\begin{lem} \label{lem:enforce-conc}
  If a Borel measurable map $p : X \to Y$ enforces $\varepsilon$-concentration
  of $X$ to $Y$, then 
  \[
  \dconc(X,Y) \le 2\,d_P(p_*\mu_X,\mu_Y) + \varepsilon.
  \]
\end{lem}

\begin{proof}
  We take a parameter $\varphi : I \to X$.
  The map $\psi := p\circ\varphi : I \to Y$ is a parameter
  of $(Y,p_*\mu_X)$.
  By Lemma \ref{lem:pullback-me},
  \begin{align*}
    d_H(\varphi^*\Lip_1(X),\psi^*\Lip_1(Y))
    &= d_H(\varphi^*\Lip_1(X),\varphi^*p^*\Lip_1(Y))\\
    &= d_H(\Lip_1(X),p^*\Lip_1(Y)) \le \varepsilon,
  \end{align*}
  which implies $\dconc(X,(Y,p_*\mu_X)) \le \varepsilon$.
  Since
  \[
  \dconc(Y,(Y,p_*\mu_X)) \le \square(Y,(Y,p_*\mu_X))
  \le 2\,d_P(\mu_Y,p_*\mu_X),
  \]
  the lemma follows from a triangle inequality.
\end{proof}

\begin{lem}
  For a Borel measurable map $p : X \to Y$,
  the following {\rm(1)} and {\rm(2)} are equivalent to each other.
  \begin{enumerate}
  \item $p^*\Lip_1(Y) \subset \Lip_1(X)$.
  \item $p : X \to Y$ is $1$-Lipschitz continuous.
  \end{enumerate}
\end{lem}

\begin{proof}
  `(2) $\implies$ (1)' is obvious.

  We prove `(1) $\implies$ (2)'.
  Take any two points $x,y \in X$ and fix them.
  The function $f := d_Y(p(x),\cdot)$ belongs to $\Lip_1(Y)$,
  which together with (1) implies
  $p^*f \in p^*\Lip_1(Y) \subset \Lip_1(X)$.
  Therefore,
  \[
  d_Y(p(x),p(y)) = |p^*f(x)-p^*f(y)| \le d_X(x,y).
  \]
  This completes the proof.
\end{proof}

\begin{lem} \label{lem:1-Lip-up-to-Lip1}
  For a Borel measurable map $p : X \to Y$
  and for two real numbers $\varepsilon,\delta > 0$,
  we consider the two following conditions.
  \begin{enumerate}
  \item[(A${}_\varepsilon$)] $p^*\Lip_1(Y) \subset B_\varepsilon(\Lip_1(X))$.
  \item[(B${}_\delta$)] $p$ is $1$-Lipschitz up to $\delta$.
  \end{enumerate}
  Then we have the following {\rm(1)} and {\rm(2)}.
  \begin{enumerate}
  \item There exists a real number $\delta = \delta(Y,\varepsilon) > 0$
    for any $\varepsilon > 0$
    such that $\lim_{\varepsilon\to 0} \delta(Y,\varepsilon) = 0$
    and if {\rm(A${}_\varepsilon$)} holds and
    if $d_P(p_*\mu_X,\mu_Y) < \varepsilon$,
    then we have {\rm(B${}_\delta$)}.
  \item 
    If {\rm(B${}_\delta$)} holds, then we have {\rm(A${}_\delta$)}.
  \end{enumerate}
\end{lem}

\begin{proof}
  We prove (1).
  For a number $\varepsilon' > 0$, let
  $N(\varepsilon')$ be the infimum of $\#\cN$,
  where $\cN$ runs over all nets in $Y$ such that
  there is a Borel subset $Y_0 \subset Y$
  with the property that $\mu_Y(Y_0) \ge 1-\varepsilon'$
  and $\cN \subset Y_0$ is an $\varepsilon'$-net of $Y_0$.
  Since we have a compact subset of $Y$ whose $\mu_Y$-measure
  is arbitrarily close to $1$, the number $N(\varepsilon')$ is finite.
  For any $\varepsilon > 0$,
  there is a number $\varepsilon' = \varepsilon'(\varepsilon) > 0$
  such that $\lim_{\varepsilon\to 0} \varepsilon' = 0$ and
  $N(\varepsilon') \le 1/\sqrt{\varepsilon}$.
  We find a Borel subset $Y_0 \subset Y$ and an $\varepsilon'$-net
  $\cN \subset Y_0$ such that
  \[
  \mu_Y(Y_0) \ge 1-\varepsilon' \quad\text{and}\quad
  \#\cN = N(\varepsilon') \le \frac{1}{\sqrt{\varepsilon}}.
  \]
  It follows from $d_P(p_*\mu_X,\mu_Y) < \varepsilon$
  that
  \[
  \mu_X(p^{-1}(B_\varepsilon(Y_0))) = p_*\mu_X(B_\varepsilon(Y_0))
  \ge \mu_Y(Y_0)-\varepsilon \ge 1-\varepsilon-\varepsilon'.
  \]
  Let $y \in \cN$ be any point and set
  $f_y := d_Y(y,\cdot)$.
  (A${}_\varepsilon$) implies that
  $p^*f_y \in p^*\Lip_1(Y) \subset B_\varepsilon(\Lip_1(X))$,
  so that we have a function $g_y \in \Lip_1(X)$ with
  $\dKF(p^*f_y,g_y) \le \varepsilon$, namely
  $\mu_X(|p^*f_y - g_y| > \varepsilon) \le \varepsilon$.
  Setting
  \[
  X_0 := p^{-1}(B_\varepsilon(Y_0)) \setminus
  \bigcup_{y\in\cN} \{\;|p^*f_y - g_y| > \varepsilon\;\},
  \]
  we have
  \[
  \mu_X(X \setminus X_0) \le \#\cN\cdot\varepsilon + \varepsilon + \varepsilon'
  \le \sqrt{\varepsilon} + \varepsilon + \varepsilon'.
  \]
  Let $\pi : Y \to \cN$ be a Borel measurable
  nearest point projection.
  We take any two points $x,x' \in X_0$.
  Since $p(x) \in B_\varepsilon(Y_0)$ and  $\pi(p(x)) \in \cN$,
  we have
  \begin{align*}
    d_Y(p(x),p(x')) &\le d_Y(\pi(p(x)),p(x')) + \varepsilon + \varepsilon'
    = p^*f_{\pi(p(x))}(x') + \varepsilon + \varepsilon'\\
    &\le g_{\pi(p(x))}(x') + 2\varepsilon + \varepsilon',\\
    g_{\pi(p(x))}(x) &\le p^*f_{\pi(p(x))}(x) + \varepsilon
    = d_Y(\pi(p(x)),p(x)) + \varepsilon\\
    &\le 2\varepsilon + \varepsilon'
  \end{align*}
  and hence
  \begin{align*}
    d_Y(p(x),p(x')) &\le g_{\pi(p(x))}(x') - g_{\pi(p(x))}(x)
    + 4\varepsilon+2\varepsilon'\\
    &\le d_X(x,x') + 4\varepsilon+2\varepsilon'.
  \end{align*}
  Setting $\delta := \max\{4\varepsilon+2\varepsilon',
  \sqrt{\varepsilon} + \varepsilon + \varepsilon'\}$,
  the map $p$ is $1$-Lipschitz up to $\delta$.

  We prove (2).
  Take any $f \in \Lip_1(Y)$.
  (B${}_\delta$) implies that $p^*f$ is $1$-Lipschitz up to $\delta$.
  By Lemma \ref{lem:Lip-approx},
  there is a function $\tilde{f} \in \Lip_1(X)$ such that
  $\dKF(\tilde{f},p^*f) \le \delta$.
  Therefore we have
  \[
  p^*\Lip_1(Y) \subset B_\delta(\Lip_1(X)).
  \]
  This completes the proof.
\end{proof}

\begin{lem} \label{lem:M-p}
  Let $Y$ be an mm-space.
  For any $\varepsilon > 0$,
  there exists a natural number $N = N(Y,\varepsilon)$
  depending only on $Y$ and $\varepsilon$ such that,
  if $\cM(Y;N) \subset B_\varepsilon(\cM(X;N))$ for an mm-space $X$,
  then there exists a Borel measurable map $p : X \to Y$
  that is $1$-Lipschitz up to $5\varepsilon$ and satisfies
  \[
  d_P(p_*\mu_X,\mu_Y) \le 15\varepsilon.
  \]
\end{lem}


\begin{proof}
  Corollary \ref{cor:XN} implies that
  there is a natural number $N(Y,\varepsilon)$ and
  a measure $\underline{\mu}_N \in \cM(Y;N)$
  such that $\square(\underline{Y}_N,Y) < \varepsilon/3$,
  where $\underline{Y}_N := (\R^N,\|\cdot\|_\infty,\underline{\mu}_N)$.
  By Lemma \ref{lem:box-eps-mm-iso},
  there is an $\varepsilon$-mm-isomorphism $\Psi : \underline{Y}_N \to Y$.
  We find a Borel subset $\underline{Y}_{N,0} \subset \underline{Y}_N$
  such that $\underline{\mu}_N(\underline{Y}_{N,0}) \ge 1-\varepsilon$
  and
  \[
  |\;d_Y(\Psi(u),\Psi(v)) - \|u-v\|_\infty\;| \le \varepsilon
  \]
  for any $u,v \in \underline{Y}_{N,0}$.
  Since $\underline{\mu}_N \in \cM(Y;N) \subset B_\varepsilon(\cM(X;N))$,
  there is a $1$-Lipschitz map $\Phi' : X \to (\R^N,\|\cdot\|_\infty)$
  such that
  \[
  d_P(\underline{\mu}_N,\Phi'_*\mu_X) \le \varepsilon.
  \]
  We see that
  \[
  \Phi'_*\mu_X(B_\varepsilon(\underline{Y}_{N,0}))
  \ge \underline{\mu}_N(\underline{Y}_{N,0}) - \varepsilon
  \ge 1-2\varepsilon.
  \]
  By Lemma \ref{lem:eps-proj}, there is a Borel measurable
  $\varepsilon$-projection
  $\pi : \R^N \to \underline{Y}_{N,0}$ with
  $\pi|_{\underline{Y}_{N,0}} = \id_{\underline{Y}_{N,0}}$.
  Let $\Psi' := \Psi\circ\pi : \R^N \to Y$.
  For any $u,v \in B_\varepsilon(\underline{Y}_{N,0})$,
  \begin{align*}
    &|\;d_Y(\Psi'(u),\Psi'(v)) - \|u-v\|_\infty\;|\\
    &\le |\;d_Y(\Psi(\pi(u)),\Psi(\pi(v))) - \|\pi(u)-\pi(v)\|_\infty\;|
    + 4\varepsilon\\
    &\le 5\varepsilon.
  \end{align*}
  By Lemma \ref{lem:push-di-Lip},
  \[
  d_P(\Psi'_*\underline{\mu}_N,\Psi'_*\Phi'_*\mu_X) \le
  d_P(\underline{\mu}_N,\Phi'_*\mu_X) + 10\varepsilon \le 11\varepsilon.
  \]
  It follows from $\underline{\mu}_N(\underline{Y}_{N,0}) \ge 1-\varepsilon$
  that $d_P(\pi_*\underline{\mu}_N,\underline{\mu}_N)
  \le d_{KF}^{\underline{\mu}_N}(\pi,\id_{\R^N}) \le \varepsilon$.
  Besides we have $\pi_*\underline{\mu}_N(\underline{Y}_{N,0}) = 1
  \ge \underline{\mu}_N(\underline{Y}_{N,0}) \ge 1-\varepsilon$.
  Applying Lemma \ref{lem:push-di-Lip} yields
  \[
  d_P(\Psi'_*\underline{\mu}_N,\Psi_*\underline{\mu}_N)
  = d_P(\Psi_*\pi_*\underline{\mu}_N,\Psi_*\underline{\mu}_N)
  \le d_P(\pi_*\underline{\mu}_N,\underline{\mu}_N) + 2\varepsilon
  \le 3\varepsilon
  \]
  and hence, by a triangle inequality,
  \[
  d_P((\Psi'\circ\Phi')_*\mu_X,\Psi_*\underline{\mu}_N)
  \le 14\varepsilon.
  \]
  Since $\Psi : \underline{Y}_N \to Y$ is an $\varepsilon$-mm-isomorphism,
  we have $d_P(\Psi_*\underline{\mu}_N,\mu_Y) \le \varepsilon$,
  which together with the above inequality implies
  \begin{align*}
    d_P((\Psi'\circ\Phi')_*\mu_X,\mu_Y)
    \le 15\varepsilon.
  \end{align*}
  Setting $X_0 := {\Phi'}^{-1}(B_\varepsilon(\underline{Y}_{N,0})) \subset X$,
  we have, for any $x,y \in X_0$,
  \begin{align*}
    d_Y(\Psi'\circ\Phi'(x),\Psi'\circ\Phi'(y))
    \le \|\Phi'(x) - \Phi'(y)\|_\infty + 5\varepsilon
    \le d_X(x,y) + 5\varepsilon.
  \end{align*}
  Moreover we have
  \[
  \mu_X(X_0) = \Phi'_*\mu_X(B_\varepsilon(\underline{Y}_{N,0}))
  \ge 1-2\varepsilon.
  \]
  The desired map is $p := \Psi'\circ\Phi' : X \to Y$.
  This completes the proof.
\end{proof}

\begin{lem} \label{lem:me-diff-di}
  For any measurable maps $F = (f_1,\dots,f_N) : X \to \R^N$,
  $G = (g_1,\dots,g_N) : Y \to \R^N$, and for any $i,j = 1,2,\dots,N$,
  we have
  \begin{align}
    |\;\dKF(f_i,f_j) - \dKF(g_i,g_j)\;| &\le 2\,d_P(F_*\mu_X,G_*\mu_Y), \tag{1}\\
    |\;\dKF([f_i],[f_j]) - \dKF([g_i],[g_j])\;| &\le 2\,d_P(F_*\mu_X,G_*\mu_Y).
    \tag{2}
  \end{align}
\end{lem}

\begin{proof}
  We prove (1).
  We take any $i$ and $j$ in $\{1,2,\dots,N\}$ and fix them.
  (1) is trivial if $i = j$.   We then assume $i \neq j$.
  Suppose that $d_P(F_*\mu_X,G_*\mu_Y) < \varepsilon$ and
  $\dKF(f_i,f_j) < \rho$ for two numbers $\varepsilon$ and $\rho$.
  We set
  \[
  \rho' := \rho + 2\varepsilon, \quad
  X_0 := \{\;|x_i-x_j| \ge \rho'\;\}, \quad
  d_{X_0}(x) := \inf_{x' \in X_0} \|x-x'\|_\infty,
  \]
  where $\{\;|x_i-x_j| \ge \rho'\;\} :=
  \{\;(x_1,x_2,\dots,x_N) \in \R^N \mid |x_i-x_j| \ge \rho'\;\}$.
  Let us now prove
  \begin{align}
    d_{X_0}(x) = \max\left\{\;\frac{1}{2}(\rho'-|x_i-x_j|),
      0\;\right\}.
    \label{eq:me-diff-di-1}
  \end{align}
  Let $r$ be the right-hand side of \eqref{eq:me-diff-di-1}.
  Taking any point $x' = (x_1',\dots,x_N') \in X_0$, we have
  $|x_i'-x_j'| \ge \rho'$.
  If $r > 0$, then $\rho' = |x_i-x_j|+2r$ and so
  $|x_i'-x_j'| \ge |x_i-x_j|+2r$.  This implies
  \begin{align*}
    2\|x-x'\|_\infty \ge |x_i-x_i'| + |x_j-x_j'|
    \ge |x_i'-x_j'| - |x_i-x_j| \ge 2r
  \end{align*}
  and thus $d_{X_0}(x) \ge r$.
  
  We next prove $d_{X_0}(x) \le r$.  This is trivial if $x \in X_0$.
  We suppose $x \notin X_0$.  Then,
  \[
  r = \frac{1}{2} (\rho'-|x_i-x_j|) > 0.
  \]
  Put
  \begin{align*}
    x_i' :=
    \begin{cases}
      x_i + r & \text{if $x_i \ge x_j$},\\
      x_i - r & \text{if $x_i < x_j$},
    \end{cases}
    \quad\text{and}\quad
    x_j' :=
    \begin{cases}
      x_j - r & \text{if $x_i \ge x_j$},\\
      x_j + r & \text{if $x_i < x_j$}.
    \end{cases}
  \end{align*}
  For any $k \in \{1,2,\dots,N\}$ with $k \neq i,j$, we set
  $x_k' := x_k$.
  Since $|x_i'-x_j'| = |x_i-x_j| + 2r = \rho'$,
  the point $x' = (x_1',\dots,x_N')$ belongs to $X_0$,
  which together with $\|x-x'\|_\infty = r$ implies $d_{X_0}(x) \le r$.
  We thus obtain \eqref{eq:me-diff-di-1}.

  It follows from \eqref{eq:me-diff-di-1} that
  \begin{align*}
    B_\varepsilon(X_0) = \{\;|x_i-x_j| \ge \rho\;\},
  \end{align*}
  which together with $\dKF(f_i,f_j) < \rho$ and
  $d_P(F_*\mu_X,G_*\mu_Y) < \varepsilon$ leads to
  \begin{align*}
    \rho &\ge \mu_X(|f_i-f_j| \ge \rho) = F_*\mu_X(|x_i-x_j| \ge \rho)
    = F_*\mu_X(B_\varepsilon(X_0))\\
    &\ge G_*\mu_Y(X_0) - \varepsilon
    = \mu_Y(|g_i-g_j| \ge \rho') - \varepsilon
  \end{align*}
  and so $\dKF(g_i,g_j) \le \rho + 2\varepsilon$.
  Letting $\varepsilon \to d_P(F_*\mu_X,G_*\mu_Y)$ and
  $\rho \to \dKF(f_i,f_j)$ yields
  \[
  \dKF(g_i,g_j) \le \dKF(f_i,f_j) + 2\,d_P(F_*\mu_X,G_*\mu_Y).
  \]
  Since this also hold if we exchange $i$ and $j$,
  we obtain (1).

  We prove (2).
  It holds that for any $\mathbf{c} \in \R^N$,
  \[
  d_P((F+\mathbf{c})_*\mu_X,(G+\mathbf{c})_*\mu_Y) = d_P(F_*\mu_X,G_*\mu_Y).
  \]
  This together with (1) implies that
  \begin{align*}
    \dKF([f_i],[f_j]) &\le \dKF(f_i+c,f_j+c')\\
    &\le \dKF(g_i+c,g_j+c') + 2\,d_P(F_*\mu_X,G_*\mu_Y)
  \end{align*}
  for any real numbers $c$ and $c'$.
  Taking the infimum of the right-hand side over all
  $c$ and $c'$ yields
  \begin{align*}
    \dKF([f_i],[f_j]) \le \dKF([g_i],[g_j]) + 2\,d_P(F_*\mu_X,G_*\mu_Y).
  \end{align*}
  This completes the proof.
\end{proof}

\begin{lem} \label{lem:KN-Lip}
  For any natural number $N$ we have
  \[
  d_H(K_N(\Lo(X)),K_N(\Lo(Y)))
  \le 2 \, d_H(\cM(X;N),\cM(Y;N)).
  \]
\end{lem}

\begin{proof}
  Assume that $d_H(\cM(X;N),\cM(Y;N)) < \varepsilon$ for a number
  $\varepsilon$.
  It suffices to prove that $d_H(K_N(\Lo(X)),K_N(\Lo(Y))) \le 2\varepsilon$.
  For this, we are going to prove
  $K_N(\Lo(X)) \subset B_{2\varepsilon}(K_N(\Lo(Y)))$.
  Take any matrix $A \in K_N(\Lo(X))$.
  We find $N$ functions $f_i \in \Lip_1(X)$, $i=1,2,\dots,N$,
  such that $A = (\dKF([f_i],[f_j]))_{ij}$.
  Set $F := (f_1,\dots,f_N) : X \to \R^N$.
  By the assumption, there is a
  $1$-Lipschitz map $G = (g_1,\dots,g_N) : X \to (\R^N,\|\cdot\|_\infty)$
  such that $d_P(F_*\mu_X,G_*\mu_Y) < \varepsilon$.
  Lemma \ref{lem:me-diff-di}(2) implies that
  \[
  |\;\dKF([f_i],[f_j]) - \dKF([g_i],[g_j])\;| < 2\varepsilon
  \]
  for any $i,j = 1,\dots,N$.
  Letting $B := (\dKF([g_i],[g_j]))_{ij}$ we have
  $\|A - B\|_\infty < 2\varepsilon$ and
  $B \in K_N(\Lo(Y))$.
  We therefore obtain $K_N(\Lo(X)) \subset B_{2\varepsilon}(K_N(\Lo(Y)))$.
  This also holds if we exchange $X$ and $Y$.
  The proof of the lemma is completed.
\end{proof}

\begin{lem} \label{lem:pb-me-di}
  Let $p : X \to Y$ be a Borel measurable map.
  For any two functions $f,g \in \Lip_1(Y)$ we have
  \begin{align}
    |\;\dKF(p^*f,p^*g) - \dKF(f,g)\;| &\le 2\,d_P(p_*\mu_X,\mu_Y), \tag{1}\\
    |\;\dKF([p^*f],[p^*g]) - \dKF([f],[g])\;| &\le 2\,d_P(p_*\mu_X,\mu_Y).\tag{2}
  \end{align}
  In particular, if $p_*\mu_X = \mu_Y$, then the map
  \[
  p^* : \Lo(Y) \ni [f] \mapsto p^*[f] := [p^*f]
  \]
  is isometric with respect to $\dKF$.
\end{lem}

\begin{proof}
  We prove (1).
  Let us first prove that
  \begin{align} \label{eq:pb-me-di1}
    B_\varepsilon(\{|f-g| \ge \rho+2\varepsilon\})
    \subset \{\;|f-g| \ge \rho\;\}
  \end{align}
  for any $\rho,\varepsilon > 0$.
  In fact, if we take a point
  $y \in B_\varepsilon(\{|f-g| \ge \rho+2\varepsilon\})$,
  then there is a point $y' \in Y$ such that $d_Y(y,y') \le \varepsilon$
  and $|f(y')-g(y')| \ge \rho+2\varepsilon$,
  which together with the $1$-Lipschitz continuity of $f$ and $g$
  imply $|f(y)-g(y)| \ge \rho$.
  This proves \eqref{eq:pb-me-di1}.

  Assume that $\dKF(p^*f,p^*g) < \rho$ and $d_P(p_*\mu_X,\mu_Y) < \varepsilon$
  for two numbers $\rho$ and $\varepsilon$.
  It follows from \eqref{eq:pb-me-di1} that
  \begin{align*}
    \mu_Y(|f-g| \ge \rho+2\varepsilon)
    &\le p_*\mu_X(B_\varepsilon(\{|f-g|\ge \rho+2\varepsilon\})) + \varepsilon\\
    &\le p_*\mu_X(|f-g| \ge \rho) + \varepsilon\\
    &= \mu_X(|p^*f-p^*g| \ge \rho) + \varepsilon
    \le \rho + \varepsilon,
  \end{align*}
  which implies $\dKF(f,g) \le \rho + 2\varepsilon$.
  Therefore,
  \begin{equation}
    \label{eq:pb-me-di2}
    \dKF(f,g) \le \dKF(p^*f,p^*g) + 2\,d_P(p_*\mu_X,\mu_Y).
  \end{equation}
  Using \eqref{eq:pb-me-di1} we also have
  \begin{equation}
    \label{eq:pb-me-di3}
    \dKF(p^*f,p^*g) \le \dKF(f,g) + 2\,d_P(p_*\mu_X,\mu_Y)
  \end{equation}
  in the same way as above.
  Combining \eqref{eq:pb-me-di2} and \eqref{eq:pb-me-di3}
  implies (1).

  We prove (2).  By (1), we have, for any two real numbers $c$ and $c'$,
  \begin{align*}
    \dKF([p^*f],[p^*g]) &\le \dKF(p^*f+c,p^*g+c')\\
    &\le \dKF(f+c,g+c') + 2\,d_P(p_*\mu_X,\mu_Y).
  \end{align*}
  Taking the infimum of the right-hand side over all $c$ and $c'$ yields
  \[
  \dKF([p^*f],[p^*g]) \le \dKF([f],[g]) + 2\,d_P(p_*\mu_X,\mu_Y).
  \]
  In the same way we have
  \[
  \dKF([f],[g]) \le \dKF([p^*f],[p^*g]) + 2\,d_P(p_*\mu_X,\mu_Y).
  \]
  Combining these two inequalities implies (2).

  This completes the proof.
\end{proof}

Let $\cF$ be a metric space with metric $d_{\cF}$.
We give an isometric action of a group $G$
on $\cF$,
\[
G \times \cF \ni (g,x) \mapsto g\cdot x \in \cF,
\]
with the property that any $G$-orbit is closed in $\cF$.
Then, the quotient space $\cF/G$ is a metric space,
where the metric $d_{\cF/G}$ on $\cF/G$ is defined by
\[
d_{\cF/G}([x],[y]) := \inf_{x' \in [x],\ y' \in [y]} d_{\cF}(x',y')
\]
for any $[x],[y] \in \cF/G$.

We later apply the following lemma for $\Lip_1(X)$ and $\Lo(X)$.

\begin{lem} \label{lem:action-dH}
  For any $G$-invariant subset
  $\cL,\mathcal{L'} \subset \cF$, we have
  \begin{enumerate}
  \item for any real number $\varepsilon \ge 0$,
    we have the equivalence
    \[
    \cL' \subset B_\varepsilon(\cL) \Longleftrightarrow
    \cL'/G \subset B_\varepsilon(\cL/G),
    \]
  \item $d_H(\cL,\cL') = d_H(\cL/G,\cL'/G)$.
  \end{enumerate}
\end{lem}

\begin{proof}
  (2) follows from (1).

  We prove (1).
  Assume that $\cL' \subset B_\varepsilon(\cL)$.
  We take any point $x \in \cL'$.
  Since $x \in B_\varepsilon(\cL)$,
  there is a sequence of points $x_n \in \cL$, $n=1,2,\dots$,
  such that
  $\limsup_{n\to\infty} d_{\cF}(x_n,x) \le \varepsilon$.
  This implies
  \[
  \limsup_{n\to\infty} d_{\cF}([x_n],[x])
  \le \limsup_{n\to\infty} d_{\cF}(x_n,x) \le \varepsilon
  \]
  and so $[x] \in B_\varepsilon(\cL/G)$.
  Therefore we have $\cL'/G \subset B_\varepsilon(\cL/G)$.

  We conversely assume that
  $\cL'/G \subset B_\varepsilon(\cL/G)$.
  Take any point $x \in \cL'$.
  Since $[x] \in \cL'/G \subset B_\varepsilon(\cL/G)$,
  there is a sequence of points $x_n \in \cL$, $n=1,2,\dots$,
  such that $\limsup_{n\to\infty} d_{\cF/G}([x_n],[x]) \le \varepsilon$.
  Hence, there are elements $g_n,h_n \in G$, $n=1,2,\dots$, such that
  \[
  \limsup_{n\to\infty} d_{\cF}(g_n\cdot x,h_n\cdot x_n)
  \le \varepsilon.
  \]
  Since $\cL$ is $G$-invariant,
  the point $x_n' := g_n^{-1}h_n\cdot x_n$ belongs to $\cL$
  and satisfies $d_{\cF}(g_n\cdot x,h_n\cdot x_n) = 
  d_{\cF}(x,x_n')$, so that
  \[
  \limsup_{n\to\infty} d_{\cF}(x_n',x) \le \varepsilon.
  \]
  This implies that $x \in B_\varepsilon(\cL)$.
  We therefore have $\cL' \subset B_\varepsilon(\cL)$.
  This completes the proof.
\end{proof}

\begin{lem} \label{lem:L-Ln}
  Let $\cL$ and $\cL_n$, $n=1,2,\dots$, be
  compact metric spaces such that
  $\cL_n$ Gromov-Hausdorff converges to $\cL$ as $n\to\infty$.
  Let $q_n : \cL \to \cL_n$, $n=1,2,\dots$, be
  $\varepsilon_n$-isometric maps with $\varepsilon_n \to 0$.
  Then, there exists a sequence of real numbers $\varepsilon_n' \to 0+$
  such that $q_n(\cL)$ is $\varepsilon_n'$-dense in $\cL_n$,
  i.e., $B_{\varepsilon_n'}(q_n(\cL)) = \cL_n$.
\end{lem}

\begin{proof}
  By $\lim_{n\to\infty} d_{GH}(\cL_n,\cL) = 0$,
  there are numbers $\varepsilon_n' \to 0+$ and
  $\varepsilon_n'$-isometries $q_n' : \cL_n \to \cL$.
  Suppose that the lemma does not hold.
  Then, there is a number $\delta > 0$ and
  a sequence of points $x_n \in \cL_n$, $n=1,2,\dots$,
  such that $d_{\cL_n}(x_n,q_n(\cL)) \ge \delta$
  for every sufficiently large $n$.
  This implies that
  $d_{\cL}(q_n'(x_n),q_n'\circ q_n(\cL)) \ge \delta/2$
  for every sufficiently large $n$.
  There is a subsequence $\{n_i\}$ of $\{n\}$ such that
  $q_{n_i}'\circ q_{n_i}$ and $q_{n_i}'(x_{n_i})$ both converges as $i\to\infty$.
  The limit, say $f : \cL \to \cL$, of $q_{n_i}'\circ q_{n_i}$
  is isometric.  The image $f(\cL)$ does not contain 
  the limit of $q_{n_i}'(x_{n_i})$ and, in particular, $f$ is not surjective.
  This is a contradiction (see \cite{BBI}*{Thm 1.6.14}).
\end{proof}

\begin{lem} \label{lem:A}
  Assume {\rm(2)} of Theorem \ref{thm:A}.
  Then, there exist Borel measurable maps $p_n : X_n \to Y$, $n=1,2,\dots$,
  that enforce $\varepsilon_n$-concentration of $X_n$ to $Y$
  and $d_P((p_n)_*\mu_{X_n},\mu_Y) \le \varepsilon_n$ for all $n$
  and for some sequence $\varepsilon_n \to 0$.
\end{lem}

\begin{proof}
  By the assumption and Lemma \ref{lem:M-p},
  there are Borel measurable maps $p_n : X_n \to X$ that are
  $1$-Lipschitz up to some $\alpha_n \to 0$,
  $n=1,2,\dots$, such that $d_P((p_n)_*\mu_{X_n},\mu_X) \le \alpha_n$.
  By Lemma \ref{lem:1-Lip-up-to-Lip1}(2),
  we have
  $p_n^*\Lip_1(X) \subset B_{\alpha_n}(\Lip_1(X_n))$.
  By setting $\cL := \Lo(X)$ and $\cL_n := \Lo(X_n)$,
  Lemma \ref{lem:action-dH}(1) yields
  $p_n^*\cL\;\subset B_{\alpha_n}(\cL_n)$.
  By Lemma \ref{lem:pb-me-di}(2),
  the map $p_n^* : \cL\;\to B_{\alpha_n}(\cL_n)$
  is $2\alpha_n$-isometric with respect to $\dKF$.
  Let $\pi_n : B_{\alpha_n}(\cL_n) \to \cL_n$
  be a nearest point projection.
  Then it is a $2\alpha_n$-isometry.
  Therefore, $\pi_n \circ p_n^* : \cL\;\to \cL_n$
  is $4\alpha_n$-isometric for any $n$.
  The assumption together with Lemma \ref{lem:KN-Lip} implies that
  $\lim_{n\to\infty} d_H(K_N(\cL_n),K_N(\cL)) = 0$
  for any $N$.
  By Lemma \ref{lem:KN-GH-2},
  $\cL_n$ Gromov-Hausdorff converges to $\cL$ as $n\to\infty$.
  From Lemma \ref{lem:L-Ln}, we find a sequence of numbers
  $\alpha_n' \to 0+$ such that
  $\pi_n \circ p_n^*(\cL)$ is $\alpha_n'$-dense in $\cL_n$.
  Thereby, $p_n^*\cL$ is $\varepsilon_n$-dense in
  $B_{\alpha_n'}(\cL_n)$,
  where $\varepsilon_n := \alpha_n + \alpha_n'$.
  This implies $d_H(p_n^*\cL,\cL_n) \le \varepsilon_n$.
  By Lemma \ref{lem:action-dH}, we have
  \[
  d_H(p_n^*\Lip_1(X),\Lip_1(X_n)) \le \varepsilon_n,
  \]
  i.e., $p_n : X_n \to X$ enforces $\varepsilon_n$-concentration of $X_n$ to $X$.
  This completes the proof.
\end{proof}

\begin{proof}[Proof of Theorem \ref{thm:A}]
  Recall that `(1) $\implies$ (2)' follows from Lemma \ref{lem:M-dconc}.
  `(2) $\implies$ (1)' follows from
  Lemma \ref{lem:A} and \ref{lem:enforce-conc}.
  This completes the proof.
\end{proof}

\begin{cor} \label{cor:enforce}
  Let $X_n$ and $Y$ be mm-spaces, where $n = 1,2,\dots$.
  Then the following {\rm(1)} and {\rm(2)} are equivalent to each other.
  \begin{enumerate}
  \item $X_n$ concentrates to $Y$ as $n\to\infty$.
  \item There exists a sequence of Borel measurable maps
    $p_n : X_n \to Y$, $n=1,2,\dots$, that enforce $\varepsilon_n$-concentration
    of $X_n$ to $Y$ and $d_P((p_n)_*\mu_{X_n},\mu_Y) \le \varepsilon_n$
    for all $n$ and for some sequence $\varepsilon_n \to 0$.
  \end{enumerate}
\end{cor}

\begin{proof}
  `(1) $\implies$ (2)' follows from Theorem \ref{thm:A} and Lemma \ref{lem:A}.

  `(2) $\implies$ (1)' follows from Lemma \ref{lem:enforce-conc}.
\end{proof}

\section{$(N,R)$-Measurement}

A main purpose of this section is to prove
that the convergence of $(N,R)$-measurement for any $R > 0$ is equivalent to
that of $N$-measurement, which is necessary in the later sections.

\begin{defn}[$(N,R)$-Measurement]
  \index{measurement} \index{NR-measurement@$(N,R)$-measurement}
  For an mm-space $X$, a natural number $N$, and a real number $R > 0$,
  we define
  \[
  \cM(X;N,R) := \{\;\mu \in \cM(X;N) \mid
  \supp\mu \subset B^N_R\;\},
  \]
  \index{MXNR@$\cM(X;N,R)$}
  where $B^N_R := \{\,x \in \R^N \mid \|x\|_\infty \le R\,\}$.
  \index{BNR@$B^N_R$}
  We call $\cM(X;N,R)$ the \emph{$(N,R)$-measurement of $X$}.
\end{defn}

$\cM(X;N,R)$ is a compact subset of $\cM(N)$ (see Lemma \ref{lem:conv-meas}).

\begin{defn}[$\pi_{\xi,R}$]
  \index{pi@$\pi_{\xi,R}$}
  For a point $\xi \in \R^N$ and a real number $R \ge 0$,
  we define a map $\pi_{\xi,R} : \R^N \to \R^N$ in the following.
  For a given point $x = (x_1,\dots,x_N) \in \R^N$
  we determine a point $y = (y_1,\dots,y_N) \in \R^N$ as,
  for $i=1,\dots,N$,
  \[
  y_i :=
  \begin{cases}
    \xi_i+R &\text{if $x_i > \xi_i+R$},\\
    \xi_i-R &\text{if $x_i < \xi_i-R$},\\
    x_i &\text{if $\xi_i-R \le x_i \le \xi_i+R$}.
  \end{cases}
  \]
  We then define $\pi_{\xi,R}(x) := y$.
\end{defn}

We see that $\pi_{\xi,R}(\R^N) =
B^N_R(\xi) := \{\,x \in \R^N \mid \|x-\xi\|_\infty \le R\,\}$
\index{BNRxi@$B^N_R(\xi)$}
and $\pi_{\xi,R}$ is a unique nearest point projection to $B^N_R(\xi)$
with respect to $\|\cdot\|_\infty$.
Letting $\pi_R := \pi_{o,R}$ \index{piR@$\pi_R$} we have
$\pi_{\xi,R}(x) = \pi_R(x-\xi)+\xi$ for any $x \in \R^N$.

\begin{lem}
  The map $\pi_{\xi,R} : \R^N \to \R^N$ is $1$-Lipschitz
  with respect to $\|\cdot\|_\infty$.
\end{lem}

\begin{proof}
  Take any two points $x,x' \in \R^N$ and set
  $y := \pi_{\xi,R}(x)$, $y' := \pi_{\xi,R}(x')$.
  It follows from the definition of $\pi_{\xi,R}$ that
  $|y_i-y_i'| \le |x_i-x_i'|$ for $i = 1,\dots,N$,
  which implies $\|y-y'\|_\infty \le \|x-x'\|_\infty$.
  This completes the proof.
\end{proof}

\begin{lem} \label{lem:npp}
  Let $N$ be a natural number, $R > 0$ a real number,
  and $\mu$ a Borel probability measure on $\R^N$.
  Then, for any two points $\xi,\eta \in \R^N$ we have
  \begin{align}
    \sup_{x \in \R^N} \|\pi_{\xi,R}(x) - \pi_{\eta,R}(x)\|_\infty
    &\le \|\xi-\eta\|_\infty,\tag{1}\\
    d_P((\pi_{\xi,R})_*\mu,(\pi_{\eta,R})_*\mu) &\le \|\xi-\eta\|_\infty.\tag{2}
  \end{align}
\end{lem}

\begin{proof}
  We prove (1).
  Let $x \in \R^N$ be any point and set
  $y := \pi_{\xi,R}(x)$, $y' := \pi_{\eta,R}(x)$.
  By the definition of $\pi_{\xi,R}$,
  we have $|y_i-y_i'| \le |\xi_i-\eta_i|$ for $i=1,\dots,N$,
  which implies $\|y-y'\|_\infty \le \|\xi-\eta\|_\infty$.
  (1) has been obtained.

  We prove (2).  It follows from (1) that
  \[
  d_P((\pi_{\xi,R})_*\mu,(\pi_{\eta,R})_*\mu)
  \le d_{KF}^\mu(\pi_{\xi,R},\pi_{\eta,R}) \le \|\xi-\eta\|_\infty.
  \]
  This completes the proof.
\end{proof}

\begin{defn}[Perfect set of measures]
  \index{perfect set}
  Let $N$ be a natural number.
  A subset $\cA \subset \cM(N)$ is said to be \emph{perfect}
  if for any $\mu \in \cA$ and $\nu \in \cM(N)$ with
  $(\R^N,\|\cdot\|_\infty,\nu) \prec (\R^N,\|\cdot\|_\infty,\mu)$,
  we have $\nu \in \cA$.
  A subset $\cA \subset \cM(N,R)$ is said to be \emph{perfect on $B^N_R$}
  \index{perfect on BNR@perfect on $B^N_R$}
  if for any $\mu \in \cA$ and $\nu \in \cM(N,R)$ with
  $(B^N_R,\|\cdot\|_\infty,\nu) \prec (B^N_R,\|\cdot\|_\infty,\mu)$,
  we have $\nu \in \cA$.
\end{defn}

Note that $\cM(X;N)$ is perfect and $\cM(X;N,R)$
is perfect on $B^N_R$.
We see that $(\pi_R)_*\cM(X;N) = \cM(X;N,R)$.

\begin{lem} \label{lem:MR-half}
  For any two perfect subsets $\cA,\cB \subset \cM(N)$
  and for any real number $R > 0$, we have
  \[
  d_H((\pi_R)_*\cA,(\pi_R)_*\cB)
  \le 2\,d_H(\cA,\cB),
  \]
  where the Hausdorff distance $d_H$ is defined with respect to
  the Prohorov metric $d_P$ on $\cM(N)$
\end{lem}

\begin{proof}
  Let $\varepsilon := d_H(\cA,\cB)$.
  Note that the perfectness of $\cA$ implies $(\pi_R)_*\cA \subset \cA$.
  For any measure $\mu \in (\pi_R)_*\cA$ there is a measure
  $\nu \in \cB$ such that $d_P(\mu,\nu) \le \varepsilon$.
  We have
  \begin{equation}
    \label{eq:MR-half}
    \nu(B_\varepsilon(B^N_R)) \ge \mu(B^N_R) - \varepsilon = 1-\varepsilon.
  \end{equation}
  Since $\pi_R|_{B^N_R} = \id_{B^N_R}$ and by \eqref{eq:MR-half}, we have
  $d_P((\pi_R)_*\nu,\nu) \le d_{KF}^\nu(\pi_R,\id_{\R^N}) \le \varepsilon$
  and hence
  \[
  d_P(\mu,(\pi_R)_*\nu) \le d_P(\mu,\nu) + d_P(\nu,(\pi_R)_*\nu)
  \le 2\varepsilon,
  \]
  so that $(\pi_R)_*\cA \subset B_{2\varepsilon}((\pi_R)_*\cB)$.
  Exchanging $\cA$ and $\cB$ yields
  $(\pi_R)_*\cB \subset B_{2\varepsilon}((\pi_R)_*\cA)$.
  We thus obtain
  \[
  d_H((\pi_R)_*\cA,(\pi_R)_*\cB) \le 2\varepsilon.
  \]
  This completes the proof.
\end{proof}



\begin{lem} \label{lem:diam-box}
  Let $X$ and $Y$ be two mm-spaces and $\varepsilon > 0$ a real number.
  If a real number $\delta$ satisfies $\square(X,Y) < \delta$, then
  \[
  \diam(Y;1-\varepsilon-\delta)
  \le \diam(X;1-\varepsilon) + \delta.
  \]
\end{lem}

\begin{proof}
  By $\square(X,Y) < \delta$, there are two parameters
  $\varphi : I \to X$, $\psi : I \to Y$, and a Borel subset
  $I_0 \subset I$ such that
  $\cL^1(I_0) \ge 1-\delta$ and
  \[
  |\;\varphi^*d_X(s,t)-\psi^*d_Y(s,t)\;| \le \delta
  \]
  for any $s,t \in I_0$.
  Therefore we have
  \begin{align*}
    &\diam(Y;1-\varepsilon-\delta)\\
    &= \inf\{\;\diam A \mid A \subset Y,
    \ \mu_Y(A) \ge 1-\varepsilon-\delta\;\}\\
    &= \inf\{\;\diam(J,\psi^*d_Y) \mid J \subset I,
    \ \cL^1(J) \ge 1-\varepsilon-\delta\;\}\\
    &\le \inf\{\;\diam(J,\psi^*d_Y) \mid J \subset I_0,
    \ \cL^1(J) \ge 1-\varepsilon-\delta\;\}\\
    &\le \inf\{\;\diam(J,\varphi^*d_X) \mid J \subset I_0,
    \ \cL^1(J) \ge 1-\varepsilon-\delta\;\} + \delta\\
    &\le \inf\{\;\diam(J_1\cup J_2,\varphi^*d_X) \mid J_1 \subset I_0,
    \ J_2 \subset I \setminus I_0,\ \cL^1(J_1) \ge 1-\varepsilon-\delta\;\}
    + \delta\\
    &\le \inf\{\;\diam(J_1\cup J_2,\varphi^*d_X) \mid J_1 \subset I_0,
    \ J_2 \subset I \setminus I_0,\ \cL^1(J_1\cup J_2) \ge 1-\varepsilon\;\}
    + \delta\\
    &= \diam(X;1-\varepsilon) + \delta.
  \end{align*}
  This completes the proof.
\end{proof}



\begin{lem} \label{lem:MR}
  Let $\cA$ and $\cA_n$, $n=1,2,\dots$, be
  perfect subsets of $\cM(N)$ such that
  \[
  \sup_{\mu \in \cA} \diam(\mu;1-\kappa) < +\infty
  \]
  for any real number $\kappa$ with $0 < \kappa < 1$,
  where $\diam$ is defined for the $l_\infty$ norm on $\R^N$.
  Then, the following {\rm(1)} and {\rm(2)} are equivalent
  to each other.
  \begin{enumerate}
  \item $\cA_n$ Hausdorff converges to $\cA$
    as $n \to \infty$.
  \item $(\pi_R)_*\cA_n$ Hausdorff converges to $(\pi_R)_*\cA$
    as $n \to \infty$ for any real number $R > 0$.
  \end{enumerate}
\end{lem}

\begin{proof}
  `(1) $\implies$ (2)' follows from Lemma \ref{lem:MR-half}.

  We prove `(2) $\implies$ (1)'.
  Take any $\varepsilon > 0$ and fix it.
  Let us first prove that
  \begin{equation}
    \label{eq:MR1}
    \cA \subset B_\varepsilon(\cA_n)
  \end{equation}
  for every sufficiently large $n$.
  For any $\mu \in \cA$ there is a number $R > 0$
  such that $d_P((\pi_R)_*\mu,\mu) < \varepsilon/2$.
  (2) proves that $(\pi_R)_*\mu \in B_{\varepsilon/2}((\pi_R)_*\cA_n)
  \subset B_{\varepsilon/2}(\cA_n)$
  for every sufficiently large $n$.
  Therefore, $\mu$ belongs to $B_\varepsilon(\cA_n)$
  for every sufficiently large $n$, which implies \eqref{eq:MR1}.
  
  Let $0 < \varepsilon < 1$.
  It suffices to prove that
  $\cA_n \subset B_{3\varepsilon}(\cA)$
  if $n$ is large enough.
  We take any sequence $\mu_n \in \cA_n$, $n=1,2,\dots$.
  Set
  \[
  R := \max\{\sup_{\mu \in \cA} \diam(\mu;1-\varepsilon),1\}
  \quad\text{and}\quad
  R' := 100R.
  \]
  It follows from (2) that
  \begin{align} \label{eq:MR2}
    d_H((\pi_{R'})_*\cA_n,(\pi_{R'})_*\cA) < \varepsilon/2
  \end{align}
  for every $n$ large enough.
  From now on we assume $n$ to be sufficiently large.
  Let us prove the following claim.
  \begin{clm}
    We have
    \begin{align} \label{eq:MR3}
      \diam(\mu_n;1-2\varepsilon) \le R + \varepsilon.
    \end{align}
  \end{clm}
  \begin{proof}
    Let $x \in \R^N$ be any point and
    $\iota_x : \R^N \to \R^N$
    the translation defined by $\iota_x(y) = y-x$, $y \in \R^N$.
    The perfectness of $\cA_n$ proves
    $(\iota_x)_*(\pi_{x,R'})_*\mu_n \in (\pi_{R'})_*\cA_n$.
    By \eqref{eq:MR2},
    there is a measure $\nu_n \in (\pi_{R'})_*\cA$ such that
    $d_P((\iota_x)_*(\pi_{x,R'})_*\mu_n,\nu_n) < \varepsilon/2$.
    Since $\iota_x$ is an isometry,
    \begin{align*}
      &\square((\R^N,\|\cdot\|_\infty,(\pi_{x,R'})_*\mu_n),
      (\R^N,\|\cdot\|_\infty,\nu_n))\\
      &\le 2\,d_P((\iota_x)_*(\pi_{x,R'})_*\mu_n,\nu_n) < \varepsilon.
    \end{align*}
    Applying Lemma \ref{lem:diam-box} yields
    \[
    \diam((\pi_{x,R'})_*\mu_n;1-2\varepsilon)
    < \diam(\nu_n;1-\varepsilon) + \varepsilon \le R + \varepsilon.
    \]
    Therefore, there is a Borel subset $A_x \subset B^N_{R'}(x)$
    such that
    $(\pi_{x,R'})_*\mu_n(A_x) \ge 1-2\varepsilon$ and
    $\diam(A_x,\|\cdot\|_\infty) \le R+\varepsilon$.
    If there is a point $x_0 \in \R^N$ such that $A_{x_0}$ belongs to
    the interior $U^N_{R'}(x_0)$ of $B^N_{R'}(x_0)$, then we have \eqref{eq:MR3}.
    We are going to prove the existence of such a point $x_0$.
    Suppose that we have no such point $x_0$.
    Then, since $\diam(A_x,\|\cdot\|_\infty) \le R+\varepsilon < 2R$,
    the set $A_x$ does not intersect $U^N_{98R}(x)$ and
    the Euclidean distance between $x$ and $A_x$ is not less than $98R$.
    Let $a_x \in \R^N$ be the center of mass of $A_x$
    with respect to the $N$-dimensional Lebesgue measure.
    We see that $\|a_x-x\|_2 \ge 96R$, where $\|\cdot\|_2$ denotes
    the Euclidean or $l_2$ norm on $\R^N$.
    Let
    \[
    V_x := \frac{1}{\|a_x-x\|_2}(a_x-x).
    \]
    Then, $V$ is a (not necessarily continuous) unit vector field on $\R^N$.
    The continuity of the map $\R^N \ni x \mapsto (\pi_{x,R'})_*\mu_n$
    (see Lemma \ref{lem:npp}) proves that
    if two points $x, y \in \R^N$ are close enough to each other,
    then $d_{\R^N}(A_x,A_y)$ is small enough and so
    the angle between $V_x$ and $V_y$ is less than $\pi/4$.
    There is a compact subset $K \subset \R^N$ such that
    $\mu_n(K) > 2\varepsilon$.
    The set $A_x$ intersects $\pi_{x,R'}(K)$
    and, if $\|x\|_2$ is sufficiently large, then
    $\pi_{x,R'}(K)$ is contained in
    an $(N-1)$-dimensional face of $\partial B^N_{x,R'}$
    containing $\pi_{x,R'}(o)$.
    Therefore, $a_x$ belongs to the $4R$-neighborhood of the face
    if $\|x\|_2$ is large enough.
    This proves that
    \[
    \lim_{\|x\|_2 \to +\infty} \angle\left(V_x,-\frac{x}{\|x\|_2}\right)
    \le \frac{\pi}{3},
    \]
    where $\angle(\cdot,\cdot)$ denotes the angle.
    From a standard mollifier argument,
    we find a $C^\infty$ unit vector filed $\tilde V$ on $\R^N$
    and a large number $C > 0$
    in such a way that,
    if $\|x\|_2 \ge C$, then the angle between $\tilde{V}_x$ and
      $-\frac{x}{\|x\|_2}$ is less than $\pi/2$.
    Applying the Poincar\'e-Hopf theorem to the vector field $\tilde{V}$
    on the Euclidean ball $\{\,x \in \R^n \mid \|x\|_2 \le C\,\}$,
    we have a contradiction.
    The claim follows.
  \end{proof}
  
  By \eqref{eq:MR3}, there is a point $x_n \in \R^N$ such that
  $\mu_n(B^N_{2R}(x_n)) \ge 1-2\varepsilon$.
  For the translation $\iota_{x_n} : \R^N \to \R^N$, we have
  \[
  d_P((f_n)_*\mu_n,(\pi_{R'})_*(f_n)_*\mu_n) \le d_{KF}^{(f_n)_*\mu_n}(\pi_{R'},\id_{\R^N})
  \le 2\varepsilon.
  \]
  By \eqref{eq:MR2}, there is a measure $\nu_n \in (\pi_{R'})_*\cA$ such that
  \[
  d_P((\pi_{R'})_*(f_n)_*\mu_n,\nu_n) < \varepsilon.
  \]
  A triangle inequality proves
  \[
  d_P(\mu_n,(f_n^{-1})_*\nu_n) = d_P((f_n)_*\mu_n,\nu_n) < 3\varepsilon.
  \]
  Since $(f_n^{-1})_*\nu_n \in \cA$, we have
  $\mu_n \in B_{3\varepsilon}(\cA)$.
  Therefore we obtain $\cA_n \subset B_{3\varepsilon}(\cA)$.
  This completes the proof of the lemma.
\end{proof}

\chapter{The space of pyramids}
\label{chap:pyramid}

\section{Tail and pyramid}

\begin{defn}[Tail]
  \index{tail} \index{TXn@$\cT\{X_n\}$}
  Let $\{X_n\}_{n=1}^\infty$ be a sequence of mm-spaces.
  The \emph{tail $\cT\{X_n\}$ of $\{X_n\}$} is defined to be
  the set of mm-spaces that are the $\square$-limits of
  mm-spaces $Y_n$, $n=1,2,\dots$, such that
  each $Y_n$ is dominated by $X_n$.
\end{defn}

If $\{X_{n_i}\}_{i=1}^\infty$ is a subsequence of $\{X_n\}_{n=1}^\infty$,
then $\cT\{X_{n_i}\} \supset \cT\{X_n\}$.

\begin{prop} \label{prop:conc-tail}
  Let $\{X_n\}_{n=1}^\infty$ be a sequence of mm-spaces.
  \begin{enumerate}
  \item If $X_n$ concentrates to $X$ as $n\to\infty$,
    then $X$ is a maximal element of $\cT\{X_n\}$,
    i.e., $X \in \cT\{X_n\}$ and
    $Y \prec X$ for any $Y \in \cT\{X_n\}$.
  \item If $X$ is a maximal element of $\cT\{X_n\}$
    and if $\cT\{X_{n_i}\} = \cT\{X_n\}$
    for any subsequence $\{X_{n_i}\}$ of $\{X_n\}$,
    then $X_n$ concentrates to $X$.
  \end{enumerate}
\end{prop}

\begin{proof}
  We prove (1).  Assume that $X_n$ concentrates to $X$ as $n\to\infty$.
  By Corollary \ref{cor:XN},
  there are measures $\underline{\mu}_N \in \cM(X;N)$, $N=1,2,\dots$,
  such that
  $\underline{X}_N = (\R^N,\|\cdot\|_\infty,\underline{\mu}_N)$
  $\square$-converges to $X$ as $N\to\infty$.
  Since Theorem \ref{thm:A} implies that $\cM(X_n;N)$ Hausdorff
  converges to $\cM(X;N)$, we find $1$-Lipschitz maps
  $F_{N,n} : X_n \to (\R^N,\|\cdot\|_\infty)$, $n=1,2,\dots$, for each $N$
  such that $\lim_{n\to\infty} d_P((F_{N,n})_*\mu_{X_n},\underline{\mu}_N) = 0$.
  The mm-space $X_{N,n} := (\R^N,\|\cdot\|_\infty,(F_{N,n})_*\mu_{X_n})$
  is dominated by $X_n$ and $\square$-converges to
  $\underline{X}_N = (\R^N,\|\cdot\|_\infty,\underline{\mu}_N)$ as $n\to\infty$.
  Therefore, there is a monotone nondecreasing function $m : \N \to \N$
  such that $m(N) \to \infty$ as $N \to \infty$ and
  $\square(X_{N,n},\underline{X}_N) < 1/N$ for any
  $n,N \in \N$ with $n \ge m(N)$.
  Setting $M(n) := \max\{\;j \in \N \mid m(j) \le n\;\}$ for $n \in \N$,
  we see that $M(n)$ is monotone nondecreasing in $n$ and
  $M(n) \to \infty$ as $n \to \infty$.
  By $m(M(n)) \le n$,
  \[
  \square(X_{M(n),n},\underline{X}_{M(n)}) < 1/M(n) \to 0 \ \text{as}\ n\to\infty
  \]
  and hence
  \[
  \square(X_{M(n),n},X) \le \square(X_{M(n),n},\underline{X}_{M(n)})
  + \square(\underline{X}_{M(n)},X) \to 0 \ \text{as}\  n \to \infty,
  \]
  which proves $X \in \cT\{X_n\}$.

  Let us prove the maximality of $X$ in $\cT\{X_n\}$.
  Take any mm-space $Y \in \cT\{X_n\}$.
  There is a sequence of mm-spaces $Y_n$, $n=1,2,\dots$,
  with $Y_n \prec X_n$ that $\square$-converges to $Y$ as $n\to\infty$.
  By Lemma \ref{lem:dom-M}, $\cM(Y_n;N) \subset \cM(X_n;N)$.
  Theorem \ref{thm:A} implies that
  $\cM(X_n;N)$ and $\cM(Y_n;N)$ Hausdorff converges
  to $\cM(X;N)$ and $\cM(Y;N)$, respectively.
  Therefore we have $\cM(Y;N) \subset \cM(X;N)$ for any $N$,
  which together with Lemma \ref{lem:dom-M} implies that
  $Y$ is dominated by $X$.
  We thus obtain the maximality of $X$ in $\cT\{X_n\}$.
  (1) has been proved.

  We prove (2).
  Since $X \in \cT\{X_n\}$, we find a sequence of mm-spaces $Y_n$,
  $n=1,2,\dots$, with $Y_n \prec X_n$ that $\square$-converges to $X$.
  By Proposition \ref{prop:dconc-box},
  $Y_n$ concentrates to $X$ as $n\to\infty$,
  so that, by Theorem \ref{thm:A} and Lemma \ref{lem:MR},
  $\cM(Y_n;N,R)$ Hausdorff converges to $\cM(X;N,R)$ as $n\to\infty$
  for any $N$ and $R > 0$.
  We see $\cM(Y_n;N,R) \subset \cM(X_n;N,R)$ by Lemma \ref{lem:dom-M}.
  Theorem \ref{thm:A} and Lemma \ref{lem:MR} together tell us that,
  to prove the concentration of $X_n$ to $X$,
  it suffices to show that $\cM(X_n;N,R)$ Hausdorff converges to $\cM(X;N,R)$.
  Suppose that $\cM(X_n;N,R)$ does not Hausdorff converge to $\cM(X;N,R)$
  for some $N$ and $R > 0$.
  Then, there are a number $\delta > 0$,
  a sequence of natural numbers $n_i \to \infty$,
  and measures $\mu_{n_i} \in \cM(X_{n_i};N,R)$, $i=1,2,\dots$,
  such that $d_P(\mu_{n_i},\cM(X;N,R)) \ge \delta$ for any $i$.
  Replacing $\{\mu_{n_i}\}$ with a subsequence
  we assume that $\mu_{n_i}$ converges weakly to some measure $\mu_\infty$
  on $B^N_R$.  The mm-space $(\R^N,\|\cdot\|_\infty,\mu_{n_i})$
  is dominated by $X_{n_i}$ and $\square$-converges to
  $(\R^N,\|\cdot\|_\infty,\mu_\infty)$, so that
  $(\R^N,\|\cdot\|_\infty,\mu_\infty)$ belongs to the tail
  $\cT\{X_{n_i}\} = \cT\{X_n\}$.
  The maximality of $X$ in $\cT\{X_n\}$ yields
  that $X$ dominates $(\R^N,\|\cdot\|_\infty,\mu_\infty)$
  and so $\mu_\infty$ is an element of $\cM(X;N,R)$,
  which is a contradiction.
  This completes the proof.
\end{proof}

\begin{defn}[Pyramid] \label{defn:pyramid}
  \index{pyramid}
  Recall that $\cX$ is the mm-isomorphism classes of mm-spaces.
  A subset $\cP \subset \cX$ is called a \emph{pyramid}
  if it satisfies the following (1), (2), and (3).
  \begin{enumerate}
  \item If $X \in \cP$ and if $Y \prec X$, then $Y \in \cP$.
  \item For any two mm-spaces $X, X' \in \cP$,
    there exists an mm-space $Y \in \cP$ such that
    $X \prec Y$ and $X' \prec Y$.
  \item $\cP$ is nonempty and $\square$-closed.
  \end{enumerate}
  (2) is called the Moore-Smith property and a pyramid is a directed subfamily of $\cX$.
  We denote the set of pyramids by $\Pi$.
  \index{Pi@$\Pi$}

  For an mm-space $X$ we define
  \[
  \cP_X := \{\;X' \in \cX \mid X' \prec X\;\}.
  \]
  \index{PX@$\cP_X$}
  It follows from Theorem \ref{thm:dom} that $\cP_X$ is a pyramid.
  We call $\cP_X$ the \emph{pyramid associated with $X$}.
  \index{pyramid associated with X@pyramid associated with $X$}
\end{defn}

We see that $X \prec Y$ if and only if $\cP_X \subset \cP_Y$.
Note that for any two mm-spaces $X$ and $X'$ we always have
an mm-space dominating both $X$ and $X'$, in fact
$X \times X'$ with product measure and $l_p$ product metric,
$1 \le p \le +\infty$,
is an example of such an mm-space.
It is trivial that $\cX$ itself is a pyramid.

In Gromov's book \cite{Gromov}, the definition of a pyramid
is only by (1) and (2) of Definition \ref{defn:pyramid}.
We here put (3) as an additional condition for later convenience.

\section{Weak Hausdorff convergence}

All the discussions in this section also work in the case where
$\cX$ is a general separable metric space.

\begin{defn}[Weak (Hausdorff) convergence] \label{defn:w-conv}
  \index{weak convergence} \index{weak Hausdorff convergence}
  \index{converges weakly}
  Denote by $\cF(\cX)$ the set of $\square$-closed subsets of $\cX$.
  Let $\cY_n, \cY \in \cF(\cX)$
  (where $\cY_n$ are $\cY$ may be empty).
  We say that \emph{$\cY_n$ converges weakly to $\cY$}
  as $n\to\infty$
  if the following (1) and (2) are both satisfied.
  \begin{enumerate}
  \item For any mm-space $X \in \cY$, we have
    \[
    \lim_{n\to\infty} \square(X,\cY_n) = 0.
    \]
  \item For any mm-space $X \in \cX \setminus \cY$, we have
    \[
    \liminf_{n\to\infty} \square(X,\cY_n) > 0.
    \]
  \end{enumerate}
  We here agree that $\square(X,\emptyset) = +\infty$.
\end{defn}

\begin{lem} \label{lem:w-subconv}
  Any sequence $\{\cY_n\}_{n=1}^\infty \subset \cF(\cX)$
  has a weakly convergent subsequence.
\end{lem}

\begin{proof}
  Let $\{\cY_n\}_{n=1}^\infty \subset \cF(\cX)$ be a given sequence.
  We first prove the following
  \begin{clm}
    There exists a subsequence $\{\cY_{n_i}\}$ of $\{\cY_n\}$
    such that, for any mm-space $X \in \cX$, the limit
    \[
    \lim_{i\to\infty} \square(X,\cY_{n_i}) \in [\,0,+\infty\,]
    \]
    exists.
  \end{clm}

  \begin{proof}
    According to Proposition \ref{prop:separable},
    we find a dense countable subset $\{X_k\}_{k=1}^\infty \subset \cX$.
    There is a subsequence $\{\cY_{n^1_i}\}$ of
    $\{\cY_n\}$ such that
    the limit $\lim_{i\to\infty} \square(X_1,\cY_{n^1_i}) \in [\,0,+\infty\,]$
    exists.
    There is also a subsequence $\{\cY_{n^2_i}\}$ of
    $\{\cY_{n^1_i}\}$ such that
    the limit $\lim_{i\to\infty} \square(X_2,\cY_{n^2_i}) \in [\,0,+\infty\,]$
    exists.
    Repeating this procedure we have the limits
    $\lim_{i\to\infty} \square(X_j,\cY_{n^j_i}) \in [\,0,+\infty\,]$
    for all $j=1,2,3,\dots$.
    By a diagonal argument, we are able to choose a subsequence
    $\{\cY_{n_i}\}$
    of $\{\cY_n\}$ in such a way that the limit
    $\lim_{i\to\infty} \square(X_j,\cY_{n_i}) \in [\,0,+\infty\,]$
    exists for all $j=1,2,\dots$.

    If 
    $\lim_{i\to\infty} \square(X_{j_0},\cY_{n_i}) = +\infty$ for some $j_0$,
    then a triangle inequality proves that
    $\lim_{i\to\infty} \square(X,\cY_{n_i}) = +\infty$
    for any mm-space $X \in \cX$.

    Assume that, for any $j$, the limit
    \[
    r_j := \lim_{i\to\infty} \square(X_j,\cY_{n_i})
    \]
    is finite.
    Let $X$ be any mm-space.
    Since $\{X_k\}_{k=1}^\infty$ is $\square$-dense in $\cX$,
    there is a natural number $j_k$ for any $k$
    such that $\square(X,X_{j_k}) < 1/k$.
    By a triangle inequality,
    \[
    |\;\square(X,\cY_{n_i}) - \square(X_{j_k},\cY_{n_i})\;|
    \le \square(X,X_{j_k}) < 1/k
    \]
    and therefore
    \[
    r_{j_k}-1/k \le \liminf_{i\to\infty} \square(X,\cY_{n_i})
    \le \limsup_{i\to\infty} \square(X,\cY_{n_i}) \le r_{j_k}+1/k,
    \]
    which proves the existence of the limit
    $\lim_{i\to\infty} \square(X,\cY_{n_i})$.
    The claim has been proved.
  \end{proof}
  Let $\cY$ be the set of mm-spaces $X$
  satisfying $\lim_{i\to\infty}\square(X,\cY_{n_i}) = 0$.
  \begin{clm}
    $\cY$ is $\square$-closed.
  \end{clm}
  \begin{proof}
    Assume that a sequence of mm-spaces $X_j \in \cY$
    $\square$-converges to an mm-space $X$.
    A triangle inequality implies
    \[
    \square(X,\cY_{n_i})
    \le \square(X,X_j) + \square(X_j,\cY_{n_i}),
    \]
    where we have $\lim_{j\to\infty} \square(X,X_j) = 0$ and
    $\lim_{i\to\infty} \square(X_j,\cY_{n_i})
    = 0$ because $X_j \in \cY$.
    Therefore, $\lim_{i\to\infty} \square(X,\cY_{n_i}) = 0$
    and $X \in \cY$.
    We have the claim.
  \end{proof}
  From the definition of $\cY$,
  it is easy to prove
  the weak convergence of $\cY_{n_i}$ to $\cY$.
  This completes the proof of Lemma \ref{lem:w-subconv}.
\end{proof}

\begin{prop}[Down-to-earth criterion for weak convergence] \label{prop:w-conv}
  \index{down-to-earth criterion for weak convergence}
  \ \\
  For given $\cY_n,\cY \in \cF(\cX)$, $n=1,2,\dots$,
  the following {\rm(1)} and {\rm(2)} are equivalent
  to each other.
  \begin{enumerate}
  \item $\cY_n$ converges weakly to $\cY$.
  \item Let $\underline{\cY}_\infty$ be the set of the limits
    of convergent sequences $Y_n \in \cY_n$,
    and $\overline{\cY}_\infty$ the set of the limits of
    convergent subsequences of $Y_n \in \cY_n$.
    Then we have
    \[
    \cY = \underline{\cY}_\infty = \overline{\cY}_\infty.
    \]
  \end{enumerate}
\end{prop}

Note that we have
$\underline{\cY}_\infty \subset \overline{\cY}_\infty$ in general.

\begin{proof}
  We prove `(1) $\implies$ (2)'.
  Assume that $\cY_n$ converges weakly to $\cY$.

  Let us first prove $\cY \subset \underline{\cY}_\infty$.
  We take any mm-space $X \in \cY$.
  From Definition \ref{defn:w-conv}(1),
  we have $\lim_{n\to\infty} \square(X,\cY_n) = 0$.
  There is a sequence of mm-spaces
  $X_n \in \cY_n$ that $\square$-converges to $X$,
  i.e., $X \in \underline{\cY}_\infty$.
  Thus we obtain $\cY \subset \underline{\cY}_\infty$.

  Let us next prove $\overline{\cY}_\infty \subset \cY$.
  Take any mm-space $X \in \overline{\cY}_\infty$.
  There is a sequence of mm-spaces $X_i \in \cY_{n_i}$
  with $n_i \to \infty$ that $\square$-converges to $X$.
  If $X$ does not belong to $\cY$, then,
  by Definition \ref{defn:w-conv}(2), we have
  $\liminf_{n\to\infty} \square(X,\cY_n) > 0$,
  which contradicts that $\cY_{n_i} \ni X_i \overset{\square}{\to} X$
  as $i\to\infty$.
  Thus, $X$ belongs to $\cY$ and
  we have $\overline{\cY}_\infty \subset \cY$.
  We obtain (2).

  We prove `(2) $\implies$ (1)'.
  Assume that
  $\cY = \underline{\cY}_\infty = \overline{\cY}_\infty$.
%
%
%
  Let us verify Definition \ref{defn:w-conv}(1).
  Take any mm-space $X \in \cY$.
  Since $X \in \underline{\cY}_\infty$,
  there is a sequence of mm-spaces $X_n \in \cY_n$, $n=1,2,\dots$,
  that $\square$-converges to $X$.
  Therefore,
  \[
  \limsup_{n\to\infty} \square(X,\cY_n)
  \le \lim_{n\to\infty} \square(X,X_n) = 0.
  \]

  Let us verify Definition \ref{defn:w-conv}(2).
  Suppose that $\liminf_{n\to\infty} \square(X,\cY_n) = 0$
  for an mm-space $X$.
  It suffices to prove that $X$ belongs to $\cY$.
  We find a subsequence $\{\cY_{n_i}\}$ of $\{\cY_n\}$
  in such a way that $\lim_{i\to\infty} \square(X,\cY_{n_i}) = 0$.
  There is an mm-space $X_i \in \cY_{n_i}$ for each $i$
  such that $X_i$ $\square$-converges to $X$ as $i\to\infty$.
  Therefore, $X$ belongs to $\overline{\cY}_\infty = \cY$.

  This completes the proof.
\end{proof}

\section{Weak convergence of pyramids}

To show that the weak limit of pyramids is also a pyramid,
we prove the following lemma.

\begin{lem} \label{lem:lim-pyramid}
  \begin{enumerate}
  \item If a sequence of mm-spaces $X_n$, $n=1,2,\dots$, $\square$-converges
    to an mm-space $X$ and if $X$ dominates an mm-space $Y$,
    then there exists a sequence of mm-spaces $Y_n$ $\square$-converging
    to $Y$ such that $X_n$ dominates $Y_n$ for each $n$.
  \item If two sequences of mm-spaces $X_n$ and $Y_n$, $n=1,2,\dots$,
    $\square$-converge to $X$ and $Y$, respectively, and if $X_n$ and $Y_n$
    are both dominated by an mm-space $\tilde{Z_n}$ for each $n$,
    then there exists a sequence of mm-spaces $Z_n$
    such that $X_n,Y_n \prec Z_n \prec \tilde{Z_n}$
    and $\{Z_n\}$ has a $\square$-convergent subsequence.
  \end{enumerate}
\end{lem}

Note that the limit of the $\square$-convergent subsequence of $\{Z_n\}$
dominates $X$ and $Y$ by Theorem \ref{thm:dom}.

\begin{proof}
  We prove (1).
  By Corollary \ref{cor:XN}, $Y$ is approximated by
  $\underline{Y}_N = (\R^N,\|\cdot\|_\infty,\underline{\mu}_N)$,
  where $\Phi_N : Y \to (\R^N,\|\cdot\|_\infty)$ is a $1$-Lipschitz map
  and $\underline{\mu}_N := (\Phi_N)_*\mu_Y$.
  Since $X \succ Y$, we find a $1$-Lipschitz map $F : X \to Y$
  with $F_*\mu_X = \mu_Y$.
  There is an $\varepsilon_n$-mm-isomorphism $p_n : X_n \to X$
  with $\varepsilon_n \to 0$.
  Then the composition $f_n := \Phi_N \circ F \circ p_n
  : X_n \to \R^N$
  is $1$-Lipschitz up to $\varepsilon_n$.
  Applying Lemma \ref{lem:Lip-approx} we find
  a $1$-Lipschitz map $\tilde{f}_n : X_n \to (\R^N,\|\cdot\|_\infty)$
  in such a way that
  \[
  \dKF(\tilde{f}_n,f_n) \le \varepsilon_n.
  \]
  By Corollary \ref{cor:push-di-Lip},
  \begin{align*}
    d_P((f_n)_*\mu_{X_n},\underline{\mu}_N)
    &= d_P((\Phi_N\circ F)_*(p_n)_*\mu_{X_n},(\Phi_N\circ F)_*\mu_X)\\
    &\le d_P((p_n)_*\mu_{X_n},\mu_X) \le \varepsilon_n.
  \end{align*}
  Note that this holds for all $N$ and $n$.
  Take a sequence of natural numbers $N_n$, $n=1,2,\dots$,
  divergent to infinity.  We see
  \[
  \lim_{n\to\infty} d_P((f_n)_*\mu_{X_n},\underline{\mu}_{N_n}) = 0.
  \]
  Letting $Y_n' := (\R^{N_n},\|\cdot\|_\infty,(f_n)_*\mu_{X_n})$,
  we have, by Proposition \ref{prop:box-di},
  \[
  \lim_{n\to\infty} \square(Y_n',\underline{Y}_{N_n}) = 0,
  \]
  which together with the convergence
  $\underline{Y}_{N_n} \overset{\square}{\to} Y$ implies that
  $Y_n'$ $\square$-converges to $Y$ as $n\to\infty$.
  By setting $Y_n := (\R^N,\|\cdot\|_\infty,(\tilde{f}_n)_*\mu_{X_n})$,
  it follows from Proposition \ref{prop:box-di} and Lemma \ref{lem:di-me}
  that
  \[
  \square(Y_n,Y_n') \le 2\,d_P((\tilde{f}_n)_*\mu_{X_n},(f_n)_*\mu_{X_n}) \le 
  2\dKF(\tilde{f}_n,f_n) \le 2\varepsilon_n,
  \]
  so that $Y_n$ $\square$-converges to $Y$.
  Since $\tilde{f}_n : X_n \to Y_n$ is $1$-Lipschitz,
  we have $Y_n \prec X_n$.
  (1) has been proved.

  We prove (2).
  Since $X_n,Y_n \prec \tilde{Z}_n$,
  we find $1$-Lipschitz maps $\tilde{f}_n : \tilde{Z}_n \to X_n$,
  $\tilde{g}_n : \tilde{Z}_n \to Y_n$ such that
  $(\tilde{f}_n)_*\mu_{\tilde{Z}_n} = \mu_{X_n}$ and
  $(\tilde{g}_n)_*\mu_{\tilde{Z}_n} = \mu_{Y_n}$.
  We set, for $x,y \in \tilde{Z}_n$,
  \[
  d_n(x,y) := \max\{\;d_{X_n}(\tilde{f}_n(x),\tilde{f}_n(y)),
  d_{Y_n}(\tilde{g}_n(x),\tilde{g}_n(y))\;\}.
  \]
  It is easy to see that $d_n$ is a pseudo-metric on $\tilde{Z}_n$.
  Let $\hat{Z}_n$ be the quotient space of $\tilde{Z}_n$
  modulo $d_n=0$.
  Then, $d_n$ induces a metric on $\hat{Z}_n$.
  Let $(Z_n,d_{Z_n})$ be the completion of $\hat{Z}_n$.
  Note that $\hat{Z}_n$ is naturally embedded in $Z_n$.
  If $d_n(x,y) = 0$ for two points $x,y \in \tilde{Z}_n$,
  then $\tilde{f}_n(x) = \tilde{f}_n(y)$ and $\tilde{g}_n(x) = \tilde{g}_n(y)$.
  Therefore, setting $f_n([x]) := \tilde{f}_n(x)$
  and $g_n([x]) := \tilde{g}_n(x)$ for an equivalence class $[x] \in \hat{Z}_n$,
  we obtain two maps $f_n : \hat{Z}_n \to X_n$ and $g_n : \hat{Z}_n \to Y_n$
  for each $n$.
  It follows from the $1$-Lipschitz continuity of
  $\tilde{f}_n$ and $\tilde{g}_n$ that $f_n$ and $g_n$ are both $1$-Lipschitz
  continuous, so that both of them extend to $1$-Lipschitz
  maps $f_n : Z_n \to X_n$ and $g_n : Z_n \to Y_n$.
  Let $\pi_n : \tilde{Z}_n \to Z_n$ be the natural projection
  and let $\mu_{Z_n} := (\pi_n)_*\mu_{\tilde{Z}_n}$.
  Then, $Z_n = (Z_n,d_{Z_n},\mu_{Z_n})$ is an mm-space dominated by $\tilde{Z}_n$.
  Since
  \[
  (f_n)_*\mu_{Z_n} = (\tilde{f}_n)_*\mu_{\tilde{Z}_n} = \mu_{X_n}
  \quad\text{and}\quad
  (g_n)_*\mu_{Z_n} = (\tilde{g}_n)_*\mu_{\tilde{Z}_n} = \mu_{Y_n},
  \]
  we have $X_n, Y_n \prec Z_n$.
  The $\square$-precompactness of $\{X_n\}$ and $\{Y_n\}$ together with
  Lemma \ref{lem:precpt} implies that for any $\varepsilon > 0$
  there is a number $\Delta(\varepsilon)$ such that
  for each $n$ we find Borel subsets $K_{nj} \subset X_n$,
  $j=1,\dots,N$, and $K_{nj}' \subset Y_n$, $j=1,\dots,N'$,
  with $N,N' \le \Delta(\varepsilon)$ in such a way that
  \begin{align*}
    \diam K_{nj}, \; &\diam K_{nj}' \le \varepsilon,\\
    \diam \bigcup_{j=1}^N K_{nj},
    \; &\diam \bigcup_{j=1}^{N'} K_{nj}' \le \Delta(\varepsilon),\\
    \mu_{X_n}\left(X_n \setminus \bigcup_{j=1}^N K_{nj}\right),
    \; &\mu_{Y_n}\left(Y_n \setminus \bigcup_{j=1}^{N'} K_{nj}\right) \le \varepsilon.
  \end{align*}
  Set $K_{njk} := f_n^{-1}(K_{nj}) \cap g_n^{-1}(K_{nk}')$
  for $j=1,\dots,N$ and $k=1,\dots,N'$.
  To prove the $\square$-precompactness of $\{Z_n\}$,
  we are going to verify (4) of Lemma \ref{lem:precpt}.

  Let us prove that $\diam K_{njk} \le \varepsilon$.
  In fact, if we take two points $x,y \in K_{njk}$, then
  since $f_n(x),f_n(y) \in K_{nj}$, $g_n(x),g_n(y) \in K_{nk}'$,
  we have $d_{X_n}(f_n(x),f_n(y)), d_{Y_n}(g_n(x),g_n(y)) \le \varepsilon$
  and so $d_{Z_n}(x,y) \le \varepsilon$.

  Let us prove that
  $\diam \bigcup_{1\le j\le N, 1\le k\le N'}K_{njk} \le \Delta(\varepsilon)$.
  For any two points $x,y \in \bigcup_{1\le j\le N, 1\le k\le N'}K_{njk}$,
  there are four numbers $j(x)$, $k(x)$, $j(y)$, and $k(y)$ such that
  $x \in K_{nj(x)k(x)}$ and $y \in K_{nj(y)k(y)}$.
  Since $f_n(x) \in K_{nj(x)}$, $g_n(x) \in K_{nk(x)}'$,
  $f_n(y) \in K_{nj(y)}$, and $g_n(y) \in K_{nk(y)}'$, we have
  \begin{align*}
    d_{X_n}(f_n(x),f_n(y)) &\le \diam \bigcup_{j=1}^N K_{nj}
    \le \Delta(\varepsilon),\\
    d_{Y_n}(g_n(x),g_n(y)) &\le \diam \bigcup_{k=1}^N K_{nk}'
    \le \Delta(\varepsilon),
  \end{align*}
  which implies $d_{Z_n}(x,y) \le \Delta(\varepsilon)$.

  Let us prove that $\mu_{Z_n}(Z_n \setminus \bigcup_{j,k}K_{njk}) \le 2\varepsilon$.
  In fact, since $Z_n \setminus \bigcup_{j,k} K_{njk}
  = f_n^{-1}(X_n \setminus \bigcup_j K_{nj})
  \cup g_n^{-1}(Y_n \setminus \bigcup_k K_{jk}')$,
  we have
  \begin{align*}
    \mu_{Z_n}\left(Z_n \setminus \bigcup_{j,k}K_{njk}\right)
    &\le \mu_{X_n}\left(X_n \setminus \bigcup_j K_{nj}\right)
    + \mu_{Y_n}\left(Y_n \setminus \bigcup_k K_{jk}'\right)\\
    &\le 2\varepsilon.
  \end{align*}

  Applying Lemma \ref{lem:precpt} yields 
  the $\square$-precompactness of $\{Z_n\}$.
  In particular, it has a $\square$-convergent subsequence.
  This completes the proof.
\end{proof}

Combining Lemma \ref{lem:lim-pyramid} and
Theorem \ref{thm:dom} implies the following.

\begin{prop} \label{prop:lim-pyramid}
  If a sequence of pyramids converges weakly,
  then the weak limit is also a pyramid.
\end{prop}

Lemma \ref{lem:w-subconv} together with 
this proposition implies the following theorem.

\begin{thm} \label{thm:cpt-pyramid}
  The set $\Pi$ of pyramids is sequentially compact,
  i.e., any sequence of pyramids has a subsequence that converges weakly to
  a pyramid.
\end{thm}

\begin{prop} \label{prop:conc-pyramid}
  For given mm-spaces $X$ and $X_n$, $n=1,2,\dots$,
  the following {\rm(1)} and {\rm(2)} are equivalent to each other.
  \begin{enumerate}
  \item $X_n$ concentrates to $X$ as $n\to\infty$.
  \item $\cP_{X_n}$ converges weakly to $\cP_X$ as $n\to\infty$.
  \end{enumerate}
\end{prop}

\begin{proof}
  We prove `(1) $\implies$ (2)'.
  Assume that $X_n$ concentrates $X$ as $n\to\infty$.
  Let us first prove that $\cP_X = \cT\{X_n\}$.
  Proposition \ref{prop:conc-tail}(1) says that
  $X$ is a maximal element of $\cT\{X_n\}$, which implies
  $\cT\{X_n\} \subset \cP_X$.
  It follows from $X \in \cT\{X_n\}$ that there is a sequence of mm-spaces
  $X_n'$ with $X_n' \prec X_n$, $n=1,2,\dots$,
  such that $X_n'$ $\square$-converges to $X$ as $n\to\infty$.
  Let $Y \in \cP_X$ be any mm-space.
  Lemma \ref{lem:lim-pyramid}(1) tells us that
  there is a sequence $Y_n$
  $\square$-converging to $Y$ such that $Y_n \prec X_n' \prec X_n$.
  This proves that $Y$ belongs to $\cT\{X_n\}$.
  We thus have $\cP_X = \cT\{X_n\}$.

  We apply the above discussion to any subsequence $\{X_{n_i}\}$ of $\{X_n\}$
  to obtain $\cP_X = \cT\{X_{n_i}\}$.
  Therefore, $\cY_n := \cP_{X_n}$ satisfies
  $\underline{\cY}_\infty = \overline{\cY}_\infty = \cP_X$,
  so that Proposition \ref{prop:w-conv} proves
  the weak convergence of $\cP_{X_n}$ to $\cP_X$.

  We prove `(2) $\implies$ (1)'.
  Assume that $\cP_{X_n}$ converges weakly to $\cP_X$.
  Since $\cT\{X_n\}$ is the set of the limits of mm-spaces in $\cP_{X_n}$,
  Proposition \ref{prop:w-conv} shows that,
  for any subsequence $\{X_{n_i}\}$ of $\{X_n\}$,
  we have $\cP_X = \cT\{X_n\} = \cT\{X_{n_i}\}$.
  In particular, $X$ is a maximal element of $\cT\{X_n\}$.
  By Proposition \ref{prop:conc-tail}(2),
  $X_n$ concentrates to $X$.
\end{proof}

Proposition \ref{prop:conc-pyramid} means that
the map
\[
\iota : \cX \ni X \longmapsto \cP_X \in \Pi
\]
is a topological embedding map with respect to
the concentration topology on $\cX$,
where the topology (metric) on the set $\Pi$ of pyramids
is introduced in the next section.
In Section \ref{sec:comp}, we prove that
the $\dconc$-completion of $\cX$ is also embedded in $\Pi$.

\begin{cor}
  Let $\{X_n\}_{n=1}^\infty$ be a L\'evy family.
  Then, any mm-space $Y_n$ with $Y_n \prec X_n$
  either $\square$-converges to a one-point space $*$,
  or $\square$-diverges.
\end{cor}

\begin{proof}
  Since $X_n$ concentrates to the one-point space $*$ as $n\to\infty$,
  the associated pyramid $\cP_{X_n}$ converges weakly to the pyramid $\{*\}$.
  This proves the corollary.
\end{proof}

\begin{cor}
  Any pyramid is $\dconc$-closed.
\end{cor}

\begin{proof}
  Let $\cP$ be a pyramid and
  assume that a sequence of mm-spaces $X_n \in \cP$, $n=1,2,\dots$,
  concentrates to an mm-space $X$.
  By Proposition \ref{prop:conc-pyramid},
  $\cP_{X_n}$ converges weakly to $\cP_X$ as $n\to\infty$.
  It follows from $\cP_{X_n} \subset \cP$ and the $\square$-closedness of $\cP$
  that $\cP_X$ is contained in $\cP$ and in particular
  $X$ belongs to $\cP$.
  This completes the proof.
\end{proof}

\section{Metric on the space of pyramids}

A main purpose in this section is to introduce
a compact metric on the set of pyramids compatible with
weak convergence.

\begin{defn}[$\cX(N,R)$]
  \index{XNR@$\cX(N,R)$}
  Let $N$ be a natural number and $R$ a nonnegative extended real number,
  i.e., $0 \le R \le +\infty$.
  Denote by $\cM(N,R)$ the set of $\mu \in \cM(N)$ with
  $\supp\mu \subset B^N_R$, where $B^N_\infty := \R^N$.
  \index{BNinfinity@$B^N_\infty$}
  Note that $\cM(N,\infty) = \cM(N)$.
  We define
  \[
  \cX(N,R) := \{\;(B^N_R,\|\cdot\|_\infty,\mu) \mid \mu \in \cM(N,R)\;\}.
  \]
\end{defn}

If $R < +\infty$, then $\cX(N,R)$ is $\square$-compact.

\begin{lem} \label{lem:X-dense}
  $\bigcup_{N=1}^\infty \cX(N,N)$ is $\square$-dense in $\cX$.
\end{lem}

\begin{proof}
  We take any mm-space $X$ and fix it.
  By Corollary \ref{cor:XN}, there is a sequence of mm-spaces
  $\underline{X}_N = (\R^N,\|\cdot\|_\infty,\underline{\mu}_N)$, $N=1,2,\dots$,
  with $\underline{\mu}_N \in \cM(X;N)$ that $\square$-converges to $X$.
  Let $\pi_R : \R^N \to B^N_R$, $R > 0$, be the nearest point projection.
  For $R \in (\,0,+\infty\,)$,
  the push-forward measure $\underline{\mu}_{N,R} := (\pi_R)_*\underline{\mu}_N$
  belongs to $\cM(X;N,R)$ and converges weakly to $\underline{\mu}_N$
  as $R \to +\infty$.
  The mm-space
  $\underline{X}_{N,R} := (B^N_R,\|\cdot\|_\infty,\underline{\mu}_{N,R})$
  belongs to $\bigcup_{N=1}^\infty \cX(N,N)$
  and $\square$-converges to $\underline{X}_N$ as $R\to +\infty$
  (see Proposition \ref{prop:box-di}).
  By using a triangle inequality,
  there is a sequence $R_N \to +\infty$ such that
  $\underline{X}_{N,R_N}$ $\square$-converges to $X$ as $N\to\infty$.
  This completes the poof.
\end{proof}

\begin{lem} \label{lem:down-XNR}
  Let $N$ be a natural number and let $0 \le R \le +\infty$.
  If a sequence of mm-spaces $X_n$, $n=1,2,\dots$, $\square$-converges
  to $(B^N_R,\|\cdot\|_\infty,\mu)$ for a measure $\mu \in \cM(N,R)$,
  then there exists a sequence of measures $\mu_n \in \cM(X_n;N,R)$,
  $n=1,2,\dots$, converging weakly to $\mu$.
  In particular, setting
  $X_n' := (B^N_R,\|\cdot\|_\infty,\mu_n)$, we have
  $X_n' \in \cX(N,R)$, $X_n' \prec X_n$, and
  $\lim_{n\to\infty} \square(X_n',X) = 0$.
\end{lem}

\begin{proof}
  Assume that a sequence of mm-spaces $X_n$, $n=1,2,\dots$,
  $\square$-converges to $(B^N_R,\|\cdot\|_\infty,\mu)$
  for a measure $\mu \in \cM(N,R)$.
  Then, there is an
  $\varepsilon_n$-mm-isomorphism $f_n : X_n \to X = (\R^N,\|\cdot\|_\infty,\mu)$
  with $\varepsilon_n \to 0$.
  This satisfies $d_P((f_n)_*\mu_{X_n},\mu) \le \varepsilon_n$.
  Apply Lemma \ref{lem:Lip-approx}
  to obtain a $1$-Lipschitz map $\tilde{f}_n : X_n \to (\R^N,\|\cdot\|_\infty)$
  such that $\dKF(\tilde{f}_n,f_n) \le \varepsilon_n$.
  We therefore have
  \begin{align*}
    d_P((\tilde{f}_n)_*\mu_{X_n},\mu)
    &\le d_P((\tilde{f}_n)_*\mu_{X_n},(f_n)_*\mu_{X_n}) 
    + d_P((f_n)_*\mu_{X_n},\mu)\\
    &\le \dKF(\tilde{f}_n,f_n) + \varepsilon_n
    \le 2\varepsilon_n.
  \end{align*}
  In the case where $R = \infty$, the measure
  $\mu_n := (\tilde{f}_n)_*\mu_{X_n}$ is a desired one.
  In the case where $R < \infty$,
  since $\supp\mu \subset B^N_R$,
  the measure $\mu_n := (\pi_R)_*(\tilde{f}_n)_*\mu_{X_n}$
  satisfies
  $\mu_n \in \cM(X_n;N,R)$ and $\mu_n \to \mu$ weakly,
  where $\pi_R : \R^N \to B^N_R$ is the nearest point projection.
  We thus obtain the first part of the lemma.
  The rest is clear.
  This completes the proof.
\end{proof}

\begin{lem} \label{lem:pyramid-conv-dH}
  For given pyramids $\cP$ and $\cP_n$, $n=1,2,\dots$,
  the following {\rm(1)} and {\rm(2)} are equivalent to each other.
  \begin{enumerate}
  \item $\cP_n$ converges weakly to $\cP$ as $n\to\infty$.
  \item For any natural number $N$ and for any real number $R \ge 0$,
    the set $\cP_n\cap\cX(N,R)$ Hausdorff
    converges to $\cP\cap\cX(N,R)$ as $n\to\infty$, where the Hausdorff distance
    is induced from the box metric.
  \end{enumerate}
\end{lem}

\begin{proof}
  We prove `(1) $\implies$ (2)'.
  Suppose that $\cP_n$ converges weakly to $\cP$, but
  $\cP_n\cap\cX(N,R)$ does not Hausdorff converge to $\cP\cap\cX(N,R)$
  for some $N$ and $R$.
  We find a subsequence $\{\cP_{n_i}\}$ of $\{\cP_n\}$ in such a way that
  $\liminf_{n\to\infty} d_H(\cP_n\cap\cX(N,R),\cP\cap\cX(N,R)) > 0$.
  Since $\cX(N,R)$ is $\square$-compact,
  Lemma \ref{lem:dH}(2) tells us the $d_H$-compactness of
  $\cF(\cX(N,R))$.
  By replacing $\{\cP_{n_i}\}$ with a subsequence,
  $\cP_{n_i} \cap \cX(N,R)$ Hausdorff converges to
  some compact subset $\cP_\infty \subset \cX(N,R)$,
  that is different from $\cP\cap\cX(N,R)$.
  Since any mm-space $X \in \cP_\infty$ is the limit of
  some $X_i \in \cP_{n_i} \cap \cX(N,R)$, $i=1,2,\dots$,
  the set $\cP_\infty$ is contained in $\cP$, so that
  $\cP_\infty \subset \cP\cap\cX(N,R)$.
  For any mm-space $X \in \cP \cap \cX(N,R)$,
  there is a sequence of mm-spaces $X_i \in \cP_{n_i}$ $\square$-converging
  to $X$ as $i\to\infty$.
  By Lemma \ref{lem:down-XNR},
  we find a sequence of mm-spaces $X_i' \in \cX(N,R)$ with $X_i' \prec X_i$
  that $\square$-converges to $X$.
  Since $X_i' \in \cP_{n_i} \cap \cX(N,R)$,
  the space $X$ belongs to $\cP_\infty$.
  Thus we have $\cP_\infty = \cP \cap \cX(N,R)$.
  This is a contradiction.

  We prove `(2) $\implies$ (1)'.
  We assume (2).
  Let $\underline{\cP}_\infty$ be the set of the limits of
  convergent sequences of mm-spaces $X_n \in \cP_n$, and
  $\overline{\cP}_\infty$ the set of the limits of
  convergent subsequences of mm-spaces $X_n \in \cP_n$.
  We have $\underline{\cP}_\infty \subset \overline{\cP}_\infty$ in general.
  By Proposition \ref{prop:w-conv},
  it suffices to prove that $\underline{\cP}_\infty = \overline{\cP}_\infty = \cP$.

  Let us first prove $\cP \subset \underline{\cP}_\infty$.
  Take any mm-space $X \in \cP$.
  By Lemma \ref{lem:X-dense},
  there is a sequence of mm-spaces $X_i \in \bigcup_{N=1}^\infty \cX(N,N)$
  that $\square$-converges to $X$.
  For each $i$ we find a natural number $N_i$ with $X_i \in \cX(N_i,N_i)$.
  By (2), there is a sequence of mm-spaces
  $X_{in} \in \cP_n \cap \cX(N_i,N_i)$, $n=1,2,\dots$,
  that $\square$-converges to $X_i$ for each $i$.
  There is a sequence $i_n \to \infty$ such that
  $X_{i_nn}$ $\square$-converges to $X$, so that
  $X$ belongs to $\underline{\cP}_\infty$.
  We obtain $\cP \subset \underline{\cP}_\infty$.

  The rest of the proof is to show that $\overline{\cP}_\infty \subset \cP$.
  Take any mm-space $X \in \overline{\cP}_\infty$.
  By Corollary \ref{cor:XN},
  $X$ is approximated by
  some $\underline{X}_N = (\R^N,\|\cdot\|_\infty,\underline{\mu}_N)$,
  $\underline{\mu}_N \in \cM(X;N)$.
  Since $(\pi_R)_*\underline{\mu}_N \to \underline{\mu}_N$
  as $R \to +\infty$, where
  $\pi_R : \R^N \to B^N_R$ is the nearest point projection,
  $X$ is approximated by some $X' \in \cX(N,R)$ with $X' \prec X$.
  By the $\square$-closedness of $\cP$,
  it suffices to prove that $X'$ belongs to $\cP$.
  It follows from $X \in \overline{\cP}_\infty$ that
  there are sequences $n_i \to \infty$ and $X_i \in \cP_{n_i}$
  such that $X_i$ $\square$-converges to $X$.
  By Lemma \ref{lem:lim-pyramid}(1),
  we find a sequence of mm-spaces
  $X_i'$ with $X_i' \prec X_i$ that $\square$-converges to $X'$.
  Lemma \ref{lem:down-XNR} implies
  the existence of a sequence $X_i'' \in \cX(N,R)$ such that
  $X_i'' \prec X_i'$ for any $i$ and $X_i''$ converges to $X'$ as $i\to\infty$.
  Since $\cP_{n_i}$ is a pyramid, $X_i''$ belongs to $\cP_{n_i}$.
  By (2), $X'$ is an element of $\cP$.
  This completes the proof.
\end{proof}

\begin{lem} \label{lem:pyramid-conv-mono}
  Let $\cP$ and $\cP_n$, $n=1,2,\dots$, be pyramids.
  Let $N$ be a natural number and let $R' \ge R \ge 0$.
  If $\cP_n\cap\cX(N,R')$ Hausdorff converges to $\cP\cap\cX(N,R')$
  as $n\to\infty$, then
  $\cP_n\cap\cX(N,R)$ Hausdorff converges to $\cP\cap\cX(N,R)$
  as $n\to\infty$.
\end{lem}

\begin{proof}
  Assume that $\cP_n\cap\cX(N,R')$ Hausdorff converges to $\cP\cap\cX(N,R')$
  as $n\to\infty$.
  By the compactness of $\cX(N,R)$,
  it suffices to prove that the limit set of any subsequence
  of $\{\cP_n\cap\cX(N,R)\}$
  coincides with $\cP\cap\cX(N,R)$.

  Take any mm-space $X \in \cP\cap\cX(N,R)$.
  By the assumption, there are mm-spaces $X_n \in \cP_n\cap\cX(N,R')$
  such that $X_n$ $\square$-converges to $X$.
  Lemma \ref{lem:down-XNR} implies that
  there are mm-spaces $X_n' \in \cP_n\cap\cX(N,R)$
  such that $X_n' \prec X_n$ for each $n$ and $X_n'$ $\square$-converges
  to $X$.  Therefore, the limit set of any subsequence of
  $\{\cP_n\cap\cX(N,R)\}$ contains $\cP\cap\cX(N,R)$.

  Let $\{\cP_{n_i}\cap\cX(N,R)\}$ be a subsequence of
  $\{\cP_n\cap\cX(N,R)\}$,
  and let $X_i \in \cP_{n_i}\cap\cX(N,R)$ $\square$-converge to
  an mm-space $X$ as $i\to\infty$.
  Since $\cP_n\cap\cX(N,R')$ Hausdorff converges to $\cP\cap\cX(N,R')$,
  the limit mm-space $X$ belongs to $\cP$, so that $X \in \cP\cap\cX(N,R)$.
  This completes the proof.
\end{proof}

\begin{defn}[Metric on the space of pyramids]
  \label{defn:metric-Pi}
  \index{metric on the space of pyramids}
  \index{rho, rhok@$\rho$, $\rho_k$}
  Define
  for a natural number $k$ and for two pyramids $\cP$ and $\cP'$,
  \begin{align*}
    \rho_k(\cP,\cP') &:= \frac{1}{4k} d_H(\cP\cap\cX(k,k),\cP'\cap\cX(k,k)),\\
    \rho(\cP,\cP') &:= \sum_{k=1}^\infty 2^{-k} \rho_k(\cP,\cP').
  \end{align*}
\end{defn}

\begin{thm} \label{thm:metric-Pi}
  $\rho$ is a metric on the space $\Pi$ of pyramids
  that is compatible with weak convergence.
  $\Pi$ is compact with respect to $\rho$.
\end{thm}

\begin{proof}
  We first prove that $\rho$ is a metric.
  Since $\square \le 1$, we have $\rho_k \le 1/(4k)$ for each $k$
  and then $\rho \le 1/4$.
  Each $\rho_k$ is a pseudo-metric on $\Pi$ and so is $\rho$.
  If $\rho(\cP,\cP') = 0$ for two pyramids $\cP$ and $\cP'$,
  then $\rho_k(\cP,\cP') = 0$ for any $k$,
  which implies $\cP = \cP'$.
  Thus, $\rho$ is a metric on $\Pi$.

  We next prove the compatibility of the metric $\rho$
  with weak convergence in $\Pi$.
  It follows from Lemmas \ref{lem:pyramid-conv-dH} and
  \ref{lem:pyramid-conv-mono} that
  a sequence of pyramids $\cP_n$, $n=1,2,\dots$,
  converges weakly to a pyramid $\cP$
  if and only if $\lim_{n\to\infty} \rho_k(\cP_n,\cP) = 0$
  for any $k$, which is also equivalent to
  $\lim_{n\to\infty} \rho(\cP_n,\cP) = 0$.

  Since $\Pi$ is sequentially compact
  (see Theorem \ref{thm:cpt-pyramid}),
  it is compact with respect to $\rho$.
  This completes the proof.
\end{proof}

\begin{lem} \label{lem:PX-MNR-dH}
  Let $X$ and $Y$ be two mm-spaces, $N$ a natural number,
  and $R$ a nonnegative extended real number.
  Then we have
  \[
  d_H(\cP_X \cap \cX(N,R),\cP_Y \cap \cX(N,R))
  \le 2\,d_H(\cM(X;N,R),\cM(Y;N,R)).
  \]
\end{lem}

\begin{proof}
  The lemma follows from Proposition \ref{prop:box-di}.
\end{proof}

\begin{thm} \label{thm:rho-dconc}
  For any two mm-spaces $X$ and $Y$, we have
  \[
  \rho(\cP_X,\cP_Y) \le \dconc(X,Y),
  \]
  i.e., the embedding map $\iota : \cX \ni X \mapsto \cP_X \in \Pi$
  is $1$-Lipschitz continuous.
\end{thm}

\begin{proof}
  By Lemmas \ref{lem:PX-MNR-dH}, \ref{lem:MR-half}, and \ref{lem:M-dconc},
  \begin{align*}
    & d_H(\cP_X \cap \cX(N,R),\cP_Y \cap \cX(N,R))\\
    &\le 2 d_H(\cM(X;N,R),\cM(Y;N,R))\\
    &\le 4 d_H(\cM(X;N),\cM(Y;N))
    \le 4N \dconc(X,Y),
  \end{align*}
  so that $\rho_k(\cP_X,\cP_Y) \le \dconc(X,Y)$ for any $k$.
  This proves the theorem.
\end{proof}

\begin{rem}
  As is shown in Section \ref{sec:Spheres-Gaussians},
  the sequence of spheres $X_n := S^n(\sqrt{n})$, $n=1,2,\dots$,
  satisfies that $\cP_{X_n}$ converges weakly (Theorem \ref{thm:sphere-Gaussian})
  and that $\{X_n\}$ has no $\dconc$-Cauchy subsequence
  (Corollary \ref{cor:sphere-Gaussian}).
  This implies that it is impossible to estimate $\rho(\cP_X,\cP_Y)$ 
  below by $\dconc(X,Y)$.
  However, by Proposition \ref{prop:conc-pyramid} and
  Theorem \ref{thm:metric-Pi},
  the map $\iota : \cX \ni X \mapsto \cP_X \in \Pi$
  is an embedding map with respect to $\dconc$ and $\rho$.
\end{rem}


\chapter{Asymptotic concentration}
\label{chap:asymp-conc}

\section{Compactification
  of the space of ideal mm-spaces}
\label{sec:comp}

In this section, we prove the main theorem in this book
that the $\dconc$-completion of the space of mm-spaces
is embedded in the set of pyramids.

\begin{defn}[Asymptotic sequence of mm-spaces and asymptotic concentration]
  \index{asymptotic} \index{asymptotic concentration}
  A sequence of mm-spaces $X_n$, $n=1,2,\dots$, is said to be
  \emph{asymptotic} if $\cP_{X_n}$ converges weakly as $n\to\infty$.
  We say that a sequence of mm-spaces \emph{asymptotically concentrates}
  if it is a $\dconc$-Cauchy sequence.
  \index{asymptotically concentrate}
\end{defn}

\begin{prop} \label{prop:asymp}
  If a sequence of mm-spaces asymptotically concentrates,
  then it is asymptotic.
\end{prop}

\begin{proof}
  Let $\{X_n\}$ be a sequence of mm-spaces that asymptotically concentrates.
  Then, Theorem \ref{thm:rho-dconc} proves that
  $\{\cP_{X_n}\}$ is a $\rho$-Cauchy sequence
  and converges in $\Pi$ by the compactness of $(\Pi,\rho)$.
  This completes the proof.
\end{proof}


It is clear that any monotone
nondecreasing (with respect to the Lipschitz order) sequence of mm-spaces,
$X_n$, $n=1,2,\dots$, is asymptotic,
where the limit pyramid is the $\square$-closure of
$\bigcup_{n=1}^\infty \cP_{X_n}$.
In particular, for any given mm-spaces $F_n$, $n=1,2,\dots$,
the product space
\[
F_1 \times F_2 \times \dots \times F_n
\]
with $d_{l_p}$, $1 \le p \le +\infty$, and product measure
$\bigotimes_{i=1}^n \mu_{F_i}$ is asymptotic.

\begin{defn}[Completion of $\cX$ and ideal mm-space]
  \index{Completion of X@completion of $\cX$}
  \index{ideal mm-space}
  Denote by $\bar{\cX}$ the $\dconc$-completion of the set $\cX$
  \index{Xbar@$\bar{\cX}$}
  of mm-isomorphism classes of mm-spaces, and let
  $\partial\cX := \bar{\cX} \setminus \cX$.
  \index{boundaryX@$\partial\cX$}
  We call each element of $\partial\cX$ an \emph{ideal mm-space}.
\end{defn}

Since the map $\iota : \cX \ni X \mapsto \cP_X \in \Pi$ is $1$-Lipschitz
continuous
with respect to $\dconc$ and $\rho$ (Theorem \ref{thm:rho-dconc}),
this uniquely extends to a $1$-Lipschitz continuous
map
\[
\iota : \bar\cX \ni \bar{X} \longmapsto \cP_{\bar{X}} \in \Pi.
\]
A main purpose of this section is to prove that the map $\iota$
is a topological embedding map (see Theorem \ref{thm:emb-pyramid})
by using many statements proved in the previous sections.
Since $\Pi$ is compact and $\iota(\cX)$ is dense in $\Pi$,
this turns out to be a compactification of $\bar\cX$ (and of $\cX$)
with respect to the concentration topology.

\begin{defn}[Measurement of a pyramid]
  \label{defn:m-pyramid} \index{measurement of a pyramid} \index{measurement}
  Let $\cP$ be a pyramid.
  For a natural number $N$ and a nonnegative extended real number $R$,
  we define
  \begin{align*}
    \cM(\cP;N,R) &:= \{\;\mu \in \cM(N,R) \mid
    (B^N_R,\|\cdot\|_\infty,\mu) \in \cP\;\},\\
    \cM(\cP;N) &:= \cM(\cP;N,+\infty).
  \end{align*}
  $\cM(\cP;N,R)$ and $\cM(\cP;N)$ are respectively called
  the \emph{$(N,R)$-measurement} and the \emph{$N$-measurement} of $\cP$.
  \index{NR-measurement@$(N,R)$-measurement}
  \index{N-measurement@$N$-measurement}
\end{defn}

We see that $\cM(\cP_X;N,R) = \cM(X;N,R)$ for any mm-space $X$.
It is obvious that $\cM(\cP;N,R)$ is perfect on $B^N_R$ for any pyramid
$\cP$.

\begin{lem} \label{lem:perfect-lim}
  Let $\cA_n,\cA \subset \cM(N,R)$, $n=1,2,\dots$,
  be closed subsets, where
  $N$ is a natural number and $R$ a nonnegative extended real number.
  If $\cA_n$ Hausdorff converges to $\cA$ as $n\to\infty$
  and if each $\cA_n$ is perfect on $B^N_R$, then
  $\cA$ is perfect on $B^N_R$.
\end{lem}

\begin{proof}
  Assume that $(B^N_R,\|\cdot\|_\infty,\nu) \prec (B^N_R,\|\cdot\|_\infty,\mu)$
  for two measures $\mu \in \cA$ and $\nu \in \cM(N,R)$.
  Since $d_H(\cA_n,\cA) \to 0$,
  there is a sequence of measures $\mu_n \in \cA_n$
  that converges weakly to $\mu$.
  Since $\square((B^N_R,\|\cdot\|_\infty,\mu_n),(B^N_R,\|\cdot\|_\infty,\mu))
  \le d_P(\mu_n,\mu) \to 0$ as $n\to\infty$ and
  by Lemma \ref{lem:lim-pyramid}(1),
  there is a sequence of mm-spaces $X_n$ such that
  $X_n \prec (B^N_R,\|\cdot\|_\infty,\mu_n)$
  and $X_n$ $\square$-converges to $(B^N_R,\|\cdot\|_\infty,\nu)$.
  Applying Lemma \ref{lem:down-XNR},
  we find a sequence of measures $\nu_n \in \cM(X_n;N,R)$
  converging weakly to $\nu$.
  Since $(B^N_R,\|\cdot\|_\infty,\nu_n) \prec X_n \prec
  (B^N_R,\|\cdot\|_\infty,\mu_n)$, the perfectness of $\cA_n$
  implies that $\nu_n$ belongs to $\cA_n$,
  so that $\nu$ belongs to $\cA$.
  This completes the proof.
\end{proof}


\begin{thm}[Observable criterion for asymptotic concentration]
  \label{thm:cri-asymp-conc}
  \index{observable criterion for asymptotic concentration}
  \ \\
  Let $\{X_n\}_{n=1}^\infty$ be a sequence of mm-spaces
  and $\bar{X} \in \partial\cX$ an ideal mm-space.
  Then, the following {\rm(1)}, {\rm(2)}, and {\rm(3)}
  are equivalent to each other.
  \begin{enumerate}
  \item $X_n$ asymptotically concentrates to $\bar{X}$ as $n\to\infty$.
  \item For any natural number $N$,
    the $N$-measurement $\cM(X_n;N)$ Hausdorff converges to $\cM(\cP_{\bar{X}};N)$
    as $n\to\infty$.
  \item For any natural number $N$ and any nonnegative real number $R$,
    the $(N,R)$-measurement $\cM(X_n;N,R)$ Hausdorff converges to
    $\cM(\cP_{\bar{X}};N,R)$ as $n\to\infty$.
  \end{enumerate}
\end{thm}

\begin{proof}[Proof of `{\rm(1)} $\implies$ {\rm(2)}']
  Assume that a sequence of mm-spaces $X_n$, $n=1,2,\dots$,
  asymptotically concentrates to $\bar{X}$.
  We take any natural number $N$ and fix it.
  Lemma \ref{lem:M-dconc} implies that
  $\{\cM(X_n;N)\}_n$ is $d_H$-Cauchy.
  By Lemma \ref{lem:dH}(1),
  $\{\cM(X_n;N)\}_n$ Hausdorff converges to a closed subset
  $\cA \subset \cM(N)$.
  Let us prove
  \begin{align} \label{eq:cL-barX}
    \cP_{\bar{X}} \cap \cX(N) = \{\;(\R^N,\|\cdot\|_\infty,\mu) \mid
    \mu \in \cA\;\}.
  \end{align}

  To prove `$\supset$', we take any measure $\mu \in \cA$.
  There is a sequence of measures $\mu_n \in \cM(X_n;N)$, $n=1,2,\dots$,
  converging weakly to $\mu$.
  $(\R^N,\|\cdot\|_\infty,\mu_n)$ $\square$-converges to
  $(\R^N,\|\cdot\|_\infty,\mu)$ as $n\to\infty$.
  Since $(\R^N,\|\cdot\|_\infty,\mu_n) \in \cP_{X_n}$ and
  since $\cP_{X_n}$ converges weakly to $\cP_{\bar{X}}$,
  the space $(\R^N,\|\cdot\|_\infty,\mu)$ belongs to $\cP_{\bar{X}}$.
    
  To prove `$\subset$', we take any mm-space
  $(\R^N,\|\cdot\|_\infty,\mu) \in \cP_{\bar{X}} \cap \cX(N)$.
  Since $\cP_{X_n}$ converges weakly to $\cP_{\bar{X}}$ as $n\to\infty$,
  there is a sequence of mm-spaces $X_n' \in \cP_{X_n}$
  that $\square$-converges to $(\R^N,\|\cdot\|_\infty,\mu)$.
  Applying Lemma \ref{lem:down-XNR},
  we find a sequence of measures $\mu_n \in \cM(X_n;N)$, $n=1,2,\dots$,
  converging weakly to $\mu$.
  We therefore have $\mu \in \cA$.
  \eqref{eq:cL-barX} has been proved.

  By Lemma \ref{lem:perfect-lim}, $\cA$ is perfect,
  which together with \eqref{eq:cL-barX} proves that
  $\cA = \cM(\cP_{\bar{X}};N)$.
  We obtain (2).
\end{proof}

For the rest of the proof of Theorem \ref{thm:cri-asymp-conc},
we need several lemmas.

\begin{lem} \label{lem:dGH-dconc}
  For any two mm-spaces $X$ and $Y$ we have
  \[
  d_{GH}(\Lo(X),\Lo(Y)) \le \dconc(X,Y).
  \]
\end{lem}

\begin{proof}
  Let $\varphi : I \to X$ and $\psi : I \to Y$ be any parameters.
  By Lemmas \ref{lem:pb-me-di} and \ref{lem:action-dH}(2),
  \begin{align*}
    d_{GH}(\Lo(X),\Lo(Y))
    &\le d_H(\varphi^*\Lo(X),\psi^*\Lo(Y))\\
    &= d_H(\varphi^*\Lip_1(X),\psi^*\Lip_1(Y)).
  \end{align*}
  Taking the infimum of the right-hand side over all $\varphi$ and $\psi$,
  we have the lemma.
\end{proof}

Denote by $\mathcal{H}$ the set of isometry classes of compact metric spaces.
Lemma \ref{lem:dGH-dconc} tells us that the map
\[
\Lo : \cX \ni X \longmapsto \Lo(X) \in \mathcal{H}
\]
is $1$-Lipschitz continuous with respect to $\dconc$ and $d_{GH}$.
Since $(\mathcal{H},d_{GH})$ is a complete metric space
(see Lemma \ref{lem:dGH-complete}),
the map $\Lo$ extends to
\[
\Lo : \bar{\cX} \to \mathcal{H}
\]
as an $1$-Lipschitz map.
If a sequence of mm-spaces $X_n$, $n=1,2,\dots$,
asymptotically concentrates to an ideal mm-space $\bar{X} \in \partial\cX$,
then $\Lo(X_n)$ Gromov-Hausdorff converges to $\Lo(\bar{X})$ as $n\to\infty$.
\index{L1(barX)@$\Lo(\bar{X})$}

\begin{lem} \label{lem:isom-L}
  \begin{enumerate}
  \item For any two mm-spaces $X$ and $Y$ with $Y \prec X$,
    there exists an isometric embedding
    $\iota_Y : \Lo(Y) \hookrightarrow \Lo(X)$.
  \item For any ideal mm-space $\bar{X} \in \partial\cX$
    and for any mm-space $Y \in \cP_{\bar{X}}$,
    there exists an isometric embedding
    $\iota_Y : \Lo(Y) \hookrightarrow \Lo(\bar{X})$.
  \end{enumerate}
\end{lem}

\begin{proof}
  We prove (1).
  By $Y \prec X$, we find a $1$-Lipschitz map $p : X \to Y$
  with $p_*\mu_X = \mu_Y$.
  By Lemma \ref{lem:pb-me-di},
  $\iota_Y := p^* : \Lo(Y) \to \Lo(X)$ is an isometric embedding.

  We prove (2).
  Let $\bar{X} \in \bar\cX$ and $Y \in \cP_{\bar{X}}$.
  There is a sequence of mm-spaces $X_n$, $n=1,2,\dots$,
  that asymptotically concentrates to $\bar{X}$.
  Since $\cP_{X_n}$ converges weakly to $\cP_{\bar{X}}$,
  there is a sequence of mm-spaces $Y_n \in \cP_{X_n}$
  that $\square$-converges to $Y$.
  By $Y_n \prec X_n$ and (1), we find an isometric embedding
  $\iota_{Y_n} : \Lo(Y_n) \hookrightarrow \Lo(X_n)$.
  Since $\Lo(X_n)$ and $\Lo(Y_n)$ Gromov-Hausdorff converge to
  $\Lo(\bar{X})$ and $\Lo(Y)$ respectively as $n\to\infty$,
  some subsequence of $\{p_n^*\}$ converges to
  an isometric embedding $\iota_Y : \Lo(Y) \hookrightarrow \Lo(\bar{X})$
  (see \cite{Petersen}*{\S 10.1.3}).
\end{proof}

\begin{defn}[Concentrated pyramid]
  \index{concentrated pyramid}
  A pyramid $\cP$ is said to be \emph{concentrated}
  if $\{\Lo(X)\}_{X \in \cP}$ is $d_{GH}$-precompact.
\end{defn}

Lemma \ref{lem:isom-L} implies

\begin{cor} \label{cor:concentrated}
  For any {\rm(}ideal{\rm)} mm-space $\bar{X} \in \bar{\cX}$,
  the associated pyramid $\cP_{\bar{X}}$ is concentrated.
\end{cor}

\begin{lem} \label{lem:Sep-CapLo}
  For any mm-space $X$ and any real number $\kappa > 0$,
  we have
  \[
  \Sep(X;\kappa,\kappa) \le \kappa(\Cap_\kappa(\Lo(X))+1).
  \]
\end{lem}

\begin{proof}
  If $\Sep(X;\kappa,\kappa) = 0$, then the lemma is trivial.
  Assume $\Sep(X;\kappa,\kappa) > 0$, and
  let $r$ be any real number with $0 < r < \Sep(X;\kappa,\kappa)$.
  We find two Borel subsets $A, B \subset X$ in such a way that
  $\mu_X(A),\mu_X(B) \ge \kappa$ and $d_X(A,B) > r$.
  Let $f(x) := d_X(x,A)$, $x \in X$.
  It is clear that
  $\lambda f$ belongs to $\Lip_1(X)$ for any $\lambda \in [\,-1,1\,]$.
  For two real numbers $\lambda$ and $\lambda'$
  we estimate $\dKF([\lambda f],[\lambda'f])$.
  Let $c$ be any real number.
  If $x \in A$, then $|\lambda f(x)-\lambda'f(x)-c| = |c|$.
  If $x \in B$, then $|\lambda f(x)-\lambda'f(x)-c|
  \ge |\lambda-\lambda'|f(x)-|c| \ge |\lambda-\lambda'|r - |c|$.
  Therefore, if $|c| \ge |\lambda-\lambda'| r/2$, then
  \begin{equation}
    \label{eq:sup-ObsDiam}
    |\lambda f(x) - \lambda'f(x) - c| \ge \frac{|\lambda-\lambda'| r}{2}
  \end{equation}
  for all $x \in A$.
  If $|c| < |\lambda-\lambda'| r/2$, then
  \eqref{eq:sup-ObsDiam} holds for all $x \in B$.
  We thus have
  \[
  \mu_X\left(|\lambda f(x) - \lambda'f(x) - c|
    \ge \frac{|\lambda-\lambda'| r}{2}\right) \ge \kappa,
  \]
  so that, if $|\lambda-\lambda'| \ge 2\kappa/r$, then
  $\dKF([\lambda f],[\lambda'f]) \ge \kappa$.
  Setting $N := [r/(2\kappa)]$ and $\lambda_k := 2k\kappa/r$
  for $k = 0,\pm 1,\pm 2,\dots,\pm N$,
  we see that
  $-1 \le \lambda_{-N} < \dots < \lambda_{N} \le 1$.
  For two different integers $k$ and $l$ in $\{0,\pm 1,\pm 2,\dots,\pm N\}$,
  we have $|\lambda_k-\lambda_l| \ge 2\kappa/r$
  and so $\dKF([\lambda_k f],[\lambda_l f]) \ge \kappa$.
  Namely, $\{[\lambda_k f]\}_k$ is a $\kappa$-discrete net
  of $\Lo(X)$ and therefore $2N+1 \le \Cap_\kappa(X)$.
  Since $N > r/(2\kappa) -1$, we have
  \[
  \frac{r}{\kappa}-1 < \Cap_\kappa(X).
  \]
  This completes the proof.
\end{proof}

\begin{cor} \label{cor:sup-diam-cP}
  Let $\cP$ be a concentrated pyramid.
  Then we have
  \[
  \sup_{\mu\in\cM(\cP;N)} \diam(\mu;1-\kappa) < +\infty.
  \]
  for any natural number $N$ and any real number $\kappa$ with $0 < \kappa < 1$.
\end{cor}

\begin{proof}
  It follows from Lemmas \ref{lem:dGH-precpt} and \ref{lem:Sep-CapLo} that
  \[
  \sup_{X\in\cP} \Sep(X;\kappa,\kappa) < +\infty
  \]
  for any $\kappa > 0$,
  which together with Proposition \ref{prop:ObsDiam-Sep} and
  Lemma \ref{lem:ObsDiamRN-ObsDiam} implies
  \[
  \sup_{X\in\cP} \ObsDiam_{(\R^N,\|\cdot\|_\infty)}(X;-\kappa) < +\infty
  \]
  for any $N$ and $\kappa$ with $0 < \kappa < 1$.
  This proves the corollary.
\end{proof}

\begin{proof}[Proof of `{\rm(2)} $\Longleftrightarrow$ {\rm(3)}'
  of Theorem \ref{thm:cri-asymp-conc}]
  Let $\cA := \cM(X;N)$ and $\cA_n := \cM(X_n;N)$.
  Since $\cP_{\bar{X}}$ is concentrated (see Corollary \ref{cor:concentrated}),
  Corollary \ref{cor:sup-diam-cP} implies
  the assumption of Lemma \ref{lem:MR}.
  Applying Lemma \ref{lem:MR} completes the proof.
\end{proof}


\begin{defn}[Approximate a pyramid]
  \index{approximate a pyramid}
  Let $\cP$ be a pyramid and $\{Y_m\}_{m=1}^\infty$ a sequence of mm-spaces.
  We say that $\{Y_m\}_{m=1}^\infty$ \emph{approximates $\cP$} if
  \[
  Y_1 \prec Y_2 \prec \dots \prec Y_m \prec \dots\quad\text{and}\quad
  \overline{\bigcup_{m=1}^\infty \cP_{Y_m}}^\square = \cP,
  \]
  where the upper bar with $\square$ means the $\square$-closure.
\end{defn}

We see that, if $\{Y_m\}_{m=1}^\infty$ approximates a pyramid $\cP$,
then $\cP_{Y_m}$ converges weakly to $\cP$ as $m\to\infty$.

\begin{lem} \label{lem:approx-pyramid}
  For any pyramid $\cP$,
  there exists a sequence of mm-spaces that approximates $\cP$.
\end{lem}

\begin{proof}
  The $\square$-separability of $\cX$ (see
  Proposition \ref{prop:separable}) implies that
  there is a dense countable set $\{Y_m'\}_{m=1}^\infty \subset \cP$.
  Let $Y_1 := Y_1'$.
  There is an mm-space $Y_2 \in \cP$ dominating $Y_1$ and $Y_2'$.
  There is also an mm-space $Y_3 \in \cP$
  dominating $Y_2$ and $Y_3'$.
  Repeating this procedure, we have a sequence of mm-spaces
  $Y_m \in \cP$, $m=1,2,\dots$, with the property that
  $Y_1 \prec Y_2 \prec \dots \prec Y_m \prec \cdots$
  and $Y_m' \prec Y_m \in \cP$ for any $m$.
  We see that
  \[
  \{Y_m'\}_{m=1}^\infty \subset \bigcup_{m=1}^\infty \cP_{Y_m}
  \subset \cP.
  \]
  This completes the poof.
\end{proof}

Let $\cP$ be a concentrated pyramid and
$\{Y_m\}_{m=1}^\infty$ a sequence of mm-spaces
that approximates $\cP$.
By the monotonicity of $\{Y_m\}$ and by Lemma \ref{lem:isom-L}(1),
we have a sequence of isometric embeddings
\[
\Lo(Y_1) \hookrightarrow \Lo(Y_2) \hookrightarrow \dots
\hookrightarrow \Lo(Y_m) \hookrightarrow \dots.
\]
Since $\{\Lo(Y_m)\}$ is $d_{GH}$-precompact,
it Gromov-Hausdorff converges to a compact metric space,
say $\Lo(\cP)$. \index{L1(P)@$\Lo(\cP)$}
Each $\Lo(Y_m)$ is embedded into $\Lo(\cP)$ isometrically.

\begin{lem} \label{lem:cL-prime}
  Let $\cP$ be a concentrated pyramid.
  For any mm-space $Z \in \cP$, there exists an isometric embedding
  $\iota_Z : \Lo(Z) \hookrightarrow \Lo(\cP)$.
\end{lem}

\begin{proof}
  Let $\{Y_m\}_{m=1}^\infty$ be as above.
  Take any mm-space $Z \in \cP$ and fix it.
  There is a sequence of mm-spaces $Z_m \in \cP_{Y_m}$, $m=1,2,\dots$,
  that $\square$-converges to $Z$.
  It follows from Lemma \ref{lem:dGH-dconc} and Proposition \ref{prop:dconc-box}
  that $\Lo(Z_m)$ Gromov-Hausdorff converges to $\Lo(Z)$ as $m\to\infty$.
  By Lemma \ref{lem:isom-L}(1),
  each $\Lo(Z_m)$ is isometrically embedded in $\Lo(Y_m)$.
  Since $\Lo(Y_m)$ Gromov-Hausdorff converges to $\Lo(\cP)$,
  the space $\Lo(Z)$ is isometrically embedded in $\Lo(\cP)$.
  This completes the proof.
\end{proof}

\begin{prop}
  For a concentrated pyramid $\cP$,
  the metric space $\Lo(\cP)$ is {\rm(}up to isometry{\rm)}
  independent of a defining sequence $\{Y_m\}$
  of mm-spaces approximating $\cP$.
\end{prop}

\begin{proof}
  Let $\{Y_m\}$ and $\{Y_m'\}$ be two sequences of mm-spaces
  that approximate $\cP$, and assume that
  $\Lo(Y_m)$ and $\Lo(Y_m')$ Gromov-Hausdorff converge
  to $\Lo(\cP)$ and $\Lo'(\cP)$ respectively as $m\to\infty$.
  By Lemma \ref{lem:cL-prime},
  each $\Lo(Y_m)$ is isometrically embedded into $\Lo'(\cP)$,
  so that $\Lo(\cP)$ is isometrically embedded into $\Lo'(\cP)$.
  In the same way, $\Lo'(\cP)$ is isometrically embedded into $\Lo(\cP)$.
  Therefore, $\Lo(\cP)$ and $\Lo'(\cP)$ are isometric to each other
  (see \cite{BBI}*{1.6.14}).
  This completes the proof.
\end{proof}

\begin{lem} \label{lem:KNcP}
  Let $X$ be an mm-space and $\bar{X} \in \partial\cX$ an ideal mm-space.
  Then, for any natural number $N$, we have
  \[
  d_H(K_N(\Lo(X)),K_N(\Lo(\cP_{\bar{X}})))
  \le 2 d_H(\cM(X;N),\cM(\cP_{\bar{X}};N)).
  \]

  In particular, if $\{X_n\}$ is a sequence of mm-spaces such that
  $\cM(X_n;N)$ Hausdorff converges to $\cM(\cP_{\bar{X}};N)$
  as $n\to\infty$ for any $N$, then
  $\Lo(X_n)$ Gromov-Hausdorff converges to $\Lo(\cP_{\bar{X}})$.
\end{lem}

\begin{proof}
  The proof is similar to that of Lemma \ref{lem:KN-Lip}.
  Assume that $d_H(\cM(X;N),\cM(\cP_{\bar{X}};N)) < \varepsilon$
  for a real number $\varepsilon$ and for a natural number $N$.

  Let us first prove that
  $K_N(\Lo(X)) \subset B_{2\varepsilon}(K_N(\Lo(\cP_{\bar{X}})))$.
  We take any matrix $A = (\dKF([f_i],[f_j]))_{ij} \in K_N(\Lo(X))$,
  where $f_i \in \Lip_1(X)$, $i=1,2,\dots,N$.
  We set $F := (f_1,f_2,\dots,f_N) : X \to \R^N$.
  Since $F_*\mu_X \in \cM(X;N)$,
  there is a measure $\nu \in \cM(\cP_{\bar{X}};N)$
  such that $d_P(F_*\mu_X,\nu) < \varepsilon$.
  For $i=1,2,\dots,N$, we define
  $g_i : \R^N \to \R$ by $g_i(x_1,x_2,\dots,x_N) := x_i$
  for $(x_1,x_2,\dots,x_N) \in \R^N$.
  Since $(g_1,\dots,g_n) = \id_{\R^N}$,
  Lemma \ref{lem:me-diff-di}(2) implies that
  \[
  |\;\dKF([f_i],[f_j]) - d_{KF}^\nu([g_i],[g_j])\;|
  \le 2\,d_P(F_*\mu_X,\nu) < 2\varepsilon
  \]
  for any $i,j = 1,\dots,N$.
  Setting $B := (d_{KF}^\nu([g_i],[g_j]))_{ij}$ we have
  $\|A-B\|_\infty < 2\varepsilon$.
  $[g_i] \in \Lo(Y)$ implies $B \in K_N(\Lo(Y))$.
  Let $Y := (\R^N,\|\cdot\|_\infty,\nu)$.
  Since $Y \in \cP_{\bar{X}}$ and by Lemma \ref{lem:cL-prime},
  we have $B \in K_N(\Lo(Y)) \subset K_N(\Lo(\cP_{\bar{X}}))$
  and therefore $A \in B_{2\varepsilon}(K_N(\Lo(\cP_{\bar{X}})))$.
  We obtain $K_N(\Lo(X)) \subset B_{2\varepsilon}(K_N(\Lo(\cP_{\bar{X}})))$.
  
  Let us next prove that
  $K_N(\Lo(\cP_{\bar{X}})) \subset B_{2\varepsilon}(K_N(\Lo(X)))$.
  Take any matrix $B \in K_N(\Lo(\cP_{\bar{X}}))$.
  Let $\{Y_m\}$ be a sequence of mm-spaces approximating $\cP_{\bar{X}}$.
  Since $\Lo(Y_m)$ Gromov-Hausdorff converges to $\Lo(\cP_{\bar{X}})$
  as $m\to\infty$ and by Lemma \ref{lem:KN-GH},
  $K_N(\Lo(Y_m))$ Hausdorff converges to $K_N(\Lo(\cP_{\bar{X}}))$
  as $m\to\infty$, so that there is a matrix $B_m \in K_N(\Lo(Y_m))$
  such that $\lim_{m\to\infty} \|B_m-B\|_\infty = 0$.
  We find functions $g_{mi} \in \Lip_1(Y_m)$, $i=1,2,\dots,N$,
  with $B_m = (\dKF([g_{mi}],[g_{mj}]))_{ij}$.
  Since $Y_m \in \cP_{\bar{X}}$,
  setting $G_m := (g_{m1},\dots,g_{mN}) : Y_m \to \R^N$
  we have $(G_m)_*\mu_{Y_m} \in \cM(\cP_{\bar{X}};N)$.
  By the assumption, there is a $1$-Lipschitz map
  $F_m : X \to (\R^N,\|\cdot\|_\infty)$ such that
  \[
  d_P((F_m)_*\mu_X,(G_m)_*\mu_{Y_m}) < \varepsilon.
  \]
  Setting $(f_{m1},f_{m2},\dots,f_{mN}) := F_m$, we have,
  by Lemma \ref{lem:me-diff-di},
  \[
  |\;\dKF([f_{mi}],[f_{mj}])-\dKF([g_{mi}],[g_{mj}])\;|
  < 2\varepsilon.
  \]
  Let $A_m := (\dKF([f_{mi}],[f_{mj}]))_{ij}$.
  Since $\|A_m-B_m\|_\infty < 2\varepsilon$ and
  $A_m \in K_N(\Lo(X))$, we have
  $B_m \in B_{2\varepsilon}(K_N(\Lo(X)))$
  and therefore $B \in B_{2\varepsilon}(K_N(\Lo(X)))$.
  We obtain the first part of the lemma.
  The rest follows from Lemma \ref{lem:KN-GH-2}.
  This completes the proof.
\end{proof}

\begin{prop} \label{prop:cL-barX}
  Let $\bar{X} \in \partial\cX$ be an ideal mm-space.
  Then, $\Lo(\bar{X})$ and $\Lo(\cP_{\bar{X}})$ are isometric to each other.
\end{prop}

\begin{proof}
  Let $\bar{X} \in \partial\cX$ be an ideal mm-space.
  We take a sequence of mm-spaces $X_n$, $n=1,2,\dots$,
  that asymptotically concentrates to $\bar{X}$.
  It follows from the definition of $\Lo(\bar{X})$ that
  $\Lo(X_n)$ Gromov-Hausdorff converges to $\Lo(\bar{X})$.
  By  `(1) $\implies$ (2)' of Theorem \ref{thm:cri-asymp-conc},
  the measurement $\cM(X_n;N)$ Hausdorff converges to $\cM(\cP_{\bar{X}};N)$
  as $n\to\infty$ for any $N$.
  By Lemma \ref{lem:KNcP}, 
  $\Lo(X_n)$ Gromov-Hausdorff converges to $\Lo(\cP_{\bar{X}})$
  as $n\to\infty$.
  This completes the proof.
\end{proof}

Since the following lemma is proved in the same way as in Lemma \ref{lem:L-Ln},
we omit the proof.

\begin{lem} \label{lem:L-dense}
  Let $\cL$, $\cL_\varepsilon$, and $\cL_\varepsilon'$ be compact metric spaces
  and $q_\varepsilon : \cL_\varepsilon \to \cL_\varepsilon'$ be
  an $\varepsilon$-isometric map for every positive number $\varepsilon$.
  If $d_{GH}(\cL_\varepsilon,\cL) \le \varepsilon$ and if
  $d_{GH}(\cL_\varepsilon',\cL) \le \varepsilon$,
  then $q_\varepsilon(\cL_\varepsilon)$
  is $\varepsilon'$-dense in $\cL_\varepsilon'$,
  where $\varepsilon' = \varepsilon'(\varepsilon)$ is a function
  with $\lim_{\varepsilon\to 0} \varepsilon' = 0$.
\end{lem}

\begin{lem} \label{lem:cri-asymp-conc}
  Let $X_n$ and $Y_m$, $m,n=1,2,\dots$, be mm-spaces
  such that $\{Y_m\}$ approximates a concentrated pyramid $\cP$.
  Assume that
  \begin{enumerate}
  \item $\cM(X_n;N)$ Hausdorff converges to $\cM(\cP;N)$ as $n\to\infty$ for any natural number $N$,
  \item $\Lo(X_n)$ Gromov-Hausdorff converges to $\Lo(\cP)$
    as $n\to\infty$.
  \end{enumerate}
  Then, as $m,n\to\infty$,
  both $X_n$ and $Y_m$ asymptotically concentrate to
  some common ideal mm-space $\bar{X} \in \partial\cX$
  with $\cP = \cP_{\bar{X}}$.
\end{lem}

\begin{proof}
  Let $\varepsilon > 0$ be any number.
  There is a number $m_0 = m_0(\varepsilon)$ such that
  $d_{GH}(\Lo(Y_m),\Lo(\cP)) < \varepsilon$ for any $m \ge m_0$.
  Let $m$ be any number with $m \ge m_0$.
  Let $N = N(Y_m,\varepsilon)$ be as in Lemma \ref{lem:M-p}.
  By the assumption, there is a number $n_0 = n_0(m,\varepsilon)$
  such that $d_{GH}(\Lo(X_n),\Lo(\cP)) < \varepsilon$ and
  $d_H(\cM(X_n;N),\cM(\cP;N)) < \varepsilon$ for any $n \ge n_0$.
  We therefore have
  \[
  \cM(Y_m;N) \subset \cM(\cP;N) \subset B_\varepsilon(\cM(X_n;N)).
  \]
  By Lemma \ref{lem:M-p},
  there is a Borel map $p_{mn} : X_n \to Y_m$ that is $1$-Lipschitz
  up to $5\varepsilon$ and satisfies
  $d_P((p_{mn})_*\mu_{X_n},\mu_{Y_m}) \le 15\varepsilon$.
  Using Lemma \ref{lem:1-Lip-up-to-Lip1}(2) we have
  $p_{mn}^*\Lip_1(Y_m) \subset B_{5\varepsilon}(\Lip_1(X_n))$
  and hence $p_{mn}^*\Lo(Y_m) \subset B_{5\varepsilon}(\Lo(X_n))$.
  Lemma \ref{lem:pb-me-di}(2) tells us that
  $p_{mn}^* : \Lo(Y_m) \to B_{5\varepsilon}(\Lo(X_n))$
  is $30\varepsilon$-isometric.
  Let $\pi : B_{5\varepsilon}(\Lo(X_n)) \to \Lo(X_n)$ be a nearest point projection.
  Since this is $5\varepsilon$-isometric,
  $\pi \circ p_{mn}^* : \Lo(Y_m) \to \Lo(X_n)$ is $35\varepsilon$-isometric.
  We apply Lemma \ref{lem:L-dense}
  for $\cL := \Lo(\cP)$, $\cL_\varepsilon := \Lo(Y_m)$,
  $\cL_\varepsilon' := \Lo(X_n)$, and $q_\varepsilon := \pi\circ p_{mn}^*$.
  Then, $\pi \circ p_{mn}^*(\Lo(Y_m))$ is $\varepsilon'$-dense in $\Lo(X_n)$,
  where $\varepsilon' \to 0$ as $\varepsilon \to 0$.
  This implies that
  \[
  \lim_{m\to\infty}\limsup_{n\to\infty} d_H(p_{mn}^*\Lo(Y_m),\Lo(X_n)) = 0,
  \]
  which together with Lemma \ref{lem:enforce-conc} proves
  \[
  \lim_{m\to\infty}\limsup_{n\to\infty} \dconc(X_n,Y_m) = 0.
  \]
  The two sequences $\{X_n\}$ and $\{Y_m\}$ are both $\dconc$-Cauchy
  and asymptotically concentrate to some common
  ideal mm-space $\bar{X} \in \partial\cX$.
  The pyramid $\cP_{Y_m}$ converges weakly to both $\cP$ and $\cP_{\bar{X}}$,
  so that we have $\cP = \cP_{\bar{X}}$.
  This completes the proof.
\end{proof}

\begin{cor} \label{cor:inj-pyramid}
  If $\bar{X},\bar{X}' \in \partial\cX$ are two ideal mm-spaces
  with $\cP_{\bar{X}} = \cP_{\bar{X}'}$, then we have $\bar{X} = \bar{X}'$.
  Namely, the map $\iota : \cX \to \Pi$ is injective.
\end{cor}

\begin{proof}
  Assume that $\cP_{\bar{X}} = \cP_{\bar{X}'}$.
  There are mm-spaces $X_n$ and $X_n'$, $n=1,2,\dots$,
  such that $\{X_n\}$ and $\{X_n'\}$ asymptotically concentrate to
  $\bar{X}$ and $\bar{X}'$, respectively.
  $\Lo(X_n)$ and $\Lo(X_n')$ Gromov-Hausdorff converges to
  $\Lo(\bar{X})$ and $\Lo(\bar{X}')$, respectively.
  Moreover, $\Lo(\bar{X})$ and $\Lo(\bar{X}')$ are isometric to
  $\Lo(\cP_{\bar{X}})$ and $\Lo(\cP_{\bar{X}'})$ respectively
  by Proposition \ref{prop:cL-barX}.
  By `(1)$\implies$(2)' of Theorem \ref{thm:cri-asymp-conc},
  $\cM(X_n;N)$ and $\cM(X_n';N)$ both Hausdorff converge
  to $\cM(\cP_{\bar{X}};N) = \cM(\cP_{\bar{X}'};N)$ as $n\to\infty$
  for any $N$.
  Take a sequence of mm-spaces $Y_n$, $n=1,2,\dots$,
  that approximates the pyramid $\cP_{\bar{X}} = \cP_{\bar{X}'}$.
  It then follows from Lemma \ref{lem:cri-asymp-conc}
  that $X_n$, $X_n'$, and $Y_m$ all asymptotically concentrate
  to a common ideal mm-space as $m,n\to\infty$,
  so that we obtain $\bar{X} = \bar{X}'$.
  This completes the proof.
\end{proof}

\begin{proof}[Proof of `{\rm(2)} $\implies$ {\rm(1)}'
  of Theorem \ref{thm:cri-asymp-conc}]
  We assume (2).
  It then follows from Lemma \ref{lem:KNcP}
  that $\Lo(X_n)$ Gromov-Hausdorff converges to $\Lo(\cP_{\bar{X}})$.
  Taking a sequence of mm-spaces $Y_m$, $m=1,2,\dots$,
  approximating $\cP_{\bar{X}}$,
  we apply Lemma \ref{lem:cri-asymp-conc} for $\cP := \cP_{\bar{X}}$
  to prove that 
  $X_n$ and $Y_m$ both asymptotically concentrate to a common (ideal) mm-space
  $\bar{X}' \in \bar{\cX}$ with $\cP_{\bar{X}} = \cP_{\bar{X}'}$.
  By Corollary \ref{cor:inj-pyramid}, we have $\bar{X} = \bar{X}'$.
  This completes the proof of Theorem \ref{thm:cri-asymp-conc}.
\end{proof}

The following is clear.

\begin{cor} \label{cor:Ym}
  If a sequence of mm-spaces approximates $\cP_{\bar{X}}$
  for an ideal mm-space $\bar{X}$, then
  it asymptotically concentrates to $\bar{X}$.
\end{cor}

\begin{lem} \label{lem:conv-pyramid-MNR}
  Let $\cP$ and $\cP_n$, $n=1,2,\dots$, be pyramids,
  $N$ a natural number, and $R > 0$ a real number.
  If $\cP_n\cap\cX(N,R)$ Hausdorff converges to $\cP\cap\cX(N,R)$
  as $n\to\infty$ with respect to the box metric $\square$ on $\cX$, then
  $\cM(\cP_n;N,R)$ Hausdorff converges to $\cM(\cP;N,R)$ as $n\to\infty$.
\end{lem}

\begin{proof}
  Assume that $\cP_n\cap\cX(N,R)$ Hausdorff converges to $\cP\cap\cX(N,R)$
  as $n\to\infty$.
  By the compactness of $B^N_R$ and by Lemma \ref{lem:dH}(2),
  there is a sequence $n_i \to \infty$ such that
  $\cM(\cP_{n_i};N,R)$ Hausdorff converges to some compact subset
  $\cA \subset \cM(N,R)$ as $i\to\infty$.
  Let us prove $\cA = \cM(\cP;N,R)$.
  By Proposition \ref{prop:box-di},
  \begin{align*}
    &d_H(\cP_{n_i}\cap\cX(N,R),\{(\R^N,\|\cdot\|_\infty,\mu)\mid\mu\in\cA\})\\
    &\le 2\, d_H(\cM(\cP_{n_i};N,R),\cA) \to 0 \quad\text{as $n\to\infty$},
  \end{align*}
  which together with the assumption implies that
  \[
  \{(\R^N,\|\cdot\|_\infty,\mu)\mid\mu\in\cA\} = \cP\cap\cX(N,R).
  \]
  Since Lemma \ref{lem:perfect-lim} implies the perfectness of $\cA$
  on $B^N_R$, we obtain $\cA = \cM(\cP;N,R)$.
  Since this holds for any convergent subsequence $\{\cM(\cP_{n_i};N,R)\}$
  of $\{\cM(\cP_n;N,R)\}$,
  the measurement $\cM(\cP_n;N,R)$ Hausdorff converges to $\cM(\cP;N,R)$
  as $n\to\infty$.  This completes the proof.
\end{proof}

\begin{cor} \label{cor:asymp-conc-pyramid}
  Let $\{X_n\}$ be a sequence of mm-spaces
  and $\bar{X} \in \partial\cX$ an ideal mm-space.
  Then, the following {\rm(1)} and {\rm(2)} are equivalent to each other.
  \begin{enumerate}
  \item $X_n$ asymptotically concentrates to $\bar{X}$ as $n\to\infty$.
  \item $\cP_{X_n}$ converges weakly to $\cP_{\bar{X}}$ as $n\to\infty$.
  \end{enumerate}
\end{cor}

\begin{proof}
  `(1) $\implies$ (2)' is clear by the definition of $\cP_{\bar{X}}$.

  We prove `(2) $\implies$ (1)'.
  If $\cP_{X_n}$ converges weakly to $\cP_{\bar{X}}$ as $n\to\infty$,
  then, by Lemmas \ref{lem:pyramid-conv-dH} and \ref{lem:conv-pyramid-MNR},
  $\cM(X_n;N,R)$ Hausdorff converges to $\cM(\cP_{\bar{X}};N,R)$ as $n\to\infty$
  for any $N$ and $R$.
  By Theorem \ref{thm:cri-asymp-conc} we obtain (1).
\end{proof}

\begin{prop} \label{prop:concentrated}
  For a given pyramid $\cP$,
  the following {\rm(1)} and {\rm(2)} are equivalent to each other.
  \begin{enumerate}
  \item $\cP$ is concentrated.
  \item There exists an {\rm(}ideal{\rm)} mm-space $\bar{X} \in \bar{\cX}$ such that
    $\cP = \cP_{\bar{X}}$.
  \end{enumerate}
\end{prop}

\begin{proof}
  `(2) $\implies$ (1)' follows from Corollary \ref{cor:concentrated}.

  We prove `(1) $\implies$ (2)'.
  Let $\{X_n\}$ be a sequence of mm-spaces that approximates $\cP$.
  $\Lo(X_n)$ Gromov-Hausdorff converges to $\Lo(\cP)$.
  Since $\cP_{X_n}$ converges weakly to $\cP$ as $n\to\infty$,
  and by Lemmas \ref{lem:pyramid-conv-dH} and \ref{lem:conv-pyramid-MNR},
  $\cM(X_n;N,R)$ Hausdorff converges to $\cM(\cP;N,R)$ as $n\to\infty$
  for any $N$ and $R$.
  We apply Lemma \ref{lem:MR}, where the assumption of the lemma
  follows from
  Corollary \ref{cor:sup-diam-cP}.
  Thus,
  $\cM(X_n;N)$ Hausdorff converges to $\cM(\cP;N)$ as $n\to\infty$
  for any $N$.
  Lemma \ref{lem:cri-asymp-conc} proves (2).
\end{proof}

Proposition \ref{prop:conc-pyramid} and Corollary \ref{cor:asymp-conc-pyramid}
extend to the following

\begin{thm} \label{thm:conc-pyramid}
  Let $\{\bar{X}_n\}_{n=1}^\infty \subset \bar{\cX}$ and
  $\bar{X} \in \bar\cX$.
  Then, the following {\rm(1)} and {\rm(2)} are equivalent to each other.
  \begin{enumerate}
  \item $\bar{X}_n$ $\dconc$-converges to $\bar{X}$ as $n\to\infty$.
  \item $\cP_{\bar{X}_n}$ converges weakly to $\cP_{\bar{X}}$ as $n\to\infty$.
  \end{enumerate}
\end{thm}

\begin{proof}
  We prove `(1)$\implies$(2)'.
  Suppose that 
  $\bar{X}_n$ $\dconc$-converges to $\bar{X}$
  and $\cP_{\bar{X}_n}$ does not converge weakly to $\cP_{\bar{X}}$ as $n\to\infty$.
  By replacing $\{\bar{X}_n\}$ with a subsequence,
  we assume that $\cP_{\bar{X}_n}$ converges weakly to a pyramid $\cP$
  with $\cP \neq \cP_{\bar{X}}$ as $n\to\infty$.
  For each $n$ there is a sequence of mm-spaces
  $X_{n,m} \in \cP_{\bar{X}_n}$, $m=1,2,\dots$, that $\dconc$-converges
  to $\bar{X}_n$.
  Corollary \ref{cor:asymp-conc-pyramid} implies that
  $\cP_{X_{n,m}}$ converges weakly to $\cP_{\bar{X}_n}$ as $m\to\infty$.
  There is a sequence of numbers $m(n)$, $n=1,2,\dots$,
  such that, as $n\to\infty$, $\cP_{X_{n,m(n)}}$ converges weakly to $\cP$
  and $X_{n,m(n)}$ $\dconc$-converges to $\bar{X}$.
  By Corollary \ref{cor:asymp-conc-pyramid}
  and Proposition \ref{prop:conc-pyramid}, we have $\cP = \cP_{\bar{X}}$,
  which is a contradiction.

  We prove `(2)$\implies$(1)'.
  Suppose that 
  $\cP_{\bar{X}_n}$ converges weakly to $\cP_{\bar{X}}$
  and $\bar{X}_n$ does not $\dconc$-converges to $\bar{X}$ as $n\to\infty$.
  Replacing $\{\bar{X}_n\}$ with a subsequence,
  we assume that $\dconc(\bar{X}_n,\bar{X})$, $n=1,2,\dots$, are
  bounded away from zero.
  There is a natural number $n(k)$ for each $k$
  such that $\rho(\cP_{\bar{X}_{n(k)}},\cP_{\bar{X}}) < 1/k$.
  By Corollary \ref{cor:Ym},
  there is an mm-space $Y_k \in \cP_{\bar{X}_{n(k)}}$
  for each $k$ such that
  $\rho(\cP_{\bar{X}_{n(k)}},\cP_{Y_k}) \le \dconc(\bar{X}_{n(k)},Y_k) < 1/k$.
  A triangle inequality yields that
  $\rho(\cP_{Y_k},\cP_{\bar{X}}) < 2/k$ for any $k$
  and $\cP_{Y_k}$ converges weakly to $\cP_{\bar{X}}$ as $k\to\infty$.
  By Corollary \ref{cor:asymp-conc-pyramid},
  as $k\to\infty$, $Y_k$ $\dconc$-converges to $\bar{X}$ 
  and therefore $\bar{X}_{n(k)}$ $\dconc$-converges to $\bar{X}$,
  which is a contradiction.
  This completes the proof.
\end{proof}

We finally obtain the main theorem in this book.

\begin{thm}[Embedding theorem] \label{thm:emb-pyramid}
  \index{embedding theorem}
  The map
  \[
  \iota : \bar{\cX} \ni \bar{X} \longmapsto \cP_{\bar{X}} \in \Pi
  \]
  is a topological embedding map.
  The space $\Pi$ of pyramids is a compactification of $\cX$ and $\bar{\cX}$.
\end{thm}

\begin{proof}
  The first part follows from Corollary \ref{cor:inj-pyramid}
  and Theorem \ref{thm:conc-pyramid}.
  It follows from Lemma \ref{lem:approx-pyramid}
  that $\iota(\cX)$ is dense in $\Pi$.
  This completes the proof.
\end{proof}

\begin{rem}
  Since $\cX$ itself is a non-concentrated pyramid,
  it does not belongs to $\iota(\bar{\cX})$.
  In particular, $\iota(\bar{\cX})$ is a proper subset of $\Pi$.
  As is seen in Example \ref{ex:prod-sph},
  $\cX$ is a proper subset of $\bar{\cX}$.
  We have
  \[
  \iota(\cX) \subsetneq \iota(\bar{\cX}) \subsetneq \Pi.
  \]
  and $\iota(\bar{\cX})$ is not a closed subset of $\Pi$.
\end{rem}

\begin{prop}
  The map $\Lo : \bar{\cX} \to \mathcal{H}$ is proper
  with respect to $\dconc$ and $d_{GH}$.
\end{prop}

\begin{proof}
  We take a sequence of (ideal) mm-spaces $\bar{X}_n \in \bar\cX$,
  $n=1,2,\dots$,
  such that $\Lo(\bar{X}_n)$ Gromov-Hausdorff converges to
  a compact metric space $\cL$ as $n\to\infty$.
  It suffices to prove that $\{\bar{X}_n\}$ has a subsequence
  that is $\dconc$-convergent in $\bar{\cX}$.
  Theorem \ref{thm:cpt-pyramid} implies that,
  by replacing $\{\bar{X}_n\}$ with a subsequence,
  the pyramid $\cP_{\bar{X}_n}$ converges weakly to a pyramid $\cP$
  as $n\to\infty$.
  By Proposition \ref{prop:concentrated} and
  Theorem \ref{thm:conc-pyramid},
  it suffices to prove that $\cP$ is concentrated.
  Take any sequence $\{Y_i\}_{i=1}^\infty \subset \cP$ and fix it.
  For each $i$, there are a number $n(i)$ and an mm-space
  $X_i' \in \cP_{\bar{X}_{n(i)}}$ such that $\square(Y_i,X_i') \le 1/i$.
  Lemma \ref{lem:isom-L} implies that
  $\Lo(X_i')$ is isometrically embedded into $\Lo(\bar{X}_{n(i)})$.
  Since $\Lo(\bar{X}_{n(i)})$ Gromov-Hausdorff converges to $\cL$ as $i\to\infty$,
  $\{\Lo(X_i')\}$ has a subsequence converging to a compact subset
  $\cL' \subset \cL$.  We replace $\{\Lo(X_i')\}$ with such a subsequence.
  We have
  $d_{GH}(\Lo(Y_i),\Lo(X_i')) \le \dconc(Y_i,X_i')
  \le \square(Y_i,X_i') \le 1/i$ and hence
  $\Lo(Y_i)$ Gromov-Hausdorff converges to $\cL'$ as $i\to\infty$.
  Thus, $\cP$ is concentrated.
  This completes the proof.
\end{proof}

\section{Infinite product, II}

In this section, we see asymptotically concentrating
product spaces.

\begin{defn}[L\'evy radius]
  \index{Levy radius@L\'evy radius}
  \index{LeRad@$\LeRad(X;-\kappa)$}
  Let $X$ be an mm-space and $\kappa > 0$ a real number.
  The \emph{L\'evy radius of $X$} is defined to be
  \begin{align*}
    \LeRad(X;-\kappa)
    := \inf\{\; & \rho > 0 \mid \mu_X(|f-m_f|> \rho) \le \kappa\\
    &\text{for any $1$-Lipschitz function $f : X \to \R$}\;\},
  \end{align*}
  where $m_f$ is the L\'evy mean of $f$.
\end{defn}

\begin{lem} \label{lem:LeRad-ObsDiam}
  Let $X$ be an mm-space.
  For any $\kappa$ with $0 < \kappa < 1/2$, we have
  \[
  \LeRad(X;-\kappa) \le \ObsDiam(X;-\kappa).
  \]
\end{lem}

\begin{proof}
  Assume that $\ObsDiam(X;-\kappa) < \varepsilon$ for a real number
  $\varepsilon$.
  For any $1$-Lipschitz function $f : X \to \R$,
  there is a closed interval $A \subset \R$ such that
  $f_*\mu_X(A) \ge 1-\kappa$ and $\diam A < \varepsilon$.
  By $\kappa < 1/2$, we have $f_*\mu_X(A) > 1/2$,
  so that any median of $f$ belongs to $A$ and
  so does $m_f$.
  Since $A \subset \{\;y\in\R \mid |y-m_f| \le \varepsilon\;\}$,
  we have
  \[
  1-\kappa \le f_*\mu_X(A) \le \mu_X(|f-m_f| \le \varepsilon)
  \]
  and so
  \[
  \mu_X(|f-m_f| > \varepsilon) \le \kappa,
  \]
  which implies $\LeRad(X;-\kappa) \le \varepsilon$.
  This completes the proof.
\end{proof}

\begin{prop} \label{prop:prod-dconc}
  Let $Y$ and $Z$ be two mm-spaces and let $1 \le p \le +\infty$.
  We equip the product space $X := Y \times Z$ with the $d_{l_p}$ metric.
  If $\ObsDiam(Z) < 1/2$, then we have
  \[
  \dconc(X,Y) \le \ObsDiam(Z).
  \]
\end{prop}

\begin{proof}
  Assume that $\ObsDiam(Z) < \varepsilon < 1/2$.
  Letting $p : X \to Y$ be the projection,
  we shall prove that
  $p$ enforces $\varepsilon$-concentration of $X$ to $Y$.
  Since $p$ is $1$-Lipschitz continuous,
  we have $p^*\Lip_1(Y) \subset \Lip_1(X)$.
  It suffices to prove that $\Lip_1(X) \subset B_\varepsilon(p^*\Lip_1(Y))$.
  Take any function $f \in \Lip_1(X)$.
  Since $f : Y \times Z \to \R$ is $1$-Lipschitz with respect to $d_{l_p}$,
  for any point $y \in Y$, the function
  $f(y,\cdot) : Z \to \R$ is $1$-Lipschitz continuous.
  Denote by $\underline{m}(y)$ the minimum of medians of $f(y,\cdot)$.
  Let us prove that $\underline{m}(y)$ is $1$-Lipschitz continuous
  in $y \in Y$.
  In fact, for any two points $y_1,y_2 \in Y$,
  we have $f(\cdot,y_2) - d_Y(y_1,y_2) \le f(\cdot,y_1)$ and hence
  \[
  \mu_X\{f(\cdot,y_2) \le \underline{m}(y_1) + d_Y(y_1,y_2)\}
  \ge \mu_X\{f(\cdot,y_1) \le \underline{m}(y_1)\} \ge \frac{1}{2},
  \]
  which together with the minimality of $\underline{m}(y_2)$
  among all medians of $f(\cdot,y_2)$
  proves that $\underline{m}(y_1) + d_Y(y_1,y_2) \ge \underline{m}(y_2)$.
  Exchanging $y_1$ and $y_2$ we have
  $\underline{m}(y_2) + d_Y(y_1,y_2) \ge \underline{m}(y_1)$.
  Therefore, $\underline{m}$ is $1$-Lipschitz continuous.
  In the same way, we see that the maximum of medians of $f(y,\cdot)$,
  say $\overline{m}(y)$, is also $1$-Lipschitz continuous in $y \in Y$.
  Define a map $\bar{f} : X = Y \times Z \to \R$ by
  \[
  \bar{f}(y,z) := \frac{\overline{m}(y)-\underline{m}(y)}{2}
  \]
  for $(y,z) \in Y \times Z$.
  $\bar{f}$ is $1$-Lipschitz with respect to $d_{l_p}$ and belongs to
  $p^*\Lip_1(Y)$.
  Since $\ObsDiam(Z) < \varepsilon < 1/2$ implies
  $\ObsDiam(Z;-\varepsilon) < \varepsilon < 1/2$,
  we have $\LeRad(Z;-\varepsilon) < \varepsilon$ 
  by Lemma \ref{lem:LeRad-ObsDiam} and therefore
  \[
  \mu_Z(|f(y,\cdot)-\bar{f}(y,\cdot)| > \varepsilon)
  \le \LeRad(Z;-\varepsilon) < \varepsilon.
  \]
  By Fubini's theorem,
  \begin{align*}
    \mu_X(|f-\bar{f}| > \varepsilon)
    &= \int_Y \mu_Z(|f(y,\cdot)-\bar{f}(y,\cdot)| > \varepsilon)\;d\mu_Y(y)
    \le \varepsilon,
  \end{align*}
  i.e., $\dKF(f,\bar{f}) \le \varepsilon$.
  We obtain $\Lip_1(X) \subset B_\varepsilon(p^*\Lip_1(Y))$ and
  $p$ enforces $\varepsilon$-concentration of $X$ to $Y$.
  Since $p_*\mu_X = \mu_Y$, Lemma \ref{lem:enforce-conc} implies
  $\dconc(X,Y) \le \varepsilon$.
  This completes the proof.
\end{proof}

\begin{prop} \label{prop:prod-asympconc}
  Let $F_n$, $n=1,2,\dots$, be mm-spaces and let $1 \le p \le +\infty$.
  We define
  \[
  X_n := F_1 \times F_2 \times \dots \times F_n,
  \qquad\Phi_{ij} := F_i \times F_{i+1} \times \dots \times F_j.
  \]
  and equip $X_n$ and $\Phi_{ij}$ with the product measure and
  the $d_{l_p}$ metrics.
  If $\ObsDiam(\Phi_{ij}) \to 0$ as $i,j \to +\infty$,
  then $\{X_n\}$ asymptotically concentrates.
\end{prop}

\begin{proof}
  Let $i < j$.
  Since $X_j = X_i \times \Phi_{i+1,j}$,
  Proposition \ref{prop:prod-dconc} proves that
  $\dconc(X_i,X_j) \le \ObsDiam(\Phi_{ij})$
  for any $i$ and $j$ large enough.
  This completes the proof.
\end{proof}

Let $X$ and $Y$ be two compact Riemannian manifolds
and $X \times Y$ the Riemannian product of $X$ and $Y$.
Denote by $\sigma(\Delta_X)$ the spectrum of the Laplacian $\Delta_X$ on $X$.
It is well-known that
\[
\sigma(\Delta_{X\times Y}) = \{\;\lambda+\mu \mid
\lambda \in \sigma(\Delta_X),\;\mu \in \sigma(\Delta_Y)\;\},
\]
and, for any $i \ge 0$,
\[
m_i(X \times Y) = \sum_{j,k\, :\,\lambda_i(X \times Y) = \lambda_j(X) + \lambda_k(Y)}
(m_j(X) + m_k(Y) - 1),
\]
where $m_i(X)$ denotes the multiplicity of $\lambda_i(X)$, i.e.,
the number of $j$'s with $\lambda_i(X) = \lambda_j(X)$.
In particular we have
\[
\lambda_1(X \times Y) = \min\{\lambda_1(X),\lambda_1(Y)\}.
\]

\begin{cor} \label{cor:prod-asympconc}
  Let $F_n$, $n=1,2,\dots$, be compact Riemannian manifolds
  such that $\lambda_1(F_n)$ is divergent to infinity
  as $n\to\infty$.
  Let $X_n := F_1 \times F_2 \times \dots \times F_n$
  be the Riemannian product space.
  Then, $\{X_n\}$ asymptotically concentrates.
\end{cor}

\begin{proof}
  Let $i < j$.
  As $i,j \to \infty$,
  $\lambda_1(\Phi_{ij}) = \min_{k=i}^j \lambda_1(F_k)$
  is divergent to infinity.
  Corollary \ref{cor:ObsDiam-spec} implies that
  $\ObsDiam(\Phi_{ij}) \to 0$ as $i,j \to \infty$.
  The corollary follows from Proposition \ref{prop:prod-asympconc}.
\end{proof}

\begin{ex} \label{ex:prod-sph}
  Let
  \[
  X_n := S^1 \times S^2 \times \dots \times S^n
  \]
  be the Riemannian product of unit spheres in Euclidean spaces.
  Since $\lambda_1(S^n) = n$, Corollary \ref{cor:prod-asympconc}
  proves that $\{X_n\}$ asymptotically concentrates.
  The infinite product space
  \[
  X_\infty := \prod_{n=1}^\infty S^n
  \]
  is not really an mm-space, but is considered to be an ideal mm-space.
\end{ex}

\begin{prop} \label{prop:prod-non-conc}
  Let $F$ be an mm-space and let $1 \le p \le +\infty$.
  Then, $\{(F^n,d_{l_p},\mu_F^{\otimes n})\}_{n=1}^\infty$ does not asymptotically
  concentrate unless $F$ consists of a single point.
\end{prop}

Note that $\{(F^n,d_{l_p},\mu_F^{\otimes n})\}_{n=1}^\infty$ is asymptotic.

\begin{proof}
  Assume that $F$ contains at least two different points.
  We then find a non-constant $1$-Lipschitz function $\varphi : F \to \R$.
  Let $\varphi_{n,i} : F^n \to \R$, $i=1,2,\dots,n$, be the functions defined by
  \[
  \varphi_{n,i}(x_1,x_2,\dots,x_n) := \varphi(x_i),
  \qquad (x_1,x_2,\dots,x_n) \in F^n.
  \]
  Each $\varphi_{n,i}$ is $1$-Lipschitz continuous with respect to
  $d_{l_p}$.
  For any different $i$ and $j$,
  \begin{align*}
    d_{KF}^{\mu_F^{\otimes n}}([\varphi_{n,i}],[\varphi_{n,j}])
    = d_{KF}^{\mu_F^{\otimes 2}}([\varphi_{2,1}],[\varphi_{2,2}])
    = d_{KF}^{\mu_F^{\otimes 2}}(\varphi_{2,1},\varphi_{2,2}+c)
  \end{align*}
  for some real number $c$ (see Lemma \ref{lem:me-const}).
  If $d_{KF}^{\mu_F^{\otimes 2}}(\varphi_{2,1},\varphi_{2,2}+c) = 0$ were to hold,
  then $\varphi_{2,1} = \varphi_{2,2} + c$ $\mu_F^{\otimes 2}$-almost everywhere
  and hence $\varphi(x_1) = \varphi(x_2) + c$ for all $x_1,x_2 \in F$,
  which is a contradiction.
  Thus, $d_{KF}^{\mu_F^{\otimes n}}([\varphi_{n,i}],[\varphi_{n,j}])$ is a positive
  constant, say $\varepsilon_0$,
  independent of $n$, $i$, and $j$ with $i \neq j$.
  This implies that $\Cap_{\varepsilon_0/2}(\Lo(F^n,d_{l_p},\mu_F^{\otimes n})) \ge n$
  and $\{\Lo(F^n,d_{l_p},\mu_F^{\otimes n})\}_{n=1}^\infty$
  is not $d_{GH}$-precompact by Lemma \ref{lem:dGH-precpt}.
  By Lemma \ref{lem:dGH-dconc},
  $\{(F^n,d_{l_p},\mu_F^{\otimes n})\}_{n=1}^\infty$ does not asymptotically
  concentrate.
  The proof is completed.
\end{proof}

\section{Spheres and Gaussians}
\label{sec:Spheres-Gaussians}

\begin{defn}[Gaussian space]
  \index{Gaussian space} \index{Gamma-n@$\Gamma^n$}
  Let $\Gamma^n :=  (\R^n,\|\cdot\|_2,\gamma^n)$.
  We call the mm-space $\Gamma^n$
  the \emph{$n$-dimensional} (\emph{standard}) \emph{Gaussian
  space}.
\end{defn}

Recall that the $n$-dimensional Gaussian measure $\gamma^n$
coincides with the $n^{th}$ power of the one-dimensional
Gaussian measure $\gamma^1$, so that
we have the monotonicity of the Gaussian spaces
\[
\Gamma^1 \prec \Gamma^2 \prec \cdots \prec \Gamma^n \prec \cdots.
\]
Therefore, the sequence $\{\Gamma^n\}_{n=1}^\infty$ is asymptotic
and converges weakly to the $\square$-closure of the union of
$\cP_{\Gamma^n}$, $n=1,2,\dots$, say $\cP_{\Gamma^\infty}$.
Note that $\{\Gamma^n\}_{n=1}^\infty$ does not asymptotically
concentrate by Proposition \ref{prop:prod-non-conc}.
Although the infinite-dimensional Gaussian space $\Gamma^\infty
:= (\R^\infty,\|\cdot\|_2,\gamma^\infty)$ is not an mm-space
(the infinite-dimensional Gaussian measure $\gamma^\infty$
is not a Borel measure with respect to $\|\cdot\|_2$;
cf.~\cite{Bog:Gm}*{\S 2.3}),
we consider the pyramid $\cP_{\Gamma^\infty}$
as a substitute for $\Gamma^\infty$.

\begin{defn}[Virtual infinite-dimensional Gaussian space]
  \index{virtual infinite-dimensional Gaussian space}
  \index{P-Gamma-infty@$\cP_{\Gamma^\infty}$}
  We call $\cP_{\Gamma^\infty}$ the \emph{virtual infinite-dimensional}
  (\emph{standard}) \emph{Gaussian space}.
\end{defn}

In this section we prove

\begin{thm}
  \label{thm:sphere-Gaussian}
  The pyramid $\cP_{S^n(\sqrt{n})}$ associated with $S^n(\sqrt{n})$
  converges weakly to the virtual infinite-dimensional Gaussian space
  $\cP_{\Gamma^\infty}$ as $n\to\infty$,
  where $S^n(\sqrt{n})$ is equipped with the Euclidean distance function.
\end{thm}

We need a lemma.

\begin{lem} \label{lem:sphere-Gaussian}
  For any real number $\theta$ with $0 < \theta < 1$, we have
  \[
  \lim_{n\to\infty} \gamma^{n+1}\{\;x\in\R^{n+1}
  \mid \|x\|_2 \le \theta\sqrt{n}\;\} = 0.
  \]
\end{lem}

\begin{proof}
  Considering the poler coordinates on $\R^n$,
  we see that
  \[
  \gamma^{n+1}\{\;x\in\R^{n+1} \mid \|x\|_2 \le r\;\}
  = \frac{\int_0^r t^ne^{-t^2/2}\,dt}{\int_0^\infty t^ne^{-t^2/2}\,dt}.
  \]
  Integrating the both sides of $(\log(t^n e^{-t^2/2}))'' = -n/t^2-1 \le -1$
  over $[\,t,\sqrt{n}\,]$ with $0 < t \le \sqrt{n}$ yields
  \[
  -(\log(t^n e^{-t^2/2}))'
  = (\log(t^n e^{-t^2/2}))'|_{t=\sqrt{n}} - (\log(t^n e^{-t^2/2}))'
  \le t-\sqrt{n}.
  \]
  Integrating this again over $[\,t,\sqrt{n}\,]$ implies
  \[
  \log(t^ne^{-t^2/2}) - \log(n^{n/2}e^{-n/2}) \le -\frac{(t-\sqrt{n})^2}{2},
  \]
  so that $t^ne^{-t^2/2} \le n^{n/2}e^{-n/2}e^{-(t-\sqrt{n})^2/2}$ and then,
  for any $r$ with $0 \le r \le \sqrt{n}$,
  \[
  \int_0^{\sqrt{n}-r} t^ne^{-t^2/2}\,dt
  \le n^{n/2}e^{-n/2} \int_r^{\sqrt{n}} e^{-t^2/2}\,dt
  \le n^{n/2}e^{-n/2}e^{-r^2/2}.
  \]
  By setting
  \[
  I_n := \int_0^\infty t^ne^{-t^2/2}\,dt,
  \]
  Stirling's approximation implies
  \[
  I_n = 2^{\frac{n-1}{2}} \int_0^\infty s^{\frac{n-1}{2}} e^{-s}\,ds
  \sim \sqrt{\pi} (n-1)^{\frac{n}{2}} e^{-\frac{n-1}{2}}.
  \]
  Therefore,
  \[
  \gamma^{n+1}\{\;x\in\R^{n+1} \mid \|x\|_2 \le \theta\sqrt{n}\;\}
  \le \frac{n^{n/2}e^{-n/2}e^{-(1-\theta)^2n/2}}{I_n} \to 0
  \quad n\to\infty.
  \]
  This completes the proof.
\end{proof}

\begin{proof}[Proof of Theorem \ref{thm:sphere-Gaussian}]
  Suppose that $\cP_{S^n(\sqrt{n})}$ does not converge weakly to
  $\cP_{\Gamma^\infty}$ as $n\to\infty$.
  Then, by the compactness of $\Pi$, there is a subsequence
  $\{\cP_{S^{n_i}(\sqrt{n_i})}\}$ of $\{\cP_{S^n(\sqrt{n})}\}$
  that converges weakly to a pyramid $\cP$ with $\cP \neq \cP_{\Gamma^\infty}$.
  It follows from the Maxwell-Boltzmann distribution law
  (Proposition \ref{prop:MB-law}) that
  $\Gamma^k$ belongs to $\cP$ for any $k$,
  so that $\cP_{\Gamma^\infty} \subset \cP$.
  We take any real number $\theta$ with $0 < \theta < 1$ and fix it.
  Let $\theta\cP := \{\;\theta X \mid X \in \cP\;\}$.
  We see that $\cP_{S^{n_i}(\theta\sqrt{n_i})}$ converges weakly to $\theta\cP$
  as $i\to\infty$.
  Define a function $f_{\theta,n} : \R^{n+1} \to \R^{n+1}$ by
  \[
  f_{\theta,n}(x) :=
  \begin{cases}
    \frac{\theta\sqrt{n}}{\|x\|_2} x & \text{if $\|x\|_2 > \theta\sqrt{n}$},\\
    x & \text{if $\|x\|_2 \le \theta\sqrt{n}$},
  \end{cases}
  \]
  for $x \in \R^{n+1}$.
  $f_{\theta,n}$ is $1$-Lipschitz continuous with respect to
  the $l_2$ norm on $\R^{n+1}$.
  Let $\sigma_\theta^n$ be the normalized volume measure on
  $S^n(\theta\sqrt{n})$.  We consider $\sigma_\theta^n$
  as a measure on $\R^{n+1}$ via the natural embedding
  $S^n(\theta\sqrt{n}) \subset \R^{n+1}$.
  From Lemma \ref{lem:sphere-Gaussian}, we have
  \[
  d_P((f_{\theta,n})_*\gamma^{n+1},\sigma_\theta^n)
  \le \gamma^{n+1}\{\;x\in\R^{n+1} \mid \|x\|_2 < \theta\sqrt{n}\;\}
  \to 0 \ \text{as $n\to\infty$},
  \]
  so that the box distance between
  $S_{\theta,n} := (\R^{n+1},\|\cdot\|_2,(f_{\theta,n})_*\gamma^{n+1})$
  and $S^n(\theta\sqrt{n})$
  converges to zero as $n\to\infty$.
  By Proposition \ref{prop:dconc-box} and Theorem \ref{thm:rho-dconc},
  $\cP_{S_{\theta,n_i}}$ converges weakly to $\theta\cP$
  as $i\to\infty$.  Since $S_{\theta,n} \prec (\R^{n+1},\|\cdot\|_2,\gamma^{n+1})$,
  we have $\cP_{S_{\theta,n}} \subset \cP_{\Gamma^{n+1}} \subset \cP_{\Gamma^\infty}$.
  We thus obtain $\theta\cP \subset \cP_{\Gamma^\infty} \subset \cP$
  for any $\theta$ with $0 < \theta < 1$.
  Since $\theta\cP$ converges weakly to $\cP$ as $\theta \to 1$,
  we obtain $\cP = \cP_{\Gamma^\infty}$, which is a contradiction.
  This completes the proof.
\end{proof}

\begin{cor} \label{cor:sphere-Gaussian}
  The virtual infinite-dimensional Gaussian space $\cP_{\Gamma^\infty}$
  is non-concentrated,
  or equivalently, 
  neither of $S^n(\sqrt{n})$ nor $\Gamma^n$
  asymptotically concentrates as $n\to\infty$.
\end{cor}

\begin{proof}
  Proposition \ref{prop:prod-non-conc} implies
  that $\Gamma^n$ does not
  asymptotically concentrates as $n\to\infty$.
  By Proposition \ref{prop:concentrated},
  $\cP_{\Gamma^\infty}$ is non-concentrated.
  This completes the proof.
\end{proof}


\section{Spectral concentration}

In this section, we prove that a spectrally compact sequence
of mm-spaces asymptotically concentrates
if the observable diameter is uniformly bounded above.

\begin{defn}[Gradient, energy, and $\lambda_1$]
  \index{gradient} \index{energy} \index{lambda1X@$\lambda_1(X)$}
  Let $X$ be an mm-space.
  For a locally Lipschitz continuous function $f : X \to \R$,
  we define
  \begin{align*}
    |\grad f|(x) &:= \limsup_{\substack{y \to x\\ y\neq x}} \frac{|f(x)-f(y)|}{d_X(x,y)},
    \quad x \in X,\\
    \cE(f) &:= \int_X |\grad f|^2 \; d\mu_X \quad (\le +\infty).
  \end{align*}
  \index{grad f@$\vert\grad f\vert$} \index{E f@$\cE(f)$}
  We also define
  \[
  \lambda_1(X) := \inf_f \frac{\cE(f)}{\|f\|_{L_2}^2},
  \]
  where $f$ runs over all locally Lipschitz continuous functions on $X$
  with $\int_X f\,d\mu_X = 0$.
\end{defn}

If $X$ is a compact Riemannian manifold, then
$\lambda_1(X)$ defined here coincides with that defined in
\S\ref{sec:spec-sep}.

\begin{defn}[Spectral compactness] \label{defn:spec-cpt}
  \index{spectral compactness}
  For an mm-space $X$, we denote by $\Dir_1(X)$
  the set of locally Lipschitz continuous functions $f : X \to \R$ with
  $\cE(f) \le 1$.
  A subset $\cY \subset \cX$ is said to be
  \emph{spectrally compact} if 
  $(\Dir_1(X) \cap B^{L_2}_1(0),\|\cdot\|_{L_2})$ is compact
  for each $X \in \cY$ and if
  $\{(\Dir_1(X) \cap B^{L_2}_1(0),\|\cdot\|_{L_2})\}_{X \in \cY}$
  is $d_{GH}$-precompact.
  Here, $B^{L_2}_1(0)$ denotes the set of $L_2$ functions $f : X \to \R$
  with $\|f\|_{L_2} \le 1$.
  We say that a pyramid is \emph{spectrally concentrated}
  \index{spectrally concentrated}
  if it has a $\square$-dense spectrally compact subfamily.
\end{defn}

Note that, in \cite{Gromov}*{\S 3$\frac{1}{2}$},
a spectrally concentrated pyramid is defined to be
a spectrally compact pyramid, which is slightly stronger than
Definition \ref{defn:spec-cpt}.
The reason for Definition \ref{defn:spec-cpt} is
that we assume the $\square$-closedness of a pyramid, namely
we always consider the $\square$-closure of a pyramid,
which is not spectrally compact in general even if the pyramid is
spectrally compact.

For a subset $\cY \subset \cX$, we consider the condition:
\begin{equation}
  \label{eq:fin-ObsDiam}
  \sup_{X \in \cY} \ObsDiam(X;-\kappa) < +\infty
  \qquad\text{for any $\kappa > 0$}.
\end{equation}

\begin{thm} \label{thm:spec-cpt}
  If a subset $\cY \subset \cX$ satisfies \eqref{eq:fin-ObsDiam}
  and is spectrally compact,
  then $\{\Lo(X)\}_{X\in\cY}$ is $d_{GH}$-precompact.
\end{thm}

We prove the theorem later.

\begin{cor} \label{cor:spec-cpt}
  If a pyramid $\cP$ spectrally concentrates and satisfies
  \eqref{eq:fin-ObsDiam}, then $\cP$ is concentrated.
\end{cor}

\begin{proof}
  We have a $\square$-dense spectrally compact subfamily $\cP' \subset \cP$.
  Theorem \ref{thm:spec-cpt} implies that
  $\{\Lo(X)\}_{X\in\cP'}$ is $d_{GH}$-precompact.
  It follows from Lemma \ref{lem:dGH-dconc}
  and Proposition \ref{prop:dconc-box} that
  $\{\Lo(X)\}_{X\in\cP}$ is contained in
  the $d_{GH}$-closure of $\{\Lo(X)\}_{X\in\cP'}$.
  This completes the proof.
\end{proof}

\begin{rem}
  The condition \eqref{eq:fin-ObsDiam} is necessary for
  Theorem \ref{thm:spec-cpt} and Corollary \ref{cor:spec-cpt}.
  In fact, 
  let $Y_0$ and $Y_1$ be two compact Riemannian manifolds
  with diameter $\le 1$,
  and $X_n$ the disjoint union of $Y_0$ and $Y_1$.
  Define a metric $d_{X_n}$ on $X_n$ by
  \[
  d_{X_n}(x,y) :=
  \begin{cases}
    d_{Y_i}(x,y) &\text{for $x,y \in Y_i$, $i=0,1$,}\\
    n &\text{for $x \in Y_i$ and $y \in Y_{1-i}$, $i=0,1$.}
  \end{cases}
  \]
  and $\mu_{X_n} := (1/2)\mu_{Y_0} + (1/2)\mu_{Y_1}$.
  Then each $X_n$ is an mm-space.
  We see that $\{X_n\}$ is spectrally compact,
  but $\{\Lo(X_n)\}$ is not $d_{GH}$-precompact.

  It follows from Lemmas \ref{lem:dGH-precpt}, \ref{lem:Sep-CapLo},
  and \ref{prop:ObsDiam-Sep}
  that the $d_{GH}$-precompactness of $\{\Lo(X)\}_{X\in\cY}$ implies
  \eqref{eq:fin-ObsDiam}.
\end{rem}

For the proof of Theorem \ref{thm:spec-cpt},
we need some lemmas.

\begin{lem} \label{lem:me-L2}
  Let $X$ be an mm-space.
  For any two $\mu_X$-measurable functions $f,g : X \to \R$,
  we have
  \[
  \dKF(f,g) \le \|f-g\|_{L_2}^{2/3}.
  \]
\end{lem}

\begin{proof}
  Setting $\varepsilon := \dKF(f,g)$,
  we have $\mu_X(|f-g| \ge \varepsilon) \ge \varepsilon$.
  (If otherwise, then we find $\varepsilon'$ such that
  $0 < \varepsilon' < \varepsilon$ and
  $\mu_X(|f-g| > \varepsilon') < \varepsilon$, which contradicts
  $\varepsilon = \dKF(f,g)$.)
  Therefore,
  \begin{align*}
    \|f-g\|_{L_2}^2
    &\ge \int_{\{|f-g|\ge\varepsilon\}} |f-g|^2 \;d\mu_X\\
    &\ge \varepsilon^2 \mu_X(|f-g| \ge \varepsilon) \ge \varepsilon^3
    = \dKF(f,g)^3,
  \end{align*}
  which implies the lemma.
\end{proof}

\begin{lem} \label{lem:spec-cpt}
  Let $X$ be an mm-space.
  Then, for any real number $\varepsilon > 0$
  there exists a subset $\Lip_1(X;\varepsilon) \subset \Lip_1(X)$
  such that
  \begin{enumerate}
  \item the image of $\Lip_1(X;\varepsilon)$ by
    the projection $\Lip_1(X) \to \Lo(X)$ coincides with $\Lo(X)$,
  \item we have $\Lip_1(X;\varepsilon) \subset
    B^{\dKF}_\varepsilon(\Dir_1(X) \cap B^{L_2}_{D_\varepsilon}(0))$,
    where $B^{\dKF}_\varepsilon(A)$ is the set of $f \in A$ such that
    $\dKF(f,A) \le \varepsilon$, and we set
    $D_\varepsilon := \ObsDiam(X;-\varepsilon)$.
\end{enumerate}

\end{lem}

\begin{proof}
  Take any $\varepsilon > 0$.
  For any function $f \in \Lip_1(X)$
  there are two real numbers $a_f$ and $b_f$ such that
  $a_f \le b_f$, $f_*\mu_X[\,a_f,b_f\,] \ge 1-\varepsilon$,
  and $b_f-a_f \le D_\varepsilon$.
  We may assume that $a_{(f+c)} = a_f + c$ for any $f \in \Lip_1(X)$
  and for any constant $c$.
  Set
  \[
  \Lip_1(X;\varepsilon) := \{\;f\in\Lip_1(X) \mid a_f = 0\;\}.
  \]
  Then, (1) is clear.  The rest is to prove (2).
  For any $f \in \Lip_1(X;\varepsilon)$ we set
  \[
  \tilde{f}(x) := \max\{\min\{f(x),D_\varepsilon\},0\}, \quad x \in X.
  \]
  Since $\tilde{f} = f$ on
  $f^{-1}[\,0,b_f\,]$ and $\mu_X(f^{-1}[\,0,b_f\,]) \ge 1-\varepsilon$,
  we have $\dKF(\tilde{f},f) \le \varepsilon$.
  Moreover, $0 \le \tilde{f} \le D_\varepsilon$ implies
  $\|\tilde{f}\|_{L_2} \le D_\varepsilon$.
  We thus have $\tilde{f} \in \Dir_1(X) \cap B^{L_2}_{D_\varepsilon}(0)$.
  This completes the proof.
\end{proof}

\begin{proof}[Proof of Theorem \ref{thm:spec-cpt}]
  Take any $\varepsilon > 0$ and fix it.
  We set
  \[
  D_\varepsilon := \max\{\sup_{X \in \cY} \ObsDiam(X;-\varepsilon),1\},
  \]
  which is finite by the condition \eqref{eq:fin-ObsDiam}.
  By the spectral compactness of $\cY$, the family
  $\{D_\varepsilon(\Dir_1(X) \cap B^{L_2}_1(0))\}_{X \in \cY}$
  of $D_\varepsilon$-scaled sets
  is $d_{GH}$-precompact with respect to the $L_2$ norm.
  Since $D_\varepsilon(\Dir_1(X) \cap B^{L_2}_1(0))
  = \Dir_{D_\varepsilon^2}(X) \cap B^{L_2}_{D_\varepsilon}(0)$ contains
  $\Dir_1(X) \cap B^{L_2}_{D_\varepsilon}(0)$,
  the family
  $\{\Dir_1(X) \cap B^{L_2}_{D_\varepsilon}(0)\}_{X \in \cY}$
  is $d_{GH}$-precompact.  By Lemma \ref{lem:dGH-precpt},
  there is a natural number $N_\varepsilon$ such that,
  for any mm-space $X \in \cY$,
  we find an $\varepsilon$-net $\cN \subset
  \Dir_1(X) \cap B^{L_2}_{D_\varepsilon}(0)$ with $\#\cN \le N_\varepsilon$.
  By Lemma \ref{lem:me-L2} we have
  \[
  \Dir_1(X) \cap B^{L_2}_{D_\varepsilon}(0) \subset B^{L_2}_\varepsilon(\cN)
  \subset B^{\dKF}_{\varepsilon^{2/3}}(\cN),
  \]
  which together with Lemma \ref{lem:spec-cpt} implies
  \[
  \Lip_1(X;\varepsilon) \subset B^{\dKF}_{\varepsilon+\varepsilon^{2/3}}(\cN).
  \]
  Let $\cN'$ be the image of $\cN$ by a nearest point projection
  to $\Lip_1(X;\varepsilon)$.
  We see that $\cN'$ is a $(2(\varepsilon+\varepsilon^{2/3}))$-net
  of $\Lip_1(X;\varepsilon)$.
  Since the projection $\Lip_1(X;\varepsilon) \to \Lo(X)$ is
  $1$-Lipschitz continuous and surjective,
  the image of $\cN'$ by the projection
  is a $(2(\varepsilon+\varepsilon^{2/3}))$-net of $\Lo(X)$.
  This completes the proof.
\end{proof}

\begin{prop} \label{prop:spec-cpt}
  Let $\cY \subset \cX$ be a family of
  compact Riemannian manifolds.
  Then, the following {\rm(1)} and {\rm(2)} are equivalent to each other.
  \begin{enumerate}
  \item $I(\lambda) :=
    \sup_{X \in \cY} \max\{\; i \mid \lambda_i(X) \le \lambda\;\}
    < +\infty$ for any $\lambda > 0$.
  \item $\cY$ is spectrally compact.
  \end{enumerate}
\end{prop}

\begin{proof}
  We prove (1) $\implies$ (2).
  Let $X \in \cY$.
  We take a complete orthonormal basis $\{\varphi_i\}_{i=0}^\infty$ on $L_2(X)$
  such that $\Delta\varphi_i = \lambda_i(X)\varphi_i$ for any $i$.
  For any function $u \in L_2(X)$,
  we set $u_i := (u,\varphi_i)_{L_2}$.
  We see that $u = \sum_{i=0}^\infty u_i\varphi_i$,
  $\|u\|_{L_2}^2 = \sum_{i=0}^\infty u_i^2$,
  and, by the Green formula,
  \begin{align*}
    \cE(u) &= \int_X \langle du,du\rangle \; d\mu_X
    = \int_X u \Delta u \; d\mu_X\\
    &= \sum_{i,j=0}^\infty \int_X \lambda_i(X) u_i u_j \varphi_i \varphi_j \; d\mu_X
    = \sum_{i=0}^\infty \lambda_i(X)u_i^2,
  \end{align*}
  where $\langle\cdot,\cdot\rangle$ is the Riemannian metric
  on the cotangent space.
  For a given $\varepsilon > 0$, we set
  $i_\varepsilon := I(16/\varepsilon^2)$.
  Note that, if $i > i_\varepsilon$, then $\lambda_i(X) > 16/\varepsilon^2$.
  Define an orthogonal projection $\pi_\varepsilon : L_2(X) \to
  L_\varepsilon := \langle\varphi_i\mid i=0,1,\dots,i_\varepsilon\rangle$
  by
  \[
  \pi_\varepsilon(u) := \sum_{i=0}^{i_\varepsilon} u_i\varphi_i.
  \]
  For any $u \in \Dir_1(X)$,
  \[
  1 \ge \cE(u) = \sum_{i=0}^\infty \lambda_i(X)u_i^2
  \ge \sum_{i=i_\varepsilon + 1}^\infty \lambda_i(X)u_i^2
  \ge \frac{16}{\varepsilon^2} \sum_{i=i_\varepsilon + 1}^\infty u_i^2
  \]
  and hence
  \[
  \| u - \pi_\varepsilon(u) \|_{L_2}
  = \left(\sum_{i=i_\varepsilon + 1}^\infty u_i^2\right)^{1/2}
  \le \frac{\varepsilon}{4}.
  \]
  This together with triangle inequalities proves that,
  for any $u,v \in \Dir_1(X)$ with $\|u-v\|_{L_2} > \varepsilon$,
  we have
  \[
  \|\pi_\varepsilon(u) - \pi_\varepsilon(v)\|_{L_2} > \frac{\varepsilon}{2}.
  \]
  Since $L_\varepsilon$ is isometric to $\R^{i_\varepsilon+1}$,
  \[
  \Cap_\varepsilon(\Dir_1(X) \cap B^{L_2}_1(0))
  \le \Cap_{\varepsilon/2}(B^{l_2}_1(o;\R^{i_\varepsilon+1})),
  \]
  where $B^{l_2}_1(o;\R^{i_\varepsilon})$ denotes the $i_\varepsilon$-dimensional
  closed Euclidean unit ball.
  Since the right-hand side of the above inequality
  depends only on $\varepsilon$,
  Lemma \ref{lem:dGH-precpt} proves (2).

  We prove (2) $\implies$ (1).
  Suppose that (1) does not hold.
  Then, there are a number $\lambda > 0$ and a sequence
  $\{X_n\} \subset \cY$ such that
  $N_n := \max\{\; i \mid \lambda_i(X_n) \le \lambda\;\} \to +\infty$
  as $n\to\infty$.
  We may assume that $\lambda \ge 1$.
  Let $\{\varphi^{(n)}_i\}_{i=0}^\infty$ be an complete orthonormal basis
  of $L_2(X_n)$ such that
  $\Delta\varphi^{(n)}_i = \lambda_i(X_n)\varphi^{(n)}_i$ for any $i$.
  Set $\psi^{(n)}_i := \varphi^{(n)}_i/\sqrt{\lambda}$.
  If $i \le N_n$, then $\|\psi^{(n)}_i\|_{L_2} = 1/\sqrt{\lambda} \le 1$
  and
  \begin{align*}
    \cE(\psi^{(n)}_i) &= \int_X \psi^{(n)}_i \Delta \psi^{(n)}_i \; d\mu_X
    = \lambda_i(X) \int_X (\psi^{(n)}_i)^2 \; d\mu_X\\
    &= \lambda_i(X) \|\psi^{(n)}_i\|_{L_2}^2
    = \lambda_i(X_n)/\lambda \le 1,
  \end{align*}
  which imply $\psi^{(n)}_i \in \Dir_1(X) \cap B^{L_2}_1(0)$.
  Moreover,
  \[
  \|\psi^{(n)}_i - \psi^{(n)}_j\|_{L_2} = \sqrt{2/\lambda}
  \]
  for all different $i$ and $j$.
  We thus have
  \[
  \Cap_{\lambda^{-1/2}}(\Dir_1(X) \cap B^{L_2}_1(0)) \ge N_n \to \infty,
  \]
  so that $\{X_n\}$ is not spectrally compact because of
  Lemma \ref{lem:dGH-precpt}.
  $\cY$ is not spectrally compact either.
  This completes the proof.
\end{proof}

\begin{lem} \label{lem:dom-Dir}
  Let $X$ and $Y$ be two mm-spaces.
  If $X$ dominates $Y$, then
  $\Dir_1(Y)\cap B_1^{L_2}(0)$ is $L_2$-isometrically embedded into
  $\Dir_1(X)\cap B_1^{L_2}(0)$.
\end{lem}

\begin{proof}
  There is a $1$-Lipschitz map $f : X \to Y$ with $f_*\mu_X = \mu_Y$.
  A required embedding map is defined to be
  \[
  f^* : L_2(Y) \ni u \mapsto u\circ f \in L_2(X),
  \]
  which is a linear isometric embedding since,
  for any $u \in L_2(Y)$,
  \[
  \|f^*u\|_{L_2}^2 = \int_X (u\circ f)^2 \,d\mu_X
  = \int_Y u^2 \,d\mu_Y = \|u\|_{L_2}^2.
  \]
  For any $u \in \Dir_1(Y)$ and any $x \in X$,
  \begin{align*}
    |\grad(f^*u)|(x)
    &= \limsup_{y\to x} \frac{|u(f(x)) - u(f(y))|}{d_X(x,y)}\\
    &\le \limsup_{y\to x} \frac{|u(f(x)) - u(f(y))|}{d_Y(f(x),f(y))}\\
    &\le \limsup_{y'\to f(x)} \frac{|u(f(x)) - u(y')|}{d_Y(f(x),y')}
    = |\grad u|(f(x)),
  \end{align*}
  which proves that $\cE(f^*u) \le \cE(u) \le 1$.
  Therefore, $f^*(\Dir_1(Y)\cap B_1^{L_2}(0))
  \subset \Dir_1(X)\cap B_1^{L_2}(0)$.
  This completes the proof.
\end{proof}

\begin{thm} \label{thm:spec-conc}
  Let $\{X_n\}_{n=1}^\infty$ be a spectrally compact and asymptotic sequence
  of mm-spaces such that
  \[
  \sup_n \ObsDiam(X_n;-\kappa) < +\infty \quad\text{for any $\kappa > 0$}.
  \]
  Then, $\{X_n\}$ asymptotically concentrates and
  the limit pyramid is spectrally concentrated.
\end{thm}

\begin{proof}
  Theorem \ref{thm:spec-cpt} together with the assumption
  implies that $\{\Lo(X_n)\}$ is $d_{GH}$-precompact,
  so that $\{X_n\}$ asymptotically concentrates.
  By Lemma \ref{lem:dom-Dir},
  $\bigcup_{n=1}^\infty \cP_{X_n}$ is spectrally compact.
  Since the weak limit of $\cP_{X_n}$ is contained in the $\square$-closure
  of $\bigcup_{n=1}^\infty \cP_{X_n}$, it is spectrally concentrated.
  This completes the proof.
\end{proof}

\begin{defn}[Spectral concentration] \label{defn:spec-conc}
  \index{spectral concentration} \index{spectral concentrate}
  A sequence of mm-spaces is said to \emph{spectrally concentrates}
  if it spectrally compact and asymptotically concentrates.
\end{defn}

If a sequence of mm-spaces spectrally concentrates,
then the limit pyramid is spectrally concentrated.

The following corollary is stronger than
Corollary \ref{cor:prod-asympconc}.

\begin{cor} \label{cor:prod-spec-conc}
  Let $F_n$, $n=1,2,\dots$, be compact Riemannian manifolds
  such that $\lambda_1(F_n)$ is divergent to infinity
  as $n\to\infty$.
  Let $X_n := F_1 \times F_2 \times \dots \times F_n$
  be the Riemannian product space.
  Then, $\{X_n\}$ spectrally concentrates.
\end{cor}

\begin{proof}
  Since $X_n$ is monotone increasing in $n$ with respect to
  the Lipschitz order, $\{X_n\}$ is asymptotic.
  By Corollary \ref{cor:ObsDiam-spec},
  $\ObsDiam(X_n;\kappa)$, $n=1,2,\dots$, are bounded from above
  for any $\kappa > 0$.
  It follows from $\lambda_1(F_n) \to\infty$ and Proposition \ref{prop:spec-cpt}
  that $\{X_n\}$ is spectrally compact.
  By Theorem \ref{thm:spec-conc}, $\{X_n\}$ is asymptotically
  concentrates.
\end{proof}

\begin{rem}
  In the proof of Corollary \ref{cor:prod-spec-conc},
  we do not rely on Proposition \ref{prop:prod-asympconc}.
  Namely, we have an alternative proof of
  Corollary \ref{cor:prod-asympconc}.
\end{rem}

\begin{ex}[Compare Example \ref{ex:prod-sph}] \label{ex:prod-sph-spec}
  Let
  $X_n := S^1 \times S^2 \times \dots \times S^n$ (Riemannian product).
  Since $\lambda_1(S^n) = n$, Corollary \ref{cor:prod-spec-conc}
  proves that $\{X_n\}$ spectrally concentrates.
\end{ex}



\chapter{Dissipation}
\label{chap:dissipation}

\section{Basics for dissipation}

Dissipation is an opposite notion to concentration.

\begin{defn}[Dissipation]
  \index{dissipation}
  Let $X_n$, $n=1,2,\dots$, be mm-spaces and $\delta > 0$
  a real number.
  We say that $\{X_n\}$ \emph{$\delta$-dissipates}
  \index{delta-dissipate@$\delta$-dissipate}
  if for any real numbers $\kappa_0,\kappa_1,\dots,\kappa_N > 0$
  with $\sum_{i=0}^N \kappa_i < 1$, we have
  \[
  \liminf_{n\to\infty} \Sep(X_n;\kappa_0,\kappa_1,\dots,\kappa_N) \ge \delta.
  \]
  We say that $\{X_n\}$ \emph{infinitely dissipates}
  \index{infinitely dissipate}
  if it $\delta$-dissipates for any $\delta > 0$.
\end{defn}

The following is obvious.

\begin{prop}
  Let $X_n$ and $Y_n$, $n=1,2,\dots$, be mm-spaces such that
  $X_n \prec Y_n$ for any $n$.
  If $\{X_n\}$ $\delta$-dissipates for a positive real number $\delta$
  {\rm(}resp.~infinitely dissipates{\rm)},
  then so does $\{Y_n\}$.
\end{prop}

The following lemma is useful to detect dissipating families.

\begin{lem} \label{lem:diss}
  Let $X_n$, $n=1,2,\dots$, be mm-spaces and let $\delta > 0$ be a real number.
  Assume that for each $n$ there exists an at most countable sequence
  $\{C_{ni}\}_{i=1}^{k_n}$, $k_n \le +\infty$, of mutually disjoint Borel subsets
  of $X_n$ such that
  \begin{align}
    \tag{1} &\lim_{n\to\infty} \mu_X(\bigcup_{i=1}^{k_n} C_{ni}) = 1,\\
    \tag{2} &\min_{i\neq j} d_{X_n}(C_{ni},C_{nj}) \ge \delta,\\
    \tag{3} &\lim_{n\to\infty} \sup_{i=1}^{k_n} \mu_{X_n}(C_{ni}) = 0.
  \end{align}
  Then, $\{X_n\}$ $\delta$-dissipates.
\end{lem}

The proof of the lemma is easy and omitted.

\begin{ex}
  \begin{enumerate}
  \item Let $X$ be a compact Riemannian manifold and
    $\{t_n\}_{n=1}^\infty$ a sequence of positive real numbers divergent to
    infinity.  Then, $\{t_n X\}$ infinitely dissipates, where $t_n X$
    denotes the manifold $X$ with the Riemannian distance function
    multiplied by $t_n$.
  \item Let $T_n$ be a perfect rooted binary tree with root $o$
    and depth $n$,
    i.e., all vertices of $T_n$ are within distance at most $n$ from $o$,
    and the number of vertices $v$ of $T_n$ with $d_{T_n}(o,v) = k$
    is equal to $2^k$ for any $k \in \{1,2,\dots,n\}$,
    where $d_{T_n}$ denotes the path metric on $T_n$.
    Consider the equally distributed probability measure $\mu_{T_n}$ on
    the set of vertices of $T_n$, i.e.,
    \[
    \mu_{T_n} = \sum_v \frac{1}{2^{n+1}-1} \delta_v,
    \]
    where $v$ runs over all vertices of $T_n$.
    Then, $\{(T_n,d_{T_n},\mu_{T_n})\}$ infinitely dissipates,
    and $\{(T_n,(1/n)d_{T_n},\mu_{T_n})\}$ $2$-dissipates.
  \item Let $H$ be a complete simply connected hyperbolic space,
    $\{t_n\}_{n=1}^\infty$ a sequence of positive numbers divergent to
    infinity, and $B_n$ a metric ball of radius $t_n$ in $H$.
    Then, $\{B_n\}$ infinitely dissipates, and
    $\{(1/t_n)B_n\}$ $2$-dissipates.
  \end{enumerate}
\end{ex}

\begin{prop} \label{prop:dissipate}
  Let $X_n$, $n=1,2,\dots$, be mm-spaces and $\delta > 0$
  a real number.
  \begin{enumerate}
  \item $\{X_n\}$ $\delta$-dissipates if and only if
    the weak limit of any weakly convergent subsequence
    of $\{\cP_{X_n}\}$ contains all mm-spaces with diameter $\le \delta$.
  \item $\{X_n\}$ infinitely dissipates if and only if
    $\cP_{X_n}$ converges weakly to $\cX$ as $n\to\infty$.
  \end{enumerate}
\end{prop}

\begin{proof}
  We prove (1).  Denote by $\cX_{\delta}$ the set of mm-spaces with diameter
  $\le \delta$.
  Assume that $\{X_n\}$ $\delta$-dissipates and let $\cP$ be the weak limit
  of a weakly convergent subsequence of $\{\cP_{X_n}\}$.
  We are going to prove $\cX_\delta \subset \cP$.
  Since the set of finite mm-spaces with diameter $< \delta$ is $\square$-dense
  in $\cX_\delta$, it suffices to prove that
  any such mm-space belongs to $\cP$.
  Let $Y$ be any finite mm-space with diameter $< \delta$.
  We take any number $\varepsilon$ with
  $0 < \varepsilon < \delta-\diam Y$.
  Let $\{y_0,y_1,\dots,y_N\} := Y$ and take
  real numbers $\kappa_0,\kappa_1,\dots,\kappa_N$ such that
  $0 < \kappa_i \le \mu_Y\{y_i\}$ for any $i$ and
  \[
  1-\varepsilon < \sum_{i=0}^N \kappa_i < 1.
  \]
  There is a natural number $n(\varepsilon)$ such that
  $\Sep(X_n;\kappa_0,\dots,\kappa_N) > \delta-\varepsilon$
  for any $n \ge n(\varepsilon)$.
  Let $n \ge n(\varepsilon)$.
  Then, there are Borel subsets $A_{n1},\dots A_{nN} \subset X_n$
  such that $\mu_{X_n}(A_{ni}) \ge \kappa_i$,
  $d_{X_n}(A_{ni},A_{nj}) > \delta-\varepsilon > \diam Y$
  for any different $i$ and $j$.
  Embedding $Y$ into $(\R^M,\|\cdot\|_\infty)$ isometrically,
  we assume that $Y$ is a subset of $\R^M$.
  We define a map $f_n : \bigcup_{i=0}^N A_{ni} \to \R^M$ by
  $f_n|_{A_{ni}} := y_i$ for any $i$.
  $f_n$ is a $1$-Lipschitz map and extends
  to a $1$-Lipschitz map $f_n : X_n \to \R^M$.
  We see that
  $(f_n)_*\mu_{X_n}(\{y_i\}) = \mu_{X_n}(f_n^{-1}(y_i)) \ge \mu_{X_n}(A_{ni})
  \ge \kappa_i$.
  Setting $\nu := \sum_{i=0}^N \kappa_i \delta_{y_i}$,
  we have
  $\nu \le \mu_Y$, $\nu \le (f_n)_*\mu_{X_n}$, and
  $\nu(\R^M) = \sum_{i=0}^N \kappa_i > 1-\varepsilon$.
  We apply Strassen's theorem (Theorem \ref{thm:di-tra})
  to obtain $d_P((f_n)_*\mu_{X_n},\mu_Y) < \varepsilon$.
  As $\varepsilon \to 0$,
  $(f_{n(\varepsilon)})_*\mu_{X_n}$ converges weakly to $\mu_Y$.
  Thus, $X_{n(\varepsilon)} \succ
  (\R^M,\|\cdot\|_\infty,(f_{n(\varepsilon)})_*\mu_{X_{n(\varepsilon)}})
  \overset{\square}\to Y$
  as $\varepsilon \to 0$, so that $Y$ belongs to $\cP$.

  We next prove the converse.
  Assume that the weak limit of any weakly convergent subsequence
  of $\{\cP_{X_n}\}$ contains all mm-spaces with diameter $\le \delta$.
  We take any real numbers $\kappa_0,\dots,\kappa_N > 0$ with
  $\sum_{i=0}^N \kappa_i < 1$ and set
  \[
  \varepsilon_0 := \frac{1}{N+1}\left( 1-\sum_{i=0}^N \kappa_i \right).
  \]
  We see that $\sum_{i=0}^N(\kappa_i+\varepsilon_0) = 1$.
  Let $Y = \{y_0,\dots,y_N\}$ be a finite metric space such that
  $d_Y(y_i,y_j) = \delta$ for any different $i$ and $j$,
  and let $\mu_Y := \sum_{i=0}^N (\kappa_i+\varepsilon_0)\delta_{y_i}$.
  Then, $\mu_Y$ is a probability measure and $Y = (Y,d_Y,\mu_Y)$
  is an mm-space.
  Let $\cP$ be the weak limit of a weakly convergent subsequence
  of $\{\cP_{X_n}\}$.
  The assumption implies that $\cP$ contains $Y$.
  Thereby, there are mm-spaces $Y_n$, $n=1,2,\dots$,
  such that $X_n \succ Y_n \overset{\square}\to Y$ as $n\to\infty$.
  By Lemma \ref{lem:box-eps-mm-iso},
  we have an $\varepsilon_n$-mm-isomorphism
  $f_n : Y \to Y_n$, $\varepsilon_n \to 0$.
  For each $n$ there is a Borel subset $Y_0 \subset Y$ such that
  $d_P((f_n)_*\mu_Y,\mu_{Y_n}) \le \varepsilon_n$,
  $\mu_Y(Y_0) \ge 1-\varepsilon_n$, and
  \begin{equation}
    \label{eq:diss}
    |d_Y(y_i,y_j)-d_{Y_n}(f_n(y_i),f_n(y_j))| \le \varepsilon_n
  \end{equation}
  for any $y_i,y_j \in Y_0$.
  We take $n$ so large that
  $\varepsilon_n < \min\{\delta,\varepsilon_0\}$.
  Then, since $\mu_Y(Y_0) \ge 1-\varepsilon_n > 1-\varepsilon_0$
  and $\mu_Y\{y_i\} > \varepsilon_0$ for any $i$,
  we have $Y_0 = Y$.
  Since $d_Y(y_i,y_j) = \delta$ for $i\neq j$ and by \eqref{eq:diss},
  we have $d_{Y_n}(f_n(y_i),f_n(y_j)) > 0$ for $i\neq j$,
  i.e., $f_n$ is injective.
  Let $\varepsilon_n'$, $n=1,2,\dots$, be numbers
  such that $\lim_{n\to\infty} \varepsilon_n' = 0$
  and $\varepsilon_n < \varepsilon_n' < \varepsilon_0$ for any $n$.
  Set $A_{ni} := B_{\varepsilon_n'}(f_n(y_i))$.
  Since $(f_n)_*\mu_Y = \sum_{i=0}^N (\kappa_i+\varepsilon_0)\delta_{f_n(y_i)}$,
  the inequality $d_P((f_n)_*\mu_Y,\mu_{Y_n}) \le \varepsilon_n$
  proves that
  \[
  \mu_{Y_n}(A_{ni}) \ge (f_n)_*\mu_Y\{f_n(y_i)\}-\varepsilon_n'
  = \kappa_i+\varepsilon_0-\varepsilon_n' > \kappa_i.
  \]
  We also have, for $i \neq j$,
  \begin{align*}
    d_{Y_n}(A_{ni},A_{nj}) &\ge d_{Y_n}(f_n(y_i),f_n(y_j))-2\varepsilon_n'\\
    &\ge d_Y(y_i,y_j)-\varepsilon_n-2\varepsilon_n'
    = \delta-\varepsilon_n-2\varepsilon_n'.
  \end{align*}
  Therefore, $\Sep(Y_n;\kappa_0,\dots,\kappa_N)
  \ge \delta-\varepsilon_n-2\varepsilon_n'$, so that,
  by Lemma \ref{lem:Sep-prec}, we obtain
  $\Sep(X_n;\kappa_0,\dots,\kappa_N)
  \ge \delta-\varepsilon_n-2\varepsilon_n'$.
  (1) has been proved.

  We prove (2).
  If $\{X_n\}$ infinitely dissipates, then
  (1) implies that the weak limit of any subsequence of $\{\cP_{X_n}\}$
  contains any mm-spaces with finite diameter.
  Since the set of mm-spaces with finite diameter is $\square$-dense
  in $\cX$, we have $\cP = \cX$.
  The converse also follows from (1).
  This completes the proof of the proposition.
\end{proof}

For an mm-space $F$, we denote by $F^n$ the $n^{th}$ power product space of $F$,
and by $\mu_F^{\otimes n}$ the product measure on $F^n$ of $\mu_F$. 
Let $d_{l_p}$, $1 \le p \le +\infty$, be the $l_p$ metric on $F^n$
induced from $d_F$.

\begin{prop} \label{prop:disconn-dissipation}
  If $F$ is a compact disconnected mm-space,
  then $\{(F^n,d_{l_p},\mu_F^{\otimes n})\}_{n=1}^\infty$
  $\delta$-dissipates for some $\delta > 0$ and
  for any $p$ with $1 \le p \le +\infty$.
\end{prop}

\begin{proof}
  Since $(F^n,d_{l_\infty},\mu_F^{\otimes n}) \prec (F^n,d_{l_p},\mu_F^{\otimes n})$,
  it suffices to prove the proposition for $p = +\infty$.
  Since $F$ is disconnected,
  there are two disjoint subsets $F_1, F_2 \subset F$
  with $F_1 \cup F_2 = F$ such that $F_1$ and $F_2$ are both open and closed.
  The compactness of $F_1$ and $F_2$ proves that
  $\delta := d_F(F_1,F_2) > 0$.
  Set $a := \mu_F(F_1)$ and $b := \mu_F(F_2)$.
  We have $0 < a,b < 1$ because $F_1$ and $F_2$ are open.
  For $i_1,i_2,\dots,i_n = 1,2$, we set
  \[
  F_{i_1i_2\dots i_n} := F_{i_1} \times F_{i_2} \times \dots F_{i_n} \subset F^n.
  \]
  Since
  \[
  \mu_F^{\otimes n}(F_{i_1i_2\dots i_n}) = a^k b^{n-k},
  \]
  where $k$ is the number of $j$'s with $i_j = 1$,
  we have
  \[
  \lim_{n\to\infty} \max_{i_1,i_2,\dots,i_n} \mu_F^{\otimes n}(F_{i_1i_2\dots i_n})
  = 0.
  \]
  We also see
  \[
  d_{l_\infty}(F_{i_1i_2\dots i_n},F_{j_1j_2\dots j_n}) =
  \begin{cases}
    \delta &\text{if $(i_1,i_2,\dots,i_n) \neq (j_1,j_2,\dots,j_n)$},\\
    0 &\text{if $(i_1,i_2,\dots,i_n) = (j_1,j_2,\dots,j_n)$}.
  \end{cases}
  \]
  The proposition follows from Lemma \ref{lem:diss}.
\end{proof}

The following is another example of dissipation.

\begin{prop} \label{prop:Sn-dissipate}
  Let $r_n$, $n=1,2,\dots$, be positive real numbers.
  Then, $\{S^n(r_n)\}$ infinitely dissipates if and only if
  $r_n/\sqrt{n} \to +\infty$ as $n\to\infty$.
\end{prop}

\begin{proof}
  We first assume $r_n/\sqrt{n} \to +\infty$ as $n\to\infty$.
  Take any finitely many positive real numbers
  $\kappa_0,\kappa_1,\dots,\kappa_N$ with $\sum_{i=0}^N \kappa_i < 1$,
  and fix them.
  We find positive real numbers
  $\kappa_0',\kappa_1',\dots,\kappa_N'$
  in such a way that $\kappa_i < \kappa_i'$ for any $i$
  and $\sum_{i=0}^N \kappa_i' < 1$.
  For any $\varepsilon > 0$,
  there are Borel subsets $A_0,A_1,\dots,A_N \subset \R$ such that
  $\gamma^1(A_i) \ge \kappa_i'$ for any $i$
  and
  \[
  \min_{i\neq j} d_{\R}(A_i,A_j)
  > \Sep((\R,\gamma^1);\kappa_0',\dots,\kappa_N')-\varepsilon.
  \]
  We may assume that all $A_i$ are open subsets of $\R$.
  Denote by $\pi_n : S^n(\sqrt{n}) \to \R$ the orthogonal projection
  as in Section \ref{sec:obs-sphere}.
  By the Maxwell-Boltzmann distribution law,
  we have
  $\liminf_{n\to\infty}(\pi_n)_*\sigma^n(A_i) \ge \gamma^1(A_i) \ge \kappa_i'$
  for $i=0,1,\dots,N$.
  There is a natural number $n_0$ such that
  $\sigma^n(\pi_n^{-1}(A_i)) \ge \kappa_i$ for any $i$ and $n \ge n_0$.
  The $1$-Lipschitz continuity of $\pi_n$ implies
  that $d_{S^n(\sqrt{n})}(\pi_n^{-1}(A_i),\pi_n^{-1}(A_j)) \ge d_{\R}(A_i,A_j)$.
  Therefore, for any $n \ge n_0$,
  \[
  \Sep(S^n(\sqrt{n});\kappa_0,\dots,\kappa_N)
  > \Sep((\R,\gamma^1);\kappa_0',\dots,\kappa_N')-\varepsilon,
  \]
  which proves
  \[
  \liminf_{n\to\infty} \Sep(S^n(\sqrt{n});\kappa_0,\dots,\kappa_N)
  \ge \Sep((\R,\gamma^1);\kappa_0',\dots,\kappa_N') > 0.
  \]
  Since $r_n/\sqrt{n} \to +\infty$ as $n\to\infty$,
  \begin{equation}
    \label{eq:Levy-dissipate-sphere}
    \Sep(S^n(r_n);\kappa_0,\dots,\kappa_N)
    = \frac{r_n}{\sqrt{n}} \Sep(S^n(\sqrt{n});\kappa_0,\dots,\kappa_N)
  \end{equation}
  is divergent to infinity and so $\{S^n(r_n)\}$ infinitely dissipates.

  We next prove the converse.
  Assume that $\{S^n(r_n)\}$ infinitely dissipates
  and $r_n/\sqrt{n}$ is not divergent to infinity.
  Then, there is a subsequence $\{r_{n_i}\}$ of $\{r_n\}$
  such that $r_{n_i}/\sqrt{n_i}$ is bounded for all $i$.
  By \eqref{eq:Levy-dissipate-sphere},
  $\{S^{n_i}(\sqrt{n_i})\}$ infinitely dissipates.
  However, for each fixed $\kappa$ with $0 < \kappa < 1/2$,
  $\ObsDiam(S^n(\sqrt{n});-\kappa)$ is bounded for all $n$
  and, by Proposition \ref{prop:ObsDiam-Sep},
  so is $\Sep(S^n(\sqrt{n});\kappa,\kappa)$,
  which contradicts that $\{S^{n_i}(\sqrt{n_i})\}$ infinitely dissipates.
  This completes the proof.
\end{proof}

\section{Obstruction for dissipation}

In this section, we study an obstruction for dissipation and prove
the following

\begin{thm}[Non-dissipation theorem] \label{thm:non-dissipation}
  \index{non-dissipation theorem}
  Let $F$ be a compact, connected, and locally connected mm-space.
  Then, $\{(F^n,d_{l_\infty},\mu_F^{\otimes n})\}_{n=1}^\infty$ does not dissipate.
\end{thm}

Note that the connectivity of $F$ is necessary as is seen in
Proposition \ref{prop:disconn-dissipation}.

\begin{rem}
  As is seen in Proposition \ref{prop:prod-non-conc},
  $\{(F^n,d_{l_p},\mu_F^{\otimes n})\}_{n=1}^\infty$ for any $0 \le p \le +\infty$
  does not asymptotically concentrate
  unless $F$ consists of a single point.
  In particular, $\{(F^n,d_{l_\infty},\mu_F^{\otimes n})\}$
  neither asymptotically concentrate nor dissipate
  for any connected and compact Riemannian manifold $F$
  not consisting of a single point.
  Since $\{(F^n,d_{l_\infty},\mu_F^{\otimes n})\}$ is a monotone increasing sequence
  with respect to the Lipschitz order, this also holds for any
  subsequence of $\{(F^n,d_{l_\infty},\mu_F^{\otimes n})\}$.
\end{rem}

\begin{defn}[Expansion coefficient]
  \index{expansion coefficient} \index{exp@$\Exp(X;\kappa,\rho)$}
  Let $X$ be an mm-space and let $\kappa,\rho > 0$.
  The \emph{expansion coefficient $\Exp(X;\kappa,\rho)$ of $X$}
  is defined to be the supremum of real numbers $\xi \ge 1$ such that,
  if $\mu_X(A) \ge \kappa$ for a Borel subset $A \subset X$,
  then $\mu_X(B_\rho(A)) \ge \xi\kappa$.
\end{defn}

It is clear that $\Exp(X;\kappa,\rho) \ge 1$.




\begin{lem} \label{lem:exp-diss}
  Let $X_n$, $n=1,2,\dots$, be mm-spaces and let $\delta > 0$
  be a real number.
  If there exist two real numbers $\kappa$ and $\rho$
  with $\kappa > 0$ and $0 < \rho < \delta$ such that
  \[
  \inf_n \Exp(X_n;\kappa,\rho) > 1,
  \]
  then $\{X_n\}$ does not $\delta$-dissipate.
\end{lem}

\begin{proof}
  By the assumption, we find a constant $c > 1$
  such that $\Exp(X_n;\kappa,\rho) > c$ for any $n$.
  Take a number $\kappa'$ in such a way that
  $1-c\kappa < \kappa' < 1-\kappa$.
  Suppose that $\{X_n\}$ $\delta$-dissipate.
  Then, we have
  $\Sep(X_n;\kappa,\kappa') > \rho$ for every sufficiently large $n$.
  For such a number $n$, we find two Borel subset $A,A' \subset X_n$
  in such a way that $d_X(A,A') > \rho$, $\mu_{X_n}(A) \ge \kappa$,
  and $\mu_{X_n}(A') \ge \kappa'$.
  It follows from $\Exp(X_n;\kappa,\rho) > c$
  that $\mu_{X_n}(B_\rho(A)) \ge c\kappa$.
  Since $B_\rho(A)$ and $A'$ does not intersect to each other,
  we have
  \[
  \mu_{X_n}(B_\rho(A) \cup A') \ge c\kappa + \kappa' > 1,
  \]
  which is a contradiction.
  This completes the proof.
\end{proof}

\begin{prop} \label{prop:exp-lam1}
  Let $X$ be an mm-space.
  Then we have
  \[
  \Exp(X;\kappa,\rho) \ge \min\left\{1+\frac{\lambda_1(X)\rho^2}{4},2\right\}
  \]
  for any real numbers $\kappa$ and $\rho$ with $0 < \kappa \le 1/4$
  and $\rho > 0$.
\end{prop}

\begin{proof}
  If $\lambda_1(X) = 0$, then the proposition is trivial.

  Assume that $\lambda_1(X) > 0$.
  It suffices to prove that
  \begin{equation}
    \label{eq:exp-lam1}
    \mu_X(B_\rho(A)) \ge e_\rho\kappa
  \end{equation}
  for any Borel subset $A \subset X$ with $\mu_X(A) \ge \kappa$,
  where
  \[
  e_\rho := \min\left\{1+\frac{\lambda_1(X)\rho^2}{4},2\right\}.
  \]
  Let us first prove it in the case where $A$ is an open subset of $X$.
  For two constants $c$ and $r > 0$,
  we define a function $f_{c,r} : X \to \R$ by
  \[
  f_{c,r}(x) :=
  \begin{cases}
    c + \frac{r}{\rho} d_X(x,A)
    &\text{for $x \in B_\rho(A)$},\\
    c+r &\text{for $x \in X \setminus B_\rho(A)$}.
  \end{cases}
  \]
  $f_{c,r}$ is Lipschitz continuous with Lipschitz constant
  $r/\rho$.
  We have $f_{c,r} = c + f_{0,r}$ and so
  \[
  \int_X f_{c_r,r}\,d\mu_X = 0,\qquad
  c_r := -\int_X f_{0,r}\,d\mu_X.
  \]
  Since $A \cup (X \setminus B_\rho(A))$ is open,
  we have $|\grad f| = 0$ on $A \cup (X \setminus B_\rho(A))$,
  so that
  \[
  \cE(f_{c,r}) = \frac{r^2}{\rho^2}E,
  \qquad E := \int_{B_\rho(A)\setminus A} |\grad d_X(\cdot,A)|^2\,d\mu_X.
  \]
  If $E = 0$, then the Rayleigh quotient satisfies $R(f_{c_r,r}) = 0$,
  which contradicts $\lambda_1(X) > 0$.
  We therefore have $E > 0$.
  Setting $f := f_{c_r,r}$ for $r := \rho/E^{1/2}$,
  we have $\cE(f) = 1$ and $\int_X f\,d\mu_X = 0$.
  Since $|\grad d_X(\cdot,A)| \le 1$,
  \[
  1 = \cE(f) \le \frac{r^2}{\rho^2}
  (\mu_X(B_\rho(A)) - \mu_X(A)),
  \]
  which implies
  \[
  r \ge \frac{\rho}{\sqrt{\mu_X(B_\rho(A)) - \mu_X(A)}}.
  \]
  Since $f \ge c_r$ on $X$ and $f = c_r+r$ on $X \setminus B_\rho(A)$,
  \begin{align*}
    0 &= \int_X f\,d\mu_X \ge c_r\,\mu_X(B_\rho(A)) + (c_r+r)\mu_X(X\setminus B_\rho(A))\\
    &= c_r + r\,\mu_X(X\setminus B_\rho(A)).
  \end{align*}
  Therefore,
  \[
  -c_r \ge
  \frac{(1-\mu_X(B_\rho(A)))\rho}
  {\sqrt{\mu_X(B_\rho(A))-\mu_X(A)}}.
  \]
  This together with
  \begin{align*}
    \frac{1}{\lambda_1(X)} \ge \|f\|_{L_2}^2 \ge c_r^2\, \mu_X(A).
  \end{align*}
  implies
  \begin{align} \label{eq:exp-lam1-2}
    \mu_X(B_\rho(A))
    \ge \{1 + \lambda_1(X)(1-\mu_X(B_\rho(A)))^2\rho^2\}\,
    \mu_X(A)
  \end{align}
  If $\mu_X(B_\rho(A)) \ge 1/2$, then
  $\mu_X(B_\rho(A)) \ge 1/2 \ge 2\kappa \ge e_\rho\kappa$.
  If $\mu_X(B_\rho(A)) < 1/2$, then
  \eqref{eq:exp-lam1-2} leads us to
  \[
    \mu_X(B_\rho(A))
    \ge \left( 1 + \frac{\lambda_1(X)\rho^2}{4} \right)\kappa \ge e_\rho\kappa.
  \]
  We thus obtain \eqref{eq:exp-lam1} for any open subset $A \subset X$
  with $\mu_X(A) \ge \kappa$.

  Let $A$ be any Borel subset of $X$ with $\mu_X(A) \ge \kappa$.
  We take any number $\delta$ with $0 < \delta < \rho$.
  Applying \eqref{eq:exp-lam1} for $U_\delta(A)$ and $\rho-\delta$
  yields that
  \[
  \mu_X(B_\rho(A)) \ge \mu_X(B_{\rho-\delta}(U_\delta(A)))
  \ge e_{\rho-\delta}\,\kappa,
  \]
  so that $\Exp(X;\kappa,\rho) \ge e_{\rho-\delta}$
  for any $\delta$ with $0 < \delta < \rho$.
  This completes the proof.
\end{proof}

Lemma \ref{lem:exp-diss} and Proposition \ref{prop:exp-lam1}
together imply

\begin{thm} \label{thm:lam1-diss}
  Let $X_n$, $n=1,2,\dots$, be mm-spaces
  such that $\lambda_1(X_n)$ is bounded away from zero.
  Then, $\{X_n\}$ does not dissipate, i.e.,
  does not $\delta$-dissipate for any real number $\delta > 0$.
\end{thm}

\begin{cor} \label{cor:prod-diss}
  Let $F$ be a compact connected Riemannian manifold.
  Then, the sequence $\{F^n\}$ of Riemannian product spaces
  does not dissipate.
\end{cor}

\begin{proof}
  Since $\lambda_1(F^n) = \lambda_1(F) > 0$,
  Theorem \ref{thm:lam1-diss} implies the corollary.
\end{proof}

\begin{defn}[Modulus of continuity]
  \index{modulus of continuity}
  Let $f : X \to Y$ be a map between two metric spaces $X$ and $Y$.
  A function $\omega : [\,0,+\infty\,] \to [\,0,+\infty\,]$
  is called a \emph{modulus of continuity of $f$}
  if $\omega(0+) = \omega(0) = 0$ and if
  \[
  d_Y(f(x),f(y)) \le \omega(d_X(x,y))
  \]
  for any two points $x,y \in X$.
\end{defn}

A map $f : X \to Y$ has a modulus of continuity
if and only if $f$ is uniformly continuous.

\begin{rem} \label{rem:mod-cont}
  If $\omega$ is a modulus of continuity of a map $f : X \to Y$,
  then the function $\hat\omega(t) = \sup_{s \le t} \omega(s)$
  is a monotone nondecreasing modulus of continuity of $f$.
  We may assume without loss of generality that a modulus of continuity is
  monotone nondecreasing.
\end{rem}

\begin{lem} \label{lem:pre-maj}
  Let $f_n : X_n \to Y_n$, $n=1,2,\dots$, be uniformly continuous functions
  between mm-spaces $X_n$ and $Y_n$
  all which admit a common modulus of continuity
  and satisfy $(f_n)_*\mu_{X_n} = \mu_{Y_n}$.
  If $\{X_n\}$ does not dissipate, then so does $\{Y_n\}$.
\end{lem}

\begin{proof}
  Assume that $\{Y_n\}$ $\delta$-dissipate for some number $\delta > 0$.
  Let $\kappa_0,\dots,\kappa_N$ be any positive real numbers with
  $\sum_{i=0}^N \kappa_i < 1$.
  Then, there is a natural number $n_0$ such that
  $\Sep(Y_n;\kappa_0,\dots,\kappa_N) > \delta/2$
  for any $n \ge n_0$.
  For each $n \ge n_0$, we find Borel subsets $A_{n1},\dots,A_{nN} \subset Y_n$
  such that $\mu_{Y_n}(A_{ni}) \ge \kappa_i$ for any $i$
  and $d_{Y_n}(A_{ni},A_{nj}) \ge \delta/2$
  for any different $i$ and $j$.
  Let $\omega$ be the common modulus of continuity of $f_n$.
  We assume that $\omega$ is monotone nondecreasing
  (see Remark \ref{rem:mod-cont}).
  Then we have
  \[
  \omega(d_{X_n}(f_n^{-1}(A_{ni}),f_n^{-1}(A_{nj}))+0)
  \ge d_{Y_n}(A_{ni},A_{nj}) \ge \delta/2,
  \]
  so that there is a number $\delta_0 > 0$ depending only
  on $\delta$ and $\omega$ such that
  \[
  d_{X_n}(f_n^{-1}(A_{ni}),f_n^{-1}(A_{nj})) \ge \delta_0.
  \]
  We also have $\mu_{X_n}(f_n^{-1}(A_{ni})) = \mu_{Y_n}(A_{ni}) \ge \kappa_i$.
  Therefore, $\Sep(X_n;\kappa_0,\dots,\kappa_N) \ge \delta_0$
  for any $n \ge n_0$ and $\{X_n\}$ $\delta_0$-dissipates.
  This completes the proof.
\end{proof}

\begin{lem}[Majorization lemma] \label{lem:maj}
  \index{majorization lemma}
  Let $f : F_0 \to F$ be a uniformly continuous map between
  two mm-spaces $F_0$ and $F$ such that $f_*\mu_{F_0} = \mu_F$.
  If $\{(F_0^n,d_{l_\infty},\mu_{F_0}^{\otimes n})\}$ does not dissipate,
  then $\{(F^n,d_{l_\infty},\mu_F^{\otimes n})\}$ does not dissipate.
\end{lem}

\begin{proof}
  Let $f_n : F_0^n \to F^n$ be the function defined by
  \[
  f_n(x_1,x_2,\dots,x_n) := (f(x_1),f(x_2),\dots,f(x_n))
  \]
  for $(x_1,x_2,\dots,x_n) \in F_0^n$.
  We are going to apply Lemma \ref{lem:pre-maj}.
  For any Borel subsets $A_i \subset F$, $i=1,2,\dots,n$,
  we have
  \begin{align*}
    (f_n)_*\mu_{F_0^n}(A_1 \times \dots \times A_n)
    &= \mu_{F_0^n}(f^{-1}(A_1) \times \dots \times f^{-1}(A_n))\\
    &= \mu_{F^n}(A_1 \times \dots \times A_n).
  \end{align*}
  Since the family of the products of Borel subsets of $F$ generate
  the Borel $\sigma$-algebra of $F^n$, we obtain
  $(f_n)_*\mu_{F_0^n} = \mu_{F^n}$.
  If $\omega$ is a monotone nondecreasing modulus of continuity
  of $f$, then we have, for any $x = (x_1,x_2,\dots,x_n),
  y = (y_1,y_2,\dots,y_n) \in F_0$,
  \begin{align*}
    d_{l_\infty}(f_n(x),f_n(y)) &= \max_i d_F(f(x_i),f(y_i))
    \le \max_i \omega(d_{F_0}(x_i,y_i))\\
    &\le \omega(\max_i d_{F_0}(x_i,y_i)) = \omega(d_{l_\infty}(x,y)),
  \end{align*}
  so that $\omega$ is a modulus of continuity of $f_n$
  for any $n$.
  Lemma \ref{lem:pre-maj} completes the proof.
\end{proof}

\begin{proof}[Proof of Theorem \ref{thm:non-dissipation}]
  It is known (see \cite{Bog}*{9.7.1, 9.7.2})
  that there is a continuous map $f : [\,0,1\,] \to F$
  such that $f_*\cL^1 = \mu_F$.
  By the compactness of $[\,0,1\,]$,
  the map $f$ is uniformly continuous.
  Denote by $\cL^n$ the $n$-dimensional Lebesgue measure on $[\,0,1\,]^n$.
  By Corollary \ref{cor:prod-diss},
  $\{([\,0,1\,]^n,d_{l_2},\cL^n)\}$ does not dissipate, which
  together with $([\,0,1\,]^n,d_{l_\infty},\cL^n) \prec ([\,0,1\,]^n,d_{l_2},\cL^n)$
  implies that $\{([\,0,1\,]^n,d_{l_\infty},\cL^n)\}$ does not dissipate either.
  The majorization lemma, Lemma \ref{lem:maj},
  completes the proof.
\end{proof}

\chapter{Curvature and concentration}
\label{chap:curv-conc}

\section{Fibration theorem for concentration}
\label{sec:conc}

In this section, we see that a concentration $X_n \to Y$ yields
a Borel measurable map $p_n : X_n \to Y$ such that,
for every sufficiently large $n$,
(1) $p_n$ is $1$-Lipschitz up to a small additive error,
(2) every fiber of $p_n$ concentrates to a one-point mm-space,
and (3) all fibers of $p_n$ are almost parallel to each other
for $n$ large enough.

For a Borel subset $B$ of an mm-space $X$,
we equip $B$ with the restrictions of $d_X$ and $\mu_X$ on $B$,
so that $B$ becomes a (possibly incomplete) mm-space whose measure
is not necessarily probability.
For such a $B$, we define the \emph{observable diameter}
$\ObsDiam(B;-\kappa)$, $\kappa > 0$, by
\begin{align*}
  \ObsDiam(B;-\kappa) &:= \sup\{\;\diam(f_*\mu_X;\mu_X(B)-\kappa) \mid\\
  &\qquad\qquad\text{$f : B \to \R$ $1$-Lipschitz}\;\}.
\end{align*}
\index{observable diameter}
\index{obsdiam@$\ObsDiam_Y(\cdots)$, $\ObsDiam(\cdots)$}

\begin{defn}[Effectuate relative concentration]
  \index{effectuate relative concentration}
  Let $X_n$ and $Y$ be mm-spaces
  and let $p_n : X_n \to Y$ be Borel measurable maps, where $n = 1,2,\dots$.
  We say that $\{p_n\}$ \emph{effectuates relative concentration} of $X_n$
  over $Y$ if
  \begin{align*}
    \limsup_{n\to \infty} \ObsDiam(p_n^{-1}(B);-\kappa)\leq
    \diam B
  \end{align*}
  for any $\kappa>0$ and any Borel subset $B\subset Y$.
\end{defn}

$\{p_n\}$ effectuating relative concentration of $X_n$ over $Y$
means that every fiber of $p_n$ concentrates to a one-point space
as $n\to\infty$.

\begin{lem} \label{lem:Lip1-ObsDiam}
  Let $p : X \to Y$ be a Borel measurable map between
  two mm-spaces $X$ and $Y$.
  If we have
  \begin{align*}
    \Lip_1(X)\subset B_{\varepsilon}(p^*\Lip_1(Y))
  \end{align*}
  for a real number $\varepsilon > 0$,
  then
  \begin{align*}
    \ObsDiam(p^{-1}(B);-\kappa)\leq
    \diam B+ 2\varepsilon
  \end{align*}
  for any $\kappa\geq \varepsilon$ and any Borel subset
  $B\subset Y$. 
\end{lem}

\begin{proof}
  Let $f : p^{-1}(B)\to \R$ be an arbitrary $1$-Lipschitz function.
  There is a $1$-Lipschitz extension $\tilde{f} : X \to \R$ of $f$.
  From the assumption, we have a $1$-Lipschitz function $g:Y\to \R$ such that
  $\dKF(\tilde{f}, g \circ p)\leq \varepsilon$.
  Setting
  \begin{align*}
    A := \{\; x \in p^{-1}(B) \; ; \; |\tilde{f}(x)-(g \circ p)(x)|
    \le \varepsilon\;\},
  \end{align*}
  we have $\mu_{X}(p^{-1}(B)) - \mu_X(A) \le \varepsilon \le \kappa$ and hence
  \begin{align*}
    f_*(\mu_{X}|_{p^{-1}(B)})(\overline{f(A)})
    \ge \mu_{X}(A)
    \ge \mu_{X}(p^{-1}(B))-\kappa.
  \end{align*}
  For any $x,x'\in A$ we have
  \begin{align*}
    |f(x)-f(x')| &=
    |\tilde{f}(x)-\tilde{f}(x')|\\
    &\leq |\tilde{f}(x)-(g\circ p)(x)|+|(g \circ
    p)(x)-(g \circ p)(x')|\\
    &\quad +|(g \circ p)(x')-\tilde{f}(x')|\\
    &\le d_Y(p(x),p(x')) + 2\varepsilon
    \le \diam B + 2\varepsilon,
  \end{align*}
  so that $\diam f(A) \leq \diam B + 2\varepsilon$.
  This completes the proof.
\end{proof}

The following is a direct consequence of Lemma \ref{lem:Lip1-ObsDiam}.

\begin{cor} \label{cor:Lip1-ObsDiam}
  Let $p_n:X_n\to Y$ be Borel measurable maps
  that enforce $\varepsilon_n$-concentration of mm-spaces $X_n$
  to an mm-space $Y$ for some $\varepsilon_n \to 0$.
  Then, $\{p_n\}$ effectuates relative concentration of $X_n$ over $Y$.
\end{cor}

\begin{defn}[$\kappa$-distance]
  \index{kappa-distance@$\kappa$-distance}
  \index{dplus@$d_+(A_1,A_2;+\kappa)$}
  Let $\kappa > 0$ and let $A_1$ and $A_2$ be two Borel subsets
  in an mm-space $X$.  We define the \emph{$\kappa$-distance}
  \[
  d_+(A_1,A_2;+\kappa)
  \]
  between $A_1$ and $A_2$ as the supremum of $d_X(B_1,B_2)$
  over all Borel subsets $B_1 \subset A_1$ and $B_2 \subset A_2$
  with $\mu_X(B_1) \geq \kappa$ and $\mu_X(B_2)\geq \kappa$.
  If $\min\{\mu_X(A_1),\mu_X(A_2)\} < \kappa$,
  then we set $d_+(A_1,A_2;+\kappa) := 0$.
\end{defn}

\begin{lem} \label{lem:kappa-dist}
  Let $p : X \to Y$ be a Borel measurable
  map between two mm-spaces $X$ and $Y$ such that
  \begin{align*}
    \Lip_1 (X)\subset B_{\varepsilon}(p^*\Lip_1(Y))
  \end{align*}
  for a real number $\varepsilon > 0$.  Then, for any two Borel
  subsets $A_1,A_2 \subset X$
  and any real number $\kappa$ with $\varepsilon < \kappa$,
  we have
  \begin{align*}
    d_+(A_1,A_2;+\kappa) \le d_Y(p(A_1),p(A_2))
    + \diam p(A_1) + \diam p(A_2) + 2\varepsilon.
  \end{align*}
\end{lem}

\begin{proof}
  Let $A_1'\subset A_1$ and $A_2'\subset A_2$ be
  Borel subsets such that $\mu_X(A_1')\geq \kappa$ and
  $\mu_X(A_2')\geq \kappa$.
  We set
  \[
  f(x) := \min\{d_X(x,A_1'),d_X(A_1',A_2')\}, \qquad x \in X,
  \]
  which is a $1$-Lipschitz function on $X$.
  By the assumption there is a $1$-Lipschitz
  function $g : Y \to \R$ such that $\dKF(f,g\circ p)\leq \varepsilon$.
  Letting
  \begin{align*}
    B := \{\; x\in X \; ; \; |f(x)-(g\circ p)(x)|\leq \varepsilon \;\},
  \end{align*}
  we have $\mu_X(B)\geq 1-\varepsilon$.
  Since $\mu_X(A_i') \ge \kappa > \varepsilon$,
  the intersection of $A_i'$ and $B$ is nonempty for $i=1,2$.
  We take a point $x\in A_1' \cap B$ and a point $x'\in A_2'\cap B$.
  Since $f\equiv 0$ on $A_1'$ and $f\equiv d_X(A_1',A_2')$ on $A_2'$, we have
  \begin{align*}
    & |g(p(x))|=|(g\circ p)(x)-f(x)| \le \varepsilon,\\
    & |g(p(x'))-d_X(A_1',A_2')|=|(g \circ p)(x')-f(x')| \le \varepsilon,
  \end{align*}
  and therefore
  \begin{align*}
    d_X(A_1',A_2') - 2\varepsilon &\le g(p(x'))-g(p(x))
    \le d_Y(p(x),p(x'))\\
    &\le d_Y(p(A_1),p(A_2)) + \diam p(A_1) + \diam p(A_2).
  \end{align*}
  This completes the proof.
\end{proof}


\begin{lem} \label{lem:lm-tra}
  Let $\mu_1$ and $\mu_2$ be two Borel probability measures
  on a metric space $X$. Assume that there exists a
  $\rho$-transport plan $\pi$ between $\mu_1$ and $\mu_2$
  with $\defi\pi < 1-2\kappa$ for two real numbers $\rho$ and $\kappa$
  with $\rho > 0$ and $0 < \kappa < 1/2$.
  Then, for any $1$-Lipschitz function $f : X \to \R$, we have
  \begin{align*}
    &|\lm(f;\mu_1)- \lm(f;\mu_2)|\\
    &\leq \rho + \ObsDiam(\mu_1;-\kappa) + \ObsDiam(\mu_2;-\kappa),
  \end{align*}
  where $\lm(f;\mu_i)$ is the L\'evy mean of $f$ with respect to $\mu_i$.
\end{lem}

\begin{proof}
  Let $f : X \to \R$ be a $1$-Lipschitz continuous function.
  We set $D_i := \ObsDiam(\mu_i;-\kappa)$ for $i=1,2$.
  Since Lemma \ref{lem:LeRad-ObsDiam} implies $\LeRad(\mu_i;-\kappa) \le D_i$, we have
  \[
  \mu_i(|f-\lm(f;\mu_i)| \le D_i) \ge 1-\kappa
  \]
  for $i = 1,2$.
  Letting
  \begin{align*}
    I_i:=\{\;s\in \R\mid |s-\lm(f;\mu_i)| \le D_i\;\},
  \end{align*}
  we have $\mu_i(f^{-1}(I_i)) \ge 1-\kappa$.
  We prove that $\pi(f^{-1}(I_1)\times f^{-1}(I_2)) > 0$.
  In fact, since
  \begin{align*}
    1-\kappa-\pi (f^{-1}(I_1)\times X) &\leq
    \mu_1(f^{-1}(I_1))-\pi(f^{-1}(I_1)\times X)\\
    &\leq
    \mu_1(f^{-1}(I_1))-\pi (f^{-1}(I_1)\times X)\\
    &\ + \mu_1(X\setminus f^{-1}(I_1))-\pi ((X \setminus f^{-1}(I_1))\times
    X)\\
    &= 1-\pi(X\times X),
  \end{align*}
  we have
  \[
  \pi(f^{-1}(I_1)\times X) \ge \pi(X\times X) - \kappa
  \]
  and, in the same way,
  \[
  \pi(X\times f^{-1}(I_2)) \ge \pi(X\times X) - \kappa.
  \]
  We therefore obtain 
  \begin{align*}
    \pi(f^{-1}(I_1)\times f^{-1}(I_2))
    &\ge \pi(f^{-1}(I_1)\times X) + \pi(X\times f^{-1}(I_2)) - \pi(X\times X)\\
    &\ge \pi(X\times X) - 2\kappa = 1 - 2\kappa - \defi\pi > 0.
  \end{align*}

  There is a point $(x_1,x_2) \in (f^{-1}(I_1)\times f^{-1}(I_2)) \cap \supp\pi$.
  By $\supp \pi \subset \{d_X \le \rho\}$,
  we have $|f(x_1)-f(x_2)| \le d_X(x_1,x_2) \le \rho$.
  Since $f(x_i)$ belongs to $I_i$ for $i=1,2$, we obtain
  \[
  |\lm(f;\mu_1)-\lm(f;\mu_2)| \le |f(x_1)-f(x_2)| + D_1 + D_2
  \le \rho + D_1 + D_2.
  \]
  This completes the proof.
\end{proof}

\begin{defn}[$\lambda$-Prohorov metric]
  \index{lambda-Prohorov distance/metric@$\lambda$-Prohorov distance/metric}
  Let $X$ be a metric space.
  Define the \emph{$\lambda$-Prohorov distance $d_P^{(\lambda)}(\mu,\nu)$
  between two Borel probability measures $\mu$ and $\nu$ on $X$},
  $\lambda > 0$, to be the infimum of $\varepsilon > 0$ such that
  \[
  \mu(B_\varepsilon(A)) \ge \nu(A) - \lambda\varepsilon
  \]
  for any Borel subset $A \subset X$.
\end{defn}

We see that $\lambda d_P^{(\lambda)}$ coincides with the Prohorov
metric with respect to the scaled metric $\lambda d_X$.

\begin{thm}[Fibration theorem for concentration] \label{thm:pn}
  Let $p_n : X_n \to Y$ be Borel measurable maps
  between mm-spaces $X_n$ and $Y$, $n=1,2,\dots$,
  such that $(p_n)_*\mu_{X_n}$ converges weakly to $\mu_Y$ as $n\to\infty$.
  Then, each $p_n$ enforces $\varepsilon_n$-concentration of $X_n$ to $Y$
  for some sequence $\varepsilon_n \to 0$
  if and only if we have the following {\rm(1)}, {\rm(2)}, and {\rm(3)}.
  \begin{enumerate}
  \item Each $p_n$ is $1$-Lipschitz up to some additive error $\varepsilon_n'$
    such that $\varepsilon_n' \to 0$ as $n\to\infty$.
  \item $\{p_n\}$ effectuates relative concentration of $X_n$ over $Y$.
  \item For any two Borel subsets $A_1, A_2\subset Y$ and any $\kappa > 0$,
    we have
    \begin{align*}
      \limsup_{n\to\infty} d_+(p_n^{-1}(A_1),p_n^{-1}(A_2);+\kappa)
      \le d_Y(A_1,A_2)+\diam A_1 +\diam A_2.
    \end{align*}
  \end{enumerate}
\end{thm}

As is mentioned before, (2) means that every fiber of $p_n$ concentrates
to a one-point mm-space.
(3) means that all fibers of $p_n$ are almost parallel to each other
for $n$ large enough.

\begin{proof}
  Let $p_n : X_n \to Y$ be as in the theorem.
  If each $p_n$ enforces $\varepsilon_n$-concentration of $X_n$ to $Y$
  for a sequence $\varepsilon_n \to 0$, then
  Lemma \ref{lem:1-Lip-up-to-Lip1}, Corollary \ref{cor:Lip1-ObsDiam},
  and Lemma \ref{lem:kappa-dist} respectively imply
  (1), (2), and (3) of the theorem. 
    
  Conversely we assume (1), (2), and (3).
  Lemma \ref{lem:1-Lip-up-to-Lip1} yields that
  $p_n^*\Lip_1(Y) \subset B_{\varepsilon_n'}(\Lip_1(X_n))$
  for some $\varepsilon_n' \to 0$.
  Let $\varepsilon > 0$ be an arbitrary number.
  It suffices to prove that
  $\Lip_1(X_n)\subset B_{11\varepsilon}(p_n^*\Lip_1(Y))$
  for all $n$ large enough.
  Take any $1$-Lipschitz functions $f_n : X_n \to \R$, $n=1,2,\dots$.
  For the proof, it suffices to find $1$-Lipschitz functions
  $g_n : Y\to \R$ such that
  \begin{align} \label{eq:pn2}
    \dKF(f_n,g_n \circ p_n) \le 11\varepsilon
  \end{align}
  for every sufficiently large $n$.
  Our idea to find such $g_n$ is
  to take the L\'evy mean of $f_n$ along each fiber of $p_n$.
  
  There are finitely many open subsets $B_1,B_2,\dots,B_N \subset Y$
  such that $\mu_Y(\partial B_i)=0$, $\diam B_i < \varepsilon$, and
  \begin{align*}
    \mu_Y\Big(Y\setminus \bigcup_{i=1}^N B_i \Big)< \varepsilon.
  \end{align*}
  For each $i=1,2, \cdots , N$, we take a point $y_i\in B_i$.
  We put $A_{in}:=p_n^{-1}(B_i)$ and
  $\nu_{in}:=(1/\mu_{X_n}(A_{in}))\mu_{X_n}|_{A_{in}}$.
  Note that $\mu_{X_n}(A_{in})$ converges to $\mu_Y(B_i)$ as $n\to\infty$.
  Setting $\rho_{ij} := d_Y(y_i,y_j) + 5\varepsilon$,
  we take a number $\lambda$ such that $0 < \lambda \rho_{ij} < 1/4$
  for any $i,j$.
  
  \begin{clm}\label{clm:di-rho}
    For any $i,j$, and for every sufficiently large $n$, we have
    \begin{align*}
      d_P^{(\lambda)}(\nu_{in},\nu_{jn}) \leq \rho_{ij}.
    \end{align*}
  \end{clm}
  
  \begin{proof}
    We fix $i$ and $j$.
    Let $C_n \subset X_n$ be any Borel subsets.
    It suffices to prove that
    \begin{align}\label{ccs14}
      \nu_{jn}(B_{\rho_{ij}}(C_n))\geq
      \nu_{in}(C_n)-\lambda \rho_{ij}
    \end{align}
    for every sufficiently large $n$.
    Setting $C_n':=C_n\cap A_{in}$, we
    are going to prove that
    \begin{align}\label{ccs15}
      \nu_{jn}(B_{\rho_{ij}}(C_n'))\geq
      \nu_{in}(C_n')-\lambda \rho_{ij},
    \end{align}
    which is stronger than \eqref{ccs14}.
    
    We take a number $\kappa$ such that 
    \[
    0 < \kappa
    \leq \lambda \rho_{ij} \inf_n\min\{\mu_{X_n}(A_{in}),\mu_{X_n}(A_{jn})\}.
    \]
    If $\mu_{X_n}(C_n')< \kappa$, then we have
    $\nu_{in}(C_n') < \lambda \rho_{ij}$ and so
    \eqref{ccs15} holds.
    Assume that $\mu_{X_n}(C_n') \ge \kappa $.
    We define a function $F_n:A_{jn}\to \R$ by
    $F_n(x):=d_{X_n}(x,C_n')$, $x \in A_{jn}$, and set
    \begin{align*}
      D_n := \{\; x \in A_{jn} \; ; \; |F_n(x)-\lm(F_n;\nu_{jn})|
      \le \varepsilon\;\}.
    \end{align*}
    It follows from (2) and Lemma \ref{lem:LeRad-ObsDiam} that
    $\LeRad(A_{jn};-\kappa') < \varepsilon$
    for every sufficiently large $n$ and for any $\kappa' > 0$,
    which implies that $\mu_{X_n}(A_{jn} \setminus D_n)$ converges to zero
    as $n\to\infty$.
    Remarking $\kappa < \mu_{X_n}(A_{jn})/4$,
    we have 
    \begin{align*}
      \mu_{X_n}(D_n) \ge \kappa \quad\text{and}\quad
      \nu_{jn}(D_n) \ge 1 - \lambda\rho_{ij}
    \end{align*}
    for every sufficiently large $n$.
    In the following, we assume $n$ to be large enough.
    By $\diam B_i<\varepsilon$ and (3), we have
    \begin{align*}
      d_{X_n}(C_n', D_n)\leq
      d_+(A_{in},A_{jn};+\kappa) < d_Y(y_i,y_j) + 2\varepsilon.
    \end{align*}
    We take a point $a \in D_n$ in such a way that
    \begin{align*}
      F_n(a) < d_{X_n}(C_n',D_n)+\varepsilon.
    \end{align*}
    Then,
    \begin{align*}
      \lm(F_n;\nu_{jn}) &\le F_n(a) + \varepsilon
      <  d_{X_n}(C_n',D_n) + 2\varepsilon\\
      &<  d_Y(y_i,y_j) + 4\varepsilon.
    \end{align*}
    For any $x \in D_n$,
    \begin{align*}
      d_{X_n}(x,C_n') = F_n(x) \le \lm(F_n;\nu_{jn})+\varepsilon < \rho_{ij},
    \end{align*}
    which implies $B_{\rho_{ij}}(C_n') \supset D_n$ and so
    \begin{align*}
      \nu_{jn}(B_{\rho_{ij}}(C_n')) \ge \nu_{jn}(D_n)
      \ge 1-\lambda\rho_{ij}
      \ge \nu_{in}(C_n')-\lambda \rho_{ij}.
    \end{align*}
    This completes the proof of the claim.
  \end{proof}
  
  Let us define a function $\tilde{g}_n: Y \to \R$ by
  \[
  \tilde{g}_n(y) :=
  \begin{cases}
    \lm(f_n|_{A_{in}};\nu_{in}) & \text{if $y \in B_i$},\\
    0 & \text{if $y \in Y \setminus \bigcup_{i=1}^N B_i$}.
  \end{cases}
  \]  
  We are going to prove that
  $\tilde{g}_n$ is $1$-Lipschitz up to a small additive error.
  Since $\lambda\, d_P^{(\lambda)}$ coincides with the Prohorov metric
  with respect to $\lambda\, d_{X_n}$
  and by Claim \ref{clm:di-rho},
  Strassen's theorem proves that
  there is a $\rho_{ij}$-transport plan $\pi_{ijn}$ between $\nu_{in}$ and
  $\nu_{jn}$
  such that $\pi_{ijn}(X_n\times X_n)\ge 1-\lambda \rho_{ij} > 3/4$.
  Assume that $n$ is sufficiently large.
  (2) together with $\diam B_l < \varepsilon$ implies that
  $\ObsDiam(\nu_{ln};-1/4) < \varepsilon$ for every $l$.
  Applying Lemma \ref{lem:lm-tra} we have,
  for any $y\in B_i$ and $y'\in B_j$,
  \begin{align*}
    |\tilde{g}_n(y)-\tilde{g}_n(y')|
    &= |\lm(f_n|_{A_{in}};\nu_{in} ) - \lm(f_n|_{A_{jn}};\nu_{jn} )|\\
    &\le \rho_{ij}+ 2\varepsilon
    = d_Y(y_i,y_j) + 7\varepsilon
    < d_Y(y,y') + 9\varepsilon.
  \end{align*}
  Since
  \begin{align*}
    \lim_{n\to \infty} (p_n)_*\mu_{X_n}\Big(\bigcup_{i=1}^N B_i
    \Big) = \mu_Y \Big(   \bigcup_{i=1}^N     B_i  \Big)\geq 1-\varepsilon,
  \end{align*}
  the function $\tilde{g}_n$ is $1$-Lipschitz up to $9\varepsilon$
  with respect to $(p_n)_*\mu_{X_n}$ for every sufficiently large $n$.
  
  By Lemma \ref{lem:Lip-approx},
  there is a $1$-Lipschitz function $g_n : Y \to \R$ such that
  \begin{align} \label{eq:pn3}
    \dKF(g_n \circ p_n, \tilde{g}_n \circ p_n)\leq 9\varepsilon.
  \end{align}
  Let $\kappa := \varepsilon/N$.
  For every sufficiently large $n$,
  since $\LeRad(\nu_{in};-\kappa) \le \ObsDiam(\nu_{in};-\kappa)
  < \varepsilon$, we have
  \begin{align*}
    &\mu_{X_n}(|f_n-\tilde{g}_n \circ p_n| > \varepsilon)\\
    &\leq \sum_{i=1}^N\mu_{X_n}(
    \{\; x \in A_{in} \mid  |f_n(x)-\lm(f_n|_{A_{in}};\nu_{in})| >
    \varepsilon \;\} )\\
    &\quad + \mu_{X_n}\Big(X_n \setminus \bigcup_{i=1}^N A_{in}\Big)\\
    &\leq N\kappa + \varepsilon = 2\varepsilon,
  \end{align*}
  so that $\dKF(f_n,\tilde{g}_n \circ p_n)\leq 2\varepsilon$.
  Combining this with (\ref{eq:pn3}) implies (\ref{eq:pn2}). This
  completes the proof of the theorem.
\end{proof}

\section{Wasserstein distance
and curvature-dimension condition}

This section is devoted to a quick overview on
the Wasserstein distance and the curvature-dimension condition.

Let $X$ be a complete separable metric space.

\begin{defn}[Wasserstein distance] 
  \index{Wasserstein distance/metric}
  \index{Lp-Wasserstein distance/metric@$L_p$-Wasserstein distance/metric}
  Let $\mu$ and $\nu$ be two Borel probability measures on $X$.
  For $1 \le p < +\infty$,
  the \emph{$L_p$-Wasserstein distance} between $\mu$ and $\nu$
  is defined to be
  \[
  W_p(\mu,\nu) := \left( \inf_\pi \int_{X \times X} d_X(x,x')^p\;d\pi(x,x')
    \right)^{\frac{1}{p}} \quad(\le +\infty),
  \]
  where $\pi$ runs over all transport plans between $\mu$ and $\nu$.
  It is known that, if $W_p(\mu,\nu) < +\infty$, this infimum is achieved by some transport plan,
  which we call an \emph{optimal transport plan for $W_p(\mu,\nu)$}.
  \index{optimal transport plan}
\end{defn}

It follows from H\"older's inequality that
$W_p(\mu,\nu) \le  W_q(\mu,\nu)$ for $1 \le p \le q < +\infty$.

\begin{defn}[$P_p(X)$]
  \index{Pp(X)@$P_p(X)$}
  Let $1 \le p < +\infty$.
  Denote by $P_p(X)$ the set of Borel probability measures $\mu$ on $X$
  with finite \emph{$p^{th}$ moment},
  i.e.,
  \[
  W_p(\mu,\delta_{x_0})^p = \int_X d_X(x_0,x)^p \;d\mu_X(x) < +\infty
  \]
  for some point $x_0 \in X$.
  \index{moment} \index{pth moment@$p^{th}$ moment}
  The $L_p$-Wasserstein distance defines a metric on $P_p(X)$,
  called the \emph{$L_p$-Wasserstein metric}.
\end{defn}

\begin{lem}[cf.~\cite{Villani:topics}*{Theorem 7.12}] \label{lem:Wp-weakconv}
  Let $X$ be a complete separable metric space.
  Let $\mu$ and $\mu_n$, $n=1,2,\dots$, be measures
  in $P_p(X)$, $1 \le p < +\infty$.
  Then the following {\rm(1)} and {\rm(2)} are equivalent to each other.
  \begin{enumerate}
  \item $\lim_{n\to\infty} W_p(\mu_n,\mu) = 0$.
  \item $\mu_n$ converges weakly to $\mu$ as $n\to\infty$
    and
    \[
    \lim_{R\to +\infty} \limsup_{n\to\infty} \int_{X \setminus B_R(x_0)} d_X(x_0,x)^p \;d\mu_n(x)
    = 0
    \]
     for some {\rm(}and any{\rm)} point $x_0 \in X$.
  \end{enumerate}
\end{lem}

\begin{thm}[Kantorovich-Rubinstein duality;
  cf.~\cite{Villani:oldnew}*{Remark 6.5}]
  \index{Kantorovich-Rubinstein duality}
  \label{thm:KR-duality}
  Let $X$ be a complete separable metric space.
  For any measures $\mu,\nu \in P_1(X)$ we have
  \[
  W_1(\mu,\nu) = \sup_{f \in \Lip_1(X)} \left( \int_X f \;d\mu - \int_X f \;d\nu
  \right).
  \]
\end{thm}

\begin{defn}[Relative entropy]
  \index{relative entropy} \index{Ent@$\Ent$}
  Let $\mu$ and $\nu$ be two probability measures on a set.
  The \emph{relative entropy} $\Ent(\nu|\mu)$ of $\nu$
  with respect to $\mu$ is defined as follows.
  If $\nu$ is absolutely continuous with respect to $\mu$, then
  \[
  \Ent(\nu|\mu) := \int_X U\left(\frac{d\nu}{d\mu}\right) \,d\mu
  \quad (\le +\infty),
  \]
  otherwise $\Ent(\nu|\mu) := +\infty$,
  where 
  \[
  U(r) :=
  \begin{cases}
    0 &\text{if $r = 0$},\\
    r\log r &\text{if $r > 0$}.
  \end{cases}
  \]
  \index{Ur@$U(r)$}
\end{defn}

\begin{lem} \label{lem:Ent-push}
  Let $p : X \to Y$ be a Borel measurable map between
  two complete separable metric spaces $X$ and $Y$,
  and let $\mu$ and $\nu$ be two Borel probability measures on $X$ such that
  $\nu$ is absolutely continuous with respect to $\mu$.
  Then, $p_*\nu$ is absolutely continuous
  with respect to $p_*\mu$ and we have
  \[
  \Ent(p_*\nu|p_*\mu) \le \Ent(\nu|\mu).
  \]
\end{lem}

\begin{proof}
  Let $\{\mu_y\}_{y\in Y}$ be the disintegration of $\mu$ for $p : X \to Y$.
  We set $\rho := \frac{d\nu}{d\mu}$ and
  $\tilde\rho(y) := \int_{p^{-1}(y)} \rho \,d\mu_y$.
  For any bounded continuous function $f : Y \to \R$,
  \begin{align*}
    \int_Y f \,dp_*\nu &= \int_X f\circ p \,d\nu = \int_X (f\circ p) \rho\,d\mu\\
    &= \int_Y \int_{p^{-1}(y)} (f\circ p)\rho\,d\mu_y dp_*\mu(y)\\
    &= \int_Y f(y) \int_{p^{-1}(y)} \rho \,d\mu_y dp_*\mu(y)
    = \int_Y f\tilde\rho \,dp_*\mu,
  \end{align*}
  which implies that $\frac{dp_*\nu}{dp_*\mu} = \tilde\rho$.
  It follows from Jensen's inequality that
  \begin{align*}
    \Ent(p_*\nu|p_*\mu) &= \int_Y U(\tilde\rho) \,dp_*\mu
    \le \int_X \int_{p^{-1}(y)} U(\rho) \,d\mu_y dp_*\mu(y)\\
    &= \int_X U(\rho) \,d\mu = \Ent(\nu|\mu).
  \end{align*}
  This completes the proof.
\end{proof}

\begin{lem}[cf.~\cite{Villani:oldnew}*{Theorem 29.20(i)}]
  \label{lem:Ent-lim}
  Let $\mu$, $\nu$, $\mu_n$, and $\nu_n$, $n=1,2,\dots$, be
  Borel probability measures on a compact metric space.
  If $\mu_n$ and $\nu_n$ converge weakly to $\mu$ and $\nu$
  respectively as $n\to\infty$, then we have
  \[
  \Ent(\nu|\mu) \le \liminf_{n\to\infty} \Ent(\nu_n|\mu_n).
  \]
\end{lem}

\begin{defn}[Curvature-dimension condition] \label{defn:CD}
  \index{curvature-dimension condition}
  \index{CD(K,infty)@$\CD(K,\infty)$}
  Let $K$ be a real number.
  We say that an mm-space $X$ satisfies
  the \emph{curvature-dimension condition $\CD(K,\infty)$}
  if for any $\nu_0,\nu_1 \in P_2(X)$,
  any $\varepsilon > 0$, and
  any $t \in (\,0,1\,)$, there exists a measure $\nu_t \in P_2(X)$
  such that
  \begin{align}
    \label{eq:CD1}
    W_2(\nu_t,\nu_i) &\le t^{1-i}(1-t)^i W_2(\nu_0,\nu_1) + \varepsilon,
    \quad i=0,1,\\
    \label{eq:CD2}
    \Ent(\nu_t|\mu_X) &\le (1-t)\Ent(\nu_0|\mu_X) + t\Ent(\nu_1|\mu_X)\\
    &\quad -\frac{1}{2} K t(1-t)W_2(\nu_0,\nu_1)^2 + \varepsilon.\notag
  \end{align}
\end{defn}

The following important theorem was first predicted by Otto \cite{Otto}.

\begin{thm}[\cites{CMS:PL,CMS:interp,vRS,Sturm:convex}]
  Let $X$ be a complete Riemannian manifold $X$ and
  $K$ a real number.
  Then, $\CD(K,\infty)$ for $X$ is equivalent to $\Ric_X \ge K$.
\end{thm}

\begin{defn}[$P^{cb}(X)$]
  \index{Pcb(X)@$P^{cb}(X)$}
  For an mm-space $X$, we denote by $P^{cb}(X)$
  the set of Borel probability measures $\nu$ on $X$ with compact support
  such that
  $\nu$ is absolutely continuous with respect to $\mu_X$
  and that the Radon-Nikodym derivative $\frac{d\nu}{d\mu_X}$
  is essentially bounded on $X$.
\end{defn}

Note that $P^{cb}(X) \subset P_p(X)$ for any $p$ with $1 \le p < +\infty$.

\begin{lem} \label{lem:WpEnt-approx}
  Let $X$ be an mm-space and $\nu \in P_p(X)$ a measure
  with $\Ent(\nu|\mu_X) < +\infty$, where $1 \le p < +\infty$.
  Then, for any $\varepsilon > 0$
  there exists a measure $\tilde\nu \in P^{cb}(X)$ such that
  \[
  W_p(\tilde\nu,\nu) < \varepsilon \quad\text{and}\quad
  |\Ent(\tilde\nu|\mu_X) - \Ent(\nu|\mu_X)| < \varepsilon.
  \]
\end{lem}

\begin{proof}
  By the inner regularity of $\mu_X$,
  there is a monotone nondecreasing sequence of compact subsets
  $K_n \subset X$, $n=1,2,\dots$, such that
  \[
  \lim_{n\to\infty} \mu_X(X \setminus K_n) = 0.
  \]
  Let $\rho := \frac{d\nu}{d\mu_X}$ and
  \[
  \rho_n(x) :=
  \begin{cases}
    \frac{1}{c_n} \min\{\rho(x),n\} & \text{if $x \in K_n$},\\
    0 & \text{if $x \in X \setminus K_n$},
  \end{cases}
  \]
  where $c_n := \int_{K_n} \min\{\rho(x),n\} \,d\mu_X(x)$.
  The measure $\nu_n := \rho_n \mu_X$ belongs to $P^{cb}(X)$.
  It suffices to prove that 
  \begin{align}
    \label{eq:WpEnt-approx1}
    & \lim_{n\to\infty} W_p(\nu_n,\nu) = 0,\\
    \label{eq:WpEnt-approx2}
    & \lim_{n\to\infty} \Ent(\nu_n|\mu_X) = \Ent(\nu|\mu_X).
  \end{align}
  Since $\lim_{n\to\infty} c_n = 1$ and $\mu_X(\bigcup_{n=1}^\infty K_n) = 1$,
  the function $\rho_n$ converges to $\rho$ $\mu_X$-a.e.~as $n\to\infty$.
  It is easy to see that $\nu_n$ converges weakly to $\nu$ as $n\to\infty$.
  Moreover, for a point $x_0 \in X$,
  \begin{align*}
    &\lim_{R\to\infty}\limsup_{n\to\infty} \int_{X\setminus B_R(x_0)} d_X(x_0,x)^p \;d\nu_n(x)\\
    &\le \lim_{R\to\infty}\limsup_{n\to\infty} \frac{1}{c_n}
    \int_{X\setminus B_R(x_0)} d_X(x_0,x)^p \;d\nu(x) = 0.
  \end{align*}
  By Lemma \ref{lem:Wp-weakconv} we obtain \eqref{eq:WpEnt-approx1}.

  Letting
  \[
  I_n := \int_{\{0 < \rho \le 1\}} \rho_n\log\rho_n \;d\mu_X
  \quad\text{and}\quad
  J_n := \int_{\{\rho > 1\}} \rho_n\log\rho_n \;d\mu_X,
  \]
  we have $\Ent(\nu_n|\mu_X) = I_n + J_n$.
  The dominated convergence theorem implies
  \[
  \lim_{n\to\infty} I_n = \int_{\{0 < \rho \le 1\}} \rho\log\rho \;d\mu_X.
  \]
  For every sufficiently large $n$ we have $c_n \ge 1/2$,
  which implies $\rho_n \le 2\rho$ and then
  $0 \le \rho_n\log\rho_n \le 2\rho\log(2\rho)$ on $\{\rho > 1\}$.
  We see that $2\rho\log(2\rho)$ is $\mu_X$-integrable
  on $\{\rho > 1\}$.  By the dominated convergence theorem,
  \[
  \lim_{n\to\infty} J_n = \int_{\{\rho > 1\}} \rho\log\rho \;d\mu_X.
  \]
  We thus obtain \eqref{eq:WpEnt-approx2}.
  This completes the proof.
\end{proof}

\begin{lem} \label{lem:CD-Pcb}
  If we assume $\nu_0,\nu_1 \in P^{cb}(X)$ in the definition of
  $\CD(K,\infty)$, then we still have $\CD(K,\infty)$.
\end{lem}

\begin{proof}
  We assume the condition of $\CD(K,\infty)$ for any $\nu_0,\nu_1 \in P^{cb}(X)$.
  Take any measures $\nu_0,\nu_1 \in P_2(X)$.
  If $\Ent(\nu_i|\mu_X) = +\infty$ for $i=0$ or $1$,
  then \eqref{eq:CD2} is trivial and \eqref{eq:CD1} follows from
  the dense property of $P^{cb}(X)$ in $P_2(X)$.
  Assume that $\Ent(\nu_i|\mu_X) < +\infty$ for $i=0,1$.
  By Lemma \ref{lem:WpEnt-approx},
  for any $\varepsilon > 0$ we find two measures
  $\tilde\nu_0,\tilde\nu_1 \in P^{cb}(X)$
  in such a way that
  $W_2(\tilde\nu_i,\nu_i) < \varepsilon$
  and $|\Ent(\tilde\nu_i|\mu_X)-\Ent(\nu_i|\mu_X)| < \varepsilon$
  for $i=0,1$.
  By the assumption, there is a measure $\nu_t \in P_2(X)$
  for any $t \in (\,0,1\,)$
  such that $\tilde\nu_0$, $\tilde\nu_1$, and $\nu_t$ together
  satisfy \eqref{eq:CD1} and \eqref{eq:CD2}.
  Remarking that $\varepsilon$ can be taken to be arbitrarily small
  compared with $K$ and $W_2(\nu_0,\nu_1)$,
  we obtain the lemma.
\end{proof}


\begin{defn}[Length of a curve]
  \index{length}
  For a continuous curve $c : [\,a,b\,] \to X$ on a metric space $X$,
  we define the \emph{length $L(c)$ of $c$} by
  \[
  L(c) := \sup_{a = s_0 < s_1 < \dots < s_k = b} \sum_{i=1}^k d_X(c(s_{i-1}),c(s_i))
  \quad (\le +\infty).
  \]
  A curve in a metric space is said to be \emph{rectifiable}
  if the length of the curve is finite.
  \index{rectifiable}
\end{defn}

\begin{defn}[Intrinsic metric space]
  \index{intrinsic metric space} \index{length space}
  An \emph{intrinsic metric space} (or \emph{length space})
  is, by definition, a metric space such that,
  for any given two points in the space,
  the distance between them is equal to
  the infimum of the lengths of curves joining them.
  We assume that any two points in an intrinsic metric space
  has finite distance unless otherwise stated.
  In particular, any two points in an intrinsic metric space
  can always be joined by a rectifiable curve.
\end{defn}

\begin{defn}[Minimal geodesic, geodesic space]
  \index{minimal geodesic} \index{geodesic space}
  A curve $\gamma : [\,a,b\,] \to X$ on a metric space $X$
  is called a \emph{minimal geodesic} if
  $d_X(\gamma(s),\gamma(t)) = |s-t|$ for all $s,t \in [\,a,b\,]$.
  We say that a metric space $X$ is a \emph{geodesic space}
  if for any two points $x,y \in X$ there exists a minimal geodesic
  $\gamma$ joining $x$ and $y$ such that
  $d_X(x,y) = L(\gamma)$.
\end{defn}

\begin{prop}[\cite{Sturm:geoI}*{Remark 4.6(iii)}]
  If an mm-space $X$ satisfies $\CD(K,\infty)$ for some real number $K$,
  then $(P_2(X),W_2)$ and $X$ are both intrinsic metric spaces.
\end{prop}


\begin{prop} \label{prop:CD-Sep}
  If an mm-space $X$ satisfies $\CD(K,\infty)$ for a real number $K > 0$,
  then
  \begin{align}
    \tag{1}
    \Sep(X;\kappa_0,\kappa_1)
    &\le \sqrt{\frac{4}{K} \log\frac{1}{\kappa_0\kappa_1}},\\
    \tag{2}
    \ObsDiam(X;-\kappa)
    &\le \sqrt{\frac{8}{K} \log\frac{2}{\kappa}}
  \end{align}
  for any $\kappa,\kappa_0,\kappa_1 > 0$ with $\kappa_0+\kappa_1 < 1$.
\end{prop}

\begin{proof}
  (2) follows from (1) and Proposition \ref{prop:ObsDiam-Sep}.

  We prove (1).
  Let $A_0,A_1 \subset X$ be any two Borel subsets with
  $\mu_X(A_i) \ge \kappa_i$, $i=0,1$, and let
  $\nu_i := \mu_X(A_i)^{-1} \mu_X|_{A_i}$.
  By $\CD(K,\infty)$, for any $\varepsilon > 0$
  there is a measure $\nu_{1/2} \in P_2(X)$ such that
  \[
  \Ent(\nu_{1/2}|\mu_X) \le \frac{1}{2}\Ent(\nu_0|\mu_X)
  + \frac{1}{2}\Ent(\nu_1|\mu_X)
  - \frac{1}{8}K\,W_2(\nu_0,\nu_1)^2 + \varepsilon.
  \]
  We have $\Ent(\nu_i|\mu_X) = \log(1/\mu_X(A_i)) \le \log(1/\kappa_i)$
  for $i=0,1$.
  Jensen's inequality implies $\Ent(\nu_{1/2}|\mu_X) \ge 0$.
  Therefore,
  \[
  0 \le \frac{1}{2}\log\frac{1}{\kappa_0} + \frac{1}{2}\log\frac{1}{\kappa_1}
  - \frac{1}{8} K\,W_2(\nu_0,\nu_1)^2 + \varepsilon
  \]
  and, by the arbitrariness of $\varepsilon$,
  \[
  d_X(A_0,A_1)^2 \le W_2(\nu_0,\nu_1)^2
  \le \frac{4}{K} \log\frac{1}{\kappa_0\kappa_1}.
  \]
  This completes the proof.
\end{proof}

The following corollary is a direct consequence of
Proposition \ref{prop:CD-Sep}.

\begin{cor} \label{cor:CD-Sep}
  Let $X_n$, $n=1,2,\dots$, be mm-spaces.
  If each $X_n$ satisfies $\CD(K_n,\infty)$ for some sequence
  of real numbers $K_n \to +\infty$, then
  $\{X_n\}$ is a L\'evy family.
\end{cor}

\section{Stability of curvature-dimension condition}
\label{sec:stab-CD}

A main purpose of this section is to prove that
the curvature-dimension condition is stable
under concentration.




\begin{lem} \label{lem:dH-me}
  Let $p, q : X \to Y$ be two Borel measurable maps from an mm-space $X$
  to a metric space $Y$.
  Then we have
  \[
  d_H(p^*\Lip_1(Y),q^*\Lip_1(Y)) \le \dKF(p,q).
  \]
\end{lem}

\begin{proof}
  Let $f \in \Lip_1(Y)$ be any function.
  Since $|f(p(x)) - f(q(x))| \le d_Y(p(x),q(x))$ for any $x \in X$,
  we have
  \[
  \mu_X(|f\circ p - f\circ q| > \varepsilon)
  \le \mu_X(d_Y(p,q) > \varepsilon)
  \]
  for any $\varepsilon \ge 0$, which implies that
  \[
  \dKF(f\circ p,f\circ q) \le \dKF(p,q).
  \]
  This proves the lemma.
\end{proof}

\begin{prop} \label{prop:pq}
  Let $X_n$ and $Y$ be mm-spaces and let $p_n,q_n : X_n \to Y$
  be Borel measurable maps, $n=1,2,\dots$, such that
  $\dKF(p_n,q_n) \to 0$ as $n\to\infty$.
  Then we have the following {\rm(1)} and {\rm(2)}.
  \begin{enumerate}
  \item If each $p_n$ enforces $\varepsilon_n$-concentration of $X_n$ to $Y$
    for some $\varepsilon_n \to 0$,
    then so does $q_n$.
  \item If $(p_n)_*\mu_{X_n}$ converges weakly to $\mu_Y$ as $n\to\infty$,
    then so does $(q_n)_*\mu_{X_n}$.
  \end{enumerate}
\end{prop}

\begin{proof}
  (1) follows from Lemma \ref{lem:dH-me}
  and (2) from Lemma \ref{lem:di-me}.
\end{proof}

\begin{defn}[Bounded values on exceptional domains]
  \index{bounded values on exceptional domains}
  Let $X_n$, $n=1,2,\dots$, be mm-spaces and $Y$ a metric space.
  Let $p_n : X_n \to Y$, $n=1,2,\dots$, be Borel measurable maps such that
  each $p_n$ is $1$-Lipschitz up to an additive error $\varepsilon_n$
  with $\varepsilon_n \to 0$.
  We say that $\{p_n\}$ \emph{has bounded values on exceptional domains}
  if we have
  \[
  \limsup_{n\to\infty} \sup_{x \in X \setminus \tilde X_n} d_Y(p_n(x),y_0) < +\infty
  \]
  for a point $y_0 \in Y$,
  where $\tilde X_n$ is a non-exceptional domain of $p_n$
  for the additive error $\varepsilon_n$.
\end{defn}

Note that if $\diam Y < +\infty$,
then $\{p_n\}$ always has bounded values on exceptional domains.

\begin{prop} \label{prop:bdd-exc-dom}
  Assume that a sequence of mm-spaces $X_n$, $n=1,2,\dots$,
  concentrates to an mm-space $Y$.
  Then, there exist Borel measurable maps $p_n : X_n \to Y$, $n=1,2,\dots$,
  such that
  \begin{enumerate}
  \item each $p_n$ enforces $\varepsilon_n$-concentration of $X_n$ to $Y$
    for some $\varepsilon_n \to 0$,
  \item $(p_n)_*\mu_{X_n}$ converges weakly to $\mu_Y$ as $n\to\infty$,
  \item $\{p_n\}$ has bounded values on exceptional domains.
  \end{enumerate}
\end{prop}

Note that (1) implies that $p_n$ is $1$-Lipschitz up to an
additive error $\varepsilon_n \to 0$, which defines
an exceptional domain of $p_n$ for (3).

\begin{proof}
  By Corollary \ref{cor:enforce}, there is a sequence of
  Borel measurable maps $p_n' : X_n \to Y$, $n=1,2,\dots$,
  $\varepsilon_n'$-enforcing concentration of $X_n$ to $Y$ with
  $\varepsilon_n' \to 0$ such that
  $(p_n')_*\mu_{X_n}$ converges weakly to $\mu_Y$ as $n\to\infty$.
  By Theorem \ref{thm:pn}, $p_n'$ is $1$-Lipschitz up to
  some additive error $\varepsilon_n \to 0$ with a non-exceptional domain
  $\tilde{X}_n \subset X_n$.
  Define a map $p_n : X_n \to X$ by
  \[
  p_n(x) :=
  \begin{cases}
    p_n'(x) &\text{if $x \in \tilde{X}_n$},\\
    y_0 &\text{if $x \in X_n \setminus \tilde{X}_n$},
  \end{cases}
  \]
  for $x \in X$, where $y_0 \in Y$ is a fixed point.
  Each $p_n$ is Borel measurable, $1$-Lipschitz up to $\varepsilon_n$
  with the non-exceptional domain $\tilde{X}_n$, and satisfies
  $\dKF(p_n,p_n') \le \varepsilon_n$.
  (3) follows from the definition of $p_n$.
  Proposition \ref{prop:pq} proves (1) and (2).
  This completes the proof.
\end{proof}

\begin{defn}[$X^D$]
  \index{XD@$X^D$}
  For an mm-space $X = (X,d_X,\mu_X)$ and for a real number $D > 0$,
  we define an mm-space $X^D$ to be $(X,d_{X^D},\mu_X)$,
  where $d_{X^D}(x,x') := \min\{d_X(x,x'),D\}$ for $x,x' \in X$.
\end{defn}

For a Borel subset $B$ of an mm-space $X$ with $\mu_X(B) > 0$, we set
\[
\mu_B := \frac{\mu_X|_{B}}{\mu_X(B)}.
\]

\begin{lem} \label{lem:B0B1}
  Let $p_n : X_n \to Y$ be Borel measurable maps
  between mm-spaces $X_n$ and $Y$, $n=1,2,\dots$, such that
  each $p_n$ enforces $\varepsilon_n'$-concentration of $X_n$ to $Y$
  with $\varepsilon_n' \to 0$,
  and that $(p_n)_*\mu_{X_n}$ converges weakly to $\mu_Y$ as $n \to \infty$.
  For a real number $\delta > 0$,
  we give two Borel subsets $B_0, B_1 \subset Y$ such that
  \[
  \diam B_i \le \delta, \quad
  \mu_Y(B_i) > 0, \quad\text{and}\quad \mu_Y(\partial B_i) = 0
  \]
  for $i = 0,1$, and set
  \[
  \tilde{B}_i := p_n^{-1}(B_i) \cap \tilde{X}_n,
  \]
  where $\tilde{X}_n$ is a non-exceptional domain of $p_n$ for an additive error
  $\varepsilon_n \to 0$ as $n\to\infty$.
  Then, there exist Borel probability measures $\tilde\mu_0^n,\tilde\mu_1^n$
  on $X_n$
  and transport plans $\tilde\pi^n$ between $\tilde\mu_0^n$ and $\tilde\mu_1^n$,
  $n = 1,2,\dots$, such that, for every sufficiently large $n$,
  \begin{enumerate}
  \item $\tilde\mu_i^n \le (1+O(\delta^{1/2}))\mu_{\tilde{B_i}}$,
    where $O(\cdots)$ is a Landau symbol,
  \item for any $x_i \in \tilde{B}_i$, $i=0,1$,
    \[
    d_{X_n}(x_0,x_1) \ge d_Y(B_0,B_1) - \varepsilon_n,
    \]
  \item $\supp\tilde\pi^n \subset \{d_{X_n} \le d_Y(B_0,B_1) + \delta^{1/2}\}$,
  \item $-\varepsilon_n \le W_p(\tilde\mu_0^n,\tilde\mu_1^n)
    - d_Y(B_0,B_1) \le \delta^{1/2}$ for any $p \ge 1$.
  \end{enumerate}
\end{lem}

\begin{proof}
  Since $p_n$ is $1$-Lipschitz up to $\varepsilon_n$ with
  the non-exceptional domain $\tilde{X}_n \subset X_n$,
  we have $\mu_{X_n}(X_n \setminus \tilde{X}_n) \le \varepsilon_n$ and
  \[
  d_Y(p_n(x),p_n(x')) \le d_{X_n}(x,x') + \varepsilon_n
  \]
  for any $x,x' \in \tilde{X}_n$.
  This implies (2).
  Put
  \[
  D := \sup_{y_0\in B_0,\; y_1\in B_1} d_Y(y_0,y_1) + \delta^{1/2}
  \]
  and let $f \in \Lip_1(X_n^D)$ be any function.
  Note that $f \in \Lip_1(X_n)$.
  We take any number $\kappa > 0$.
  Since Theorem \ref{thm:pn} implies that
  $\{p_n\}$ effectuates relative concentration of $X_n$ over $Y$,
  we have
  $\ObsDiam(p_n^{-1}(B_i);-\kappa) < 2\delta$
  and so
  \[
  \diam(f_*(\mu_{X_n}|_{p_n^{-1}(B_i)});\mu_{X_n}(p_n^{-1}(B_i))-\kappa) < 2\delta
  \]
  for all sufficiently large $n$ and for $i=0,1$.
  There is a number $c_i \in f(X_n)$ such that
  \[
  \mu_{X_n}(\{|f-c_i| \le \delta\} \cap p_n^{-1}(B_i))
  \ge \mu_{X_n}(p_n^{-1}(B_i)) - \kappa.
  \]
  Note that $\mu_{X_n}(p_n^{-1}(B_i))$ and $\mu_{X_n}(\tilde{B}_i)$ both
  converge to $\mu_Y(B_i)$ as $n\to\infty$.
  Since $\mu_{X_n}(X_n \setminus \tilde{X}_n) \le \varepsilon_n$,
  the above inequality leads us to
  \[
  \mu_{X_n}(\{|f-c_i| \le \delta\} \cap \tilde{B}_i)
  \ge \mu_{X_n}(\tilde{B}_i) - \kappa - 2\varepsilon_n,
  \]
  so that, by chosing $\kappa$ small enough,
  \[
  \mu_{\tilde{B}_i}(|f-c_i| \le \delta)
  \ge 1-\frac{\kappa+2\varepsilon_n}{\mu_{X_n}(\tilde{B}_i)}
  > 1- \frac{\delta}{D}
  \]
  for all sufficiently large $n$ and for $i=0,1$.
  It follows from $f \in \Lip_1(X_n^D)$ that
  $|f-c_i| \le D$ on $X_n$.  We have
  \begin{align} \label{eq:int-ci}
    \left| \int_{X_n} f \,d\mu_{\tilde{B}_i} - c_i \right|
    &\le \int_{\{|f-c_i| \le \delta\}} |f-c_i| \,d\mu_{\tilde{B}_i}
    + \int_{\{\delta < |f-c_i| \le D\}} |f-c_i| \,d\mu_{\tilde{B}_i}\\
    &\le \delta + D \cdot \frac{\delta}{D} = 2\delta. \notag
  \end{align}
  Since $p_n$ enforces concentration of $X_n$ to $Y$,
  we may assume that $d_H(\Lip_1(X_n),p_n^*\Lip_1(Y)) < \varepsilon_n$.
  For the $f$ there is $g \in \Lip_1(Y)$ such that
  $\dKF(g\circ p_n,f) \le \varepsilon_n$.
  Letting $A := \{|g\circ p_n - f| \le \varepsilon_n\}$
  we have $\mu_{X_n}(A) \ge 1-\varepsilon_n$.
  Since
  $\mu_{X_n}(\{|f-c_i| \le \delta\} \cap \tilde{B}_i)
  \ge \mu_{X_n}(\tilde{B}_i)(1-\delta/D) > \varepsilon_n$
  for all sufficiently large $n$,
  the intersection
  $A \cap \{|f-c_i|\le\delta\} \cap \tilde{B}_i$ is nonempty
  for $i = 0,1$.
  We take a point $x_i$ in this set and put $y_i := p_n(x_i)$.
  It then holds that
  $|g(y_i)-f(x_i)| \le \varepsilon_n$, $|f(x_i)-c_i| \le \delta$,
  and $y_i \in B_i$.
  Therefore, setting $d_{01} := d_Y(B_0,B_1)$ we obtain
  \begin{align*}
    &\left| \int_{X_n} f \,d\mu_{\tilde{B}_0} - \int_{X_n} f \,d\mu_{\tilde{B}_1} \right|
    \le |c_0-c_1| + 4\delta
    \le |f(x_0)-f(x_1)| + 6\delta\\
    &\le | g(y_0)-g(y_1) | + 2\varepsilon_n + 6\delta
    \le d_Y(y_0,y_1) + 2\varepsilon_n + 6\delta\\
    &\le d_{01} + 2\varepsilon_n + 8\delta < d_{01} + 9\delta
  \end{align*}
  for every sufficiently large $n$.
  Note that the above estimate is uniform for all $f \in \Lip_1(X_n^D)$.
  It follows from the Kantorovich-Rubinstein duality
  (Theorem \ref{thm:KR-duality}) that
  \begin{align}
    \label{eq:W1-upper}
    \hat W_1(\mu_{\tilde{B}_0},\mu_{\tilde{B}_1}) \le d_{01} + 9\delta,
  \end{align}
  where $\hat W_1$ denotes the $L_1$-Wasserstein metric on $P_1(X_n^D)$.
  Take an optimal transport plan $\pi$ for
  $\hat W_1(\mu_{\tilde{B}_0},\mu_{\tilde{B}_1})$.
  Note that $d_{01} + \delta^{1/2} \le D$.
  Setting $\xi := d_{X_n^D} - d_{01}$ we have, by \eqref{eq:W1-upper} and
  (2),
  \begin{align*}
    9\delta &\ge \hat W_1(\mu_{\tilde{B}_0},\mu_{\tilde{B}_1}) - d_{01}
    = \int_{X_n\times X_n} \xi\,d\pi\\
    &= \int_{\{\xi < \delta^{1/2}\}} \xi\,d\pi
    + \int_{\{\xi \ge \delta^{1/2}\}} \xi\,d\pi
    \ge -\varepsilon_n + \delta^{1/2} \pi(\xi \ge \delta^{1/2})
  \end{align*}
  and so $\pi(d_{X_n} \ge d_{01} + \delta^{1/2}) = \pi(\xi \ge \delta^{1/2})
  \le 10\delta^{1/2}$ for every sufficiently large $n$.
  We put
  \begin{align*}
    V_n &:= \pi(d_{X_n} \le d_{01} + \delta^{1/2}),\\
    \tilde\pi^n &:= V_n^{-1} \pi|_{\{d_{X_n} \le d_{01} + \delta^{1/2}\}},\\
    \tilde\mu_i^n &:= (\proj_i)_*\tilde\pi^n,
  \end{align*}
  where $\proj_0 : X_n \times X_n \to X_n$ is the first projection and
  $\proj_1 : X_n \times X_n \to X_n$ the second.
  (3) follows from the definition of $\tilde\pi^n$.
  Since $\tilde\pi^n \le V_n^{-1}\pi$, we have
  \[
  \tilde\mu_i^n \le V_n^{-1}\mu_{\tilde{B}_i}
  \le (1-10\delta^{1/2})^{-1} \mu_{\tilde{B}_i},
  \]
  which implies (1).
  (4) is derived from (2) and (3).
  This completes the proof.
\end{proof}

Let $\theta(\cdot) : \R \to \R$ be a function such that
$\theta(\varepsilon) \to 0$ as $\varepsilon \to 0$,
and $\theta(\cdot|\alpha_1,\alpha_2,\dots) : \R \to \R$ a function
depending on $\alpha_1,\alpha_2,\dots$
such that $\theta(\varepsilon|\alpha_1,\alpha_2,\dots) \to 0$
as $\varepsilon \to 0$.
We use $\theta(\cdots)$ like the Landau symbols.

\begin{lem} \label{lem:CD-conc}
  Let $\{X_n\}$ be a sequence of mm-spaces satisfying $\CD(K,\infty)$
  for a real number $K$, and $Y$ an mm-space.
  Assume that a sequence of Borel measurable maps
  $p_n : X_n \to Y$, $n=1,2,\dots$, satisfies {\rm(1)}, {\rm(2)},
  and {\rm(3)} of Proposition \ref{prop:bdd-exc-dom}.
  Then, for any $\nu_0,\nu_1 \in P^{cb}(Y)$
  and any $t \in (\,0,1\,)$, there exist measures $\tilde\nu_t^n \in P(X_n)$,
  $n=1,2,\dots$, such that
  \begin{align}
    \label{eq:lem-CD1}
    \limsup_{n\to\infty} W_2((p_n)_*\tilde\nu_t^n,\nu_i)
    &\le t^{1-i}(1-t)^i W_2(\nu_0,\nu_1),
    \quad i=0,1,\\
    \label{eq:lem-CD2}
    \limsup_{n\to\infty} \Ent((p_n)_*\tilde\nu_t^n|(p_n)_*\mu_{X_n})
    &\le (1-t)\Ent(\nu_0|\mu_Y) + t\Ent(\nu_1|\mu_Y)\\
    &\quad -\frac{1}{2} K t(1-t)W_2(\nu_0,\nu_1)^2.\notag
  \end{align}
\end{lem}

\begin{proof}
  We take any $\nu_0,\nu_1 \in P^{cb}(Y)$ and fix them.
  For any natural number $m$,
  there are finitely many mutually disjoint Borel
  subsets $B_j \subset Y$, $j=1,2,\dots,J$, such that
  $\bigcup_{j=1}^J B_j = \supp\nu_0 \cup \supp\nu_1$,
  $\diam B_j \le m^{-1}$, $\mu_Y(B_j) > 0$,
  and $\mu_Y(\partial B_j) = 0$ for any $j$.
  Take a point $y_j \in B_j$ for each $j$.
  For each pair of $j,k = 1,2,\dots,J$,
  we apply Lemma \ref{lem:B0B1} for $B_j$ and $B_k$ to obtain
  measures $\tilde\mu_{jk}^{mn} \in P^{cb}(X_n)$, $n=1,2,\dots$, such that,
  for every sufficiently large $n$,
  \begin{align}
    \label{eq:tildemujk1}
    &\qquad\tilde\mu_{jk}^{mn} \le (1+\theta(m^{-1}))\mu_{\tilde{B}_j},\\
    \label{eq:tildemujk2}
    &|W_2(\tilde\mu_{jk}^{mn},\tilde\mu_{kj}^{mn})-d_Y(y_j,y_k)|
    \le \theta(m^{-1}),
  \end{align}
  where $\tilde B_j := p_n^{-1}(B_j) \cap \tilde{X}_n$.
  Replacing $\{n\}$ by a subsequence, we may assume that,
  as $n\to\infty$, $(p_n)_*\tilde\mu_{jk}^{mn}$ converges weakly to
  a measure, say $\tilde\mu_{jk}^m \in P^{cb}(Y)$,
  because $(p_n)_*\tilde\mu_{jk}^{mn}$
  is supported in the compact set $\supp\nu_0 \cup \supp\nu_1$.
  Such the subsequence of $\{n\}$ is taken to be common
  for all $j$, $k$, and $m$ by a diagonal argument.
  Let $\pi$ be an optimal transport plan for $W_2(\nu_0,\nu_1)$
  and let
  \begin{alignat*}{2}
    w_{jk} &:= \pi(B_j\times B_k), \\
    \tilde\nu_0^{mn} &:= \sum_{j,k=1}^J w_{jk} \tilde\mu_{jk}^{mn}, & \qquad
    \tilde\nu_1^{mn} &:= \sum_{j,k=1}^J w_{kj} \tilde\mu_{jk}^{mn},\\
    \tilde\nu_0^m &:= \sum_{j,k=1}^J w_{jk} \tilde\mu_{jk}^m, & \qquad
    \tilde\nu_1^m &:= \sum_{j,k=1}^J w_{kj} \tilde\mu_{jk}^m.
  \end{alignat*}
  It then holds that
  $(p_n)_*\tilde\nu_i^{mn}$ converges weakly to
  $\tilde\nu_i^m$ as $n \to \infty$ for any $m$ and $i=0,1$.
  Since $(p_n)_*\mu_{\tilde{B}_j}$ converges weakly to $\mu_{B_j}$ as $n\to\infty$
  and by \eqref{eq:tildemujk1}, we have
  \[
  \tilde\nu_0^m \le (1+\theta(m^{-1})) \sum_{j,k=1}^J w_{jk} \mu_{B_j}
  = (1+\theta(m^{-1})) \sum_{j=1}^J \nu_0(B_j) \mu_{B_j}
  \]
  as well as
  \[
  \tilde\nu_1^m \le (1+\theta(m^{-1})) \sum_{j=1}^J \nu_1(B_j) \mu_{B_j}.
  \]
  Since $\sum_{j=1}^J \nu_i(B_j) \mu_{B_j} \to \nu_i$ as $m\to\infty$,
  any weak limit of $\tilde\nu_i^m$ as $m\to\infty$ is less than
  or equal to
  $\nu_i$.  Moreover, $\tilde\nu_i^m$ and $\nu_i$ are both
  probability measures supported in the compact set
  $\supp\nu_0 \cup \supp\nu_1$
  and thus $\tilde\nu_i^m$ converges weakly to $\nu_i$
  as $m\to\infty$ for $i=0,1$.

  We next prove
  \begin{align} \label{eq:Wnu01}
    \lim_{m\to\infty}\liminf_{n\to\infty} W_2(\tilde\nu_0^{mn},\tilde\nu_1^{mn})
    = \lim_{m\to\infty}\limsup_{n\to\infty} W_2(\tilde\nu_0^{mn},\tilde\nu_1^{mn})
    = W_2(\nu_0,\nu_1)
  \end{align}
  in the following.
  Take an optimal transport plan $\tilde\pi_{jk}$ for
  $W_2(\tilde\mu_{jk}^{mn},\tilde\mu_{kj}^{mn})$ and set
  $\tilde\pi' := \sum_{j,k} w_{jk} \tilde\pi_{jk}$.
  Note that $\tilde\pi'$ is a (not necessarily optimal) transport
  plan between $\tilde\nu_0^{mn}$ and $\tilde\nu_1^{mn}$.
  By \eqref{eq:tildemujk2} we have, for every sufficiently large $n$,
  \begin{align*}
    W_2(\tilde\nu_0^{mn},\tilde\nu_1^{mn})^2
    &\le \int_{X_n \times X_n} d_{X_n}^2 \,d\tilde\pi'
    = \sum_{j,k} w_{jk} \int_{X_n \times X_n} d_{X_n}^2 \,d\tilde\pi_{jk}\\
    &= \sum_{j,k} w_{jk} W_2(\tilde\mu_{jk}^{mn},\tilde\mu_{kj}^{mn})^2\\
    &\le \sum_{j,k} w_{jk} (d_Y(y_j,y_k) + \theta(m^{-1}))^2\\
    &\le \sum_{j,k} w_{jk} (d_Y(y_j,y_k)^2 + \theta(m^{-1})d_Y(y_j,y_k))
    + \theta(m^{-1})\\
    &\le W_2(\nu_0,\nu_1)^2 + \theta(m^{-1}) W_2(\nu_0,\nu_1) + \theta(m^{-1}),
  \end{align*}
  where the last inequality follows from
  $\diam B_j \le m^{-1}$, the definition of $w_{jk}$,
  and the Schwartz inequality.

  Let us give the opposite estimate.
  We take an optimal transport plan $\tilde\pi$ for
  $W_2(\tilde\nu_0^{mn},\tilde\nu_1^{mn})$.
  Since $\tilde{\nu}_i^{mn}(X_n \setminus \tilde{X}_n) = 0$, we see that
  \begin{align*}
    W_2((p_n)_*\tilde\nu_0^{mn},(p_n)_*\tilde\nu_1^{mn})
    &\le \int_{Y \times Y} d_Y^2 \,d(p_n\times p_n)_*\tilde\pi\\
    &= \int_{X_n \times X_n} d_Y(p_n(x),p_n(x'))^2 \,d\tilde\pi(x,x')\\
    &\le \int_{X_n \times X_n} (d_{X_n}(x,x') + \varepsilon_n)^2 \,d\tilde\pi(x,x')\\
    &\le W_2(\tilde\nu_0^{mn},\tilde\nu_1^{mn})^2
    + 2\varepsilon_n W_2(\tilde\nu_0^{mn},\tilde\nu_1^{mn}) + \varepsilon_n^2.
  \end{align*}
  Here, $W_2(\tilde\nu_0^{mn},\tilde\nu_1^{mn})$
  is uniformly bounded for all large $m$ and $n$.
  Since $\lim_{m\to\infty}\lim_{n\to\infty} (p_n)_*\tilde\nu_i^{mn}
  = \lim_{m\to\infty} \tilde\nu_i^m = \nu_i$,
  we have 
  \[
  \lim_{m\to\infty}\lim_{n\to\infty}
  W_2((p_n)_*\tilde\nu_0^{mn},(p_n)_*\tilde\nu_1^{mn})
  = W_2(\nu_0,\nu_1).
  \]
  This completes the proof of \eqref{eq:Wnu01}.

  By \eqref{eq:tildemujk1} we have
  \[
  \tilde\nu_0^{mn}
  = \sum_{j,k} w_{jk} \tilde\mu_{jk}^{mn}
  \le (1+\theta(m^{-1})) \sum_j \nu_0(B_j)\mu_{\tilde{B}_j},
  \]
  which together with the monotonicity of $v(r) := U(r)/r \; (= \log r)$
  implies that
  \begin{align*}
    &\Ent(\tilde\nu_0^{mn}|\mu_{X_n})
    = \int_{X_n} v\left(\frac{d\tilde\nu_0^{mn}}{d\mu_{X_n}} \right)
    d\tilde\nu_0^{mn}\\
    &\le \int_{X_n} v\left((1+\theta(m^{-1})) \sum_j
      \frac{\nu_0(B_j)}{\mu_{X_n}(\tilde{B}_j)}
      1_{\tilde{B}_j} \right) d\tilde\nu_0^{mn}\\
    &= \sum_j v\left((1+\theta(m^{-1}))
      \frac{\nu_0(B_j)}{\mu_{X_n}(\tilde{B}_j)}
    \right) \tilde\nu_0^{mn}(\tilde{B}_j)\\
    &\le (1+\theta(m^{-1})) \sum_j v\left((1+\theta(m^{-1}))
      \frac{\nu_0(B_j)}{\mu_{X_n}(\tilde{B}_j)}
    \right) \nu_0(B_j)\\
    &= (1+\theta(m^{-1})) \sum_j U\left(
      (1+\theta(m^{-1})) \frac{\nu_0(B_j)}{\mu_{X_n}(\tilde{B}_j)}
    \right) \mu_{X_n}(\tilde{B}_j),
  \end{align*}
  which converges to
  \[
  (1+\theta(m^{-1})) \sum_j U\left(
    (1+\theta(m^{-1})) \frac{\nu_0(B_j)}{\mu_Y(B_j)}
  \right) \mu_Y(B_j)
  \]
  as $n\to\infty$.
  Since $\nu_0(B_j)/\mu_Y(B_j)$ is dominated by
  the supremum of the density $\rho_0 := \frac{d\nu_0}{d\mu_Y}$ of $\nu_0$,
  the above reduces to
  \begin{align*}
    &\sum_j U\left( \frac{\nu_0(B_j)}{\mu_Y(B_j)} \right) \mu_Y(B_j)
    + \theta(m^{-1}|\sup\rho_0)\\
    &= \Ent(\bar\nu_0^m|\mu_Y) + \theta(m^{-1}|\sup\rho_0),
  \end{align*}
  where
  \[
  \bar\nu_i^m := \sum_j \nu_i(B_j) \mu_{B_j}.
  \]
  It follows from Jensen's inequality that
  $\Ent(\bar\nu_0^m|\mu_Y) \le \Ent(\nu_0|\mu_Y)$.
  In the same way, we obtain the estimate of $\Ent(\tilde\nu_1^{mn}|\mu_{X_n})$
  and eventually, for $i=0,1$,
  \begin{align} \label{eq:Ent-tildenui}
    \limsup_{n\to\infty} \Ent(\tilde\nu_i^{mn}|\mu_{X_n})
    \le \Ent(\nu_i|\mu_Y) + \theta(m^{-1}|\sup\rho_i).
  \end{align}
  The condition $\CD(K,\infty)$ implies that,
  for any fixed $t \in (\,0,1\,)$,
  there is a measure $\tilde\nu_t^{mn} \in P_2(X_n)$
  such that
  \begin{align}
    \label{eq:CD-W2}
    W_2(\tilde\nu_t^{mn},\tilde\nu_i^{mn})
    &\le t^{1-i}(1-t)^i W_2(\tilde\nu_0^{mn},\tilde\nu_1^{mn}) + m^{-1},
    \quad i=0,1,\\
    \label{eq:CD-Ent}
    \Ent(\tilde\nu_t^{mn}|\mu_{X_n})
    &\le (1-t)\Ent(\tilde\nu_0^{mn}|\mu_{X_n})
    + t\Ent(\tilde\nu_1^{mn}|\mu_{X_n})\\
    &\quad -\frac{1}{2} K t(1-t)W_2(\tilde\nu_0^{mn}\tilde\nu_1^{mn})^2
    + m^{-1}. \notag
  \end{align}
  Lemma \ref{lem:Ent-push} implies that
  $\Ent((p_n)_*\tilde\nu_t^{mn}|(p_n)_*\mu_{X_n})
  \le \Ent(\tilde\nu_t^{mn}|\mu_{X_n})$,
  which together with
  \eqref{eq:CD-Ent}, \eqref{eq:Ent-tildenui}, and \eqref{eq:Wnu01}
  yields that
  \begin{align}
    \label{eq:pf-CD1}
    &\limsup_{m\to\infty}\limsup_{n\to\infty}
    \Ent((p_n)_*\tilde\nu_t^{mn}|(p_n)_*\mu_{X_n})\\
    \notag &\le (1-t)\Ent(\nu_0|\mu_Y) + t\Ent(\nu_1|\mu_Y)
    -\frac{1}{2} K t(1-t)W_2(\nu_0,\nu_1)^2.
  \end{align}

  We are going to estimate $W_2((p_n)_*\tilde\nu_t^{mn},\nu_i)$ for $i = 0,1$.
  Taking an optimal transport plan $\pi$ for
  $W_2(\tilde\nu_t^{mn},\tilde\nu_i^{mn})$, we see that
  \begin{align*}
    & W_2((p_n)_*\tilde\nu_t^{mn},(p_n)_*\tilde\nu_i^{mn})^2
    \le \int_{Y \times Y} d_Y^2 \,d(p_n\times p_n)_*\pi\\
    &= \int_{X_n \times X_n} d_Y(p_n(x),p_n(x'))^2 \,d\pi(x,x')
    \intertext{and since $\tilde\nu_i^{mn}(X_n \setminus \tilde{X}_n) = 0$,}
    &\le \int_{\tilde{X}_n \times \tilde{X}_n}
    (d_{X_n}(x,x') + \varepsilon_n)^2 \,d\pi(x,x')\\
    &\quad + \int_{(X_n \setminus \tilde{X}_n) \times \tilde{X}_n} d_Y(p_n(x),p_n(x'))^2 \;
    d\pi(x,x')\\
    &\le W_2(\tilde\nu_t^{mn},\tilde\nu_i^{mn})^2
    + 2\varepsilon_n W_2(\tilde\nu_t^{mn},\tilde\nu_i^{mn})
    + \varepsilon_n^2\\
    &\quad + \int_{(X_n \setminus \tilde{X}_n) \times \tilde{X}_n} d_Y(p_n(x),p_n(x'))^2 \;
    d\pi(x,x').
  \end{align*}
  Since $\{p_n\}$ has bounded values on exceptional domains
  and since
  \[
  \tilde\nu_i^{mn}(X_n \setminus p_n^{-1}(\supp\nu_0 \cup \supp\nu_1)) = 0,
  \]
  there is a constant $D > 0$ such that
  \[
  d_Y(p_n(x),p_n(x'))^2 \le D
  \]
  for $\pi$-a.e. $(x,x') \in (X_n \setminus \tilde{X}_n) \times \tilde{X}_n$
  and hence
  \begin{align*}
    &\int_{(X_n \setminus \tilde{X}_n) \times \tilde{X}_n} d_Y(p_n(x),p_n(x'))^2 \; d\pi(x,x')\\
    &\le D \, \pi((X_n \setminus \tilde{X}_n) \times X_n) = D\,
    \tilde\nu_t^{mn}(X_n \setminus \tilde{X}_n).
  \end{align*}
  Therefore,
  \begin{align}
    \label{eq:pn-tildenut}
    &\limsup_{n\to\infty} W_2((p_n)_*\tilde\nu_t^{mn},(p_n)_*\tilde\nu_i^{mn})^2\\
    \notag &\le \limsup_{n\to\infty} (W_2(\tilde\nu_t^{mn},\tilde\nu_i^{mn})^2
    + D \,\tilde\nu_t^{mn}(X_n \setminus \tilde{X}_n)).
  \end{align}

  Let us prove that
  \begin{align}
    \label{eq:nonexc0}
    \lim_{n\to\infty} \tilde\nu_t^{mn}(X_n \setminus \tilde{X}_n) = 0.
  \end{align}
  \eqref{eq:CD-Ent} and \eqref{eq:Ent-tildenui} together imply
  $\Ent(\tilde\nu_t^{mn}|\mu_{X_n}) \le C$, where $C$ is a constant
  independent of large $m$ and $n$.
  Put
  $\tilde\rho_t := \frac{d\tilde\nu_t^{mn}}{d\mu_{X_n}}$.
  Since $U(r)/r$ is monotone increasing in $r$, we have, for any $r > 0$,
  \begin{align*}
    \tilde\nu_t^{mn}(X_n \setminus \tilde{X}_n)
    &= \int_{\{\tilde\rho_t \ge r\} \setminus \tilde{X}_n} \tilde\rho_t \,d\mu_{X_n}
    + \int_{\{\tilde\rho_t < r\} \setminus \tilde{X}_n} \tilde\rho_t \,d\mu_{X_n}\\
    &\le \frac{r}{U(r)} \int_{\{\tilde\rho_t \ge r\} \setminus \tilde{X}_n}
    U(\tilde\rho_t) \,d\mu_{X_n}
    + r \mu_{X_n}(X \setminus \tilde{X}_n).
  \end{align*}
  By $\int_{\{U(\tilde\rho_t) < 0\}} U(\tilde\rho_t)\;d\mu_{X_n} \ge \inf U$,
  we have
  \[
  \int_{\{U(\tilde\rho_t) > 0\}} U(\tilde\rho_t)\;d\mu_{X_n} \le C - \inf U
  \]
  and thus
  \[
  \tilde\nu_t^{mn}(X_n \setminus \tilde{X}_n)
  \le \frac{(C - \inf U)r}{U(r)} + r \mu_{X_n}(X_n \setminus \tilde X_n)
  \]
  for any $r > 0$ with $U(r) > 0$.
  By remarking $\lim_{r\to +\infty} r/U(r) \to 0$,
  the above inequality implies \eqref{eq:nonexc0}.

  Combining \eqref{eq:nonexc0} with \eqref{eq:pn-tildenut} yields that
  \[
  \limsup_{n\to\infty} W_2((p_n)_*\tilde\nu_t^{mn},(p_n)_*\tilde\nu_i^{mn})
  \le \limsup_{n\to\infty} W_2(\tilde\nu_t^{mn},\tilde\nu_i^{mn}).
  \]
  Since $(p_n)_*\tilde\nu_i^{mn} \overset{n\to\infty}\to
  \tilde\nu_i^m \overset{m\to\infty}\to \nu_i$,
  we have
  \begin{align*}
    &\limsup_{m\to\infty}\limsup_{n\to\infty} W_2((p_n)_*\tilde\nu_t^{mn},\nu_i)\\
    &\le \limsup_{m\to\infty}\limsup_{n\to\infty}
    W_2(\tilde\nu_t^{mn},\tilde\nu_i^{mn}),
    \intertext{and by \eqref{eq:CD-W2} and \eqref{eq:Wnu01},}
    &\le \limsup_{m\to\infty} \limsup_{n\to\infty}
    t^{1-i}(1-t)^i W_2(\tilde\nu_0^{mn},\tilde\nu_1^{mn})\\
    &= t^{1-i}(1-t)^i W_2(\nu_0,\nu_1).
  \end{align*}
  By this and \eqref{eq:pf-CD1}, there is a sequence $m(n) \to \infty$
  as $n\to\infty$ such that $\tilde\nu_t^n := \tilde\nu_t^{m(n)n}$
  satisfies \eqref{eq:lem-CD1} and \eqref{eq:lem-CD2}.
  This completes the proof.
\end{proof}

\begin{lem} \label{lem:tight}
  Let $Y$ be a proper metric space and $y_0 \in Y$ a point.
  If a sequence of measures $\nu_n \in P_p(Y)$, $n=1,2,\dots$,
  has uniformly bounded $p^{th}$ moment $W_p(\nu_n,\delta_{y_0})$
  for some real number $p$ with $1 \le p < +\infty$, then $\{\nu_n\}$ is tight.
\end{lem}

\begin{proof}
  Assume that a sequence of measures $\nu_n \in P_p(Y)$, $n=1,2,\dots$,
  satisfies $W_p(\nu_n,\delta_{y_0}) \le C$ for any $n$, for some constant $C$,
  and for some $p$ with $1 \le p < +\infty$.
  By H\"older's inequality we have, for any $R > 0$,
  \begin{align*}
    C &\ge W_p(\nu_n,\delta_{y_0}) \ge W_1(\nu_n,\delta_{y_0})\\
    &\ge \int_{Y \setminus B_R(y_0)} d_Y(y,y_0) \; d\nu_n(y)
    \ge R \,\nu_n(Y \setminus B_R(y_0)).
  \end{align*}  
  For any given $\varepsilon > 0$, setting $R := C/\varepsilon$
  and $K_\varepsilon := B_R(y_0)$ yields that
  $\nu_n(Y \setminus K_\varepsilon) \le \varepsilon$.
  This completes the proof.
\end{proof}

The following is one of main theorems of this chapter.

\begin{thm} \label{thm:CD-conc}
  Let $\{X_n\}_{n=1}^\infty$ be a sequence of mm-spaces satisfying $\CD(K,\infty)$
  for a real number $K$.
  If $X_n$ concentrates to an mm-space $Y$ as $n\to\infty$,
  then we have the following {\rm(1)}, {\rm(2)}, and {\rm(3)}.
  \begin{enumerate}
  \item $Y$ is an intrinsic metric space.
  \item $Y$ is a geodesic space satisfying $\CD(K,\infty)$
    provided that $Y$ is proper.
  \item $Y$ satisfies $\CD(K,\infty)$
    provided that $(p_n)_*\mu_{X_n} = \mu_Y$ for every $n=1,2,\dots$,
    where $\{p_n : X_n \to Y\}_{n=1}^\infty$ is a sequence of Borel measurable
    maps enforcing concentration of $X_n$ to $Y$
    and with bounded values on exceptional domains.
  \end{enumerate}
\end{thm}

\begin{proof}
  (1) follows from a standard discussion (cf.~\cite{Sturm:geoI}*{Remark 4.6})
  using Lemma \ref{lem:CD-conc}.

  (3) is derived from Lemmas \ref{lem:CD-conc} and \ref{lem:CD-Pcb}.

  We prove (2).
  Since $Y$ is proper, (1) implies that $Y$ is a geodesic space.
  Take any $\nu_0,\nu_1 \in P^{cb}(Y)$ and fix them.
  Fixing any $t \in (\,0,1\,)$,
  we have measures $\nu_t^n := (p_n)_*\tilde\nu_t^n \in P_2(Y)$,
  $n=1,2,\dots$, as in Lemma \ref{lem:CD-conc}.
  Since $W_2(\nu_t^n,\nu_i)$ is uniformly bounded for all $n$,
  Lemma \ref{lem:tight} together with Prohorov's theorem
  (Theorem \ref{thm:Proh})
  proves that $\{\nu_t^n\}$
  has a subsequence converging weakly to a measure on $Y$,
  say $\nu_t$.  By \eqref{eq:lem-CD1}, we have
  \[
  W_2(\nu_t,\nu_i) = t^{1-i}(1-t)^i W_2(\nu_0,\nu_1)
  \]
  Let $C_0$ be the union of the images of all minimal geodesic segments
  joining $y$ and $y'$, where $y$ and $y'$ run over all
  points in $\supp\nu_0$ and $\supp\nu_1$ respectively.
  Since $\supp\nu_i$, $i=0,1$, are compact and $Y$ is proper,
  $C_0$ is a compact subset of $Y$.
  Note that $\nu_t$ is supported in $C_0$.
  We set $C_r := B_r(C_0)$ for $r > 0$, which is also compact.
  Let $\rho_t^n := \frac{d\nu_t^n}{d\mu_Y}$.
  Since
  \begin{align*}
    \int_{Y \setminus C_r} U(\rho_t^n) \; d(p_n)_*\mu_{X_n}
    \ge (p_n)_*\mu_{X_n}(Y \setminus C_r) \inf U,
  \end{align*}
  we have
  \[
  \liminf_{n\to\infty} \int_{Y \setminus C_r} U(\rho_t^n) \; d(p_n)_*\mu_{X_n}
  \ge \mu_Y(\overline{Y \setminus C_r}) \inf U,
  \]
  where we note that $\inf U < 0$.
  By Lemma \ref{lem:Ent-lim},
  \begin{align*}
    \liminf_{n\to\infty} \int_{C_r} U(\rho_t^n) \; d(p_n)_*\mu_{X_n}
    \ge \int_{C_r} U\Bigl(\frac{d\nu_t}{d\mu_Y}\Bigr) \; d\mu_Y
    = \Ent(\nu_t|\mu_Y).
  \end{align*}
  We thus have
  \begin{align*}
    \liminf_{n\to\infty} \Ent(\nu_t^n|(p_n)_*\mu_{X_n})
    &= \liminf_{n\to\infty} \int_Y U(\rho_t^n) \; d(p_n)_*\mu_{X_n}\\
    &\ge \Ent(\nu_t|\mu_Y) + \mu_Y(\overline{Y \setminus C_r}) \inf U,
  \end{align*}
  where $\mu_Y(\overline{Y \setminus C_r}) \to 0$ as $r \to +\infty$.
  This together with \eqref{eq:lem-CD2} implies
  \[
  \Ent(\nu_t|\mu_Y) \le (1-t)\Ent(\nu_0|\mu_Y) + t\Ent(\nu_1|\mu_Y)
    -\frac{1}{2} K t(1-t)W_2(\nu_0,\nu_1)^2.
  \]
  The proof of the theorem is now completed.
\end{proof}


\section{$k$-L\'evy family}
\label{sec:k-Levy}

\begin{defn}[$k$-L\'evy family]
  \index{k-Levy family@$k$-L\'evy family}
  Let $k$ be a natural number.
  A sequence $\{X_n\}_{n=1}^{\infty}$ of mm-spaces is called
  a \emph{$k$-L\'evy family}
  if
  \[
  \lim_{n\to\infty}\Sep(X_n;\kappa_0,\cdots,\kappa_k)=0
  \]
  for any $\kappa_0,\kappa_1,\cdots,\kappa_k >0$.
\end{defn}

It follows from Proposition \ref{prop:ObsDiam-Sep} that
a sequence of mm-spaces is a $1$-L\'evy family
if and only if it is a L\'evy family.
For $k \le k'$, any $k$-L\'evy family is a $k'$-L\'evy family.
A $k$-L\'evy family is a union of
$k$ number of L\'evy families.


The following is a direct consequence of Proposition \ref{prop:lamk-Sep}.

\begin{prop}\label{prop:k-Levy-lam}
  Let $X_n$, $n=1,2,\dots$, be closed Riemannian manifolds.
  If $\lambda_k(X_n)$ diverges to infinity as $n\to \infty$
  for a natural number $k$,
  then $\{X_n\}$ is a $k$-L\'evy family.
\end{prop}

\begin{thm} \label{thm:k-Levy-conc}
  Let $\{X_n\}_{n=1}^{\infty}$ be a $k$-L\'evy family of mm-spaces
  for a natural number $k$.
  Then we have one of the following {\rm(1)} and {\rm(2)}.
  \begin{enumerate}
  \item $\{X_n\}$ is a L\'evy family.
  \item There exist a subsequence $\{X_{n_i}\}_{i=1}^\infty$ of
    $\{X_n\}_{n=1}^{\infty}$ and a sequence of numbers $t_i$ with $0 < t_i \le 1$
    such that, as $i\to\infty$, $t_i X_{n_i}$ 
    concentrates to a finite mm-space $Y$ with $2 \le \# Y \le k$.
  \end{enumerate}
\end{thm}

Recall that $tX$ indicates the scaled mm-space $(X,td_X,\mu_X)$
for $t > 0$.

\begin{proof}
  Assume that $\{X_n\}_{n=1}^{\infty}$ is a $k$-L\'evy family and
  is not a L\'evy family.
  We may assume that $k$ is the minimal number such that
  $\{X_n\}$ is a $k$-L\'evy family.
  We have $k \ge 2$.
  Taking a subsequence of $\{X_n\}$, we have
  \begin{align}
    \Sep(X_n;\kappa_0,\cdots, \kappa_{k-1}) > c
  \end{align}
  for every $n$ and for some numbers
  $c,\kappa_0,\kappa_1,\cdots,\kappa_{k-1} > 0$.
  There are Borel subsets $A_{0n},\cdots , A_{k-1,n}\subset X_n$
  such that $\mu_{X_n}(A_{in})\geq \kappa_i$ and
  $d_{X_n}(A_{in},A_{jn}) > c$ for any different $i$ and $j$.
  For simplicity we set
  \begin{align*}
    S_n(\kappa) &:= \Sep(X_n;\kappa_0,\cdots, \kappa_{k-1},\kappa)
  \end{align*}
  for $\kappa > 0$.
  Since $\{X_n\}$ is a $k$-L\'evy family,
  there is a sequence of positive numbers $\tilde\kappa_n \to 0$
  such that $S_n(\tilde\kappa_n)$ converges to zero as $n\to\infty$.
  Let
  \begin{align*}
    \quad r_n := \max\{S_n(\tilde{\kappa}_n),1/n\}
    \quad\text{and}\quad B_{in}:=B_{2r_n}(A_{in}).
  \end{align*}
  Note that $S_n(\tilde{\kappa}_n)$ may be zero, but $r_n$ is positive
  and still converges to zero as $n\to\infty$.
  Since the distance between $X_n \setminus \bigcup_{i=0}^{k-1}B_{in}$
  and $A_{jn}$, $j=0,\dots,k-1$, is strictly greater than
  $S_n(\tilde{\kappa}_n)$,
  we have
  \begin{align} \label{eq:k-Levy-conc-1}
    \mu_{X_n}\Big(X_n \setminus \bigcup_{i=0}^{k-1}B_{in}\Big) <
    \tilde{\kappa}_n
  \end{align}
  for every $n$.

  \begin{clm} \label{clm:Bin-Levy}
    For each $i=0,1,\dots,k-1$, the sequence
    $\{(B_{in},d_{X_n},\mu_{in})\}_{n=1}^{\infty}$ is a
    L\'evy family, where $\mu_{in} := \mu_{X_n}|_{B_{in}}$.
  \end{clm}

  \begin{proof}
    Take any numbers $\kappa,\kappa' > 0$ and fix them.
    Since $\{X_n\}$ is a $k$-L\'evy family,
    \begin{align*}
      \alpha_n := \Sep(X_n;\kappa_0,\kappa_1, \cdots, \kappa_{i-1},
      \kappa,\kappa', \kappa_{i+1}, \cdots,
      \kappa_{k-1})
    \end{align*}
    converges to zero as $n\to\infty$, so that
    \[
    \alpha_n < c-4r_n \le \min_{j\neq j'}d_{X_n}(B_{jn},B_{j'n})
    \]
    for every sufficiently large $n$.
    For the claim, it suffices to prove that
    $\Sep(B_{in};\kappa,\kappa') \leq \alpha_n$.
    In fact, if $\Sep(B_{in};\kappa,\kappa') > \alpha_n$,
    then we have two Borel subsets $B_{in}', B_{in}'' \subset B_{in}$
    such that $\mu_{X_n}(B_{in}') \ge \kappa$, $\mu_{X_n}(B_{in}'')
    \ge \kappa'$, and $d_{X_n}(B_{in}',B_{in}'') > \alpha_n$.
    By remarking that
    $\mu_{X_n}(B_{jn})\geq \kappa_j$ for every $j$,
    this is a contradiction to the definition of $\alpha_n$.
  \end{proof}

  We set, for $i,j = 0,1,\dots,k-1$,
  \begin{align*}
    d_{ijn} &:=
    \sup\{\; |\lm (f;\mu_{in})- \lm (f;\mu_{jn})| ; f \in \Lip_1(X_n)\;\},\\
    D_n &:= \max_{i,j=0,1,\dots,k-1} d_{ijn}.
  \end{align*}
  Since $d_{X_n}(B_{in},B_{jn})$ is bounded away from zero for $i \neq j$,
  we have $\liminf_{n\to\infty} D_n > 0$.
  Define
  \[
  t_n :=
  \begin{cases}
    1/D_n & \text{if $D_n > 1$},\\
    1 & \text{if $D_n \le 1$}.
  \end{cases}
  \]
  It is clear that $0 < t_n \le 1$ and $0 \le t_nd_{ijn} \le 1$.
  Taking a subsequence of $\{n\}$, we see that, as $n\to\infty$,
  $t_n d_{ijn}$ converges to a number, say $d_{ij} \in [\,0,1\,]$,
  and $\mu_{X_n}(B_{in})$ to a number, say $w_i \in (\,0,1\,)$, for any $i,j$.
  We have $d_{ii} = 0$, $d_{ij} \ge 0$, $d_{ji} = d_{ij}$ for any $i,j$,
  and $\max_{i,j} d_{ij} > 0$.
  We prove that $d_{ij} \le d_{il} + d_{lj}$ for any $i, j, l$.
  In fact, for any $\varepsilon > 0$
  there is a function $f \in \Lip_1(X_n)$ such that
  \[
  |\; |\lm(f;\mu_{in}) - \lm(f;\mu_{jn})| - d_{ijn} \;|
  < \varepsilon.
  \]
  We have
  \begin{align*}
    d_{iln} &\ge |\lm(f;\mu_{in}) - \lm(f;\mu_{ln})|,\\
    d_{ljn} &\ge |\lm(f;\mu_{ln}) - \lm(f;\mu_{jn})|,
  \end{align*}
  and therefore
  $d_{ijn} -\varepsilon < d_{iln} + d_{ljn}$,
  which implies $d_{ij} \le d_{il} + d_{lj}$.

  Let $Y = \{y_0,y_1,\dots,y_{k-1}\}$ be a set consisting of
  $k$ elements and define
  $d_Y(y_i,y_j) := d_{ij}$.
  Then, $(Y,d_Y)$ is a pseudometric space.
  We define the measure $\mu_Y := \sum_{i=0}^{k-1} w_i \delta_{y_i}$ on $Y$.
  It follows from \eqref{eq:k-Levy-conc-1} that $\mu_Y$ is a
  probability measure.
  Let $Y'$ be the quotient space of $Y$ by the equivalence relation
  $d_Y = 0$, and let $[y_i] \in Y'$ denote
  the equivalence class represented by $y_i \in Y$.
  The pseudometric $d_Y$ induces a metric on $Y'$, say $d_{Y'}$.
  Let $\mu_{Y'}$ be the push-forward of $\mu_Y$
  by the projection $Y \to Y'$.
  It then follows from $\max_{i,j}d_{ij} > 0$ that
  $(Y',d_{Y'},\mu_{Y'})$ is an mm-space with $2 \le \# Y' \le k$.
  Let us prove that $X_n$ concentrates $Y'$ as $n\to\infty$.
  Define a map $p_n:X_n \to Y'$ by
  \[
  p_n(x) :=
  \begin{cases}
    [y_i] & \text{if $x \in B_{in}$},\\
    [y_0] & \text{if $x \in X_n \setminus \bigcup_{i=0}^{k-1}B_{in}$}.
  \end{cases}
  \]
  It is obvious that each $p_n$ is a Borel measurable map, and
  that $(p_n)_*\mu_{X_n}$ converges weakly to $\mu_Y$ as $n\to\infty$.
  It suffices to prove that
  $\{p_n\}$ satisfies (1), (2), and (3) of Theorem \ref{thm:pn}.

  We prove (1), i.e., $p_n$ is $1$-Lipschitz up to an additive error
  tending to zero as $n\to\infty$.
  There are $f_{ijn} \in \Lip_1(X_n)$ such that
  \[
  d_{ij} = \lim_{n\to\infty}
  t_n |\lm(f_{ijn};\mu_{in})-\lm(f_{ijn};\mu_{jn})|.
  \]
  We find a sequence $\varepsilon_n \to 0$ such that for any $i$ and $j$,
  \begin{align*}
    |\;d_{ij} - t_n|\lm(f_{ijn};\mu_{in})-\lm(f_{ijn};\mu_{jn})|\;|
    \le \varepsilon_n.
  \end{align*}
  It follows from Claim \ref{clm:Bin-Levy} and Lemma \ref{lem:LeRad-ObsDiam}
  that there is a sequence $\varepsilon_n' \to 0$
  in such a way that, for each $i$, the $\mu_{X_n}$-measure of
  \begin{align*}
    B_{in}' := \{\; x \in B_{in} \; ; \;
    & |f_{ijn}(x)-\lm(f_{ijn};\mu_{in})| < \varepsilon_n' \\
    &\ \text{for any $j$ different from $i$}\;\}
  \end{align*}
  converges to $w_i$ as $n\to\infty$.
  For any $x \in B_{in}'$ and $x' \in B_{jn}'$ we have
  \begin{align*}
    d_{Y'}(p_n(x),p_n(x')) &= d_{Y'}([y_i],[y_j]) = d_{ij}\\
    &\le t_n|\lm(f_{ijn};\mu_{in})-\lm(f_{ijn};\mu_{jn})| + \varepsilon_n\\
    &\le t_n|f_{ijn}(x)-f_{ijn}(x')| + 2t_n\varepsilon_n' + \varepsilon_n\\ 
    &\le t_n d_{X_n}(x,x') + 2t_n\varepsilon_n' + \varepsilon_n
  \end{align*}
  which implies (1).

  We prove (2), i.e., $\{p_n\}$ effectuates relative concentration of $X_n$
  over $Y'$.  Let $f_n \in \Lip_1(X_n)$ and $B \subset Y'$.
  Since $\ObsDiam(tX;-\kappa) = t\ObsDiam(X;-\kappa)$,
  it suffices to prove that
  \[
  \limsup_{n\to\infty}
  t_n \diam((f_n)_*(\mu_{X_n}|_{p_n^{-1}(B)});\mu_{X_n}(p_n^{-1}(X_n))-\kappa)
  \le \diam B.
  \]
  Let $\tilde B$ be the preimage of $B$ by the projection $Y \to Y'$
  and set $\tilde B = \{y_{i_1},\dots,y_{i_m}\}$.
  We see that
  \[
  p_n^{-1}(B) =
  \begin{cases}
    \bigcup_{l=1}^m B_{i_ln} & \text{if $y_0 \notin \tilde{B}$},\\
    \bigcup_{l=1}^m B_{i_ln} \cup (X_n \setminus \bigcup_{i=0}^k B_{in})
    & \text{if $y_0 \in \tilde{B}$}.
  \end{cases}
  \]
  Put $M_{ln} := \lm(f_n|B_{i_ln};\mu_{X_n}|_{B_{i_ln}})$.
  Since
  \[
  \LeRad(\mu_{X_n}|_{B_{i_l}};-\kappa)
  \le \ObsDiam(\mu_{X_n}|_{B_{i_l}};-\kappa) \to 0
  \ \text{as $n\to\infty$},
  \]
  we have, for any $\varepsilon > 0$,
  \begin{align*}
    &(f_n)_*(\mu_{X_n}|_{B_{i_ln}})([\,M_{ln}-\varepsilon,M_{ln}+\varepsilon\,])\\
    &= \mu_{X_n}(\{\; x \in B_{i_ln} \mid |f_n(x) - M_{ln}| \le \varepsilon\;\})
    \to w_{i_l} \ \text{as $n\to\infty$}.
  \end{align*}
  Moreover, the total measure of $(f_n)_*(\mu_{X_n}|_{p_n^{-1}(B)})$
  converges to $\sum_{l=1}^m w_{i_l}$.
  Therefore,
  \begin{align*}
  &\limsup_{n\to\infty}
  t_n \diam((f_n)_*(\mu_{X_n}|_{p_n^{-1}(B)});\mu_{X_n}(p_n^{-1}(X_n))-\kappa)\\
  &\le \limsup_{n\to\infty} \max_{1 \le l,l' \le m} t_n |M_{ln}-M_{l'n}|
  \le \limsup_{n\to\infty} \max_{1 \le l,l' \le m} t_n d_{i_l i_{l'} n}\\
  &= \max_{1 \le l,l' \le m} d_{i_l i_{l'}} = \diam B,
  \end{align*}
  which completes the proof of (2).

  For the proof of (3), we prove the following.

  \begin{clm} \label{clm:d-plus}
    For any $i,j=0,1,\dots,k-1$ and $\kappa > 0$ we have
    \begin{align*}
      \limsup_{n\to\infty} t_n d_+(B_{in},B_{jn};+\kappa) \le d_{Y'}([y_i],[y_j]).
    \end{align*}
  \end{clm}

   \begin{proof}
     We fix $i$, $j$, and $\kappa > 0$.
     There are subsets $C_n \subset B_{in}$ and
     $C_n' \subset B_{jn}$ such that
     $\mu_{X_n}(C_n)\geq \kappa$, $\mu_{X_n}(C_n')\geq \kappa$,
     and
     \[
     d_+ (B_{in},B_{jn};+\kappa) < d_{X_n}(C_n,C_n') + 1/n.
     \]
     Let $f_n : X_n \to \R$ be the distance function from $C_n$,
     i.e., $f_n(x) := d_{X_n}(x,C_n)$, $x \in X_n$.
     By Claim \ref{clm:Bin-Levy},
     there is a sequence $\varepsilon_n' \to 0$
     such that, for each $l = 0,1,\dots,k-1$, the $\mu_{X_n}$-measure of
     \begin{align*}
       D_{ln} := \{\;  x\in B_{ln} \mid
       |f_n(x)-\lm (f_n; \mu_{ln})| \le \varepsilon_n' \; \}
     \end{align*}
     converges to $w_l$ as $n\to\infty$.
     There is a natural number $N$ such that 
     the intersections $C_n \cap D_{in}$ and $C_n' \cap D_{jn}$ are both
     nonempty for any $n \ge N$.
     In what follows, let $n$ be any number with $n \ge N$.
     Taking points $x_n \in C_n \cap D_{in}$ and $x_n' \in C_n' \cap D_{jn}$,
     we have
     \begin{align*}
      |\lm (f_n;\mu_{in})|
      &= |f_n(x_n)-  \lm(f_n;\mu_{in}) | \le \varepsilon_n',\\
      \lm(f_n;\mu_{jn}) &\ge f_n(x_n') - \varepsilon_n'
      \geq d_{X_n}(C_n,C_n') - \varepsilon_n'.
    \end{align*}
    We therefore obtain
    \begin{align*}
      t_n d_+(B_{in},B_{jn};+\kappa)
      &\le t_n d_{X_n}(C_n,C_n') + t_n/n\\
      &\le t_n(\lm(f_n;\mu_{jn}) - \lm(f_n;\mu_{in}))
      + 2t_n\varepsilon_n' + t_n/n\\
      &\le t_n d_{ijn} + \varepsilon_n + 2t_n\varepsilon_n' + t_n/n\\
      &\to d_{ij} \quad\text{as $n\to\infty$},
    \end{align*}
    which implies the claim.
  \end{proof}

  Using Claim \ref{clm:d-plus} we now prove (3).
  We take $A_1, A_2 \subset Y'$ and $\kappa > 0$.
  Let $\{y_{i_1},\dots,y_{i_M}\}$ be the preimage of $A_1$ by
  the projection $Y \to Y'$ and $\{y_{j_1},\dots,y_{j_N}\}$ the preimage
  of $A_2$.  We see that $A_1 = \{[y_{i_1}],\dots,[y_{i_M}]\}$
  and $A_2 = \{[y_{j_1}],\dots,[y_{j_N}]\}$.
  Assume that $n$ is large enough.
  Since $\mu_{X_n}(X_n \setminus \bigcup_{i=0}^{k-1} B_{in}) < \kappa/2$,
  it holds that
  \begin{align*}
    & d_+(p_n^{-1}(A_1),p_n^{-1}(A_2);+\kappa)\\
    &\le d_+(B_{i_1 n} \cup \dots \cup B_{i_M n},
    B_{j_1 n} \cup \dots \cup B_{j_N n};+\kappa/2)\\
    &\le \max_{\alpha=1,\dots,M,\ \beta=1,\dots,N}
    d_+(B_{i_\alpha n},B_{j_\beta n};+\kappa/(2\max\{M,N\}))
  \end{align*}
  and by Claim \ref{clm:d-plus},
  \begin{align*}
    &\limsup_{n\to\infty} t_n d_+(p_n^{-1}(A_1),p_n^{-1}(A_2);+\kappa)\\
    &\le \max_{\alpha=1,\dots,M,\ \beta=1,\dots,N} d_{Y'}([y_{i_\alpha}],[y_{j_\beta}])\\
    &\le d_{Y'}(A_1,A_2) + \diam A_1 + \diam A_2.
  \end{align*}
  (3) has been proved.

  By Theorem \ref{thm:pn} and Corollary \ref{cor:enforce},
  this completes the proof of Theorem \ref{thm:k-Levy-conc}.
\end{proof}

\begin{cor} \label{cor:k-Levy-conc-2}
  Let $\{X_n\}_{n=1}^{\infty}$ be a $k$-L\'evy family of mm-spaces
  for a natural number $k$ such that
  \[
  \limsup_{n\to\infty} \ObsDiam(X_n;-\kappa) < +\infty
  \]
  for any $\kappa > 0$.
  Then, there exists a subsequence of $\{X_n\}$
  that concentrates to a finite mm-space $Y$ with $\# Y \le k$.
\end{cor}

\begin{proof}
  In the proof of the previous Theorem \ref{thm:k-Levy-conc},
  the assumption for the observable diameter implies that
  $d_{ijn}$ is bounded above uniformly for all $i$, $j$, and $n$,
  so that $t_n$ is bounded away from zero.
  Therefore, Theorem \ref{thm:k-Levy-conc} implies the corollary.
\end{proof}

The following is a direct consequence of Corollary \ref{cor:k-Levy-conc-2}
and Theorem \ref{thm:CD-conc}.

\begin{cor} \label{cor:lamk-conc-2}
  Let $\{X_n\}_{n=1}^\infty$ be a sequence of closed
  Riemannian manifolds with a lower bound of Ricci curvature
  and with the property that
  \[
  \limsup_{n\to\infty} \ObsDiam(X_n;-\kappa) < +\infty
  \]
  for any $\kappa > 0$.
  If $\lambda_k(X_n)$ diverges to infinity as $n \to \infty$
  for some natural number $k$, then $\{X_n\}$ is a L\'evy family.
\end{cor}

\begin{cor} \label{cor:lamk-conc}
  Let $\{X_n\}_{n=1}^\infty$ be a sequence of closed
  Riemannian manifolds of nonnegative Ricci curvature.
  If $\lambda_k(X_n)$ diverges to infinity as $n \to \infty$
  for some natural number $k$, then $\{X_n\}$ is a L\'evy family.
\end{cor}

\begin{proof}
  Let $\{X_n\}_{n=1}^\infty$ be a sequence of closed
  Riemannian manifolds of nonnegative Ricci curvature
  such that $\lambda_k(X_n)$ diverges to infinity as $n \to \infty$
  for some natural number $k$.
  By Proposition \ref{prop:k-Levy-lam}, $\{X_n\}$ is a $k$-L\'evy family.
  Suppose that $\{X_n\}$ is not a L\'evy family.
  By applying Theorem \ref{thm:k-Levy-conc}, there are
  a subsequence $\{X_{n_i}\}$ of $\{X_n\}$ and a sequence of
  numbers $t_i$ with $0 < t_i \le 1$ such that
  $t_iX_{n_i}$ concentrates to a disconnected mm-space $Y$.
  Since each $t_iX_{n_i}$ satisfies $\CD(0,\infty)$,
  so does $Y$, which is a contradiction.
  This completes the proof.
\end{proof}

It follows from Corollary \ref{cor:ObsDiam-spec} that
if $\lambda_1(X_n)$ diverges to infinity as $n\to\infty$
for a sequence of closed Riemannian manifolds $X_n$, $n=1,2,\dots$,
then it is a L\'evy family.
We have the converse under the nonnegativity of Ricci curvature.

\begin{thm}[E.~Milman; \cites{Emil:isop,Emil:role}]
  \label{thm:Levy-lam-Ric}
  If $\{X_n\}_{n=1}^{\infty}$ is a L\'evy family
  of closed Riemannian manifolds of
  nonnegative Ricci curvature,
  then $\lambda_1(X_n)$ diverges to infinity as $n \to \infty$.
\end{thm}

Combining Corollary \ref{cor:lamk-conc}
and Theorem \ref{thm:Levy-lam-Ric} yields the following equivalence
\begin{align*}
  &\text{$\{X_n\}$ is a L\'evy family}\\
  \Longleftrightarrow\  &\lambda_1(X_n) \to +\infty\\
  \Longleftrightarrow\  &\lambda_k(X_n) \to +\infty
  \ \text{for some $k$}
\end{align*}
for a sequence of closed Riemannian manifolds $X_n$, $n=1,2,\dots$,
with nonnegative Ricci curvature.

Using Theorem \ref{thm:Levy-lam-Ric}, we prove
the following theorem.

\begin{thm} \label{thm:lamk1}
  For any natural number $k$, there exists a positive constant
  $C_k$ depending only on $k$
  such that if $X$ is a closed Riemannian manifold
  of nonnegative Ricci curvature,
  then we have
  \[
  \lambda_k(X) \le C_k \lambda_1(X).
  \]
\end{thm}

\begin{proof}
  Suppose that Theorem \ref{thm:lamk1} is false.
  Then, there are a natural number $k$ and
  a sequence of closed Riemannian manifolds
  $X_n$, $n=1,2,\dots$, of nonnegative Ricci curvature
  such that $\lambda_k(X_n)/\lambda_1(X_n)$ diverges to infinity
  as $n\to\infty$. Let $X_n'$ be the scale change of $X_n$ such that
  $\lambda_1(X_n') = 1$.
  Since
  \[
  \lambda_k(X_n')
  = \frac{\lambda_k(X_n')}{\lambda_1(X_n')}
  = \frac{\lambda_k(X_n)}{\lambda_1(X_n)}
  \to +\infty \quad\text{as $n\to\infty$},
  \]
  and by Corollary \ref{cor:lamk-conc}, the sequence $\{X_n'\}$
  is a L\'evy family.
  By Theorem \ref{thm:Levy-lam-Ric}, $\lambda_1(X_n')$ must be divergent
  to infinity, which is a contradiction.
  This completes the proof.
\end{proof}

\begin{ex} \label{ex:k-spheres}
  For any natural number $k \ge 2$, we give an example
  of a sequence of closed Riemannian manifolds $X_n$, $n=1,2,\dots$,
  such that, as $n\to\infty$,
  \begin{enumerate}
  \item $\lambda_{k-1}(X_n)$ converges to zero,
  \item $\lambda_k(X_n)$ diverges to infinity,
  \item $X_n$ concentrates to a finite mm-space $Y$ with $\# Y = k$
    and, in particular, $\{X_n\}$ is not a L\'evy family.
  \end{enumerate}
  
  Such a sequence $\{X_n\}$ is constructed as follows.
  Let $S_1^n,S_2^n,\dots,S_k^n$ be the $k$ copies of
  an $n$-dimensional unit sphere in a Euclidean space,
  and $p_i^n,q_i^n \in S_i^n$ two points that are antipodal to each other
  for each $i$,
  i.e., the geodesic distance between $p_i^n$ and $q_i^n$
  is equal to $\pi$.
  Let us consider the connected sum of $S_1^n,\dots,S_k^n$ with
  small bridges.
  Let $\delta_n$ be a small positive number.
  For each $i=1,2,\dots,k-1$, we remove open geodesic metric balls
  $U_{\delta_n}(q_i^n)$ and $U_{\delta_n}(p_{i+1}^n)$ of radius $\delta_n$
  centered at $q_i^n$ and $p_{i+1}^n$ from $S_i^n$ and $S_{i+1}^n$,
  and attach a copy of the Riemannian product space
  $S^{n-1}(\sin\delta_n) \times [\,0,\delta_n\,]$ to the boundaries,
  where $S^{n-1}(r)$ is an $(n-1)$-dimensional sphere of radius $r$
  in a Euclidean space.
  The boundary of $B_{\delta_n}(x)$, $x \in S_i^n$, is
  isometric to $S^{n-1}(\sin\delta_n)$ and the attaching is by
  an isometry.
  The resulting manifold, say $X_n$, is homeomorphic to a sphere
  and has a $C^0$ Riemannian metric.
  We deform the metric of $X_n$ to a smooth one, so that
  $X_n$ is a $C^\infty$ Riemannian manifold.
  Let $\hat X_n$ be the disjoint union of $S_1^n,\dots,S_k^n$.
  Take a sequence of positive numbers $C_n \to +\infty$.
  We may assume that $\delta_n$ is so small that
  the spectrum of the Laplacian of $X_n$ in $[\,0,C_n\,]$ is
  very close to the spectrum of the Laplacian of $\hat X_n$ in $[\,0,C_n\,]$.
  A precise proof of this follows from the same discussion as in
  \cite{Fuk}*{\S 9}.
  Since $\lambda_{k-1}(\hat X_n) = 0$ and $\lambda_k(\hat X_n) = n$,
  the sequence $\{X_n\}$ satisfies (1) and (2).
  It is easy to see that $X_n$ concentrates to a finite mm-space
  $Y = \{y_1,\dots,y_k\}$ such that $d_Y(y_i,y_j) = \pi|i-j|$
  and $\mu_Y(\{y_i\}) = 1/k$ for all $i,j = 1,2,\dots,k$.
  The sequence $\{X_n\}$ do not satisfy the conclusions of
  Theorem \ref{thm:lamk1} and Corollary \ref{cor:lamk-conc}.
  It has no lower bound of Ricci curvature.
\end{ex}

\begin{rem}
  The converse of Theorem \ref{thm:k-Levy-conc} also holds
  in the following sense.
  Let $X$ be an mm-space and $Y$ an mm-space consisting of
  $k$ points. For any numbers $\kappa_0,\kappa_1,\cdots,\kappa_{k}$ with
  $k\dconc(X,Y)< \min_i\kappa_i$, we have
  \begin{align} \label{eq:Sep-dconc}
    \Sep(X;\kappa_0,\kappa_1,\cdots,\kappa_k)\leq 2\dconc(X,Y).
  \end{align}
  In particular, if a sequence of mm-spaces concentrates to
  a finite mm-space consisting of $k$ points, then
  it is a $k$-L\'evy family.
\end{rem}

\begin{proof}[Proof of \eqref{eq:Sep-dconc}]
  We take Borel subsets
  $A_0,A_1,\cdots ,A_{k}\subset X$ such that
  $\mu_X(A_i)\geq \kappa_{i}$ for $i=0,1,\cdots,k$.
  Let $\alpha$ be an arbitrary positive number such that
  \begin{align*}
    \dconc(X,Y) < \alpha < \frac{1}{k}\min_i\kappa_i.
  \end{align*}
  There are two parameters $\varphi : I \to X$
  and $\psi : I \to Y$ such that
  \begin{align*}
    d_H(\varphi^*\Lip_1(X),\psi^*\Lip_1(Y))<\alpha.
  \end{align*}
  We define functions $f_i :X\to \R $
  by $f_i(x):=d_X(x,A_i)$, $x \in X$, $i=0,1,\cdots,k$.
  Since each $f_i$ is $1$-Lipschitz, there is $g_i \in \Lip_1(Y)$ such that
  $\dKF(f_i \circ \varphi , g_i \circ \psi) < \alpha$.
  This is equivalent to $\cL^1(B_i)> 1-\alpha$,
  where 
  \begin{align*}
    B_i := \{\; s \in I \mid |(f_i \circ \varphi)(s) - (g_i \circ
    \psi)(s)|<\alpha \;\}.
  \end{align*}
  For any $j=0,1,\cdots,k$, we have
  \begin{align*}
    \cL^1\Big(\varphi^{-1}(A_j)\cap \bigcap_{i=0}^{k-1}B_i
    \Big)\geq \mu_X(A_j)+ \cL^1 \Big(\bigcap_{i=0}^{k-1}B_i
    \Big) -1 \geq \kappa_j-k\alpha>0
  \end{align*}
  Take $s_j \in \varphi^{-1}(A_j) \cap \bigcap_{i=0}^{k-1}B_i$ for each $j$.
  Since $\psi(s_j) \in Y$, $j=0,1,\dots,k$,
  and $Y$ consists of $k$ points, it follows
  from the pigeonhole principle that
  $\psi(s_{i_1})=\psi(s_{i_2})$ for some $i_1$ and $i_2$
  with $i_1<i_2$. Since $i_1\leq k-1$ and $s_{i_1}\in
  \varphi^{-1}(A_{i_1})\cap B_{i_1}$, we have
  \begin{align*}
    |(g_{i_1} \circ \psi)(s_{i_2})|
    &= |(g_{i_1} \circ \psi)(s_{i_1})|\\
    &= |d_X(\varphi(s_{i_1}),A_{i_1})-(g_{i_1} \circ \psi)(s_{i_1})|
    < \alpha.
  \end{align*}
  Combining this with $s_{i_2}\in \varphi^{-1}(A_{i_2})\cap B_{i_1}$,
  we therefore obtain
  \begin{align*}
    d_X(A_{i_1},A_{i_2})\leq d_X(\varphi(s_{i_2}),
    A_{i_1})<(g_{i_1}\circ \psi)(s_{i_2}) + \alpha <2\alpha,
  \end{align*}
  which implies \eqref{eq:Sep-dconc}.
\end{proof}

\begin{rem}
  All the results in this section also hold for
  weighted Riemannian manifolds with Bakry-\'Emery Ricci curvature
  bounded from below, instead of Ricci curvature.
  The proofs are the same.
\end{rem}

\section{Concentration of Alexandrov spaces}
\label{sec:Alex}

In this section, we prove the stability of a lower bound
of Alexandrov curvature under concentration.
We also prove a version of Corollary \ref{cor:lamk-conc}
for Alexandrov spaces.

\begin{defn}[$\tilde\angle x_1x_0x_2$]
  \index{tildeangle@$\tilde\angle x_1x_0x_2$}
  Take a real number $\kappa$ and fix it.
  Let $X$ be a metric space and $M^2(\kappa)$ a complete simply connected
  two-dimensional space form of constant curvature $\kappa$.
  For different three points $x_0,x_1,x_2 \in X$,
  we denote by $\tilde\angle x_1x_0x_2$ 
  the angle between $\tilde x_0\tilde x_1$ and $\tilde x_0\tilde x_2$,
  where $\tilde x_0$, $\tilde x_1$, $\tilde x_2$ are three points
  in $M^2(\kappa)$ such that
  $d_X(x_i,x_j) = d_{M^2(\kappa)}(\tilde{x}_i,\tilde{x}_j)$ for $i,j=0,1,2$,
  and where $\tilde x_i\tilde x_j$
  is a minimal geodesic joining $\tilde x_i$ to $\tilde x_j$.
  $\tilde\angle x_1x_0x_2$ is defined
  only if the following condition is satisfied:
  \begin{align*}
    (*)
    \begin{cases}
      &\kappa \le 0,\\
      \text{or} & \kappa > 0
      \ \text{and}\ \per(\triangle x_0x_1x_2) \le 2\pi/\sqrt{\kappa},
    \end{cases}
  \end{align*}
  where $\per(\triangle x_0x_1x_2) := d_X(x_0,x_1) + d_X(x_1,x_2) + d_X(x_2,x_3)$.
  We call ($*$) the \emph{perimeter condition for $x_0$, $x_1$, $x_2$}.
  If the perimeter condition ($*$) is satisfied
  for any different three points in $X$,
  then we say that $X$ satisfies the \emph{perimeter condition for $\kappa$}.
  \index{perimeter condition}
\end{defn}

Under the perimeter condition for $X$ and $\kappa > 0$,
we see that, if $d_X(x_0,x_i) = \pi/\sqrt{\kappa}$ for $i=1$ or $2$, then
$\tilde\angle x_1x_0x_2$ is not unique and can be taken to be
any real number between $0$ and $\pi$.
If $d_X(x_0,x_i) < \pi/\sqrt{\kappa}$ for $i=1,2$, then
$\tilde\angle x_1x_0x_2$ is uniquely determined and depends only on
$\kappa$ and $d_X(x_i,x_j)$, $i,j=0,1,2$.

\begin{defn}[Alexandrov space]
  \index{Alexandrov space} \index{Alexandrov curvature}
  A metric space $X$ is said to be of
  \emph{Alexandrov curvature $\ge \kappa$}
  if $X$ satisfies the perimeter condition for $\kappa$
  and if for any different four points $x_0,x_1,x_2,x_3 \in X$ we have
  \[
  \tilde\angle x_1x_0x_2 + \tilde\angle x_2x_0x_3 + \tilde\angle x_3x_0x_1
  \le 2\pi.
  \]
  An \emph{Alexandrov space of curvature $\ge \kappa$} is, by definition,
  an intrinsic metric space of Alexandrov curvature $\ge \kappa$.
\end{defn}

This definition of Alexandrov space may be different from
those in the other literatures, but is equivalent to them.

For an Alexandrov space,
the Hausdorff dimension and covering dimension coincide to each other
and are called the \emph{dimension}.
Note that, for an Alexandrov space, the finiteness of dimension
implies the local compactness of the space, so that
a finite-dimensional Alexandrov space is a proper geodesic space.
We refer to \cites{BGP,Plaut} for the details for Alexandrov spaces.

\begin{lem}
  \label{lem:dn}
  Let $p_n : X_n \to Y$ be Borel measurable maps between
  mm-spaces $X_n$ and $Y$, $n=1,2,\dots$. We assume that
  each $p_n$ enforces $\varepsilon_n'$-concentration of $X_n$ to $Y$
  with $\varepsilon_n' \to 0$,
  and that $(p_n)_*\mu_{X_n}$ converges weakly to $\mu_Y$ as $n\to\infty$.
  Let $B\subset Y$ be an open subset and let
  \begin{align*}
    d_n(x) := d_{X_n}(x,p_n^{-1}(B) \cap \tilde{X}_n) \quad\text{and}\quad
    \ud_n(x) := d_Y(p_n(x),B)
  \end{align*}
  for $x \in X_n$, where $\tilde{X}_n$ is a non-exceptional domain of $p_n$
  for some additive error $\varepsilon_n \to 0$ as $n\to\infty$.
  Then, for any $\varepsilon > 0$ we have
  \begin{align}
    \tag{1}
    &\limsup_{n\to\infty}
    \mu_{X_n}(d_n \ge \ud_n + \diam B + \varepsilon) = 0,\\
    \tag{2}
    &\limsup_{n\to\infty} \mu_{X_n}(d_n \le \ud_n - \varepsilon) = 0.
  \end{align}
\end{lem}

\begin{proof}
  We prove (1).
  Suppose the contrary.  Then, replacing $\{n\}$ with a subsequence,
  we may assume that
  \begin{align} \label{eq:dn3}
    \mu_{X_n}(d_n \geq \ud_n + D + \varepsilon) \ge \alpha
  \end{align}
  for all $n$ and for a constant $\alpha > 0$, where $D := \diam B$.
  There are finitely many mutually disjoint Borel subsets
  $Y_1,\dots,Y_N \subset Y$ such that
  $\mu_Y((\bigcup_{i=1}^N Y_i)^\circ) > 1-\alpha/2$ and
  $\diam Y_i \le \varepsilon/2$ for any $i=1,\dots,N$.
  Since $(p_n)_*\mu_{X_n}$ converges weakly to $\mu_Y$, we have
  \begin{align} \label{eq:dn4}
    \mu_{X_n}\Bigl(p_n^{-1}\Bigl(\bigcup_{i=1}^N Y_i\Bigr)\Bigr) > 1-\alpha/2.
  \end{align}
  for all sufficiently large $n$.
  By setting
  \[
  A_{in} := \{\;x \in p_n^{-1}(Y_i) \mid
  d_n(x) \geq \ud_n(x) + D + \varepsilon\;\},
  \]
  \eqref{eq:dn3} and \eqref{eq:dn4} together imply
  \[
  \mu_{X_n}\Bigl( \bigcup_{i=1}^N A_{in} \Bigr) > \alpha/2.
  \]
  There is a number $i_0$ such that $\mu_{X_n}(A_{i_0 n}) > \alpha/(2N)$
  for infinitely many $n$.
  Replacing $\{n\}$ with a subsequence we assume that
  $\mu_{X_n}(A_{i_0 n}) > \alpha/(2N)$ for any $n$.  
  Letting $\kappa := \min\{\alpha /(2N),\mu_Y(B)/2\}$, we observe that
  \[
  \mu_{X_n}(A_{i_0 n}) > \kappa \quad\text{and}\quad
  \mu_{X_n}(p_n^{-1}(B) \cap \tilde{X}_n) > \kappa
  \]
  for all sufficiently large $n$.
  Applying Theorem \ref{thm:pn}(3), we obtain 
  \begin{align*}
    \limsup_{n\to\infty} d_{X_n}(A_{i_0 n}, p_n^{-1}(B) \cap \tilde{X}_n)
    &\le \limsup_{n\to\infty} d_+(p_n^{-1}(Y_{i_0}),p_n^{-1}(B);+\kappa)\\
    &\le d_Y(Y_{i_0},B) + \varepsilon/2 + D,
  \end{align*}
  so that there is a point $x_n\in A_{i_0 n}$ for each $n$ such that
  \[
  \limsup_{n\to\infty} d_n(x_n) \le d_Y(Y_{i_0},B) + \varepsilon/2 + D
  \le \inf_{p_n^{-1}(Y_{i_0})} \ud_n + \varepsilon/2 + D,
  \]
  which contradicts $x_n\in A_{i_0 n}$.
  (1) has been proved.

  We prove (2).
  Suppose that (2) does not hold.
  Then, replacing with a subsequence we have
  \[
  \mu_{X_n}(d_n \le \ud_n - \varepsilon) \ge \alpha
  \]
  for all $n$ and for some $\alpha,\varepsilon > 0$.
  Since $\mu_{X_n}(\tilde{X}_n) \to 1$ as $n\to\infty$,
  there is a point $x_n \in \tilde{X}_n$ for every sufficiently large $n$
  such that $d_n(x_n) \le \ud_n(x_n) -\varepsilon$.
  We find a point $x_n' \in p_n^{-1}(B) \cap \tilde{X}_n$
  such that $|d_n(x_n) - d_{X_n}(x_n,x_n')| \le \varepsilon_n$.
  Since $p_n$ is $1$-Lipschitz up to the additive error $\varepsilon_n$,
  we have
  \[
  \ud_n(x_n) \le d_Y(p_n(x_n),p_n(x_n')) \le d_{X_n}(x_n,x_n') + \varepsilon_n
  \le d_n(x_n) + 2\varepsilon_n,
  \]
  which is a contradiction.
  This completes the proof.
\end{proof}

\begin{thm} \label{thm:Alexcurv-conc}
  Let $X_n$, $n=1,2,\dots$, be mm-spaces of Alexandrov
  curvature $\ge \kappa$ for a real number $\kappa$.
  If $X_n$ concentrates to an mm-space $Y$ as $n\to\infty$,
  then $Y$ is of Alexandrov curvature $\ge \kappa$.
\end{thm}

\begin{proof}
  By Corollary \ref{cor:enforce},
  there are Borel measurable maps $p_n : X_n \to Y$, $n=1,2,\dots$,
  enforcing concentration of $X_n$ to $Y$ such that
  $(p_n)_*\mu_{X_n}$ converges weakly to $\mu_Y$ as $n\to\infty$.
  Applying Theorem \ref{thm:pn} yields that
  $p_n$ is $1$-Lipschitz up to some additive error $\varepsilon_n \to 0$,
  so that
  \begin{equation}
    \label{eq:Alex-curv}
    d_Y(p_n(x),p_n(x')) \le d_{X_n}(x,x') + \epsilon_n
  \end{equation}
  for any $x,x' \in \tilde X_n$,
  where $\tilde X_n$ is a non-exceptional domain of $p_n$ for $\varepsilon_n$.
  We take any different four points $y_0,y_1,y_2,y_3 \in Y$ and fix them.
  Let $o \in Y$ be a point different from $y_0,y_1,y_2,y_3$.
  By Proposition \ref{prop:pq}, we may assume that
  $p_n(x) = o$ for all $x \in X_n \setminus \tilde{X}_n$.
  Let $\delta$ be any number such that $0 < \delta < d_Y(o,y_i)$
  for $i=0,1,2,3$.
  We see that $p_n^{-1}(B_\delta(y_i)) \subset \tilde{X}_n$
  for any $i$ and $n$.
  Set, for $x \in X_n$,
  \[
  d_{in}(x) := d_{X_n}(x,p_n^{-1}(B_\delta(y_i)),\quad
  \ud_{in}(x) := d_Y(p_n(x),B_\delta(y_i)).
  \]
  Lemma \ref{lem:dn} implies that
  \[
  \limsup_{n\to\infty} \mu_{X_n}(|d_{in}-\ud_{in}| > 3\delta) = 0.
  \]
  Since $\liminf_{n\to\infty} \mu_{X_n}(p_n^{-1}(U_\delta(y_0)))
  \ge \mu_Y(U_\delta(y_0)) > 0$,
  there is a point $x_{0n} \in p_n^{-1}(U_\delta(y_0))$
  for every sufficiently large $n$ such that
  \[
  |\,d_{in}(x_{0n})-\ud_{in}(x_{0n})\,| \le 3\delta
  \]
  for $i=1,2,3$.
  Since $d_Y(p_n(x_{0n}),y_0) < \delta$ and by a triangle inequality,
  we have
  \[
  |\,\ud_{in}(x_{0n})-d_Y(y_0,y_i)\,| < 2\delta.
  \]
  From the definition of $d_{in}$,
  there is a point $x_{in} \in p_n^{-1}(B_\delta(y_i))$
  such that
  \[
  |\,d_{in}(x_{0n})-d_{X_n}(x_{0n},x_{in})\,| < \delta.
  \]
  Combining these inequalities, we obtain
  \[
  |\,d_{X_n}(x_{0n},x_{in})-d_Y(y_0,y_i)\,| < 6\delta
  \]
  for $i=1,2,3$.
  This holds for every $n$ large enough compared with $\delta$.
  Besides, we have $d_Y(p_n(x_{in}),y_i) \le \delta$,
  so that \eqref{eq:Alex-curv} and
  $x_{in} \in p_n^{-1}(B_\delta(y_i)) \subset \tilde{X_n}$
  together lead us to
  \[
  d_Y(y_i,y_j) - 2\delta
  \le d_Y(p_n(x_{in}),p_n(x_{jn}))
  \le d_{X_n}(x_{in},x_{jn}) + \epsilon_n
  \]
  for $i,j=1,2,3$.
  By the arbitrariness of $\delta$,
  we eventually obtain points $x_{in} \in X_n$ such that, for $i,j=1,2,3$,
  \begin{align}
    \label{eq:Alexcurv-conc1}
    \lim_{n\to\infty} d_{X_n}(x_{0n},x_{in}) &= d_Y(y_0,y_i),\\
    \label{eq:Alexcurv-conc2}
    \liminf_{n\to\infty} d_{X_n}(x_{in},x_{jn}) &\ge d_Y(y_i,y_j).
  \end{align}
  Since $y_i$ are arbitrary points in $Y$,
  formula \eqref{eq:Alexcurv-conc2}
  and the perimeter condition for $X_n$ and $\kappa$ together
  imply the perimeter condition for $Y$ and $\kappa$.
  We also obtain, from \eqref{eq:Alexcurv-conc1} and
  \eqref{eq:Alexcurv-conc2},
  \[
  \liminf_{n\to\infty} \tilde\angle x_{in}x_{0n}x_{jn}
  \ge \tilde\angle y_iy_0y_j
  \]
  for $i,j=1,2,3$.
  Since $X_n$ is of Alexandrov curvature $\ge \kappa$,
  we have
  \[
  \tilde\angle x_{1n}x_{0n}x_{2n} + \tilde\angle x_{2n}x_{0n}x_{3n} +
  \tilde\angle x_{3n}x_{0n}x_{1n} \le 2\pi,
  \]
  which together with the previous inequality yields
  \[
  \tilde\angle y_1y_0y_2 + \tilde\angle y_2y_0y_3 + \tilde\angle y_3y_0y_1
  \le 2\pi.
  \]
  This completes the proof.
\end{proof}

\begin{thm}[Petrunin, Zhang, and Zhu; \cites{Petrunin,Zhang-Zhu}]
  \label{thm:Alex-CD}
  Let $X$ be an $n$-dimensional Alexandrov space of curvature $\ge \kappa$
  for a constant $\kappa$.
  Then, $X$ satisfies $\CD((n-1)\kappa,\infty)$.
\end{thm}

Note that they in fact proved that $X$ satisfies $\CD((n-1)\kappa,n)$
that is a stronger condition than $\CD((n-1)\kappa,\infty)$.

Let $X_n$, $n=1,2,\dots$, be compact Alexandrov spaces of
curvature $\ge \kappa$ for a real number $\kappa$.
We equip each $X_n$ with the $(\dim X_n)$-dimensional Hausdorff measure
normalized as the total measure to be one.
Assume that $X_n$ has an upper bound of diameter.
Let us discuss concentration of $X_n$ as $n\to\infty$.
In the case where the dimension of $X_n$ is bounded from above,
it is well-known (see \cite{Yamaguchi}*{Proposition A.4})
that each $X_n$ satisfies the doubling condition
and the doubling constant is bounded from above, so that
$\{X_n\}$ is a uniform family (see Remark \ref{rem:precpt}).
It then follows from Corollary \ref{cor:box-dconc-precpt} and
Remark \ref{rem:box-mGH} that
$X_n$ concentrates to an mm-space $Y$ as $n\to\infty$
if and only if $X_n$ converges to $Y$ in the sense of
measured Gromov-Hausdorff convergence.

It is more significant to consider the case where the dimension of $X_n$
diverges to infinity as $n\to\infty$.
Combining Theorem \ref{thm:Alex-CD} and Corollary \ref{cor:CD-Sep}
yields the following.

\begin{cor}
  Let $X_n$, $n=1,2,\dots$, be compact finite-dimensional Alexandrov spaces of
  curvature $\ge \kappa$ for a constant $\kappa$.
  If $\kappa$ is positive
  and if the dimension of $X_n$ diverges to infinity,
  then $\{X_n\}$ is a L\'evy family.
\end{cor}


In the case where $X_n$ has nonnegative Alexandrov curvature,
we have the following theorem.

\begin{thm} \label{thm:Alex-conc}
  Let $X_n$, $n=1,2,\dots$, be compact finite-dimensional Alexandrov spaces
  of nonnegative curvature.
  If $X_n$ concentrates to an mm-space $Y$ as $n\to\infty$,
  then $Y$ is an Alexandrov space of nonnegative curvature.
\end{thm}

Note that $Y$ maybe infinite-dimensional.

\begin{proof}
  By Theorem \ref{thm:Alexcurv-conc},
  it suffices to prove that $Y$ is an intrinsic metric space.
  Theorem \ref{thm:Alex-CD} says that each $X_n$ satisfies $\CD(0,\infty)$.
  Applying Theorem \ref{thm:CD-conc} yields that $Y$ is an intrinsic metric
  space.
  This completes the proof.
\end{proof}

\begin{rem}
  If $X_n$ have a negative lower curvature bound,
  then the limit $Y$ is not necessarily an intrinsic metric space.
  It is easy to construct such an example.
  In fact, the manifold $(\,0,1\,) \times S^n$ with metric
  $dt + f(t)\theta_n$ concentrates to a disconnected space
  as $n\to\infty$ in general, where $\theta_n$ is the metric of
  the unit sphere in $\R^{n+1}$
  and $f : (\,0,1\,) \to \R$ is a function such that
  the completion of the Riemannian manifold is diffeomorphic to a sphere,
  i.e, $f(0+0) = f(1-0) = 0$, $f'(0+0) = f'(1-0) = 0$, etc.
  If $-f''/f$ is bounded below, then the sectional curvature
  of the manifold is bounded below.
  The concentration limit is the closure of the subset of $(\,0,1\,)$ where
  $f$ takes its maximum.
  That is not necessarily connected and is not an intrinsic metric space
  in general.
\end{rem}

We denote by $\lambda_k(X)$ the $k^{th}$ nonzero eigenvalue of the Laplacian
on a compact finite-dimensional Alexandrov space $X$,
where we refer to \cites{OS,KMS} for the Riemannian structure and
the Laplacian on an Alexandrov space.
We have the following proposition
in the same way as in the proof of Proposition \ref{prop:lamk-Sep}.

\begin{prop} \label{prop:lamk-sep}
  Let $X$ be a compact finite-dimensional Alexandrov space.
  Then we have
  \[
  \lambda_k(X) \Sep(X;\kappa_0,\dots,\kappa_k)^2
  \le \frac{4}{\min_{i=0,1,\dots,k}\kappa_i}
  \]
  for any $\kappa_0,\dots,\kappa_k > 0$.

  In particular, a sequence $\{X_n\}_{n=1}^\infty$ of compact
  finite-dimensional Alexandrov spaces is a $k$-L\'evy family
  if $\lambda_k(X_n)$ diverges to infinity as $n\to\infty$.
\end{prop}

\begin{cor} \label{cor:lamk-conc-alex}
  Let $X_n$, $n=1,2,\dots$, be compact finite-dimensional Alexandrov spaces
  of nonnegative curvature.
  If $\lambda_k(X_n)$ diverges to infinity as $n \to \infty$
  for some natural number $k$, then $\{X_n\}$ is a L\'evy family.
\end{cor}

\begin{proof}
  Proposition \ref{prop:lamk-sep} says that
  $\{X_n\}$ is a $k$-L\'evy family.
  By Theorems \ref{thm:k-Levy-conc} and \ref{thm:Alex-conc},
  $\{X_n\}$ is a L\'evy family.
  This completes the proof.
\end{proof}

\begin{thm} \label{thm:lamk1-alex}
  If $\{X_n\}_{n=1}^\infty$ is a L\'evy family of compact
  Alexandrov spaces of nonnegative curvature,
  then $\lambda_1(X_n)$ diverges to infinity as $n\to\infty$.
\end{thm}

\begin{proof}
  The theorem follows from Theorem \ref{thm:Alex-CD},
  \cite{Savare:self-imp}*{Corollary 4.3},
  and the proof of \cite{Ldx:conc-isop}*{Theorem 1}.
\end{proof}

Using Theorem \ref{thm:lamk1-alex} and Corollary \ref{cor:lamk-conc-alex},
we obtain the following theorem
in the same way as in Theorem \ref{thm:lamk1}.

\begin{thm}
  For any natural number $k$, there exists a positive constant
  $C_k$ depending only on $k$
  such that if $X$ is a compact Alexandrov space
  of nonnegative curvature,
  then we have
  \[
  \lambda_k(X) \le C_k \lambda_1(X).
  \]
\end{thm}






\printindex

\end{document}